\documentclass[11pt]{article}
\usepackage[a4paper, margin=0.75in]{geometry}
\usepackage{xcolor}

\usepackage{natbib}                 
\usepackage[colorlinks=true, linkcolor=black, citecolor=blue, urlcolor=black]{hyperref}
\usepackage{bookmark}

\usepackage{enumitem}
\usepackage{graphicx}
\usepackage{subcaption}
\usepackage{mathtools}
\usepackage{amsmath, amssymb, amsthm,amsfonts}
\usepackage{bbm}
\usepackage{stackengine}
\usepackage{calc}
\usepackage{circledtext}
\usepackage{float}
\usepackage{mathrsfs}
\usepackage{braket}
\usepackage{booktabs}
\usepackage{caption}
\usepackage{tikz}
\usepackage{threeparttable}
\usetikzlibrary{matrix, positioning}

\newcommand{\Var}{{\rm Var}}

\DeclareMathOperator{\Tr}{Tr}

\DeclareMathOperator{\KL}{KL}
\DeclareMathOperator{\TV}{TV}
\DeclareMathOperator{\Ber}{Ber}
\newcommand*\circled[1]{\tikz[baseline=(char.base)]{
            \node[shape=circle,draw,inner sep=1pt] (char) {#1};}}

\allowdisplaybreaks
\newenvironment{proofof}[1]{%
  \begin{proof}[Proof of Theorem~\ref{#1}]%
}{%
  \end{proof}%
}
\newcommand{\indep}{\perp \!\!\! \perp}

\DeclareMathOperator*{\argmin}{arg\,min}

\newtheorem{theorem}{Theorem}
\newtheorem{lemma}{Lemma}
\newtheorem{proposition}{Proposition}
\newtheorem{assumption}{Assumption}
\newtheorem{example}{Example}
\title{Semi-supervised linear regression with missing covariates}
\author{Benedict M. Risebrow and Thomas B. Berrett\\ Department of Statistics, University of Warwick}
\begin{document}

\maketitle

\begin{abstract}
    Missing values in datasets are common in applied statistics. For regression problems, theoretical work thus far has largely considered the issue of missing covariates as distinct from missing responses. However, in practice, many datasets have both forms of missingness. Motivated by this gap, we study linear regression with a labelled dataset containing missing covariates, potentially alongside an unlabelled dataset. We consider both structured (blockwise-missing) and unstructured missingness patterns, along with sparse and non-sparse regression parameters. For the non-sparse case, we provide an estimator based on imputing the missing data combined with a reweighting step. For the high-dimensional sparse case, we use a modified version of the Dantzig selector. We provide non-asymptotic upper bounds on the risk of both procedures. These are matched by several new minimax lower bounds, demonstrating the rate optimality of our estimators. Notably, even when the linear model is well-specified, our results characterise substantial differences in the minimax rates when unlabelled data is present relative to the fully supervised setting. Particular consequences of our sparse and non-sparse results include the first matching upper and lower bounds on the minimax rate for the supervised setting when either unstructured or structured missingness is present. Our theory is coupled with extensive simulations and a semi-synthetic application to the California housing dataset.
\end{abstract}

\section{Introduction}\label{sec:Introduction}

\subsection{Motivation}\label{sec:motivations}
Linear regression is one of the most extensively studied problems in statistics, with a large number of proposed procedures for estimating the parameters from complete observations. In practice, a common feature of datasets is that some or all of the samples have missing values for variables of interest~\citep{tsiatis2006semiparametric}. Missingness in the response or the covariates renders most existing methods inapplicable or greatly reduces their efficacy. It is therefore of both practical and theoretical interest to develop methods that perform well in the presence of missing data.
Classical causes of missing data include non-response and technical malfunctions~\citep{bower2017addressing}. For example, in surveys, subjects may choose not to answer some of the questions~\citep{brick1996handling}. Typically, these sources of missingness lead to \textit{unstructured} or \textit{sporadic} patterns in the observed dataset, where the occurrence of missing values across different variables is largely uncorrelated.

A more recent phenomenon arises from attempts to merge existing datasets with different variables recorded in each. This leads to a \textit{blockwise-missing} or \textit{structured} pattern. A common cause of this is multimodal data, with a motivating example being the ADNI dataset~\citep{mueller2005alzheimer}, which aims to understand Alzheimer's disease. In this collection of studies, variables are organised by their measurement source, such as magnetic resonance imaging and positron emission tomography. Some of the component studies only recorded variables from a subset of the sources, leading to a blockwise-missing setup upon merging. What distinguishes this setting from that of the previous paragraph is the strong correlation between the occurrence of missing values across different variables. Such blockwise-missing structures are becoming increasingly common in healthcare data. The MIMIC datasets~\citep{gehrmann2018comparing} aggregate data across a hospital, with subsets of clinical tests being ordered for patients based on their symptoms. In electronic health record data, privacy policies differ by institution~\citep{keshta2021security}, leading to structured missingness when the available data is combined. A non-medical example is the Forest Inventory and Analysis database~\citep{tinkham2018applications}. This records variables on various plots of land in the United States, with a subset of expensive-to-collect variables being recorded for only a fraction of the plots. These blockwise-missing settings present new challenges relative to the unstructured setting~\citep{mitra2023learning}. For example, imputation schemes designed for sporadic missingness often perform poorly. 

Often, in datasets suitable for regression, recording the response variable can be costly or time-consuming~\citep{jiao2024learning}. This occurs in clinical settings when determining disease status requires expert judgement~\citep{gehrmann2018comparing}. Many---sometimes the vast majority---of the observations in such a dataset do not have their response variable recorded, leading to a \textit{semi-supervised} dataset~\citep{chapelle2006semisupervised}.

In the existing literature, most works on regression problems consider either missing responses or missing covariates separately, without exploring how unlabelled data can be leveraged to improve estimation when the labelled dataset contains missing covariates. Our goal is to introduce powerful methodology for linear regression and to characterise the fundamental statistical limits of estimation in this setting from a minimax perspective. We seek to answer the following questions:
\begin{itemize}
    \item How can unlabelled data be used when labelled samples have missing covariates?
    \item How does the missingness pattern within the covariates affect the performance of an optimal procedure?
    \item Do procedures exist that are optimal for both structured and unstructured missingness?
\end{itemize}

\subsection{Formal setting}\label{sec:Formal Setting}
We begin by defining the distribution of $(X,Y)\in \mathbb{R}^p\times\mathbb{R}$ at the population level. We denote the marginal distribution on $\mathbb{R}^p$ of our covariates $X$ by $P_{X}$. We assume that the covariates have second moment matrix $\Sigma$, i.e., $\mathbb{E}\left[XX^T\right]=\Sigma\in\mathbb{R}^{p\times p}$. To simplify the presentation, we will often refer to this as the covariance matrix. Our response $Y$ is distributed conditionally on the covariates $X$ according to the following well-specified linear model:
\begin{align}
\label{eq:datamech1}
    Y= X^T\beta^{*} + \epsilon, & \text{  where  } \mathbb{E}\left[\epsilon|X\right] =0 \text{ and }\Var\left[\epsilon|X\right] = \sigma^2.
\end{align}
Here $\beta^* \in \mathbb{R}^p$ and $\epsilon$ is the noise random variable. We denote by $P_{XY}$ the joint distribution of $(X,Y)$ when generated as above. For $a \in \mathbb{N}$, denote by $[a]$ the set $\{1,\ldots,a\}$ and, for $O\subseteq [p]$, define $P_{X_{O}Y}$ to be the marginal distribution of $(X_{O},Y)$ when $(X,Y)$ is distributed according to $P_{XY}$.

We assume that we have $K$ missingness patterns, where in the $k^{th}$ pattern we observe only the covariates indexed by $O_{k} \subseteq [p]$ and $Y$. We denote the corresponding missing variables by $M_{k}=[p]\setminus O_{k}$. Our labelled dataset is denoted by $\mathcal{L}=\{\mathcal{D}_{k}\}_{k=1}^K$, where $\mathcal{D}_{k}$ is the data in the $k^{th}$ missingness pattern. The $\{\mathcal{D}_{k}\}_{k=1}^K$ are assumed to be jointly independent, and each $\mathcal{D}_{k}$ consists of $n_{k}$ copies of $(X_{O_{k}},Y)$ drawn i.i.d. from the distribution $P_{X_{O_{k}}Y}$. The latter assumption is equivalent to the MCAR hypothesis~\citep{little2019statistical}, whereby the missingness is assumed to be independent of the data. We denote the size of the labelled dataset by $n_{\mathcal{L}}=\sum_{k=1}^Kn_{k}$. We assume throughout that each variable is observed at least once, i.e., $\cup_{j=1}^K O_{j} = [p]$. 

Alongside this labelled dataset $\mathcal{L}$, we have an unlabelled dataset $\mathcal{U}$. In the \textit{ideal semi-supervised} (ISS) setting, we know the covariance $\Sigma$ of the covariates, so $\mathcal{U}$ consists of this matrix. This is a relaxation of the standard definition of ISS, where we would assume knowledge of $P_{X}$. We say we are in the \textit{ordinary semi-supervised} (OSS) setting if we have an independent unlabelled dataset of size $N$ drawn i.i.d.~from $P_{X}$; in this setting $\mathcal{U} = \{X_{i}\}_{i=n_{\mathcal{L}}+1}^{n_{\mathcal{L}}+N}$. Again, drawing the unlabelled data from $P_{X}$ is equivalent to the MCAR hypothesis. For simplicity, we assume that the unlabelled samples have all covariates observed, but our estimators and rates can be modified in a straightforward way to allow missingness in the unlabelled data. It is rare in practice to know $\Sigma$, so the ISS case is primarily introduced to provide theoretical insight into the OSS case. We will also be interested in the supervised case where $N=0$. 

Given the data $\mathcal{L}\cup\mathcal{U}$, our goal is to estimate $\beta^*$. We work in a minimax setting, where we define the minimax risk over the class of data-generating mechanisms $\mathcal{P}$ and potential estimators $\hat{\Theta}$ via
\begin{equation}\label{eq:minimax risk def}
        \mathcal{M} = \inf_{\hat{\beta}\in\hat{\Theta}} \sup_{P \in \mathcal{P}} \mathbb{E} \left[ \|\hat{\beta} - \beta^*\|_{2}^2\right],
\end{equation}
 where $\mathcal{P}$ and $\hat{\Theta}$ are to be specified below. Our goal is to derive estimators whose worst-case performance matches \eqref{eq:minimax risk def} up to constant factors. 
 
 We can equivalently list our observations as $\{((X_{i})_{O_{\Xi(i)}},Y_{i})\}_{i=1}^{n_{\mathcal{L}}}$, where $\Xi : [n_{\mathcal{L}}] \rightarrow [K]$ and $\Xi(i)$ is the missingness pattern of the $i^{th}$ observation. We define for $k \in [K]$, $\mathcal{I}_{k}=\{i\in[n_{\mathcal{L}}]:((X_{i})_{O_{\Xi(i)}},Y_{i})\in\mathcal{D}_{k}\}$. These are the indices of the observations in $\mathcal{D}_{k}$. Further define $\mathcal{I}_{\mathcal{L}}=\cup_{k \in [K]} \mathcal{I}_k$ as the indices of the labelled data and $\mathcal{I}_{\mathcal{U}}$ as the indices of the unlabelled data.
 We now introduce a convenient reparametrisation of the missingness pattern as follows. We assume that the covariate indices are partitioned as $
    [p] = \cup_{i=1}^L L_{i},$
where each observation pattern $O_{j}$ is the union of some of the $L_{i}$. There are multiple ways of doing this, and we choose the unique way that minimises $L$. Here, $\{L_{i}\}_{i=1}^L$ are the disjoint \textit{modalities}  and $L$ denotes the number of modalities. For example, in the ADNI dataset mentioned in Section~\ref{sec:motivations} the modalities are the different sources of data, whereas in an unstructured setting we might have a modality for each variable. We further define modality and cross-modality sample sizes as follows:
\begin{align*}
    h_{i} = \sum_{j=1}^{n_{\mathcal{L}}}\mathbbm{1}_{\left\{L_{i} \subseteq O_{\Xi(j)}\right\}} \text{ for } i \in [L], \quad \quad \quad n_{g_{1},g_{2}} = \sum_{i=1}^{n_{\mathcal{L}}}\mathbbm{1}_{\{L_{g_{1}}\cup L_{g_{2}} \subseteq O_{\Xi(i)}\}} \text{ for } g_{1}, g_{2} \in [L].
\end{align*}

 We now give two important, contrasting examples of this setting that we will return to throughout the paper. 
  \begin{example}\label{examp:unstructured pattern}
     For the \textbf{unstructured pattern}, we assume there exists $\rho \in (0,1)$ and $C_{\rho}>0$ such that for all distinct $(g,h) \in [L]^2$, $\rho n_{\mathcal{L}} \leq h_{g} \leq C_{\rho}\rho n_{\mathcal{L}}$ and $\rho^2n_{\mathcal{L}}\leq n_{g,h} \leq C_{\rho}\rho^2n_{\mathcal{L}}$. This incorporates unstructured missingness, which could arise when each variable is missing with fixed probability independently of everything else. \cite{wang2019rate} and \cite{loh2011high} study this setting in the context of high-dimensional linear regression with fully labelled data.
 \end{example}
 The above example, although primarily describing unstructured missingness, includes certain structured patterns. For example, let $L=K=3$, $O_{1}=L_{1}\cup L_{2}$, $ O_{2}=L_{1}\cup L_{3}$, $O_{3}=L_{2}\cup L_{3}$ and $n_{1}=n_{2}=n_{3}$. This pattern satisfies the above condition with $\rho = \frac{\sqrt{3}}{3}$ and $C_{\rho}=\frac{2\sqrt{3}}{3}$.
 To contrast Example \ref{examp:unstructured pattern}, we give the following structured example.
 \begin{example}\label{examp:simple monotonic pattern}
     The \textbf{simple monotonic pattern} is defined to have two missingness patterns with $O_{1}=[p]$ and $O_{2}=[p-p_{0}]$, i.e., $n_{1}$ of our labelled samples have all variables observed and $n_{2}$ of our labelled samples have the final $p_{0}$ variables missing. Here, we choose $L_{1}=[p-p_{0}]$ and $L_{2}=[p]\setminus[p-p_{0}]$. In a clinical setting, the $p_{0}$ variables could be expensive to collect and so are missing for a large number of labelled samples. This is the primary setting of~\citet{robins1994estimation}, which does not leverage unlabelled data.
 \end{example}

  Both of these patterns describe the labelled portion of the data and will be considered alongside unlabelled data or knowledge of the covariance matrix $\Sigma$. 

\subsection{Existing work}\label{sec: Existing work}

Classical approaches to missing data are summarised in~\citet{little2019statistical}. Two such techniques are imputation~\citep{van2011mice} and the expectation-maximisation (EM) algorithm~\citep{dempster1977maximum}. Imputation---the process of filling in missing values---has received recent methodological attention, with~\citet{naf2024good} and~\citet{shadbahr2023impact} providing new ways of ranking different imputation methods. Nonetheless, there are few theoretical guarantees regarding the effect of imputation on the downstream statistical task. The recent work of~\citet{tan2025integrated} proposes a matrix completion-based imputation approach for dealing with both structured and unstructured missingness. For the EM algorithm,~\citet{balakrishnan2017statistical} provide guarantees for several procedures in this framework.

In the past decade, there has been a resurgence in the development of statistical methodology with missing data, often focussing on high-dimensional and non-parametric settings. Theory now exists for tasks including covariance matrix estimation~\citep{lounici2014high}, high-dimensional principal component analysis~\citep{zhu2022high}, high-dimensional linear classification~\citep{tony2019high}, non-parametric classification~\citep{sell2024nonparametric}, functional estimation~\citep{berrett2024efficient}, change-point detection~\citep{xie2012change} and the estimation of U-statistics~\citep{kim2024semi}.

 Blockwise missingness, described in Section \ref{sec:motivations}, has recently been the focus of much research~\citep{mitra2023learning}. Recent tasks tackled in this setting include: nearest neighbour classification~\citep{yu2022integrative}; failure time modelling with censored outcomes~\citep{wang2022regularized}; combining AI models to estimate parameters~\citep{xu2025blockwise}; extending the framework of prediction powered inference under blockwise missingness~\citep{zhao2025imputation}; performing multi-task learning~\citep{sui2025multi} and variable selection in longitudinal studies~\citep{ouyang2024imputation}.

Moving our focus to linear regression, several works have addressed this problem, primarily focussing on the MCAR case. Within a purely semi-supervised setup (where covariates are fully observed),~\citet{azriel2022semi} and~\citet{chakrabortty2018efficient} develop efficient low-dimensional estimators. Notably, their results show that the linear model must be misspecified in order for unlabelled data to aid in estimating $\beta^*$.~\citet{deng2024optimal} derive a rate-optimal estimator for the high-dimensional sparse case, illustrating that unlabelled data can improve the rate of convergence for particular classes of misspecified linear models. \citet{chen2025semi} further demonstrate that, in high dimensions, using unlabelled data can improve the robustness and efficiency of an estimator even when the model is well-specified. Of particular interest in the semi-supervised setting is the concept of \textit{safety}. Broadly speaking, a semi-supervised estimator is deemed safe if it always performs at least as well as the best supervised estimator. \citet{chakrabortty2018efficient} and~\citet{deng2024optimal} discuss adaptations of their method to ensure safety. 

Focussing on missingness exclusively in the covariates, recent work has studied unstructured missingness patterns for the high-dimensional sparse MCAR case. For the unstructured pattern (Example \ref{examp:unstructured pattern}),~\citet{loh2011high} propose a non-convex procedure via a LASSO-style estimator, obtaining an upper bound on the minimax risk of order $\frac{s\log(p)}{\rho^4n_{\mathcal{L}}}$\footnote{We ignore the dependence of the rate on the other parameters here for simplicity. Here $\|\beta^*\|_{0} \leq s$.}. This rate was improved to $\frac{s\log(p)}{\rho^2n_{\mathcal{L}}}$ in~\citet{wang2019rate} and to $\frac{s\log(p)}{\rho n_{\mathcal{L}}}$ when the covariance is assumed to be known.~\citet{chandrasekher2020imputation} achieve a similar rate under Gaussian assumptions via imputation, assuming that $\mathbb{E}\left[X_{O}|X_{O^c}\right]$ is known for each $O\subseteq[p]$. This unstructured assumption is rather limiting and excludes the blockwise-missing patterns described in Section \ref{sec:motivations}. For a general restricted moment model that includes linear regression,~\citet{robins1994estimation} derive semiparametrically efficient estimators of parameters for the simple monotonic pattern under the MAR assumption. From a minimax perspective,~\citet{yu2020optimal} consider high-dimensional sparse MCAR estimation in a general blockwise-missing setting, with~\citet{diakite2025adapdiscom} extending this estimator to account for measurement error. An alternative approach via imputation is offered in~\citet{xue2021integrating} when the MAR assumption holds.

A natural extension of this line of work is to explore what happens when we have both missingness in the covariates alongside unlabelled data, as occurs in many of the examples described in Section \ref{sec:motivations}. Indeed, the upper bounds in~\citet{wang2019rate} mentioned in the previous paragraph differ between the ISS and supervised cases, suggesting a difference between these settings. Several very recent works have explored the effects of having missingness in both the covariates and the response.~\citet{song2024semi} consider high-dimensional linear regression with the goal of performing inference. Their results require a hard sparsity assumption on $\beta_{O}^*+\Sigma_{OO}^{-1}\Sigma_{OO^c}\beta^*_{O^c}$ for each observation pattern $O\subseteq [p]$. A similar approach for GLMs is given in~\citet{wang2025distributed} under similar assumptions. Also in a GLM setting,~\citet{li2024adaptive} provide asymptotic results for an estimator that builds upon~\citet{robins1994estimation}, which yields an improvement in efficiency when unlabelled data is present. Overall, the results of~\citet{li2024adaptive} and~\citet{wang2019rate} suggest that incorporating unlabelled data can be beneficial when missing covariates are present in the labelled data.

\subsection{Contributions and outline}\label{sec:Our contributions}

We work in the setting described in Section \ref{sec:Formal Setting} with labelled data containing missing covariates and unlabelled data. Our contributions are as follows: 
\begin{itemize}
    \item For the low-dimensional OSS problem, we introduce two very similar estimators for the blockwise-missing and unstructured settings; in both settings, we establish upper bounds on the risk and provide matching lower bounds. Notably, for the simple monotonic pattern (Example \ref{examp:simple monotonic pattern}), the unlabelled data allows for a reduction in the effective dimensionality which is not available to procedures that do not use unlabelled data. For the unstructured pattern (Example \ref{examp:unstructured pattern}), the use of unlabelled data allows for an increase in the effective observation rate from $\rho$ to $\rho^{1/2}$.
    \item The OSS estimator extends naturally to the supervised case via two-fold cross-fitting. The rate optimality of our approach is also demonstrated via lower bounds.
    \item For the high-dimensional unstructured supervised setting and its OSS counterpart, we provide a lower bound that resolves a conjecture of~\citet{wang2019rate} as a special case. We additionally extend the estimator of~\citet{wang2019rate} to the OSS setting, providing an upper bound that matches up to constant factors.
    \item For the high-dimensional blockwise-missing supervised and OSS settings, we provide an additional analysis of the above estimator and establish its rate optimality. This highlights missingness patterns for which the upper bound of~\citet{yu2020optimal} can be improved. 
\end{itemize}
The rest of the paper is organised as follows. In Section \ref{sec:Low-dimensional results}, we introduce our low-dimensional estimator in the OSS and supervised settings, providing upper and lower bounds in each setting. The proofs for this section are in Appendix \ref{sec:Low Dimensional Proofs}. In Section \ref{sec:High-dimensional results}, we introduce our high-dimensional estimator and provide upper and lower bounds in each setting. The proofs for this section are in Appendix \ref{sec:High-Dimensional Proofs}. We test the empirical performance of our estimators in Section \ref{sec:Simulations} via a simulation study. Finally, in Section \ref{sec:Real World Dataset Application}, we apply our methodology to the California housing dataset with synthetic missingness. Additional details for this section are given in Appendix \ref{Appendix:Real World Dataset Application}.

\begin{table}[H]
\centering
\caption{We summarise the minimax rates proven in this paper, without listing all the technical conditions required. In this table, we assume $\|\beta^*\|_{2}^2\leq B^2, \sigma^2 \leq R^2$ and $B^2\leq R^2$, though this latter assumption is not required for the low-dimensional structured OSS setting. In the high-dimensional setting, $\|\beta^*\|_{0}\leq s$. The unstructured missingness setting is described in Example \ref{examp:unstructured pattern}, whereas the structured case refers to the setting in which we treat $L$, the number of modalities, as constant.}
\label{tab:rates-2x2}

\renewcommand{\arraystretch}{1.3}
\resizebox{\linewidth}{!}{%
\begin{tabular}{@{}lcc@{}}
\toprule
& \textbf{Unstructured missingness} & \textbf{Structured missingness} \\
\midrule
\textbf{Low-dimensional} & $\frac{R^2p}{\rho n_{\mathcal{L}}}+\frac{B^2p}{\rho^2n_{\mathcal{L}}+N}$ & $R^2\max_{g\in[L]}\left\{\frac{|L_{g}|}{h_{g}}\right\}+B^2\max_{g,h \in [L]}\left\{\frac{|L_{g}|}{n_{g,h}+N}\right\}$ \\
\textbf{High-dimensional} & $\frac{R^2s\log(p)}{\rho n_{\mathcal{L}}}+\frac{B^2s\log(p)}{\rho^2n_{\mathcal{L}}+N}$ & $R^2s\max_{g\in[L]}\left\{\frac{\log(|L_{g}|)}{h_{g}}\right\}+B^2s\max_{g,h \in [L]}\left\{\frac{\log(|L_{g}|)}{n_{g,h}+N}\right\}$ \\
\bottomrule
\end{tabular}%
}

\end{table}

\subsection{Notation}\label{sec:notation}
 For $x \in \mathbb{R}^p$ and $q \in (0,\infty)$, let $\|x\|_{q} = \left(\sum_{i=1}^p|x_{i}|^q\right)^{\frac{1}{q}}$, $\|x\|_{\infty} = \max\left\{|x_{i}|:i \in [p]\right\}$ and $\|x\|_{0} = \sum_{i=1}^p \mathbbm{1}_{\left\{x_{i} \neq 0\right\}}$. Denote by $e_{i} \in \mathbb{R}^p$ the $i^{th}$ standard basis vector. Let $\mathcal{S}^{p-1}=\{x\in\mathbb{R}^p:\|x\|_{2}=1\}$. For a finite set $S$, denote by $|S|$ the number of elements in $S$. For $A \in \mathbb{R}^{m\times p}$, denote by $\|A\|_{\infty} = \|\mathrm{Vec}\left(A\right)\|_{\infty}$ and $\|A\|_\mathrm{op} = \sup_{x \in \mathcal{S}^{p-1}}\|Ax\|_{2}$.  When $m=p$ and $A$ is symmetric, define $\lambda_{\min}(A)$ and $\lambda_{\max}(A)$ to be the smallest and largest eigenvalues of $A$ respectively. For two matrices $A,B \in \mathbb{R}^{p\times p}$, we write $A \preceq B$ if $x^TAx \leq x^TBx$ $\forall x \in \mathbb{R}^p$. For sets $A\subseteq [p]$, $B\subseteq [q]$ and $\Omega \in \mathbb{R}^{p\times q}$, denote by $\Omega_{AB}$ the $|A|\times|B|$ submatrix of $\Omega$ defined by extracting from $\Omega$ first the rows indexed by $A$ and then the columns indexed by $B$. Further, define $\Omega_{A} \in \mathbb{R}^{|A| \times p}$ as the rows of $\Omega$ indexed by $A$. For two functions $f,g: I \rightarrow \mathbb{R}$, with $I$ an arbitrary index set, we write
$f\lesssim g$ if there exists a universal constant $C>0$ such that $f(i)\leq Cg(i)$ for all $i \in I$ and $f\gtrsim g$ if there exists a universal constant $c>0$ such that $f(i)\geq cg(i)$ for all $i \in I$. If both $f\lesssim g$ and $f\gtrsim g$, we write $f \asymp g$. Similarly, for $p \in \mathbb{N}$ and $A,B:I \rightarrow \mathbb{R}^{p\times p}$ for some index set $I$, we write $A \precsim B$ if there exists a universal constant $c>0$ such that $A(i)\preceq cB(i)$ for all $i \in I$ and $A \succsim B$ if there exists a universal constant $c>0$ such that $A(i)\succeq cB(i)$ for all $i \in I$. For parameters of our model $a_{1},a_{2},\ldots, a_{k}$, we often write $F=F(a_{1},a_{2},\ldots,a_{k})$ for some constant $F$ depending only on $a_{1},a_{2},\ldots,a_{k}$. When such constants appear, they may vary in their dependence on $a_{1},a_{2},\ldots,a_{k}$ from line to line; where compact notation is needed, the dependence on these parameters may be suppressed.  For a topological space $F$, denote by $\mathcal{B}(F)$ the Borel $\sigma$-algebra on $F$. For a vector-valued random variable $X$, let $\|X\|_{\psi_{2}}$ and $\|X\|_{\psi_{1}}$ denote the sub-Gaussian and sub-exponential norms, respectively, of $X$ in the sense of~\citet{vershynin2025HDP}. 

\section{Low-dimensional results}\label{sec:Low-dimensional results}
 
To motivate our estimator, consider the simple monotonic pattern (Example \ref{examp:simple monotonic pattern}) in the ISS setting with $X \sim N(0,\Sigma)$, $\epsilon \sim N(0,\sigma^2)$ and $\epsilon \indep X$. One natural way of estimating $\beta^*$ would be alongside $\sigma$ via maximum likelihood estimation, i.e.,~by finding $(\hat{\beta},\hat{\sigma})$ that solve
\begin{equation} \label{eq:MLE}
    \min_{\substack{\beta \in \mathbb{R}^p,\\ \sigma >0}}\left\{\sum_{i\in \mathcal{I}_{1}} \frac{(Y_{i}-X_{i}^T\beta)^2}{\sigma^2}+\sum_{i\in \mathcal{I}_{2}}\frac{ (Y_{i}-(X_{i})_{O_{2}}^T\Sigma_{O_{2}O_{2}}^{-1}\Sigma_{O_{2}}\beta)^2}{\sigma^2+\beta_{M_{2}}^TS_{M_{2}}\beta_{M_{2}}}+\log(\sigma^{2n_{1}}(\sigma^2+\beta_{M_{2}}^TS_{M_{2}}\beta_{M_{2}})^{n_{2}})\right\},
\end{equation}
where $S_{M_{2}} = \Sigma_{M_{2}M_{2}}-\Sigma_{M_{2}O_{2}}\Sigma_{O_{2}O_{2}}^{-1}\Sigma_{O_{2}M_{2}}$. This is the style of approach pursued in~\citet{balakrishnan2017statistical} via the EM algorithm. However, \eqref{eq:MLE} is non-convex and theoretical guarantees rely on the Gaussian assumptions. In the current work, we instead construct the convex relaxation of \eqref{eq:MLE},
\begin{equation*} 
    \hat{\beta} \in \argmin_{\beta \in \mathbb{R}^p}\left\{\sum_{i\in \mathcal{I}_{1}} (Y_{i}-X_{i}^T\beta)^2+D\sum_{i\in \mathcal{I}_{2}} (Y_{i}-(X_{i})_{O_{2}}^T\Sigma_{O_{2}O_{2}}^{-1}\Sigma_{O_{2}}\beta)^2\right\},
\end{equation*}
where $D\in (0,1)$ is some constant to be specified later. Equivalently, we define $\hat{\beta}$ to solve 
\begin{equation*} 
    R(\hat{\beta},D) \equiv \sum_{i\in \mathcal{I}_{1}} X_{i}(Y_{i}-X_{i}^T\hat{\beta})+D\sum_{i\in \mathcal{I}_{2}} \Sigma_{O_{2}}^T\Sigma_{O_{2}O_{2}}^{-1}(X_{i})_{O_{2}}\{Y_{i}-(X_{i})_{O_{2}}^T\Sigma_{O_{2}O_{2}}^{-1}\Sigma_{O_{2}}\hat{\beta} \}=0.
\end{equation*}
Note that we have $\mathbb{E}[ R(\beta^*,D)]=0$ for all $D\in \mathbb{R}$ and for all covariate and noise distributions, i.e.~even without Gaussian assumptions. Thus, we can expect this convex relaxation to work in wide generality.

Returning to the general OSS case, define $P_{j} = \Sigma_{O_{j}O_{j}}^{-1}\Sigma_{O_{j}} \in \mathbb{R}^{|O_{j}|\times p}$ for $j \in [K]$. Using some positive-definite estimate of the covariance $\hat{\Sigma}$, we estimate $P_{j}$ via $\hat{P}_{j}\equiv\hat{\Sigma}_{O_{j}O_{j}}^{-1}\hat{\Sigma}_{O_{j}}$ for $j \in [K]$. For a general missingness pattern, our estimator $\hat{\beta}$ of $\beta^*$ is defined by 
\begin{equation}\label{eq:OSS estimator definition initial}
    \hat{\beta} = \argmin_{\beta \in \mathbb{R}^p}\left\{ \sum_{j=1}^K\sum_{i \in \mathcal{I}_{j}}\hat{D}_{j}(Y_{i}-(X_{i})_{O_{j}}^T\hat{P}_{j}\beta)^2\right\},
\end{equation}
where $\hat{D}_{j} \in (0,1)$ are weights to be specified later. This amounts to replacing, for $O\subset [p]$, $(X_{O},\text{NA})$ by $(X_{O},\hat{\Sigma}_{MO}\hat{\Sigma}_{OO}^{-1}X_{O})$ and performing a weighted least squares procedure on this imputed data. As we see in the simulations of Section \ref{sec:Simulations}, the weighting step is important for good performance and this estimator outperforms the naive imputation approach of setting $\hat{D}_{j}=1$.

To analyse the performance of this estimator, we work under natural assumptions from the literature. 
\begin{assumption}\label{assump:small-ball}
For $X \sim P_{X}$, there exist constants $C\geq 1$ and $\chi \in (0,1]$ such that for every $\theta \in \mathbb{R}^p\setminus\left\{0\right\}$ and $t > 0$,
\begin{equation*}
    \mathbb{P}\left(|X^T\theta| \leq t\|\theta\|_{\Sigma}\right) \leq (Ct)^{\chi},
\end{equation*}
where $\|\theta\|_{A}^2 = \theta^TA\theta$ for $A \in \mathbb{R}^{p\times p}$.
\end{assumption}
This small ball condition was introduced in~\citet{mourtada2022exact} for analysing the exact risk of the least squares estimator and has since been applied more broadly~\citep{fukuchi2023demographic,mourtada2022improper}. This is particularly useful for the structured case. It states that the distribution of the covariates does not put too much mass near any hyperplane and, in particular, implies that for any $O \subseteq [p]$,
\begin{equation*}
    \mathbb{P}\left(|X_{O}^T\theta_{O}| \leq t\|\theta_{O}\|_{\Sigma_{OO}}\right) \leq (Ct)^{\chi}.
\end{equation*}

Next, we make a sub-Gaussian assumption on the covariates that is widespread in the statistics literature.
\begin{assumption} \label{assump:subgauss}
    For some $C_{X}>0$, the distribution of the covariates satisfies $\|X\|_{\psi_{2}}\leq C_{X}$.
\end{assumption}

Further, for example as in~\citet{pensia2024robust}, we require upper and lower bounds on the eigenvalues of $\Sigma$. 
\begin{assumption}\label{assump:eigen}
    There exist constants $\lambda_{-}$ and $\lambda_{+}$ such that $ 0 < \lambda_{-} \leq \lambda_{\min}(\Sigma) \leq \lambda_{\max}(\Sigma) \leq \lambda_{+} < \infty $.
\end{assumption}

\subsection{Structured Missingness}\label{sec:LD structured missingness}
In this section we consider the blockwise-missing setting where we treat $K$, the number of missingness patterns, as fixed. This is equivalent to treating $L$, the number of modalities, as fixed.
\subsubsection{Upper bound}\label{sec:Upper bound for our estimator with known covariance}
For $k \in [K]$ write $D_{k}^*=\frac{\sigma^2}{\sigma^2+(\beta^*_{M_{k}})^TS_{M_{k}}(\beta^*_{M_{k}})}$, where we recall that $S_{M_{k}} = \Sigma_{M_{k}M_{k}}-\Sigma_{M_{k}O_{k}}\Sigma_{O_{k}O_{k}}^{-1}\Sigma_{O_{k}M_{k}}$. Now, for $i \in [p]$, define
\begin{align}\label{eq:alpha i definition}
    \alpha_{i} = \sum_{k=1}^{K}D_{k}^*n_{k}\mathbbm{1}_{\{i \in O_{k}\}}.
\end{align}
We will see below that the $\{\alpha_{i}\}_{i=1}^p$ can be interpreted as effective sample sizes for each variable and that the $\{D_{k}^*\}_{k=1}^K$ are oracle choices of the weights $\{D_{k}\}_{k=1}^K$. For convenience, define $n_{\min}=\min_{k\in[K]} n_{k}$.  

The following master theorem will be used to bound rates of convergence in both OSS and supervised settings. 

\begin{theorem}\label{thm:Master OSS with estimating weights}
Assume the noise distribution satisfies $\mathbb{E}\left[\epsilon^8\right]^{\frac{1}{4}}\leq\kappa_{\epsilon}\sigma^2$, for some $\kappa_{\epsilon}\geq 1$. Further assume the distribution of covariates satisfies Assumptions \ref{assump:small-ball}, \ref{assump:subgauss} and \ref{assump:eigen}. Suppose that for some constant $F=F(\lambda_{+},\lambda_{-},C_{X},\chi,C,K,\kappa_{\epsilon})$ and for some $\zeta \!=\! \zeta(\{n_{k}\}_{k=1}^K,\{|L_{k}|\}_{k=1}^L,p,N)\leq 1$ we have
\begin{align}\label{eq:OSS master covariance evals bound}
    &\mathbb{E}\left[\frac{\max\{1,\lambda_{\max}(\hat{\Sigma})\}^{24+16K}}{\min\{1,\lambda_{\min}(\hat{\Sigma})\}^{32+16K}}\right]\leq F, \quad 
    &\mathbb{E}\left[\|\hat{\Sigma}-\Sigma\|_{\mathrm{op}}^{16}\right]^{\frac{1}{8}} \leq F\zeta.
\end{align}
We also assume that the estimate of the covariance $\hat{\Sigma}$ is symmetric, positive-definite and independent of the labelled data used in \eqref{eq:OSS estimator definition initial}. Additionally, assume that 
\begin{equation}\label{eq:OSS master weights deviation bound}
    \max_{k \in [K]} \max \left\{\mathbb{E}\left[\left(\frac{D_{k}^*}{\hat{D}_{k}}\right)^{16}\right],\mathbb{E}\left[\left(\frac{\hat{D}_{k}}{D_{k}^*}\right)^{16}\right]\right\}\leq F.
\end{equation}
Then, for some $G=G(\lambda_{+},\lambda_{-},C_{X},\chi,C,K,\kappa_{\epsilon})$ and $H=H(\lambda_{+},\lambda_{-},C_{X},\chi,C,K,\kappa_{\epsilon})$, provided $n_{\min} \geq H p$, it holds that
    \begin{equation}\label{eq:OSS master estimation error}
        \mathbb{E}\left[\|\hat{\beta}-\beta^*\|_{2}^2\right] \leq G\left(\sigma^2\sum_{i=1}^p\frac{1}{\alpha_{i}}+\|\beta^*\|_{2}^2\zeta\right).
    \end{equation}
\end{theorem}

At a high level, this theorem guarantees that if we estimate the covariance matrix at rate $\zeta$ (condition \eqref{eq:OSS master covariance evals bound}) in the squared operator norm and estimate the oracle weights up to constant factors (condition \eqref{eq:OSS master weights deviation bound}), our procedure obtains the rate displayed in \eqref{eq:OSS master estimation error}. Estimates of the covariance satisfying the above conditions are constructed in the Appendix by Lemmas~\ref{lemma:eigenval control} and~\ref{lemma:subgaussian tail control} in the OSS case and Proposition \ref{Prop:Supervised covariance blocks} in the supervised case. 

We begin by discussing the ISS case with oracle weights. Here, we can take $\zeta=0$ in \eqref{eq:OSS master estimation error}. Performing some algebra yields that, at the population level, $\mathbb{E}[(Y-X_{O_{k}}^TP_{k}\beta^*)^2]=\sigma^2+(\beta^*_{M_{k}})^TS_{M_{k}}(\beta^*_{M_{k}})$. Note that the best linear predictor, in the sense of~\citet{buja2019models}, of $Y$ onto $X_{O_{k}}$ is $X_{O_{k}}^TP_{k}\beta^*$, i.e., $P_{k}\beta^* = \argmin_{\theta \in \mathbb{R}^{|O_{k}|}}\mathbb{E}[(Y-X_{O_{k}}^T\theta)^2]$. Thus, we can see $D_{k}^*$ as the ratio of the variance inherent in a complete case relative to the variance inherent in a case in missingness pattern $k$. Therefore, $\alpha_{i}$ from \eqref{eq:alpha i definition} admits the interpretation as the effective sample size of the $i^{th}$ predictor variable across the dataset. 


We now discuss the more realistic OSS and supervised settings, supposing also that the oracle weights are unknown. There are several ways in which the estimates of the weights can be constructed. Our theoretical results clarify that they only need to be accurate up to constant factors. Provided $\|\beta^*\|_{2}^2 \lesssim \sigma^2$, a common assumption in the literature, the weights can be set to be 1. More generally, we now give an example method for constructing these weights. 

We first note that, for $\lambda>0$, replacing $\{\hat{D}_{k}\}_{k=1}^{K}$ with $\{\lambda \hat{D}_{k}\}_{k=1}^{K}$ in \eqref{eq:OSS estimator definition initial} results in the same estimator. We therefore focus on constructing estimates of $\{\sigma^2+(\beta_{M_{k}}^*)^TS_{M_{k}}(\beta_{M_{k}}^*)\}_{k=1}^K$. Fix an observation pattern $O\subseteq [p]$ and write $P=\Sigma_{OO}^{-1}\Sigma_O$. At the population level, 
\begin{align*}
    Y=X^T\beta^*+\epsilon = X_{O}^TP\beta^*+X^T\beta^*-X_{O}^TP\beta^*+\epsilon \equiv X_{O}^T\phi^* + \eta,
\end{align*}
where $\phi^*=P\beta^*$ and $\eta=X^T\beta^*-X_{O}^TP\beta^*+\epsilon$. Note that $\mathbb{E}\left[\eta X_{O}\right]=0$ and $\mathbb{E}\left[\eta^2\right]=\sigma^2+(\beta^*_{M})^TS_{M}\beta^*_{M}\equiv d$. These equations suggest estimating $d$ by performing linear regression on this missingness pattern and then taking the residual sum of squares. Denote by $X_{O}$ the portion of the design matrix observed in this missingness pattern, $Y_{O}$ the corresponding responses and $n$ the number of samples in this pattern. Set $\hat{\phi}=(X_{O}^TX_{O})^{-1}X_{O}^TY_{O}$ and $\hat{d} = \frac{1}{n}\|Y_{O}-X_{O}\hat{\phi}\|_{2}^2$. In order to get an appropriate bound in expectation, we assume knowledge of $\kappa_{U},\kappa_{L} >0$ such that $\kappa_{L} \leq \sigma^2+(\beta^*_{M})^TS_{M}\beta^*_{M} \leq \kappa_{U}$. As our theoretical guarantees show, $\kappa_{L}$ and $\kappa_{U}$ do not need to be treated as constant and can be allowed to grow or decay with $\{n_{k}\}_{k=1}^K$ and $p$. Our final estimate of $d$ is given by 
    \begin{equation*}
        \hat{d}_{F} =
    \begin{cases}
    \kappa_{L} &\text{if } \hat{d}\leq \frac{\kappa_{L}}{2},\\
    \hat{d} & \text{if } \frac{\kappa_{L}}{2}<\hat{d}< 2\kappa_{U}, \\
    \kappa_{U} & \text{if } \hat{d} \geq 2\kappa_{U}.
    \end{cases}
    \end{equation*}
    We can repeat this for each missingness pattern to obtain $\hat{d}_{F}^{(1)},\ldots,\hat{d}_{F}^{(K)}$ which estimate the $d_{k}\equiv\sigma^2+(\beta^*_{M_{k}})^TS_{M_{k}}\beta^*_{M_{k}}$. We then set the weights via $\hat{D}_{k} = 1/\hat{d}_{F}^{(k)}$. By the invariance noted earlier, this is equivalent to setting the weights as $\hat{D}_{k} = \sigma^2/\hat{d}_{F}^{(k)}$. We therefore analyse this choice of weights.  
    \begin{theorem}\label{thm:weights initialisation guarantee}
     Given Assumptions \ref{assump:small-ball}, \ref{assump:subgauss} and \ref{assump:eigen}, further assume that for $q \geq 1$ our error distribution satisfies $\mathbb{E}\left[\epsilon^{8q}\right]^{\frac{1}{4q}}\leq\kappa_{\epsilon}\sigma^{2}$. Assume that for all $k \in [K]$, $\kappa_{L} \leq d_{k}\leq \kappa_{U}$ and $n_{k} \geq \frac{24q+6|O_{k}|}{\chi}$. There exists some constant $F=F(\lambda_{+},\lambda_{-},C,\chi,q,C_{X},\kappa_{\epsilon})$ such that if  
    \begin{equation}\label{eq:weights tuning sample size}
       \max_{k \in [K]}\left\{\left(\frac{|O_{k}|}{n_{k}}\right)^{q}\left(\frac{\kappa_{U}}{\kappa_{L}}\right)^{16}\right\}\leq \frac{1}{F}, 
    \end{equation}
    then 
        \begin{equation}\label{eq:weights initialisation conclusion}
            \max_{k \in [K]}\left\{\mathbb{E}\left[\left(\frac{\hat{D}_{k}}{D_{k}^*}\right)^{16}\right]+\mathbb{E}\left[\left(\frac{D_{k}^*}{\hat{D}_{k}}\right)^{16}\right]\right\} \leq H,
        \end{equation}
        where $H$ is a universal constant. 
    \end{theorem}
    
    We see that $\kappa_{L}$ and $\kappa_{U}$ can vary with $\{n_{k}\}_{k=1}^K$ and $p$ while condition \eqref{eq:weights tuning sample size} still holds. However, there is a trade-off between the rate at which they can vary and the number of moments we assume. The condition $\mathbb{E}\left[\epsilon^{8q}\right]^{\frac{1}{4q}}\leq\kappa_{\epsilon}\sigma^{2}$ is similar to a kurtosis bound, but for a higher moment. Alternatives to the above procedure, which assumes a large number of cases in each missingness pattern, include fitting a complete case estimator and estimating the $D_{k}^*$ based on standard OLS theory. Another option is to obtain an initial consistent estimator by setting all the weights to be equal to one and then estimating the appropriate quantities.

Applying the above procedure to tune the weights allows us to discuss the consequences of Theorem~\ref{thm:Master OSS with estimating weights} in the context of the OSS and supervised settings. In the OSS case, where $N$ is reasonably large, we construct the estimate of the covariance $\hat{\Sigma}$ used in \eqref{eq:OSS estimator definition initial} from the unlabelled data alone. As justified by Lemmas~\ref{lemma:eigenval control} and~\ref{lemma:subgaussian tail control} in the Appendix, we can then take $\zeta = \frac{p}{N}$ in Theorem~\ref{thm:Master OSS with estimating weights} to guarantee that, under appropriate conditions,
\begin{equation*}
    \mathbb{E}\left[\|\hat{\beta}-\beta^*\|_{2}^2\right]  \leq G\left(\sigma^2\sum_{i=1}^p\frac{1}{\alpha_{i}}+\|\beta^*\|_{2}^2\frac{p}{N}\right),
\end{equation*}
for some constant $G=G(\lambda_{+},\lambda_{-},C_{X},\chi,C,K)$. When $N$ is large compared to the cross modality sample sizes and the signal-to-noise ratio, Theorem \ref{thm:LD structured lower bound} below together with Theorem \ref{thm:Master OSS with estimating weights} above show that this rate is minimax optimal. Under the assumptions common in the literature, where $\alpha_{i} \asymp \rho n_{\mathcal{L}}$ for some $\rho \in (0,1)$ and $\|\beta^*\|_{2}^2\lesssim \sigma^2$, we recover existing rates of convergence \citep{wang2019rate}. Theorem~\ref{thm:OSS balanced} and the surrounding discussion consider this in more depth. 

In the supervised case, we split the indices of the data into two folds $\mathcal{F}_{1},\mathcal{F}_{2}$ by randomly splitting each missingness pattern in half. With $\mathcal{F}_{1}$ we estimate the covariance to get $\hat{\Sigma}$. In $\mathcal{F}_{2}$ we estimate the weights by the procedure preceding Theorem~\ref{thm:weights initialisation guarantee}. We then apply an appropriate version of the estimator in \eqref{eq:OSS estimator definition initial}, with $\mathcal{I}_j$ replaced by $\mathcal{I}_j \cap \mathcal{F}_2$, on this data to get $\hat{\beta}_{2}$. We repeat this procedure with the roles of the folds reversed to obtain $\hat{\beta}_{1}$. Our final estimate is $\hat{\beta}=(\hat{\beta}_{1}+\hat{\beta}_{2})/2$. Theorem \ref{thm:Master OSS with estimating weights} combined with Proposition~\ref{Prop:Supervised covariance blocks} in the Appendix and the subsequent discussion guarantees that we can take $\zeta = \max_{(g,h) \in[L]^2}\{|L_{g}|/n_{g,h}\}$ and thus our procedure has risk upper bounded by
\begin{equation*}
    \mathbb{E}\left[\|\hat{\beta}-\beta^*\|_{2}^2\right]  \leq G\left(\sigma^2\sum_{i=1}^p\frac{1}{\alpha_{i}}+\|\beta^*\|_{2}^2\max_{(g,h) \in[L]^2}\left\{\frac{|L_{g}|}{n_{g,h}}\right\}\right).
\end{equation*}
Theorem \ref{thm:LD structured lower bound} below together with Theorem \ref{thm:Master OSS with estimating weights} above show that this rate is minimax optimal, when the signal-to-noise ratio is bounded above.
In both OSS and supervised cases, the rates decouple into an ISS term and a term that depends on our ability to estimate the covariance matrix $\Sigma$. 

\subsubsection{Lower bound}\label{sec:OSS structured lower bound}
In order to state our lower bound, we first define the notation necessary to introduce our minimax framework. For $R,B>0$, $\lambda_{-}\leq\lambda_{+}$, $\chi \in (0,1]$, $C\geq 1$, $\kappa_{\epsilon}\geq 1$ and $C_{X}$, we denote by $\mathcal{P}_{OSS}=\mathcal{P}_{OSS}(\chi,C,C_{X},\lambda_{-},\lambda_{+},R,B,\kappa_{\epsilon})$ the family of data-generating mechanisms, defined in Section \ref{sec:Formal Setting}, on the measurable space 
\begin{equation*}
    (\mathcal{X},\mathcal{A}) =\left(\prod_{k=1}^K(\mathbb{R}^{|O_{k}|+1})^{n_{k}}\otimes (\mathbb{R}^{p})^N,\mathcal{B}\left(\prod_{k=1}^K(\mathbb{R}^{|O_{k}|+1})^{n_{k}}\otimes (\mathbb{R}^{p})^N\right)\right),
\end{equation*}
under the following conditions. The distribution of the covariates satisfies Assumption \ref{assump:small-ball} with parameters $\chi$ and $C$, Assumption \ref{assump:subgauss} with parameter $C_{X}$ and Assumption \ref{assump:eigen} with parameters $\lambda_{-}$ and $\lambda_{+}$; the parameters $\sigma$ and $\beta^*$ satisfy $\|\beta^*\|_{2} \leq B$ and $\sigma \leq R$. Finally, the noise satisfies $\mathbb{E}\left[\epsilon^8\right]^{\frac{1}{4}}\leq\kappa_{\epsilon}\sigma^2$. Denote by $\mathcal{P}_{OSS}^{Gauss}$ the subset of $\mathcal{P}_{OSS}$ such that $X\sim N(0,\Sigma)$, $\epsilon\sim N(0,\sigma^2)$ and $\epsilon\indep X$.
We denote by $\hat{\Theta}_{OSS}$ the set of functions $\hat{\beta}:\mathcal{X}\rightarrow\mathbb{R}^p$ such that $x \rightarrow \|\hat{\beta}(x)-\beta^*(P)\|_{2}$ is measurable for every $P \in \mathcal{P}_{OSS}$. We define $\xi : [p] \rightarrow [L], \text{ where for } j \in L_{i}, \space \xi(j) = i$. This function maps variables to their modalities. 
\begin{theorem}\label{thm:LD structured lower bound}
    Suppose we have data drawn from \eqref{eq:datamech1} with distribution lying in $\mathcal{P}_{OSS}^{Gauss}$ satisfying $C_{X}\geq \sqrt{8\lambda_{+}/3}$, $\kappa_{\epsilon}\geq 105^{1/4}$ and $\lambda_{-}<\lambda_{+}$. 
    For each $z \in \mathbb{R}^p$ and $j \in [K]$ define $D_{j}(z) = \frac{R^2}{R^2+\|z_{M_{j}}\|_{2}^2}$ and $\alpha_{i}(z) = \sum_{j=1}^Kn_{j}D_{j}(z)\mathbbm{1}_{\left\{i \in O_{j}\right\}}$. Assume that for all $i \in [p]$
    \begin{equation}\label{eq:LD struc LB SS 1}
        \inf_{\substack{z \in \mathbb{R}^p\\
        \|z\|_{2}\leq B}}\left\{\alpha_{i}(z)\right\} \geq \frac{1}{2}\max\left\{\frac{1}{\lambda_{-}},\lambda_{-}\right\}\max\left\{\frac{R^2}{B^2}, \frac{B^2}{R^2}\right\}p,
    \end{equation}
    and that 
    \begin{equation}
    \label{eq:LD struc LB SS 2}
        \min\{1,\lambda_{+}^2\}\sum_{j \in [p]\setminus\{i\}} \min\left\{\frac{B^2}{N+n_{\xi(j),\xi(i)}}, \frac{R^2}{n_{\xi(j),\xi(i)}}\right\} \leq 2B^2(\lambda_{+}-\lambda_{-})^2.
    \end{equation}
    Then, provided $|L_{i}| \geq 2$ for $i \in [L]$, it holds that
    \begin{align}\label{eq:LD struc LB conclusion}
        &\inf_{\hat{\beta}\in \hat{\Theta}_{OSS}}\sup_{P \in \mathcal{P}_{OSS}^{Gauss}}\mathbb{E}\left[\|\hat{\beta}-\beta^*(P)\|^2_{2}\right]\nonumber \\
        &\gtrsim\sup_{\substack{z \in \mathbb{R}^p\\
        \|z\|_{2}\leq B}}\left\{ \frac{R^2}{\max\{1,\lambda_{-}\}}\sum_{i=1}^p\frac{1}{\alpha_{i}(z)}\right\}+\min\{\lambda_{+}^{-2},1\}\max_{i \in [p]} 
     \sum_{j\in [p]\setminus \{i\}} \min\left\{\frac{B^2}{N+n_{\xi(j),\xi(i)}}, \frac{R^2}{n_{\xi(j),\xi(i)}}\right\}.
    \end{align}
\end{theorem}
Assuming that $C_{X}\geq \sqrt{8\lambda_{+}/3}$ and  $\kappa_{\epsilon}\geq 105^{\frac{1}{4}}$ ensures that $\mathcal{P}_{OSS}^{Gauss}$ is non-empty. \eqref{eq:LD struc LB SS 1} and \eqref{eq:LD struc LB SS 2} are mild minimum sample size conditions. When $\chi,C,C_{X},\lambda_{-}$, $\lambda_{+}, L$ and $\kappa_{\epsilon}$ are fixed, the lower bound \eqref{eq:LD struc LB conclusion} matches the upper bound in Theorem  \ref{thm:Master OSS with estimating weights} and thus establishes that the OSS estimator defined in $\eqref{eq:OSS estimator definition initial}$ attains the minimax rate in the parameters $R, B, \{|L_{l}|\}_{l=1}^L, \{n_{k}\}_{k=1}^K$ and $N$, provided each modality is not too small. In the $B \lesssim R$ regime, the supervised estimator obtains the minimax rate in $R, B, \{|L_{l}|\}_{l=1}^L$ and $ \{n_{k}\}_{k=1}^K$. 
 
Lower bounds in the unstructured high-dimensional case are given in~\citet{wang2019rate} and~\citet{loh2012corrupted} when the covariance matrix is fixed. These results also give lower bounds in the low-dimensional case. The novelty in Theorem \ref{thm:LD structured lower bound} is its characterisation of the effect of the missingness pattern and the dependence on $R$ and $B$. We recover the lower bounds from the literature for the homogeneous setting as a special case of our result. Namely, when $N=\infty$, $B \lesssim R$  and $\alpha_{i} \asymp \rho n_{\mathcal{L}}$ for some $\rho \in (0,1)$, then the risk scales like $\frac{R^2p}{\rho n_{\mathcal{L}}}$. 

To illustrate the value of using unlabelled data in a simple setting, we now discuss the simple monotonic pattern of Example \ref{examp:simple monotonic pattern} in the OSS case. We depict this pattern in Figure \ref{tikz: simple monotonic} with the covariates split into two modalities. 
\begin{figure}[H]

\begin{center}
\begin{tikzpicture}

\definecolor{myred}{RGB}{255,80,80}
\definecolor{myyellow}{RGB}{255,255,150}
\definecolor{mygreen}{RGB}{120,250,120}

\newcommand{\colOneWidth}{0.5cm}
\newcommand{\colTwoWidth}{6cm}
\newcommand{\colThreeWidth}{1.5cm}
\newcommand{\headerHeight}{0.5cm}
\newcommand{\rowOneHeight}{2cm}
\newcommand{\rowTwoHeight}{3cm}
\newcommand{\rowThreeHeight}{2cm}

\matrix (m) [matrix of nodes,
             nodes in empty cells,
             column 1/.style={nodes={minimum width=\colOneWidth}},
             column 2/.style={nodes={minimum width=\colTwoWidth}},
             column 3/.style={nodes={minimum width=\colThreeWidth}},
             row 1/.style={nodes={minimum height=\headerHeight}},
             row 2/.style={nodes={minimum height=\rowOneHeight}},
             row 3/.style={nodes={minimum height=\rowTwoHeight}},
             column sep=2pt, row sep=2pt] {
    |[fill=myred, draw]| & |[fill=myred, draw]| & |[fill=myred, draw]| \\
    |[fill=myyellow, draw]| & |[fill=myyellow, draw]| & \\
    & |[fill=mygreen, draw]| & |[fill=mygreen, draw]| \\
};

\node[left=4pt of m-1-1] {Group 1 ($n_{1}$ copies)};
\node[left=4pt of m-2-1] {Group 2 ($n_{2}$ copies)};
\node[left=4pt of m-3-1] {Unlabelled data ($N$ copies)};

\node[above=4pt of m-1-1, align=center] {Y};
\node[above=4pt of m-1-2, align=center] {Modality 1 \\ $p-p_{0}$ variables};
\node[above=4pt of m-1-3, align=center] {Modality 2 \\ $p_{0}$ variables};

\end{tikzpicture}
\caption{Simple monotonic pattern}
\label{tikz: simple monotonic}
\end{center}

\end{figure}
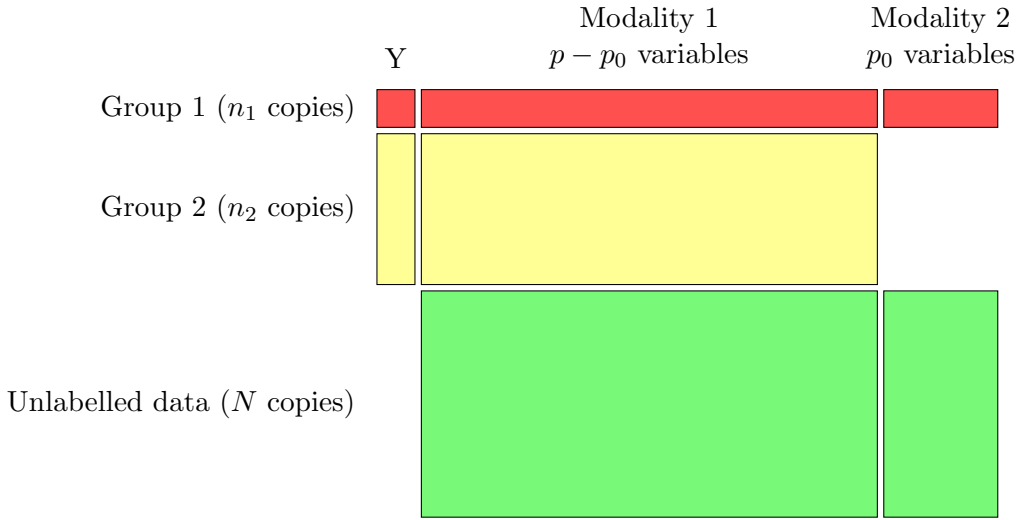

Under the mild safety requirement $N \geq \max\{\frac{B^2}{R^2},1\}n_{1}$, the method \eqref{eq:OSS estimator definition initial} achieves the minimax rate. In particular, in this regime the rate is
\begin{equation}\label{OSS:simple monotonic pattern 2}
    R^2\left(\frac{p_{0}}{n_{1}}+\frac{p-p_{0}}{n_{1}+n_{2}D^*}\right)+\frac{B^2p}{N},
\end{equation}
where $D^* = \frac{R^2}{R^2+B^2}$. This includes regimes in which the final term of \eqref{OSS:simple monotonic pattern 2} dominates the rate. Notably, for a large unlabelled dataset and a large number of observations that are missing the final $p_{0}$ covariates (large $N$ and $n_{2}$), the rate behaves like $\frac{R^2p_{0}}{n_{1}}$, i.e.~we obtain a reduction in effective dimensionality from $p$ to $p_{0}$ relative to a complete case analysis. \eqref{eq:LD struc LB conclusion} clarifies that when $N=0$, the rate is always at least $\frac{\min\{B^2,R^2\}p}{n_{1}}$. This demonstrates that any procedure that fails to use the unlabelled dataset cannot achieve such a reduction in dimensionality, at least when $B \gtrsim R$.

Another interesting implication of \eqref{eq:LD struc LB conclusion} is that in order to do consistent estimation in the supervised $(N=0)$ case, we must observe each pair of variables together infinitely often. This stands in contrast to the ideal semi-supervised case $(N=\infty)$, where we need only observe each variable infinitely often ($\alpha_\mathrm{min} \gg  R^2p$).

\subsection{Unstructured missingness}\label{sec:LD Unstructured missingness}
In contrast to the previous section, we now consider a setting in which we no longer treat the number of missingness patterns as constant. Instead, we work under the balancing assumption of Example~\ref{examp:unstructured pattern}. Recall that we assume that there exists $\rho \in (0,1)$ and $C_{\rho} \geq 1$ such that for all distinct $(g,h) \in [L]^2$, $\rho n_{\mathcal{L}} \leq h_{g} \leq C_{\rho}\rho n_{\mathcal{L}}$ and $\rho^2n_{\mathcal{L}}\leq n_{g,h} \leq C_{\rho}\rho^2n_{\mathcal{L}}$. Since we potentially have a large number of patterns, we set the weights to be equal to $1$ and instead assume that $\|\beta^*\|_{2}\lesssim \sigma$. In the previous section, our estimator took the form
\begin{equation*}
    \hat{\beta}=\hat{B}^{-1}\sum_{i=1}^{n_{\mathcal{L}}}\hat{P}_{\Xi(i)}^T(X_{i})_{O_{\Xi(i)}}Y_{i}, \text{  where  } \hat{B}=\sum_{i=1}^{n_{\mathcal{L}}}\hat{P}_{\Xi(i)}^T(X_{i})_{O_{\Xi(i)}}(X_{i})_{O_{\Xi(i)}}^T\hat{P}_{\Xi(i)},
\end{equation*}
and we have set the weights to be equal to 1. In this setting, we assume knowledge of $\lambda_{+}$ and $\lambda_{-}$ and threshold $\hat{B}$ to create $\hat{B}^{C}$ as follows
\begin{align*}
    \hat{B}^{C}=\begin{cases}
        \rho n_{\mathcal{L}}I_{p} &\text{ if either }\lambda_{\min}(\frac{\hat{B}}{\rho n_{\mathcal{L}}})< \frac{\lambda_{-}^3}{32\lambda_{+}^2} \text{ or } \lambda_{\max}(\frac{\hat{B}}{\rho n_{\mathcal{L}}})>\frac{32\lambda_{+}^3C_{\rho}}{\lambda_{-}^2},\\
        \hat{B} &\text{ otherwise.}
    \end{cases}
\end{align*}
Our estimate is then given by 
\begin{equation}\label{eq:estimator definition thresholded balanced}
    \hat{\beta}=(\hat{B}^C)^{-1}\sum_{i=1}^{n_{\mathcal{L}}}\hat{P}_{\Xi(i)}^T(X_{i})_{O_{\Xi(i)}}Y_{i}.
\end{equation}
Since $\hat{B}^{C}=\hat{B}$ with high probability, this is a very small modification of the previous estimator. The assumption of homogeneity in Example \ref{examp:unstructured pattern} allows us to change the estimator in this way, thereby removing the necessity of Assumption \ref{assump:small-ball}. The following result covers both OSS and supervised cases. 
\begin{theorem}\label{thm:OSS balanced}
Suppose that there exists $\rho \in (0,1)$ and $C_{\rho} \geq 1$ such that for all distinct $(g,h) \in [L]^2$, $\rho n_{\mathcal{L}} \leq h_{g} \leq C_{\rho}\rho n_{\mathcal{L}}$ and $\rho^2n_{\mathcal{L}}\leq n_{g,h} \leq C_{\rho}\rho^2n_{\mathcal{L}}$ and that $\|\beta^*\|_{2}^2\leq C_{*}\sigma^2$ for some $C_{*} >0$. Assume that the distribution of $X$ satisfies Assumptions \ref{assump:subgauss} and \ref{assump:eigen}. Additionally assume that
\begin{align}\label{eq:complicated eval assumption 1}
    &\mathbb{E}\left[\max\left\{\frac{(\rho^2n_{\mathcal{L}}+N)^2}{p^2}\|\hat{\Sigma}-\Sigma\|_\mathrm{op}^4,1\right\}\left(\frac{\max\{\lambda_{\max}(\hat{\Sigma}),1\}}{\min\{\lambda_{\min}(\hat{\Sigma}),1\}}\right)^8\right] \leq H,\\
    &\mathbb{P}\left(\|\hat{\Sigma}-\Sigma\|_\mathrm{op}\geq \min\left\{1,\frac{\lambda_{-}}{2}\right\}\right) \leq \frac{F_{1}}{(\rho^2n_{\mathcal{L}}+N)^{8}},\label{eq:complicated eval assumption 2}
\end{align}
where $\hat{\Sigma}$ is a symmetric, positive-definite estimate of the covariance matrix that is independent of the labelled data used in the estimator, $H=H(C_{X},C_{\rho},\lambda_{-},\lambda_{+})$ and $F_{1}=F_{1}(C_{X},C_{\rho},\lambda_{+},\lambda_{-})$. Then there exists a constant $G=G(C_{X},C_{\rho},C_{*},\lambda_{-},\lambda_{+})$ such that provided $\min\{\rho^2n_{\mathcal{L}}+N,\rho n_{\mathcal{L}}\}\geq G\log(\frac{p}{\rho})\max\{p,\frac{1}{\rho}\}$, the estimator \eqref{eq:estimator definition thresholded balanced} satisfies
\begin{align}\label{eq:OSS balanced conclusion}
    \mathbb{E}\left[\|\hat{\beta}-\beta^*\|_{2}^2\right] &\leq F\left(\frac{\sigma^2p}{\rho n_{\mathcal{L}}}+\frac{\|\beta^*\|_{2}^2p}{\rho^2n_{\mathcal{L}}+N}\right),
\end{align}
where $F=F(C_{X},C_{\rho},C_{*},\lambda_{-},\lambda_{+})$.
\end{theorem}
The conditions in the theorem on the estimate of the covariance $\hat{\Sigma}$ allow us to treat the OSS case and the supervised case in a unified manner. If $N \geq \rho^2n_{\mathcal{L}}$, a mild condition on the size of the unlabelled data, then we can construct an estimate of the covariance, such as that of Proposition~\ref{Prop:OSS covariance blocks} in the Appendix, from the unlabelled data alone. This procedure achieves the rate in \eqref{eq:OSS balanced conclusion}. In the supervised $(N=0)$ case, we apply two-fold cross-fitting exactly as in Section \ref{sec:LD structured missingness}, and estimate the covariance from the labelled data as justified by Proposition \ref{Prop: unstructured cov estimation theorem}. Again, this procedure achieves the rate \eqref{eq:OSS balanced conclusion} with $N=0$. Theorem \ref{thm:LD unstruc lower} will demonstrate that these rates are optimal. We finally remark that, in the related setting where each variable is observed independently with probability $\rho$, covariance estimates in the supervised case satisfying \eqref{eq:complicated eval assumption 1} and \eqref{eq:complicated eval assumption 2} are obtained in \cite{EJSabdalla2024covariance}. 

\subsubsection{Lower bound}\label{sec:Unstructured lower bound}

We now present a lower bound applicable in the unstructured case. Recall the family of distributions $\mathcal{P}_{OSS}^{Gauss}$ and estimators $\hat{\Theta}_{OSS}$ defined in Section \ref{sec:OSS structured lower bound}. 
\begin{theorem}\label{thm:LD unstruc lower}
    Suppose we have data drawn from \eqref{eq:datamech1} with distribution lying in $\mathcal{P}_{OSS}^{Gauss}$ satisfying $C_{X}\geq \sqrt{8\lambda_{+}/3}$, $\kappa_{\epsilon}\geq 105^{\frac{1}{4}}$, $\lambda_{-}<\lambda_{+}$ and $p \geq 2$.
    Assume that there exists $\rho \in (0,1)$ and $C_{\rho}>0$ such that for all distinct $(g,h) \in [L]^2$, $\rho n_{\mathcal{L}} \leq h_{g} \leq C_{\rho}\rho n_{\mathcal{L}}$ and $\rho^2n_{\mathcal{L}}\leq n_{g,h} \leq C_{\rho}\rho^2n_{\mathcal{L}}$. Further suppose that 
    \begin{equation}\label{eq:LD unstruc lower sample size}
        \rho n_{\mathcal{L}} \geq p \max\left\{\frac{1}{2},\frac{R^2}{4B^2\lambda_{-}}\right\} \quad\quad \text{and} \quad\quad \max\left\{\frac{N+\rho^2n_{\mathcal{L}}}{B^2},\frac{\rho^2n_{\mathcal{L}}}{R^2}\right\} \geq p \frac{\min\{1,\lambda_{+}^2\}}{2B^2(\lambda_{+}-\lambda_{-})^2}.
    \end{equation}
    Then it holds that
    \begin{align*}
    \inf_{\hat{\beta}\in \hat{\Theta}_{OSS}}\sup_{P \in \mathcal{P}_{OSS}^{Gauss}}\mathbb{E}\left[\|\hat{\beta}-\beta^*(P)\|^2_{2}\right]
    &\gtrsim \frac{1}{C_{\rho}\lambda_{-}}\frac{R^2p}{\rho n_{\mathcal{L}}}+\frac{\min\{\lambda_{+}^{-2},1\}}{C_{\rho}}\min\left\{\frac{B^2p}{N+\rho^2n_{\mathcal{L}}},\frac{R^2p}{\rho^2n_{\mathcal{L}}}\right\}.
    \end{align*}
\end{theorem}
The conditions in \eqref{eq:LD unstruc lower sample size} are mild minimum sample size requirements. Assuming that $C_{X}\geq 105^{\frac{1}{8}}$ and  $\kappa_{\epsilon}\geq 105^{\frac{1}{4}}$ ensures that $\mathcal{P}_{OSS}^{Gauss}$ is non-empty. For the rest of the discussion we treat $C_{X},\lambda_{-},\lambda_{+},C_{\rho}$ and $\kappa_{\epsilon}$ as fixed.

Consider now the regime where $\|\beta^*\|_{2}^2 \lesssim \sigma^2$. Under the unstructured assumption (Example \ref{examp:unstructured pattern}), Theorems~\ref{thm:OSS balanced} and~\ref{thm:LD unstruc lower} establish that the minimax rate in the OSS and supervised settings is
\begin{equation*}
    \frac{R^2p}{\rho n_{\mathcal{L}}}+\frac{B^2p}{\rho^2n_{\mathcal{L}}+N}.
\end{equation*}
Treating $R$ and $B$ as constants, we see that the value of a large amount of unlabelled data is an increase in the effective observation rate from $\rho$ to $\rho^{\frac{1}{2}}$. Under Gaussian assumptions and assuming the covariance matrix is known, \cite{balakrishnan2017statistical} obtains an upper bound, up to log factors, of $\frac{\max\{1,\sigma^2\}p}{n_{\mathcal{L}}}$ when $\rho$ is treated as constant. Our main novelty in this ISS case is in characterising the dependence of the minimax rates on $\rho$ and removing the Gaussian assumptions.

\section{High-dimensional results}\label{sec:High-dimensional results}

We now consider how the picture changes in a high-dimensional sparse setting. The most well-studied way of performing sparse high-dimensional linear regression is the LASSO~\citep{tibshirani1996regression}. For $n$ complete cases $(X_{1},Y_{1}), \ldots, (X_{n},Y_{n})$, this outputs an estimator $\hat{\beta}^{LASSO}$ that solves 
\begin{equation*}
    \hat{\beta}^{LASSO} \in \argmin_{\beta \in \mathbb{R}^p}\left\{\frac{1}{2n}\sum_{i=1}^n(Y_{i}-X_{i}^T\beta)^2+\lambda\|\beta\|_{1}\right\}=\argmin_{\beta \in \mathbb{R}^p}\left\{\frac{1}{2}\beta^T\hat{\Sigma}\beta-\beta^T\hat{\mathbb{E}}\left[XY\right]+\lambda\|\beta\|_{1}\right\},
\end{equation*}
for a regularisation parameter $\lambda >0$, where $\hat{\Sigma}$ is the sample covariance and $\hat{\mathbb{E}}\left[XY\right] = \frac{1}{n}\sum_{i=1}^nY_{i}X_{i}$ estimates $\mathbb{E}\left[XY\right] = \Sigma \beta^*$. A natural way to generalise this with missing data is to replace $\hat{\Sigma}$  and $\hat{\mathbb{E}}\left[XY\right]$ with elementwise estimates of $\Sigma$ and $\mathbb{E}\left[XY\right]$. However, the resulting estimate of  $\Sigma$ is likely not positive semi-definite and so this introduces non-convexity. Much of the early work on unstructured patterns considered this approach with~\citet{loh2011high} analysing the constrained non-convex problem and~\citet{datta2017cocolasso} providing a convex relaxation. However, the upper bounds in these papers did not match the lower bounds in~\citet{loh2012corrupted} in their dependence on the missingness.

In light of the success of $\ell_{1}$-regularisation and our low-dimensional estimator, a natural approach would be to add an $\ell_{1}$-penalty to \eqref{eq:OSS estimator definition initial}. In particular,  for the simple monotonic pattern (Example \ref{examp:simple monotonic pattern}), with $\lambda >0$, in the ISS setting this would take the form 

\begin{equation} \label{eq:Simple monotonic pattern with ell1 regularisation}
    \hat{\beta} \in \argmin_{\beta \in \mathbb{R}^p}\left\{\sum_{i=1}^{n_{1}} (Y_{i}-X_{i}^T\beta)^2+D\sum_{i={n_{1}+1}}^{n_1+n_2} (Y_{i}-(X_{i})_{O_{2}}^T\Sigma_{O_{2}O_{2}}^{-1}\Sigma_{O_{2}}\beta)^2+\lambda \|\beta\|_{1}\right\}.
\end{equation}
There are several problems with this approach. Firstly, tuning multiple parameters is undesirable in high dimensions. Secondly, in the OSS case, being able to invert blocks of the estimate of the covariance would require a very large number of unlabelled samples. However, even in the ISS case, the critical flaw is in choosing the regularisation parameter $\lambda$ in a principled way. Consider a regime where $n_{2}\gg n_{1}$; should we choose $\lambda \asymp \sqrt{\frac{\log(p)}{n_{1}}}$ or $\lambda \asymp \sqrt{\frac{\log(p)}{n_{2}}}$? Empirically, we have found neither choice works well with no middle ground sufficing. The theoretical reason for this is the regularisation in \eqref{eq:Simple monotonic pattern with ell1 regularisation} acts symmetrically on the coordinates of $\beta$, but we have far more information on $\Sigma_{O_{2}O_{2}}^{-1}\Sigma_{O_{2}}\beta^*$ than $\beta^*$ as a whole. Thus, while this imputation-style approach can be successful in unstructured settings~\citep{chandrasekher2020imputation}, to treat more general settings we require new techniques.

An alternative to the LASSO is the Dantzig selector~\citep{candes2007dantzig}, which for $n$ complete cases solves the following linear program 
\begin{equation*}
    \min_{\beta \in \mathbb{R}^p}\left\{\|\beta\|_{1}:\|\hat{\Sigma}\beta-\hat{\mathbb{E}}\left[XY\right]\|_{\infty} \leq \lambda\right\},
\end{equation*}
where $\hat{\Sigma}$ and $\hat{\mathbb{E}}\left[XY\right]$ are defined above and $\lambda>0$ is a regularisation parameter. An insight of~\citet{wang2019rate} is that this is still a convex program when $\hat{\Sigma}$ is replaced with a matrix with negative eigenvalues. They provide stronger estimation guarantees than those based on the LASSO for unstructured missingness. We now extend their estimator to incorporate unlabelled data and provide a refined analysis in the structured and unstructured missingness settings. 
We define  $\gamma = \mathbb{E}\left[XY\right]= \Sigma \beta^* \in \mathbb{R}^p$. We estimate this quantity in the natural way via an unbiased $\hat{\gamma}\in \mathbb{R}^p$, where for $j \in [p]$
\begin{equation}\label{eq:gamma estimate definition}
    \hat{\gamma}_{j} = \frac{1}{h_{\xi(j)}}\sum_{i \in \mathcal{I}_{\mathcal{L}}}(X_{i})_{j}Y_{i}\mathbbm{1}_{\{j \in O_{\Xi(i)}\}}.
\end{equation}
 We further estimate the covariance $\Sigma$ via a natural unbiased $\hat{\Sigma}\in \mathbb{R}^{p\times p},$ where for $(i,j) \in [p]^2$
\begin{equation}\label{eq:covariance estimate definition}
    \hat{\Sigma}_{ij} = \frac{1}{n_{\xi(i),\xi(j)}+N}\left(\sum_{k\in \mathcal{I}_{\mathcal{U}}}(X_{k})_{i}(X_{k})_{j}+\sum_{k \in \mathcal{I}_{\mathcal{L}}}(X_{k})_{i}(X_{k})_{j}\mathbbm{1}_{\{\{i,j\}\subseteq O_{\Xi(k)}\}}\right).
\end{equation}
For the above definition, we make the mild assumption that $n_{g,h}+N > 0$ for all $(g,h) \in [L]^2$. We define our estimator $\hat{\beta}$ to solve the following linear program
\begin{equation}\label{eq:ModifiedDantzig}
    \hat{\beta} \in \argmin_{\beta \in \mathbb{R}^p}\left\{\|\beta\|_{1}:\|\hat{\Sigma}\beta-\hat{\gamma}\|_{\infty} \leq \lambda\right\},
\end{equation}
where $\lambda>0$ is a regularisation parameter.

We impose sub-Gaussian assumptions and a restricted eigenvalue condition in order to analyse this program. 
\begin{assumption}\label{assump:sub-Gaussian distribution}
    $\exists R_{X}>0$ such that $\|X\|_{\psi_{2}} \leq R_{X}$.
    Also, the covariates are centred, i.e., $\mathbb{E}\left[X\right]=0$. Additionally, the noise $\epsilon$ is sub-Gaussian with parameter $\sigma$ in the sense of~\citet[][Definition 10.6.1]{samworth2024statistics}.
\end{assumption}

We use the following restricted eigenvalue assumption, based on~\citet{bickel2009simultaneous}. For a matrix $A \in \mathbb{R}^{p\times p}$ and sparsity $s \in [0,p]$, we define 
\begin{equation*}
    \phi^2(A,s) \equiv \inf_{\substack{S \subset [p]\\|S|\leq s}}\inf_{\substack{\delta \in \mathbb{R}^p: \delta \neq 0\\ \|\delta_{S^c}\|_{1} \leq \|\delta_{S}\|_{1}}}\frac{\delta^TA\delta}{\|\delta_{M(\delta)}\|_{2}^2}, 
\end{equation*}
where 
\begin{equation*}
    M(\delta) = S \cup \{i \in S^c: |\delta_{i}| \geq |\delta_{S^c}|_{(\min\{s,|S^c|\})}\}.
\end{equation*} 
Here, $|\delta_{S^c}|_{(j)}$ is the $j^{th}$ largest absolute value of $\delta_{S^c}$. For $s \in [0,p]$ and $\Phi >0$, we say that a matrix $A$ satisfies $Re(s,\Phi)$ if $\phi^2(A,s) \geq \Phi$.

\begin{assumption}\label{assump:restricted eigenvalue}
    The covariance $\Sigma$ satisfies $Re(s,\Phi_{\Sigma})$ for some constant $\Phi_{\Sigma}>0$.
\end{assumption}
If $\Sigma$ is positive-definite then it satisfies $Re(s,\lambda_{\min}(\Sigma))$. For $s \in [0,p], R>0, B>0, R_{X}>0$ and $\Phi_{\Sigma}>0$ denote by $\mathcal{P}_{HD}\equiv\mathcal{P}_{HD}(R_{X},R,B,s,\Phi_{\Sigma})$ the family of data-generating mechanisms, from Section \ref{sec:Formal Setting}, on the measurable space $(\mathcal{X},\mathcal{A}) =\left(\prod_{k=1}^K(\mathbb{R}^{|O_{k}|+1})^{n_{k}}\otimes (\mathbb{R}^{p})^N,\mathcal{B}\left(\prod_{k=1}^K(\mathbb{R}^{|O_{k}|+1})^{n_{k}}\otimes (\mathbb{R}^{p})^N\right)\right)$ that satisfy Assumptions \ref{assump:sub-Gaussian distribution}  and \ref{assump:restricted eigenvalue} as well as that $\|\beta^*\|_{0}\leq s$, $\|\beta^*\|_{2}\leq B$ and $\sigma \leq R$. Denote by $\hat{\Theta}_{HD}$ the set of functions $\hat{\beta}:\mathcal{X}\rightarrow\mathbb{R}^p$ such that $x \rightarrow \|\hat{\beta}(x)-\beta^*(P)\|_{2}$ is measurable for every $P \in \mathcal{P}_{HD}$.

\subsection{Unstructured missingness}\label{sec:Unstructured Missingness}
 In the unstructured missingness literature, it is typically assumed that each covariate is missing with probability $1-\rho$ independently of everything else. This leads to missingness patterns like that of Example~\ref{examp:unstructured pattern}. As discussed in the introduction, extensive prior work has studied this problem from a minimax perspective. In the setting of Example~\ref{examp:unstructured pattern} we give the following upper bound.
 \begin{theorem}\label{thm:high dimensional upper bound unstructured}
 Under Assumptions \ref{assump:sub-Gaussian distribution} and \ref{assump:restricted eigenvalue}, assume further that $\|\beta^*\|_{0} \leq s $ and that there exists $\rho \in (0,1)$ such that 
 for all $(g,h) \in [L]^2$, $\rho n_{\mathcal{L}} \leq h_{g}$ and $\rho^2n_{\mathcal{L}}\leq n_{g,h}$. For $A>0$, define the regularisation parameter in \eqref{eq:ModifiedDantzig} to be
 \begin{equation*}
     \lambda = A\sqrt{\max\left\{\frac{R_{X}^2(\sigma+R_{X}\|\beta^*\|_{2})^2\log(p)}{\rho n_{\mathcal{L}}},\frac{R_{X}^4\|\beta^*\|_{2}^2\log(p)}{\rho^2n_{\mathcal{L}}+N}\right\}}.
 \end{equation*} 
 Additionally, assume for some universal constant $F$ that
 \begin{align}
    \rho n_{\mathcal{L}} &\geq FA^2 \max \left\{1,\frac{(\sigma+R_{X}\|\beta^*\|_{2})^2}{R_{X}^2\|\beta^*\|_{2}^2}\right\}\log(p),\label{eq:unstructured min sample size 1}\\
     \rho^2n_{\mathcal{L}}+N &\geq F\max\left\{ 
    \frac{R_{X}^4 s^2}{\Phi_{\Sigma}^2},
    \frac{R_{X}^2 s}{\Phi_{\Sigma}},A^2,\frac{A^2R_{X}^2\|\beta^*\|_{2}^2}{(\sigma+R_{X}\|\beta^*\|_{2})^2}
    \right\}\log(p).\label{eq:unstructured min sample size 2}
 \end{align}
Then, with probability at least $1-\frac{6}{p}$, provided $A$ exceeds a sufficiently large universal constant, it holds that 
 \begin{equation*}
     \|\hat{\beta}-\beta^*\|_{2}^2 \lesssim \frac{A^2s}{\Phi_{\Sigma}^2}\left(\frac{R_{X}^2(\sigma+R_{X}\|\beta^*\|_{2})^2\log(p)}{\rho n_{\mathcal{L}}}+\frac{R_{X}^4\|\beta^*\|_{2}^2\log(p)}{\rho^2 n_{\mathcal{L}}+N}\right).
 \end{equation*}
\end{theorem} 

Conditions \eqref{eq:unstructured min sample size 1} and \eqref{eq:unstructured min sample size 2} are mild effective sample size requirements. Treating $\Phi_{\Sigma}$ and $R_{X}$ as constants and assuming $\|\beta^*\|_{2}\lesssim \sigma$, we obtain a rate
\begin{equation}\label{eq:HD unstructured rate}
    s\left(\frac{\sigma^2\log(p)}{\rho n_{\mathcal{L}}}+\frac{\|\beta^*\|_{2}^2\log(p)}{\rho^2 n_{\mathcal{L}}+N}\right).
\end{equation}
Note that~\citet{wang2019rate} obtains the same upper bound in the $N=0$ and $N=\infty$ (known covariance) settings, via essentially the same estimator. We now move to studying this setting from a lower bounds perspective.  
 Previous lower bounds on the minimax risk have focussed on the $N=0$ case and include~\citet{loh2012corrupted} and~\citet{wang2019rate}, the latter giving a lower bound of 
\begin{equation*}
    \inf_{\hat{\beta} \in \hat{\Theta}_{HD}} 
    \sup_{P \in \mathcal{P}_{HD}}\mathbb{E}\left[\|\hat{\beta}-\beta^*\|_{2}^2\right] \gtrsim \frac{R^2 s\log(p/s)}{\rho n_{\mathcal{L}}}.
\end{equation*}
Note the difference in dependence on the observation rate $\rho$ compared to \eqref{eq:HD unstructured rate}. In the $N=0$ case,~\citet{wang2019rate} also gives a lower bound of
\begin{equation*}
    \inf_{\hat{\beta} \in \hat{\Theta}_{HD}} 
    \sup_{P \in \mathcal{P}_{HD}}\mathbb{E}\left[\|\hat{\beta}-\beta^*\|_{2}^2\right] \gtrsim \frac{B^2}{\rho^2n_{\mathcal{L}}},
\end{equation*}
provided $R\gtrsim B$. This exhibits the upper bound scaling in $\rho$ and $n_{\mathcal{L}}$, but does not depend on the dimensionality. We now present a result showing that the upper bound of \eqref{eq:HD unstructured rate} is sharp in the parameters $\rho,N, n_{\mathcal{L}}, p, s$. It is also sharp in $B$ and $R$, provided $B\lesssim R$. We condition on the missingness pattern, in contrast to~\citet{loh2011high} and~\citet{wang2019rate} which work with random missingness. However, provided $\rho n_\mathcal{L}/\log(p) \rightarrow \infty$ and $(\rho^2n_\mathcal{L}+N)/\log(p) \rightarrow \infty$, the number of labelled samples with each variable present and the number of samples with each pair of variables present are stable, which would allow us to adapt our proof to this setting.
\begin{theorem}\label{thm:balancing lower bound}
    Consider the setting \eqref{eq:datamech1} with distribution lying in $\mathcal{P}_{HD}$ and $\Phi_{\Sigma} < \frac{3R_{X}^2}{16}$. Suppose for some $\rho \in (0,1)$ and $C_{\rho}>1$ there exists some $j$ such that for all $i \in [p]\setminus \{j\}$,
    \begin{align*}
         \quad \sum_{k=1}^Kn_{k}\mathbbm{1}_{\{i,j\in O_{k}\}} &\leq C_{\rho}\rho^2 n_{\mathcal{L}}.
    \end{align*}
    If $p-2 \geq 24(s-1)^2$, $s \geq 2$
    and
    \begin{equation}\label{eq: balancing minimum sample size}
        C_{\rho}\rho^2n_{\mathcal{L}}+N \geq \frac{128s\log(p-2)}{9}\max\left\{1,\frac{1}{R_{X}^4}\right\}\min\left\{\frac{R^2}{B^2},1\right\},
    \end{equation}
    then 
    \begin{equation*}
        \inf_{\hat{\beta} \in \hat{\Theta}_{HD}}\sup_{P \in \mathcal{P}_{HD}}\mathbb{E}\left[\|\hat{\beta}-\beta^*\|_{2}^2\right] \geq \frac{1}{2^{17}C_{\rho}}\min\left\{1,\frac{2^6}{3^2R_{X}^4}\right\}\frac{\min\left\{R^2,B^2\right\}s\log(\frac{p-1}{2(s-1)})}{\rho^2n_{\mathcal{L}}+N}.
    \end{equation*}
\end{theorem}
For $N=0$, this theorem resolves a conjecture of~\citet{wang2019rate}, at least in the regime that $B\lesssim R$. \eqref{eq: balancing minimum sample size} is a very mild minimum sample size condition. The condition $p-1 \geq24(s-1)^2$ is a very mild assumption on the sparsity. In regimes where $p-1 \leq 24(s-1)^2$,  $\log((p-1)/(2(s-1))) \lesssim \log(s)$  and so it is a small logarithmic correction to the $s$ term in the rate; Theorem \ref{thm:LD unstruc lower} yields a lower bound of order $\frac{s\min\{R^2,B^2\}}{\rho^2n_{\mathcal{L}}+N}$ in this regime, differing in the log factor alone. The assumption $s \geq 2$ is mild but also essential to the construction in the proof. The proof of Theorem \ref{thm:balancing lower bound} leverages an extension of the sparse Gilbert--Varshamov lemma (Lemma \ref{lemma:high-dimensional packing} in Appendix B) that may be of interest for other high-dimensional missing data problems. Combining this lower bound with the lower bound in~\citet{wang2019rate} and the upper bound of Theorem \ref{thm:high dimensional upper bound unstructured} demonstrates that when $B \lesssim R$ and $1<s^2 \lesssim p$, the minimax rate is 
\begin{equation*}
     \frac{R^2s\log(p)}{\rho n_{\mathcal{L}}}+\frac{B^2s\log(p)}{\rho^2 n_{\mathcal{L}}+N}.
\end{equation*}
Although formally our model of missingness is different to that of~\citet{wang2019rate}, it can be checked that the proof of their Theorem~2 can be adapted to our setting.
This rate is analogous to the corresponding low-dimensional rate of Section \ref{sec:LD Unstructured missingness} which we display below
\begin{equation*}
    \frac{R^2p}{\rho n_{\mathcal{L}}}+\frac{B^2p}{\rho^2n_{\mathcal{L}}+N}.
\end{equation*}
In both cases, the value of a large unlabelled dataset is an increase in the effective observation rate from $\rho$ to $\rho^{\frac{1}{2}}$. All of our lower bounds are stated in terms of expected loss, in contrast to our high-probability upper bounds in this section. However, these could be converted to lower bounds that hold with constant probability using tools as in~\citet{ma2024high}.

\subsection{Structured missingness}
We give an additional analysis in the blockwise-missing setup with a constant number of missingness patterns.  Recall that treating the number of modalities $L$ as fixed is equivalent to treating the number of patterns $K$ as fixed.
\subsubsection{Upper bound}
We begin by giving an upper bound for \eqref{eq:ModifiedDantzig} in the blockwise setting. 
\begin{theorem}\label{thm:high dimensional rate structured}
    Consider the estimator $\hat{\beta}$ in \eqref{eq:ModifiedDantzig} with regularisation parameter
    \begin{equation*}
        \lambda = A\max\left\{\max_{l \in [L]}\left\{\sqrt{\frac{R_{X}^2(\sigma+R_{X}\|\beta^*\|_{2})^2\log\left(2L|L_{l}|\right)}{h_{l}}}\right\},\max_{h\in[L]}\left\{L\sqrt{\frac{R_{X}^4\|\beta^*\|_{2}^2\log\left(2L|L_{h}|\right)}{N+\min_{g}\{n_{g,h}\}}}\right\}\right\},
    \end{equation*}
    for some $A>0$. Assume that $\|\beta^*\|_{0} \leq s $ and that Assumptions \ref{assump:sub-Gaussian distribution} and \ref{assump:restricted eigenvalue} are satisfied. Additionally, suppose that for some universal constant $F$ and all $g,h \in [L]^2$
    \begin{align}\label{eq:blockwise-missing overall sample sizes 1}
    n_{g,h} + N &\geq  F\max\left\{ 
    \frac{ R_{X}^4 s^2}{\Phi_{\Sigma}^2},
    \frac{R_{X}^2 s}{\Phi_{\Sigma}},A^2,\frac{A^2L^2R_{X}^2\|\beta^*\|_{2}^2}{(\sigma+R_{X}\|\beta^*\|_{2})^2} 
    \right\} \max\{\log(2L|L_{g}|),\log(2L|L_{h}|)\},\quad 
    \\
    h_{g}&\geq F\max\left\{A^2,\frac{A^2(\sigma+R_{X}\|\beta^*\|_{2})^2}{R_{X}^2L^2\|\beta^*\|_{2}^2}\right\}\log(2L|L_{g}|).   \label{eq:blockwise-missing overall sample sizes 2}
    \end{align}
    With probability at least
    \begin{equation}\label{eq:blockwise-missing high probability guarantee}
        1 - \frac{1}{L\min_{l \in [L]}\left\{|L_{l}|\right\}},
    \end{equation}
    provided $A$ exceeds a sufficiently large universal constant, it holds that
    \begin{equation}\label{eq:high dimensional rate}
        \|\hat{\beta}-\beta^*\|_{2}^2 \lesssim \frac{sA^2}{\Phi_{\Sigma}^2}\max\left\{\max_{l \in [L]}\left\{\frac{R_{X}^2(\sigma+R_{X}\|\beta^*\|_{2})^2\log\left(2L|L_{l}|\right)}{h_{l}}\right\},\max_{h }\left\{\frac{L^2R_{X}^4\|\beta^*\|_{2}^2\log\left(2L|L_{h}|\right)}{N+\min_{g}\{n_{g,h}\}}\right\}\right\}.
    \end{equation}
\end{theorem}
This theorem guarantees that, for a sufficiently large choice of regularisation parameter $\lambda$, the program \eqref{eq:ModifiedDantzig} yields an estimate $\hat{\beta}$ that achieves rate \eqref{eq:high dimensional rate}. In particular, treating $A,\phi_{\Sigma},R_{X}$ and $L$ as constants, we obtain a rate
\begin{equation*}
    R^2s\max_{g\in[L]}\left\{\frac{\log(|L_{g}|)}{h_{g}}\right\}+B^2s\max_{g,h \in [L]}\left\{\frac{\log(|L_{g}|)}{n_{g,h}+N}\right\}.
\end{equation*}
 The high-probability guarantee \eqref{eq:blockwise-missing high probability guarantee} is close to 1 provided the modalities are not too small. A closer inspection of the proof yields that this probability can be inflated at the cost of small logarithmic factors in \eqref{eq:high dimensional rate}, if some of the modalities are very small. Conditions \eqref{eq:blockwise-missing overall sample sizes 1} and \eqref{eq:blockwise-missing overall sample sizes 2} are mild minimum sample size requirements.

\subsubsection{Lower bound}
Recall the class of distributions $\mathcal{P}_{HD}=\mathcal{P}_{HD}(R_{X},R,B,s,\Phi_{\Sigma})$ defined just before the start of Section~\ref{sec:Unstructured Missingness}. We now present a minimax lower bound for this class in the blockwise-missing setting. 

\begin{theorem}\label{thm:High-dimensional lower bound}
    Consider the setting \eqref{eq:datamech1} with distribution lying in $\mathcal{P}_{HD}$ and $\Phi_{\Sigma} \leq \frac{3R_{X}^2}{16}$. There exists a universal constant $c_{1}$ such that if $2 \leq  s \leq c_{1}\min_{l \in [L]}|L_{l}|$ and 
    \begin{equation}\label{eq:blockwise sample size requirements}
        \begin{aligned}
        &\max_{g,h}\left\{\frac{16s\log\left(1+\frac{|L_{h}|}{2(s-1)}\right)\min\left\{1,\frac{R^2}{B^2}\right\}\min\left\{1,\frac{1}{R_{X}^4}\right\}}{9(n_{g,h}+N)}\right\}\leq 1, \quad\quad \max_{i \in [L]}\left\{\frac{R^2s\log\left(1+\frac{|L_{i}|}{2s}\right)}{16B^2\Phi_{\Sigma}h_{i}} \right\} \leq 1, 
        \end{aligned}
    \end{equation}
    then, 
    \begin{align*}
        &\inf_{\hat{\beta}\in\hat{\Theta}_{HD}}\sup_{P \in \mathcal{P}_{HD}}\mathbb{E}\left[\|\hat{\beta}-\beta^*\|_{2}^2\right]\nonumber\\ &\geq \max\left\{\max_{h \in [L]}\left\{\frac{\min\left\{1,\frac{2^6}{3^2R_{X}^4}\right\}\min\{B^2,R^2\}s\log\left(1+\frac{|L_{h}|}{2(s-1)}\right)}{2^{17}\min_{g \in [L]}\{n_{g,h}+N\}}\right\}, \max_{l \in [L]}\left\{\frac{R^2s}{2^{10}h_{l}\Phi_{\Sigma}}\log\left(1+\frac{|L_{l}|}{2s}\right)\right\}\right\}.
\end{align*}
\end{theorem}
The restriction to settings satisfying $ s \leq c_{1}\min_{l \in [L]}|L_{l}|$ is a mild assumption on the modality sizes. The assumption $2 \leq s$ is mild but, as before, essential to the construction. The conditions in \eqref{eq:blockwise sample size requirements} are weak minimum sample size requirements. 
\subsection{Discussion}
The minimax upper and lower bounds agree in $|L_{g}|$, $h_{g}$, $N$, $n_{g,h}$ and $s$. These rates are also sharp in $B$ and $R$, provided $B\lesssim R$. For the rest of the discussion in this section, we treat the parameters $L,\Phi_{\Sigma}$ and $R_{X}$ as constants. We illustrate the rate of convergence given in \eqref{eq:high dimensional rate} in several settings, comparing with existing estimators in the literature. 
\subsubsection{Simple monotonic pattern}
We first specialise our rate to the simple monotonic pattern (Example \ref{examp:simple monotonic pattern}). Assuming a bounded signal-to-noise ratio, $B \lesssim R$, we obtain the rate 
\begin{equation*}
    R^2s\left(\frac{\log(p_{0})}{n_{1}}+\frac{\log(p-p_{0})}{n_{1}+n_{2}}\right)+\frac{B^2s\log(p)}{n_{1}+N}.
\end{equation*}
We display below the analogous low-dimensional rate
\begin{equation*}
    R^2\left(\frac{p_{0}}{n_{1}}+\frac{p-p_{0}}{n_{1}+n_{2}}\right)+\frac{B^2p}{n_{1}+N}.
\end{equation*}
In both cases, when $n_{2}$ is sufficiently large the final term becomes negligible. As in the low-dimensional setting, we see the value of a large number of unlabelled samples is in a reduction of the effective dimensionality from $p$ to $p_{0}$. Again, this reduction would not be possible without the unlabelled data. Closely related work that considers this semi-supervised setting is that of \cite{song2024semi}. For the simple monotonic pattern, their upper bound coincides with the complete cases upper bound. For a general missingness pattern, the rate of convergence given by their estimator is always at least as large as that given by \eqref{eq:ModifiedDantzig}. 

\subsubsection{Supervised settings}
We now consider the supervised case when $N=0$ and compare with DISCOM~\citep{yu2020optimal}, the most relevant estimator in this context. Our optimal upper bounds provide stronger guarantees in several ways. When $\sigma \lesssim \|\beta^*\|_{2}$, the first term in \eqref{eq:high dimensional rate} is at most the second up to constants. Our rate of convergence simplifies to 
\begin{equation*}
    \max_{h }\left\{\frac{\|\beta^*\|_{2}^2s\log\left(|L_{h}|\right)}{\min_{g}\{n_{g,h}\}}\right\}\leq\frac{\|\beta^*\|_{2}^2s\log\left(p\right)}{\min_{g,h}\{n_{g,h}\}}.
\end{equation*}
DISCOM is proven to obtain a rate
\begin{equation*}
    \frac{\|\beta^*\|_{1}^2s\log(p)}{\min_{g,h}\{n_{g,h}\}}.
\end{equation*}
We have reduced the $\|\beta^*\|_{1}^2$ to $\|\beta^*\|_{2}^2$, offering an improvement of up to a factor of $s$. More generally, in any dataset with at most two missingness patterns we see the same improvement in rate. Furthermore, the guarantee  for DISCOM requires $\frac{s^2\log(p)}{\min_{g,h}n_{g,h}} = o(1)$, a stronger requirement relative to the milder minimum sample size requirements of \eqref{eq:blockwise-missing overall sample sizes 1} and \eqref{eq:blockwise-missing overall sample sizes 2}. 

The difference between the optimal rate \eqref{eq:high dimensional rate} and the rate provided by DISCOM becomes starker when there are imbalances in the sizes of the modalities and three or more modalities. Consider a setting where there are three observation patterns $O_{1} = [p]$, $O_{2} = [p-p_{0}]$, $O_{3} = [p-p_{0}-p_{1}]\cup\{[p]\setminus[p-p_{0}]\}$. We think of $n_{1}$ as small relative to $n_{2}$ and $n_{3}$, and $\log(p)$ as large relative to $\log(p_{1})$ and $\log(p_{0})$. We display this missingness pattern in Figure \ref{tikz:more complicated pattern}, where the covariates are split into their various modalities.
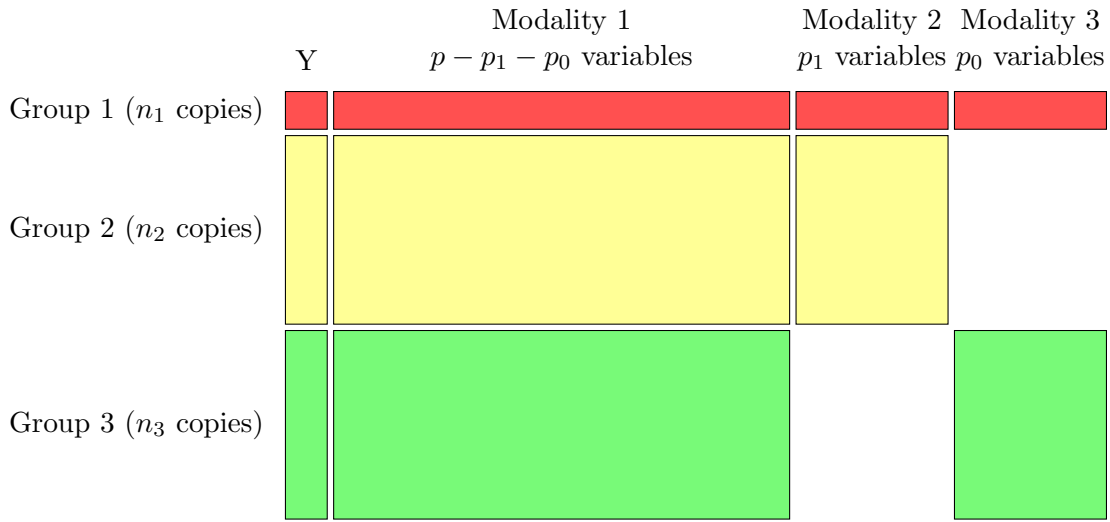
\begin{figure}[H]
\begin{center}
\begin{tikzpicture}

\definecolor{myred}{RGB}{255,80,80}
\definecolor{myyellow}{RGB}{255,255,150}
\definecolor{mygreen}{RGB}{120,250,120}

\matrix (m) [matrix of nodes,
  nodes={draw}, 
  column 1/.style={nodes={minimum width=0.55cm, minimum height=2cm}},   
  column 2/.style={nodes={minimum width=6cm, minimum height=2cm}},   
  column 3/.style={nodes={minimum width=2cm, minimum height=2cm}},   
  column 4/.style={nodes={minimum width=2cm, minimum height=2cm}}, 
  row 1/.style={nodes={minimum height=0.5cm}},   
  row 2/.style={nodes={minimum height=2.5cm}}, 
  row 3/.style={nodes={minimum height=2.5cm}},   
  column sep=2pt, row sep=2pt
] {
  |[fill=myred]| & |[fill=myred]| & |[fill=myred]| & |[fill=myred]| \\
  |[fill=myyellow]| & |[fill=myyellow]| & |[fill=myyellow]| & \\
  |[fill=mygreen]| & |[fill=mygreen]| & & |[fill=mygreen]| \\
};

\node[left=4pt of m-1-1] {Group 1 ($n_{1}$ copies)};
\node[left=4pt of m-2-1] {Group 2 ($n_{2}$ copies)};
\node[left=4pt of m-3-1] {Group 3 ($n_{3}$ copies)};

\node[above=4pt of m-1-1, align=center] {Y};
\node[above=4pt of m-1-2, align=center] {Modality 1 \\ $p-p_{1}-p_{0}$ variables};
\node[above=4pt of m-1-3, align=center] {Modality 2 \\ $p_{1}$ variables};
\node[above=4pt of m-1-4, align=center] {Modality 3 \\ $p_{0}$ variables};

\end{tikzpicture}
\caption{A non-monotonic pattern}
\label{tikz:more complicated pattern}
\end{center}
\end{figure}

The DISCOM method has an upper bound of
\begin{equation*}
    \frac{\|\beta^*\|_{1}^2s\log(p)}{n_{1}}.
\end{equation*}
By contrast, \eqref{eq:ModifiedDantzig} achieves a faster rate of 
\begin{equation*}
    \frac{\|\beta^*\|_{2}^2s\log(p-p_{1}-p_{0})}{n_{1}+\min\{n_{2},n_{3}\}}+\frac{\|\beta^*\|_{2}^2s\log(p_{1}+p_{0})}{n_{1}}.
\end{equation*}

\section{Simulations}\label{sec:Simulations}
We first investigate the empirical performance of our low-dimensional estimators from Section \ref{sec:Low-dimensional results} and our high-dimensional estimator from Section \ref{sec:High-dimensional results} through a range of simulation settings. Code for these simulations, and all other empirical work in this paper, is available via \url{https://warwick.ac.uk/fac/sci/statistics/staff/research_students/risebrow/}. The corresponding $\texttt{R}$ package is available on CRAN as LRMiss~\citep{LRMiss}.
\subsection{Simple monotonic pattern}
In order to investigate the properties of our estimators, we first consider the simple monotonic missingness pattern (Example \ref{examp:simple monotonic pattern}). Recall that, in the labelled dataset, we have $n_{1}$ complete cases and $n_{2}$ observations missing their final $p_{0}$ covariates. 

\subsubsection{Comparison with unweighted imputation}\label{sec:Comparison with single imputation}
In this section, we implement our procedure \eqref{eq:OSS estimator definition initial} in the ISS case with oracle weights. We compare this to a single imputation procedure that amounts to choosing the weights $\{\hat{D}_{k}\}_{k=1}^K$ in \eqref{eq:OSS estimator definition initial} to be 1. We choose $p = 10$, $n_{1} = 100$, $p_{0}=1$, $\beta^* = (1,1,\ldots,1,5)$, $\Sigma_{i,j} = 0.6^{|i-j|}$. We assume $X \sim N(0,\Sigma)$, $\epsilon \sim N(0,\sigma^2)$ and $\epsilon \indep X$ with the parameter $\sigma$ varying over $\{1,2,3\}$ across the plots. We plot the mean squared error of our procedure $\|\hat{\beta}-\beta^*\|_{2}^2$ relative to $n_{2}$. Each graph is formed by choosing $20$ points on the $x$-axis, repeating the simulation $1{,}000$ times at each value of $n_{2}$ and averaging the results.
\begin{figure}[H]
    \centering
    \includegraphics[
        width=1\linewidth,
        trim={0 0 0 1cm},
        clip
    ]{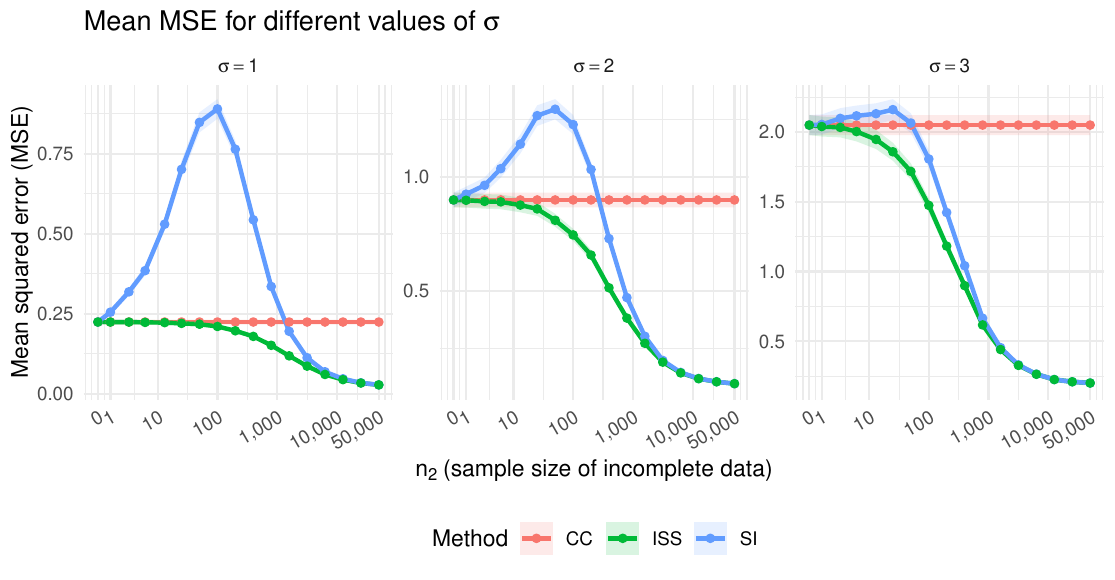}
    \caption{CC refers to a complete case analysis of the 100 complete cases via least squares. SI refers to the estimator \eqref{eq:OSS estimator definition initial} with choices of weights $\hat{D}_{1}=\hat{D}_{2}=1$. ISS refers to our estimator \eqref{eq:OSS estimator definition initial} with oracle weights $\hat{D}_{1}=1, \hat{D}_{2} = \frac{\sigma^2}{\sigma^2+(\beta^*_{M})^TS_M\beta^*_{M}}$, where $M=\{10\}$. Error bars are given by the ribbons.}
    \label{fig:Weighting Importance Example}
\end{figure}

These results highlight the need for the reweighting step in \eqref{eq:OSS estimator definition initial}. The leftmost plot has a large missing signal-to-noise ratio $\frac{(\beta^*_{M})^TS_M\beta^*_{M}}{\sigma^2}$. In this case, when $n_2$ is small to moderate, fitting the estimator without weights leads to a degradation in performance relative to the complete case estimator. For large $n_{2}$, however, the performance improves and matches that of our approach. In this setting, Theorem \ref{thm:Master OSS with estimating weights} predicts an improvement in performance for large $n_{2}$ of roughly a factor of 10 for our estimator relative to a complete case analysis by least squares. This is broadly what is observed in the simulations.

\subsubsection{Approximating the oracle weights}
This section considers the same data-generating mechanism as Section \ref{sec:Comparison with single imputation}, again in the ISS case with oracle weights. The previous section established the necessity of the weighting step. A natural question is how close the weights need to be to the oracle values to achieve good performance. For this missingness pattern, only the ratio $\hat{D}_{2}/\hat{D}_{1}$ affects the estimator. We thus consider deviations of this from the oracle value. Figure \ref{fig:2 and 4 weight plots} demonstrates that the specific choice of weight does not substantially matter beyond being of the correct order of magnitude. When the weight is set to be substantially larger than the oracle value, the performance can be worse than a complete case estimator. However, when the weight is much smaller than the oracle value, we still get an improvement on the complete case estimator. Note that in Figure~\ref{fig:2 and 4 weight plots}, for $\sigma=3$, the ratio $\hat{D}_{2}/\hat{D}_{1}$ is set to be larger than $1$ to produce the pink curve. We would never advocate for this in practice.

\begin{figure}[H]
    \centering
    \includegraphics[
        width=1\linewidth,
        trim={0 0 0 1cm},
        clip
    ]{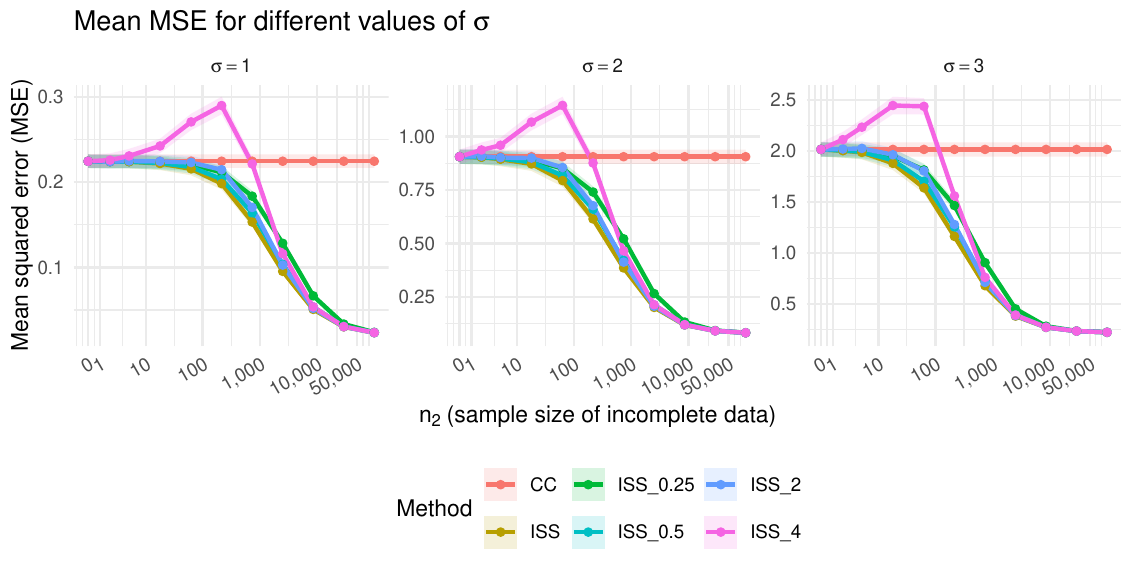}
    \caption{CC denotes the OLS estimate on complete cases. ISS is our proposed estimator \eqref{eq:OSS estimator definition initial} with oracle weights $\hat{D}_{1}=1, \hat{D}_{2}=\frac{\sigma^2}{\sigma^2+(\beta^*_{M})^TS_M\beta^*_{M}}$. For $c>0$, ISS$_c$ denotes the estimator with weights $\hat{D}_{1}=1, \hat{D}_{2}=\frac{c\sigma^2}{\sigma^2+(\beta^*_{M})^TS_M\beta^*_{M}}$. Error bars from 1{,}000 repetitions are shown by ribbons.}
    \label{fig:2 and 4 weight plots}
\end{figure}

\subsubsection{The OSS case}
We now focus on the OSS case with the same data-generating choices as for the previous two sections, changing only that $\sigma$ takes values in $\{2,4,6\}$. We explore how estimating the covariance from an unlabelled dataset of size $N$ affects the performance of the OSS estimator \eqref{eq:OSS estimator definition initial}. We initialise the weights using oracle values again. Clearly, it is not necessary to be close to the ISS setting to improve the performance relative to least squares. As predicted by the theory, as $\frac{\|\beta^*\|_{2}^2}{\sigma^2}$ decreases, the imputation error contributes less to the total mean squared error; for larger $\sigma$, this allows us to improve upon the performance of the complete case estimator even with a small number of unlabelled samples. For a sufficiently large unlabelled dataset, the OSS plots are indistinguishable from our ISS estimator.

\begin{figure}[H]
    \centering
    \includegraphics[
        width=1\linewidth,
        trim={0 0 0 1cm},
        clip
    ]{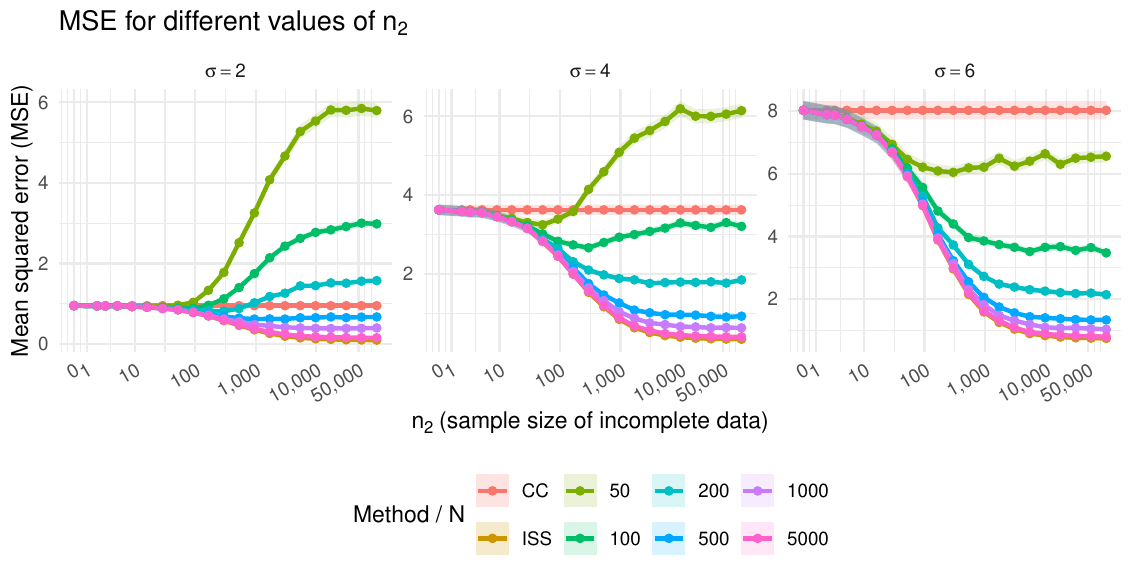}
    \caption{We compute our estimator \eqref{eq:OSS estimator definition initial} with unlabelled sample size $N$ varying from $50$ to $5{,}000$. ISS is the ideal semi-supervised estimator \eqref{eq:OSS estimator definition initial}. CC is the complete case estimator. Labelled sample sizes are $n_{1}=100$ and $n_{2}$ varying from $0$ to $100{,}000$. Error bars from 1{,}000 repetitions are shown by ribbons.}
    \label{fig:Simple monotonic OSS case}
\end{figure}

\subsection{Comparison to existing methods}\label{sec:Comparison to existing methods}
We now compare various supervised and semi-supervised approaches to this problem in the following simulation.  
We generate covariates $(X_1, \ldots, X_{10})$ from the distribution $N(0,\Sigma)$, where, for $(i,j)\in [p]^2$, $\Sigma_{i,j}=0.6^{|i-j|}$. The response is generated according to \eqref{eq:datamech1} with $\epsilon \sim N(0,\sigma^2)$, $\epsilon \indep X$ and we let $\sigma$ vary below. We set $\beta^* = (1,\ldots,1)\in \mathbb{R}^{10}$. We work in the simple monotonic pattern (Example \ref{examp:simple monotonic pattern}) with an unlabelled dataset of size $N=10{,}000$, $p_{0}=1$ and $n_1=500$ complete cases. We repeat each simulation $500$ times and compare the ratio between the complete case OLS estimator's MSE and every other estimator's MSE. A large value indicates good performance.  MICE$_{\mathcal{L}\cup\mathcal{U}}$ refers to multiple imputation via chained equations using the package associated with~\citet{van2011mice} with default options and where the missing responses in the unlabelled data are imputed; MICE$_{\mathcal{L}}$ refers to the same package but where the unlabelled data is ignored; SS refers to the method introduced in~\citet{azriel2022semi} incorporating only unlabelled data and ignoring the labelled data with missingness. CC-only  refers to the least squares estimator on the complete cases; RRZ refers to an implementation of the classical method introduced in~\citet{robins1994estimation} -- additional information is given in Appendix \ref{Appendix:Simulations}; OSS refers to the ordinary semi-supervised estimator \eqref{eq:OSS estimator definition initial} with weights estimated from the complete case estimator and covariance estimated with the unlabelled data.

\begin{table}[H]
\centering
\caption{Complete cases' MSE relative to each estimator's MSE with $\sigma=1$}
\label{tab:re_table}
\begin{tabular}{lccccc}
\toprule
\textbf{Methods} 
& \textbf{$n_{2}= 500$} 
& \textbf{$n_{2}= 1{,}000$} 
& \textbf{$n_{2}= 2{,}000$} 
& \textbf{$n_{2}= 5{,}000$} 
& \textbf{$n_{2}= 50{,}000$}\\
\midrule

MICE$_{\mathcal{L}\cup\mathcal{U}}$ & 1.526 (0.043)& 1.054 (0.033) & 0.358 (0.013) & 0.119 (0.003)& 0.056 (0.001) \\
MICE$_{\mathcal{L}}$ & 1.616 (0.038) & 1.985 (0.057)& 2.116 (0.059)  & 2.499 (0.080)& 2.815 (0.117)\\
SS & 0.968 (0.007) & 0.968 (0.007) & 0.968 (0.007)& 0.968 (0.007) & 0.968 (0.007) \\
RRZ & 1.684 (0.038)& 2.080 (0.058) & 2.359 (0.070) & 2.708 (0.073) & 3.188 (0.100)\\
CC  & 1&1 &1 &1 & 1\\
OSS &\textbf{1.839} (0.048)&\textbf{2.457} (0.071)&\textbf{3.461} (0.102) &\textbf{5.735} (0.229) &\textbf{12.164} (0.616)\\
\bottomrule
\end{tabular}
\end{table}

\begin{table}[H]
\centering
\caption{Complete cases' MSE relative to each estimator's MSE with $\sigma=3$}
\label{tab:re_table_2}
\begin{tabular}{lccccc}
\toprule
\textbf{Methods} 
& \textbf{$n_{2}= 500$} 
& \textbf{$n_{2}= 1{,}000$} 
& \textbf{$n_{2}= 2{,}000$} 
& \textbf{$n_{2}= 5{,}000$} 
& \textbf{$n_{2}= 50{,}000$}\\
\midrule

MICE$_{\mathcal{L}\cup\mathcal{U}}$ & 1.935 (0.054)& 2.527 (0.081) & 2.609 (0.131) & 1.089 (0.042)& 0.506 (0.013) \\
MICE$_{\mathcal{L}}$ & 2.120 (0.058) & 3.040 (0.105)& 3.833 (0.121)  & 6.052 (0.236)& 9.856 (0.485)\\
SS & 0.968 (0.007) & 0.968 (0.007) & 0.968 (0.007)& 0.968 (0.007) & 0.968 (0.007) \\
RRZ & 2.142 (0.057)& 3.117 (0.112) & 4.207 (0.134) & 6.558 (0.239) & 11.872 (0.610)\\
CC  & 1&1 &1 &1 & 1\\
OSS &\textbf{2.187} (0.061)&\textbf{3.153} (0.104)&\textbf{4.773} (0.166) &\textbf{8.765} (0.458) &\textbf{28.571} (2.005)\\
\bottomrule
\end{tabular}
\end{table}

The numbers in brackets are standard errors from repeating the simulation $500$ times. We see that our estimator performs the best in all of the scenarios. The other estimators are benchmarks that demonstrate the benefits of our approach, which uses all incomplete data. SS is primarily introduced to deal with misspecification and exhibits similar performance to CC. This is to be expected as the model is well-specified. Comparing with RRZ, which is semiparametrically efficient in the supervised setting, and MICE$_{\mathcal{L}}$, we see that the unlabelled data allows us to make much better use of the incomplete labelled data. Despite being the only other estimator besides OSS that makes use of the incomplete data and unlabelled data, MICE$_{\mathcal{L}\cup\mathcal{U}}$ yields very poor results. This is due to the estimator being severely biased, as discussed in \citet{li2024adaptive}.

\subsection{Comparison to other supervised imputation procedures}
We now consider the performance of our estimator when no unlabelled data is present. In this setting, we assume $p=10$, $\beta^*=(1,\ldots,1)$ and for $(i,j) \in [p]\times[p]$, $\Sigma_{i,j} = 0.6^{|i-j|}$. We generate data according to $X \sim N(0,\Sigma)$, $\epsilon \sim N(0,1)$ and $\epsilon \indep X$. As described in the figures below, we consider both a structured setting and an unstructured setting. Here, OSS refers to \eqref{eq:OSS estimator definition initial} where we estimate the covariance from the labelled data in a pairwise manner. More refined estimates of the covariance may give better empirical performance, depending on the distribution of the covariates. The weights are estimated from the initial consistent unweighted estimator. OSS\textunderscore CF is the same as OSS but with cross-fitting as outlined in the discussion after Theorem \ref{thm:weights initialisation guarantee}; CC refers to a least squares complete case analysis. Mean refers to mean imputation via the MICE package; Random refers to replacing missing entries by randomly sampling observed values in that column via the MICE package; kNN refers to a k-nearest neighbours rule for imputing missing values via the VIM package~\citep{kowarik2016imputation} with $k=5$; MI refers to multiple imputation using a random forest based procedure \cite{StekhovenBuhlmann2012}, with stochasticity introduced by aggregation and predictive mean matching \citep{Mayer2021missRangerArticle}. 
We see that in both simulations, our estimator outperforms the other imputation estimators. This is striking in the structured setting where all the other methods perform similarly to, or worse than, the standard complete case estimator. We see that the performance of our estimator is similar with and without cross-fitting.

\begin{figure}[H]
    \centering
    \begin{subfigure}{0.48\linewidth}
        \centering
        \includegraphics[
            width=\linewidth,
            trim=0 0 4cm 2cm,
            clip
        ]{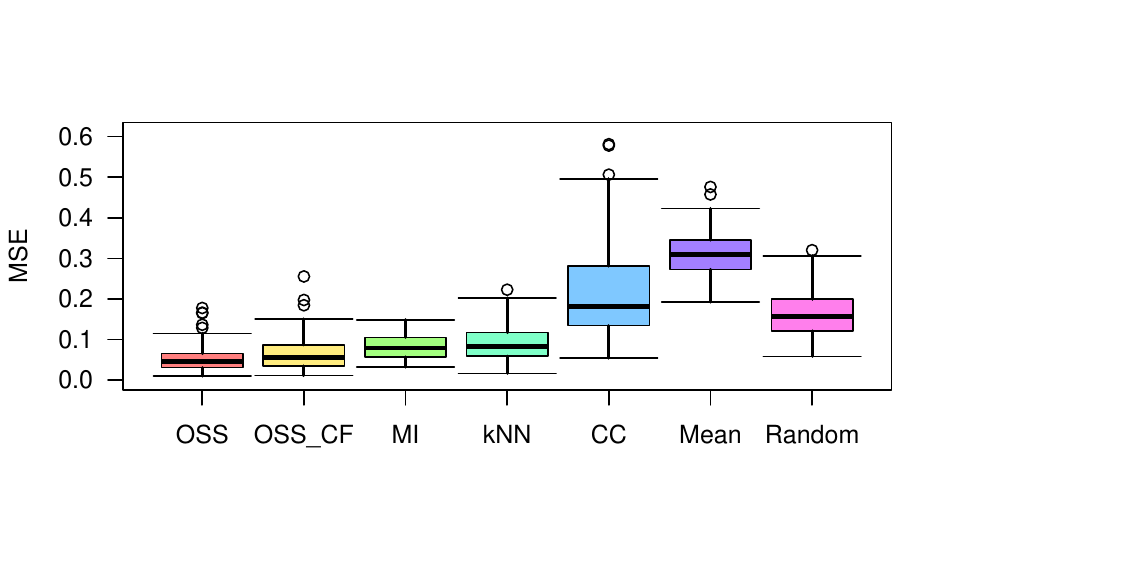}
        \caption{Unstructured missingness}
        \label{fig:unstructured-imputation-comparison}
    \end{subfigure}
    \hfill
    \begin{subfigure}{0.48\linewidth}
        \centering
        \includegraphics[
            width=\linewidth,
            trim=0 0 4cm 1.5cm,
            clip
        ]{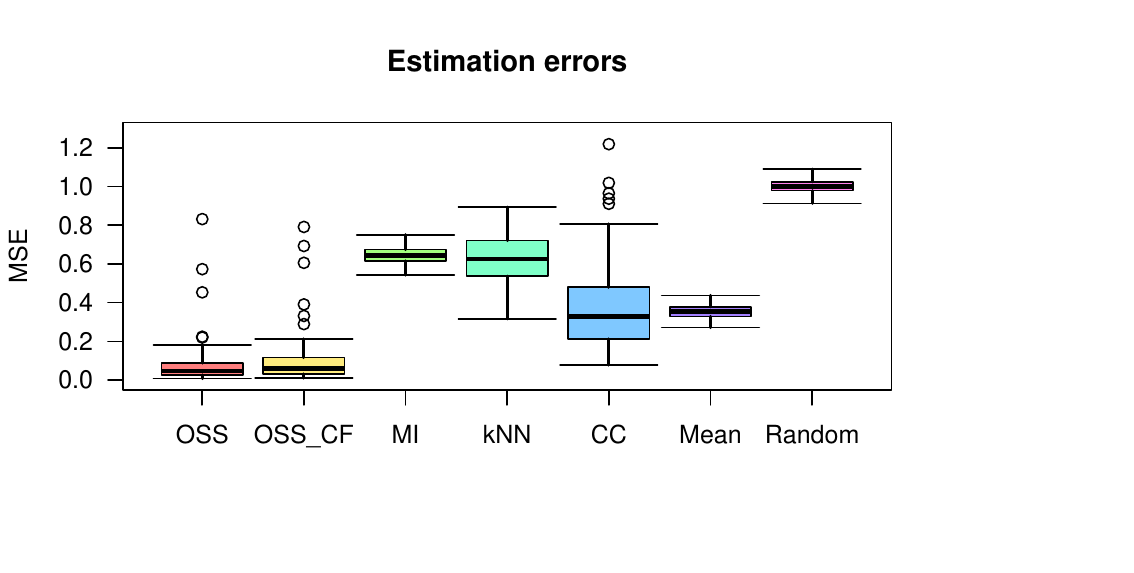}
        \caption{Structured missingness}
        \label{fig:structured-imputation-comparison}
    \end{subfigure}

    \caption{Comparison of imputation methods under structured and unstructured missingness patterns. In (a), 1{,}000 observations have each covariate missing independently with probability $0.2$. In (b), there are 50 complete cases, 500 observations missing the final covariate, and 4{,}500 missing the penultimate covariate. Results are based on $1{,}000$ repetitions.}
    \label{fig:imputation-pattern-comparison}
\end{figure}

\subsection{High-dimensional simulations}
We consider the performance of \eqref{eq:ModifiedDantzig} in both a structured setting and an unstructured setting. 
\subsubsection{Three cycle}
We consider a setting with three modalities. We use $100$ covariates and define three missingness patterns with $50$ samples missing the first $45$ covariates, $50$ samples missing the next $45$ covariates and $1{,}000$ samples missing their final $10$ variables. This is a challenging setting, since we have no complete cases and there are large imbalances in the number of times each modality is observed. $\beta^*$ is set to have $15$ entries equal to one and all other entries equal to zero, with the active set taken to be the first five variables in each modality. Setting $A=[90]$ and $B=A^c$, for each simulation, we generate a random matrix that has $\Sigma_{AA}=I_{|A|}$, $\Sigma_{BB}=I_{|B|}$ and $\Sigma_{AB}=0.3*uv^T$, where $u$ and $v$ are uniformly distributed on the unit sphere of appropriate dimension. The data is generated from \eqref{eq:datamech1} with $X \sim N(0,\Sigma)$, $\epsilon \sim N(0,\sigma^2)$ and $\epsilon \indep X$. We use our estimator \eqref{eq:ModifiedDantzig} with varying quantities of unlabelled data and apply five-fold cross-validation to choose the tuning parameter. All unlabelled data is assumed to be complete. We repeat the simulation 100 times and record a plot displaying the mean squared estimation error of each procedure. We vary $\sigma$ between the two plots. In both cases, particularly with smaller $\sigma$, the unlabelled data significantly improves the performance of the procedure.

\begin{figure}[H]
    \centering
    \begin{subfigure}{0.48\linewidth}
        \centering
        \includegraphics[width=\linewidth]{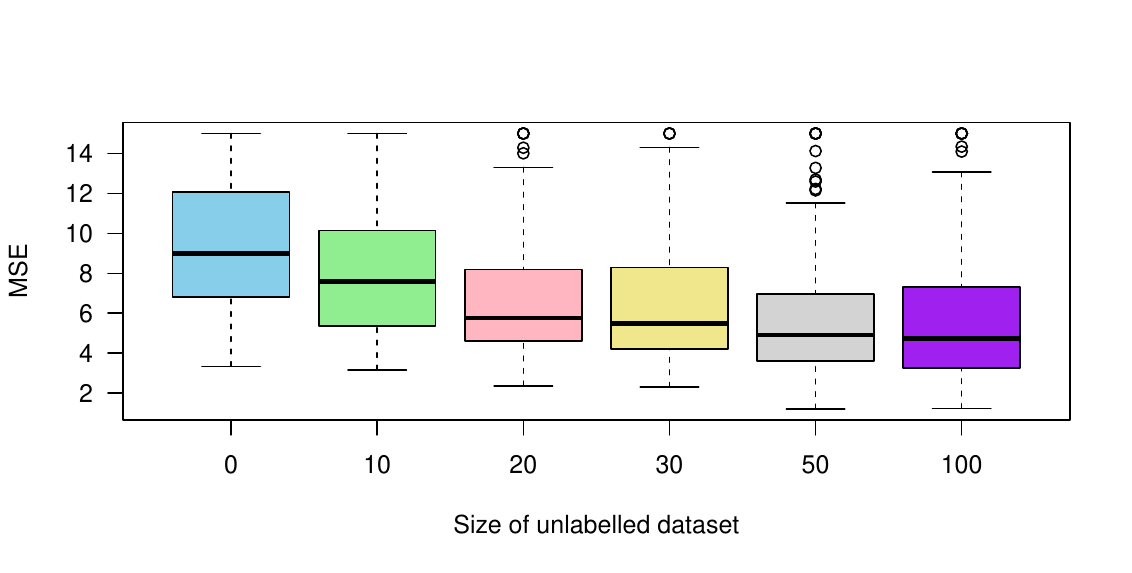}
        \caption{$\sigma = 3$}
        \label{fig:structured-sigma-3}
    \end{subfigure}
    \hfill
    \begin{subfigure}{0.48\linewidth}
        \centering
        \includegraphics[width=\linewidth]{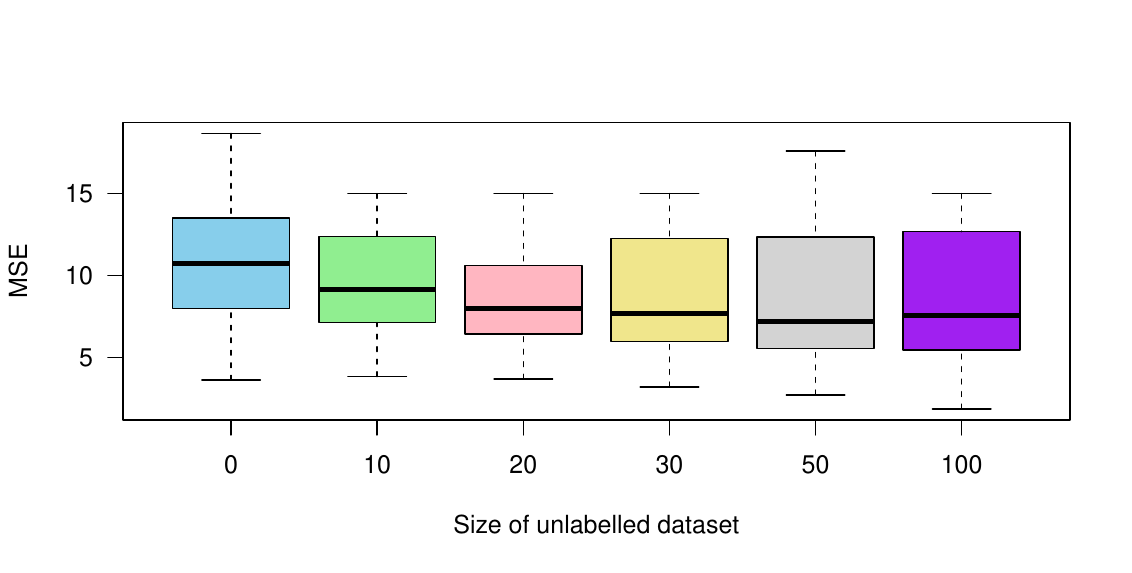}
        \caption{$\sigma = 5$}
        \label{fig:structured-sigma-5}
    \end{subfigure}

    \caption{MSE boxplots when \eqref{eq:ModifiedDantzig} is applied with the labelled dataset alongside unlabelled datasets of varying size.}
    \label{fig:structured-sigma-comparison}
\end{figure}

\subsubsection{Unstructured missingness}

In this simulation, we instead have a labelled dataset of size $1{,}000$ in $50$ dimensions with each covariate being observed independently with probability $0.2$. Setting $A=[45]$ and $B=A^c$, for each simulation, we generate a random matrix that has $\Sigma_{AA}=I_{|A|}$, $\Sigma_{BB}=I_{|B|}$ and $\Sigma_{AB}=0.3*uv^T$, where $u$ and $v$ are uniformly distributed on the unit sphere of appropriate dimension. We set $\beta^*$ to have its first five entries and last five entries to be equal to one with all other entries being zero. The data is generated according to \eqref{eq:datamech1} where $X \sim N(0,\Sigma)$, $\epsilon \sim N(0,\sigma^2)$ and $\epsilon \indep X$. We vary the number of unlabelled samples used by the procedure, assuming all covariates are observed in the unlabelled dataset. We apply five-fold cross-validation to select the regularisation parameter. The simulations are repeated $100$ times and MSE boxplots are displayed in Figure \ref{fig:unstructured-sigma-comparison}. Again, even a small number of unlabelled samples can substantially improve the performance of the procedure. As before, this improvement is greater with smaller $\sigma$. 

\begin{figure}[H]
    \centering
    \begin{subfigure}{0.48\linewidth}
        \centering
        \includegraphics[width=\linewidth]{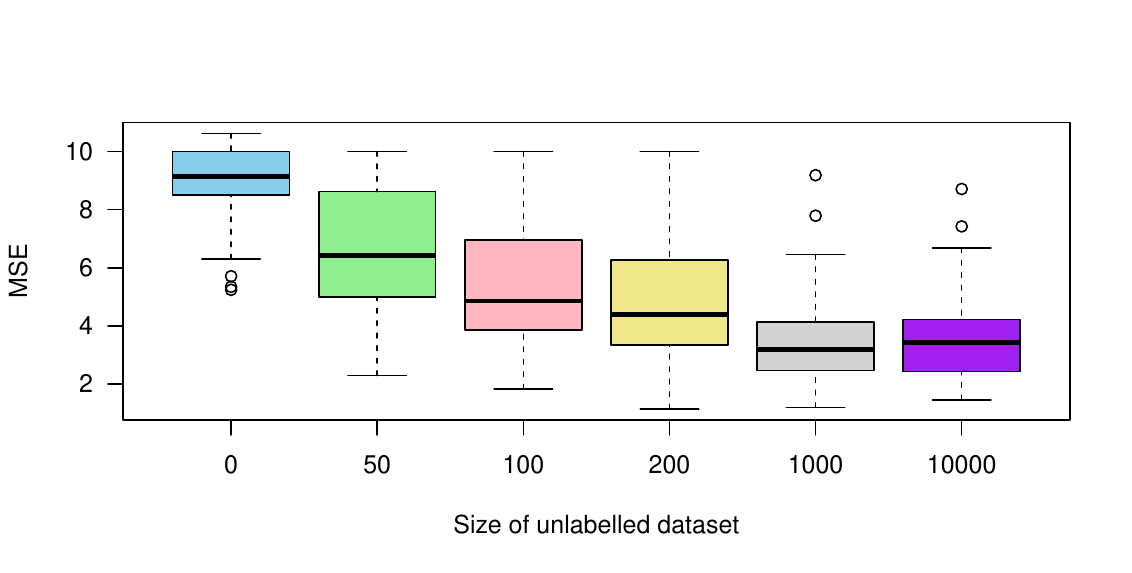}
        \caption{$\sigma = 3$}
        \label{fig:unstructured-3-sigma}
    \end{subfigure}
    \begin{subfigure}{0.48\linewidth}
        \centering
        \includegraphics[width=\linewidth]{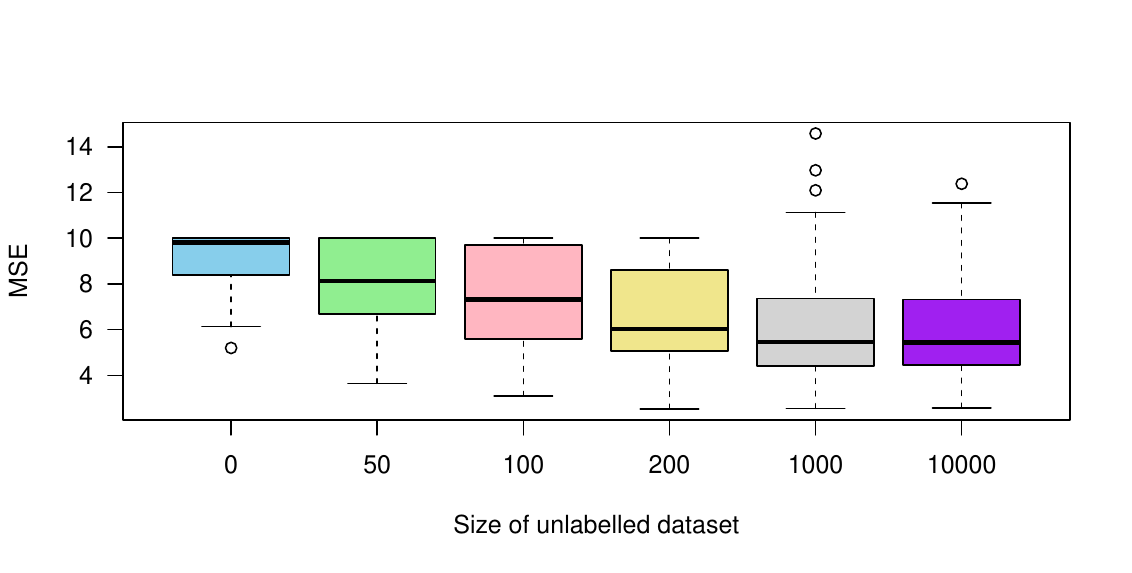}
        \caption{$\sigma = 5$}
        \label{fig:unstructured-5-sigma}
    \end{subfigure}
    \hfill
    \caption{MSE boxplots when \eqref{eq:ModifiedDantzig} is applied with the labelled dataset alongside unlabelled datasets of varying size.}
    \label{fig:unstructured-sigma-comparison}
\end{figure}

\section{Semi-synthetic real data application}\label{sec:Real World Dataset Application}
In this section, our primary goal is to illustrate real-data settings in which our estimator can offer substantial gains over the least squares estimator. We consider the California housing dataset~\citep{pace1997_sparse}. This dataset contains covariates for regions of California, which we group into the following modalities: demographic variables (median income and mean household size), geographic variables (latitude, longitude and ocean proximity) and property type variables (median house age, mean number of rooms, mean number of bedrooms). Our task is to regress median house value against these covariates. After the pre-processing steps outlined in Appendix \ref{Appendix:Real World Dataset Application}, we are left with 19,414 complete cases. We regress median house value against all the other variables via least squares and define this as the ground truth. Under MCAR mechanisms, we then introduce synthetic missingness to create two blockwise-missing settings. In both scenarios, we compare the OSS estimator to the least squares estimate on the complete cases. We plot the mean squared estimation errors for the coefficients, aggregated according to the modalities. Ribbons depicting approximate $95\%$ pointwise confidence intervals are displayed in all plots.

\subsection{Simple monotonic pattern}
In this setting, we consider the simple monotonic pattern (Example \ref{examp:simple monotonic pattern}) where in the labelled dataset we have either 400 or 1{,}000 complete cases and 5{,}000 cases missing the demographic variables. We vary the number of unlabelled samples that the OSS estimator has access to. Akin to the simulations in Section \ref{sec:Simulations}, the total mean squared error decreases as the unlabelled data size increases. With a large unlabelled dataset, the OSS estimator reduces the error in estimating the coefficients of the non-demographic variables. The difference is very small for the demographic variables, vanishing for a sufficiently large unlabelled dataset. The performance gap between the OSS and complete cases estimators is largest when the number of complete cases is small.

\begin{figure}[H]
    \centering

    \begin{subfigure}{0.48\linewidth}
        \centering
        \includegraphics[width=\linewidth]{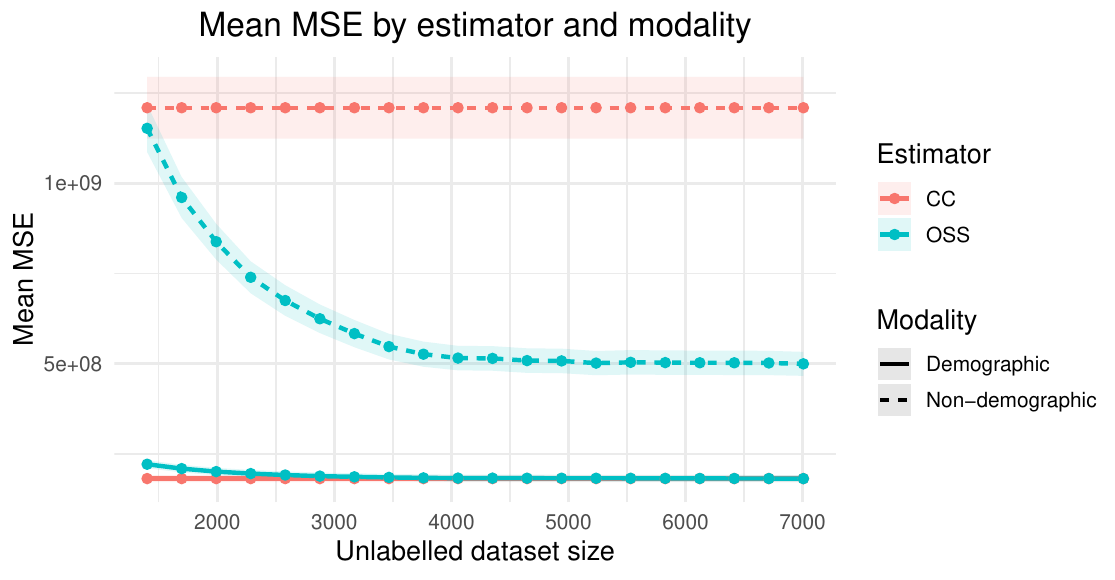}
        \label{fig:400 by modality}
    \end{subfigure}
    \hfill
    \begin{subfigure}{0.48\linewidth}
        \centering
        \includegraphics[width=\linewidth]{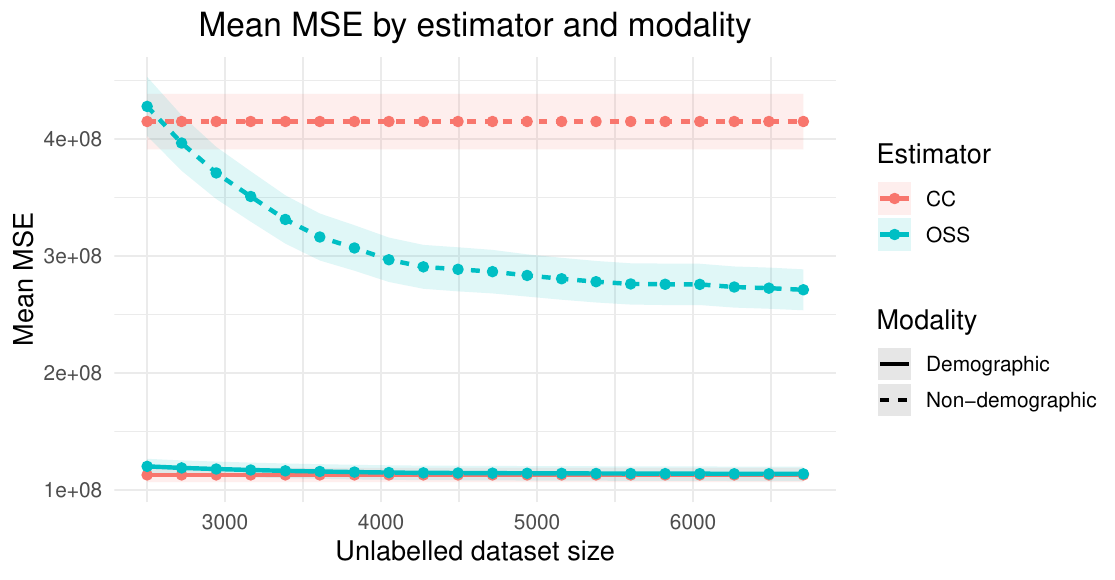}
        \label{fig:1000 by modality}
    \end{subfigure}

    \begin{subfigure}{0.48\linewidth}
        \centering
        \includegraphics[width=\linewidth]{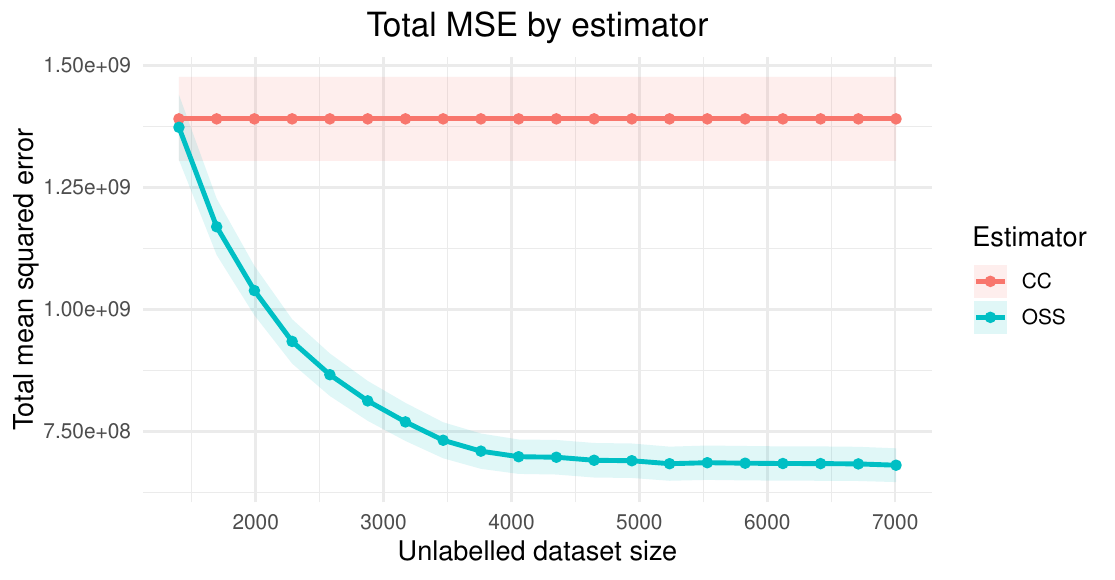}
        \label{fig:400 total MSE}
    \end{subfigure}
    \hfill
    \begin{subfigure}{0.48\linewidth}
        \centering
        \includegraphics[width=\linewidth]{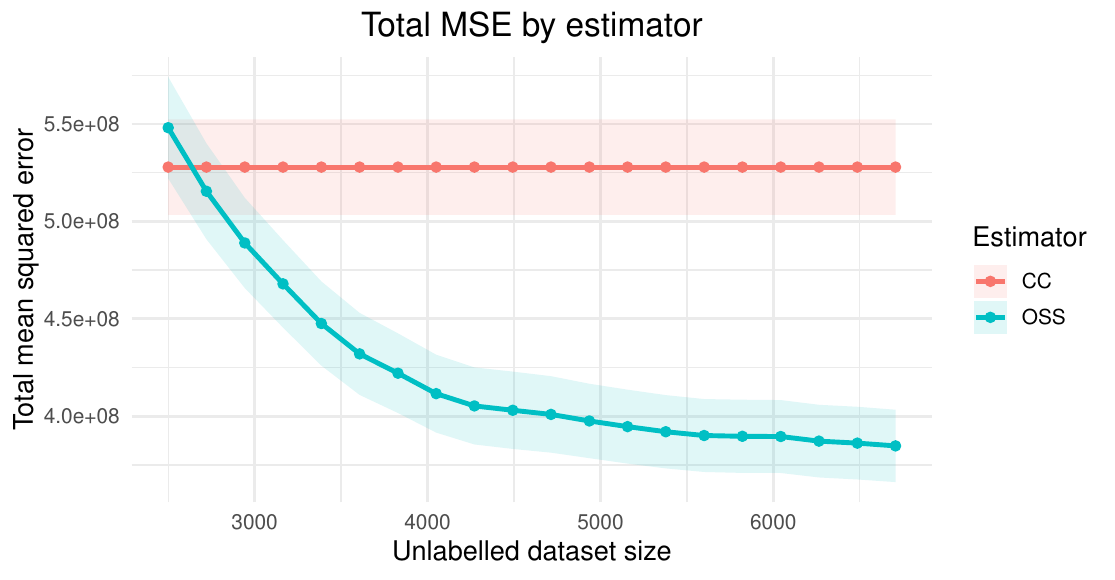}
        \label{fig:1000 total MSE}
    \end{subfigure}

    \caption{Average errors based on 1{,}000 repetitions with standard errors given as shaded ribbons. The top row of plots displays MSE by estimator, aggregated by the modality. The bottom row of plots displays the total MSE. The left-hand plots are with $400$ complete cases and the right-hand plots are with $1{,}000$ complete cases. The least squares complete case estimator is depicted in red and the OSS estimator \eqref{eq:OSS estimator definition initial} is displayed in blue.}
\end{figure}

\subsection{A more complicated pattern}
In this simulation, we consider a more complicated missingness pattern. We first assume that in the labelled dataset we have 1{,}000 complete cases, $m/2$ cases missing only demographic variables and $m/2$ cases missing only property type variables. This is exactly the setting of Figure \ref{tikz:more complicated pattern}. We consider both 10{,}000 and 5{,}000 unlabelled cases. Varying $m$, we compare the OSS estimator and the standard OLS complete case estimator. 

\begin{figure}[H]
    \centering
    \begin{subfigure}{0.48\linewidth}
        \centering
        \includegraphics[width=\linewidth]{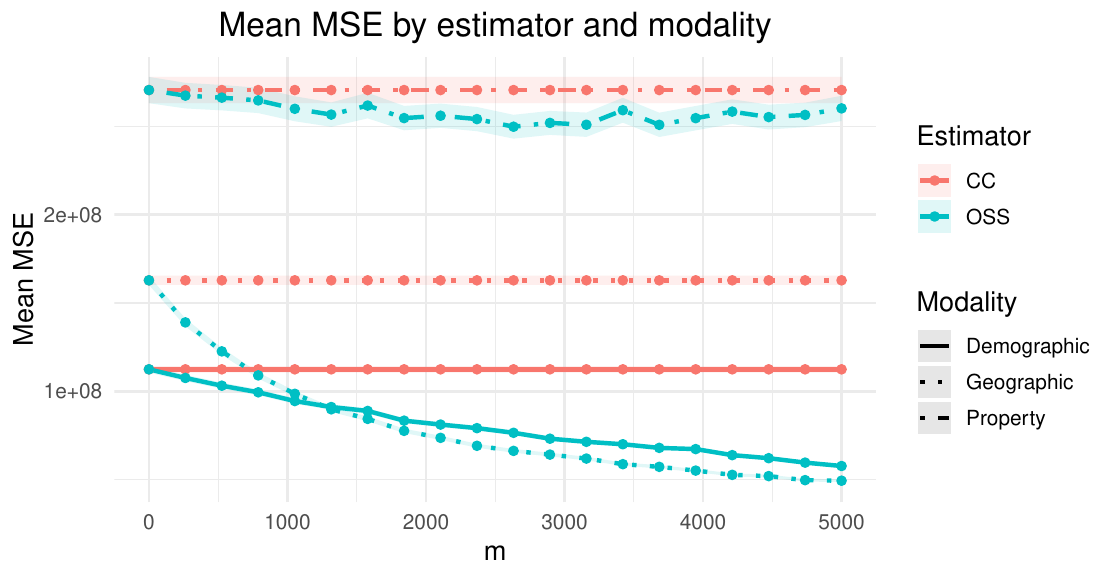}
    \end{subfigure}
    \hfill
    \begin{subfigure}{0.48\linewidth}
        \centering
        \includegraphics[width=\linewidth]{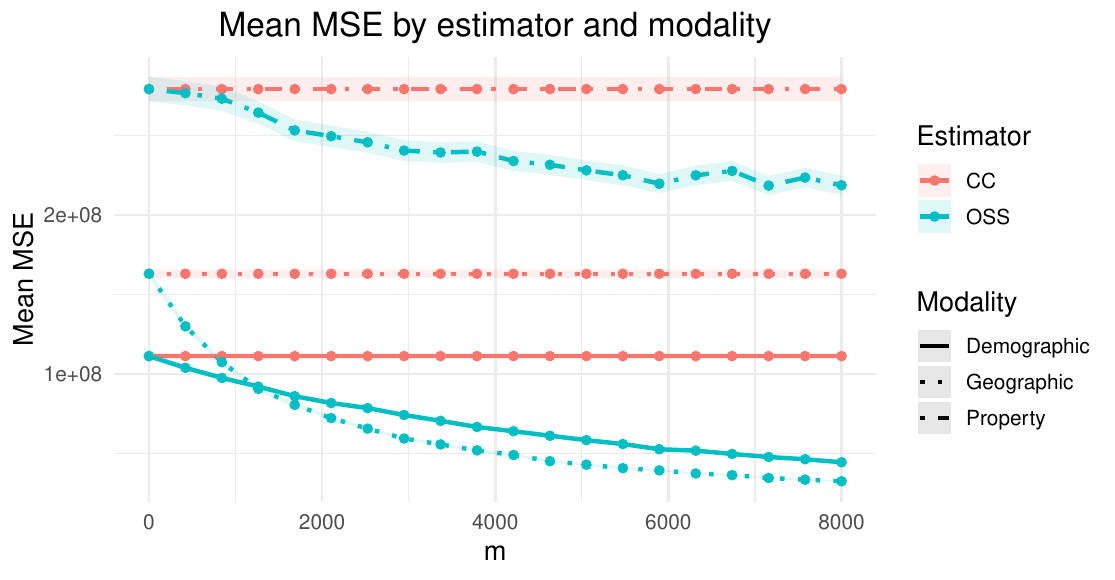}
    \end{subfigure}
    \begin{subfigure}{0.48\linewidth}
        \centering
        \includegraphics[width=\linewidth]{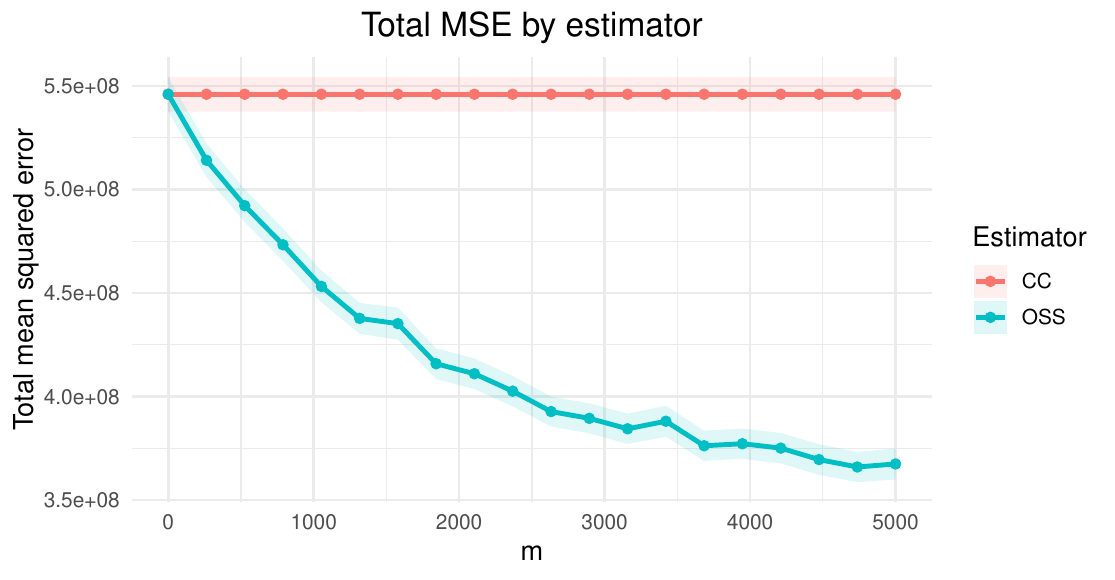}
    \end{subfigure}
    \hfill
    \begin{subfigure}{0.48\linewidth}
        \centering
        \includegraphics[width=\linewidth]{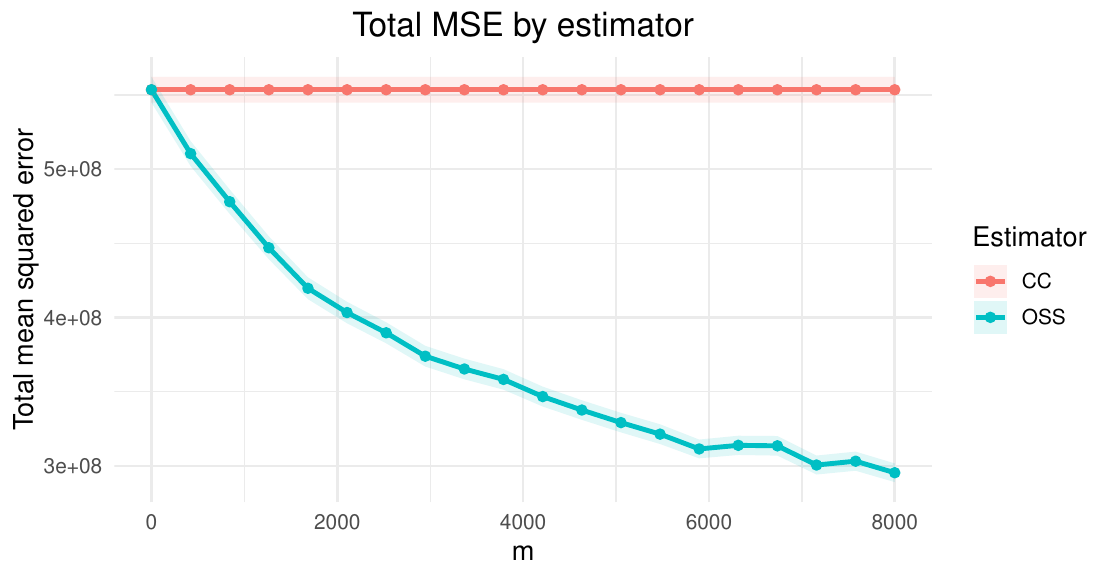}
    \end{subfigure}

    \caption{Average errors based on 10{,}000 repetitions with standard errors given as shaded ribbons. The top row of plots displays MSE by estimator, aggregated by the modality. The bottom row of plots displays the total MSE. The left-hand plots are with $5{,}000$ unlabelled cases and the right-hand plots are with $10{,}000$ unlabelled cases. The least squares complete case estimator is depicted in red and the OSS estimator \eqref{eq:OSS estimator definition initial} is displayed in blue.}
\end{figure}
Again, we see that the OSS estimator improves estimation performance relative to the complete case estimator for the coefficients in all three modalities. The most substantial gains are made in the geographic variables which are observed in all three missingness patterns. The improvements made by increasing the size of the unlabelled dataset are substantial. 

\section*{Acknowledgements}

The second author was supported by European Research Council Starting Grant 101163546.

\bibliography{references}  

\appendix
\section{Low-dimensional proofs}\label{sec:Low Dimensional Proofs}
Appendix \ref{sec:Low Dimensional Proofs} gives all the proofs in the low-dimensional setting and is split into four subsections. Section \ref{A:Preliminary Lemmas} derives some preliminary lemmas. Section \ref{A:Proofs of blockwise-missing OSS upper bounds} gives the proof of the upper bound for the estimator in the blockwise setting (Theorem \ref{thm:Master OSS with estimating weights}). It also includes results on estimating the covariance matrix at a sufficiently fast rate (Propositions \ref{Prop:Supervised covariance blocks}, \ref{Prop: unstructured cov estimation theorem} and \ref{Prop:OSS covariance blocks}). Section \ref{A:OSS Proofs under a balancing assumption} gives the proof of the upper bound for the estimator under a balancing assumption (Theorem \ref{thm:OSS balanced}). Section \ref{A:Lower Bounds} gives the proofs of the low-dimensional lower bounds (Theorems \ref{thm:LD structured lower bound} and \ref{thm:LD unstruc lower}).

\subsection{Preliminary Lemmas}\label{A:Preliminary Lemmas}
This section proves three preliminary lemmas (Lemmas \ref{lemma:hyper norm control}, \ref{lemma:eigenval control} and \ref{lemma:subgaussian tail control}).
For $k>2$, a distribution $P_{X}$ on $X \in \mathbb{R}^p$ satisfies $(k,2,C_{k})$--hypercontractivity if, for all $\theta \in \mathbb{R}^p$, $\mathbb{E}\left[|X^T\theta|^k\right]^{\frac{1}{k}} \leq C_{k}\mathbb{E}\left[(X^T\theta)^2\right]^{\frac{1}{2}}$. Observe that, for any $k >2$, Assumption~\ref{assump:subgauss} of sub-Gaussianity implies $(k,2,C_k)$ for some $C_k>0$.
\begin{lemma}\label{lemma:hyper norm control}
    Assume $P_{X}$ satisfies $(k,2,C_{k})$--hypercontractivity and that $\|\mathbb{E}\left[XX^T\right]\|_\mathrm{op} \leq \lambda_{+}$. Then, for any $1 \leq r \leq k$, it holds that 
   \begin{equation*}
    \mathbb{E}\left[\|X\|^{r}_{2}\right] \leq C_{k}^rp^{\frac{r}{2}}\lambda_{+}^{\frac{r}{2}}.
    \end{equation*}

\end{lemma}
\begin{proof}
    For $r \geq 2$, a standard property of the $\|\cdot\|_{r}$ norm is that $\|x\|_2 \leq p^{1/2-1/r} \|x\|_r$. For $2\leq r \leq k,$ it holds that
\begin{align*}
    \mathbb{E}\left[\|X\|_2^r\right] &\leq p^{r/2-1} \mathbb{E}\left[\|X\|_r^r\right] = p^{r/2} \frac{1}{p} \sum_{j=1}^p \mathbb{E}\left[ |\langle X, e_j\rangle|^r\right] \\&\leq p^{r/2} \frac{1}{p} \sum_{j=1}^p \mathbb{E}\left[ |\langle X, e_j\rangle|^k\right]^{\frac{r}{k}}\leq p^{r/2} \frac{1}{p} \sum_{j=1}^p C_k^r \{\mathbb{E}\left[X_j^2\right]\}^{r/2} \leq C_k^r p^{r/2} \lambda_{+}^{r/2},
\end{align*}
where the first inequality follows because $r \geq 2$, the second inequality follows by Jensen's inequality, the third inequality follows from the hypercontractivity condition and the final inequality holds since $\mathbb{E}\left[X_{j}^2\right]=e_{j}^T\mathbb{E}\left[XX^T\right]e_{j} \leq \lambda_{+}$.

For the case $r \in [1,2]$, it holds that
\begin{equation*}
    \mathbb{E}\left[\|X\|_2^r\right] \leq \mathbb{E}^{r/2}\left[\|X\|_2^2\right] = \Tr(\Sigma)^{\frac{r}{2}} \leq  p^{r/2} \lambda_{+}^{r/2},
\end{equation*}
where the first inequality holds by Jensen's inequality.
\end{proof}
Define for each $i \in [n_{\mathcal{L}}]$,
\begin{align*}
\hat{\lambda}_{\min} &\equiv \min_{j=1,\ldots,K} \lambda_{\min} \left( \frac{\Sigma_{O_{j}O_{j}}^{-\frac{1}{2}}X_{O_{j}}^T X_{O_{j}}\Sigma_{O_{j}O_{j}}^{-\frac{1}{2}}}{n_{j}} \right) \\&\equiv \min_{j=1,\ldots,K}\lambda_{\min} \left( \hat{\Sigma}_{j} \right),
\end{align*} where $X_{O_{j}}$ is the observed design matrix in the $j^{th}$ missingness pattern. The following lemma controls inverse powers of this quantity in expectation. 
\begin{lemma}\label{lemma:eigenval control}
Assume that $P_{X}$ satisfies Assumption \ref{assump:small-ball}. For $a>0$, provided $1\leq r \leq a$ and $n_{\min} \geq \max \left\{\frac{12r}{\chi},\frac{6p}{\chi}\right\}$, it holds that
    \begin{equation*}
        \mathbb{E}\left[\hat{\lambda}_{\min}^{-r}\right]\leq C'(\chi,C,a),
    \end{equation*}
    for some constant $C'=C'(\chi,C,a)$.
\end{lemma}
\begin{proof}
    For $q \in [1,\infty)$ and for a real-valued random variable $Y$, define $\|Y\|_{\varphi_q} = \mathbb{E}\left[|Y|^q\right]^{\frac{1}{q}}$. By Corollary 4 of~\citet{mourtada2022exact}, we have that for $q =\frac{\chi n_{\min}}{12r}$,
\begin{equation*}
    \left\|\frac{1}{\lambda_{\min}(\hat{\Sigma}_{j})^r}\right\|_{\varphi_q} \leq C'(\chi,C,a),
\end{equation*}
for some constant $C'=C'(C,\chi)$. We allow this constant to grow throughout the proof. Hence, by ~\citet[section~5.4]{pollard2024pttm}, for $r \leq a$, it holds that
\begin{align*}
    \mathbb{E}\left[\max_{j \in [K]}\left\{\frac{1}{\lambda_{\min}(\hat{\Sigma}_{j})^r} \right\}\right]&\le C'(\chi,C,a) (K)^{\frac{12r}{\chi n_{\min}}}\\
    &\leq C'(\chi,C,a) (2^{\frac{p}{n_{\min}}})^{\frac{12r}{\chi}}\\
    &\leq C'(\chi,C,a),
\end{align*}
where the last two lines follow from the fact that $K \leq 2^p$ and $p \leq n_{\min}$. 
\end{proof}

\begin{lemma}\label{lemma:subgaussian tail control}
    Let $\hat{\Sigma}=\frac{1}{n}\sum_{i=1}^nX_{i}X_{i}^T$ be the sample second moment matrix of $n$ i.i.d copies of $X$. Further assume that $\|X\|_{\psi_{2}}\leq C_{X}$, $n \geq p$ and $\lambda_{\min}(\Sigma)>0$. For $a\geq 2$, there exists a constant $C=C(C_{X},a,\lambda_{\min}(\Sigma),\lambda_{\max}(\Sigma))$ such that
    \begin{align*}
        \mathbb{E}\left[\|\hat{\Sigma}-\Sigma\|_{\mathrm{op}}^{a}\right] \leq C\left(\frac{p}{n}\right)^{\frac{a}{2}},
    \end{align*}
    and that
    \begin{align*}
        \mathbb{E}\left[\lambda_{\max}\left(\hat{\Sigma}\right)^{a}\right] \leq C.
    \end{align*}
    Let $A=c\frac{C_{X}^2}{\lambda_{\min}(\Sigma)}\|\Sigma\|_{\mathrm{op}}$ for some sufficiently large universal constant $c$ and let $b = \sqrt{\frac{p}{n}}+\frac{p}{n}$. For $x \geq bA$, it holds that
    \begin{equation*}
        \mathbb{P}\left(\|\hat{\Sigma}-\Sigma\|_{\mathrm{op}}\geq x\right) \leq 2e^{-\frac{n}{4}\min\left\{\frac{x}{A}-b,\left(\frac{x}{A}-b\right)^2\right\}}.
    \end{equation*}
\end{lemma}
\begin{proof}
We have that for all $v \in \mathcal{S}^{p-1}$
\begin{align*}
    \left\|\langle X, v \rangle\right\|_{\psi_{2}} &\leq C_{X}\\
    &\leq \frac{C_{X}\sqrt{v^T\Sigma v}}{\sqrt{\lambda_{\min}(\Sigma)}}.
\end{align*}
Thus, for some universal constant $c$, \citet[Exercise 4.49]{vershynin2025HDP} yields the following
\begin{equation*}
    \mathbb{P}\left(\|\hat{\Sigma}-\Sigma\|_{\mathrm{op}}\geq c\frac{C_{X}^2}{\lambda_{\min}(\Sigma)}\|\Sigma\|_{\mathrm{op}}\left(\sqrt{\frac{p+u}{n}}+\frac{p+u}{n}\right)\right) \leq 2e^{-u}.
\end{equation*}
Since $\sqrt{\frac{p+u}{n}}\leq \sqrt{\frac{p}{n}}+\sqrt{\frac{u}{n}}$, adjusting constants yields that
\begin{equation*}
    \mathbb{P}\left(\|\hat{\Sigma}-\Sigma\|_{\mathrm{op}}-c\frac{C_{X}^2}{\lambda_{\min}(\Sigma)}\|\Sigma\|_{\mathrm{op}}\left(\sqrt{\frac{p}{n}}+\frac{p}{n}\right)\geq c\frac{C_{X}^2}{\lambda_{\min}(\Sigma)}\|\Sigma\|_{\mathrm{op}}\left(\sqrt{\frac{u}{n}}+\frac{u}{n}\right)\right) \leq 2e^{-u}.
\end{equation*}
Define $A\equiv c\frac{C_{X}^2}{\lambda_{\min}(\Sigma)}\|\Sigma\|_{\mathrm{op}}$, $b = \sqrt{\frac{p}{n}}+\frac{p}{n}$ and $t=A\left(\sqrt{\frac{u}{n}}+\frac{u}{n}\right)$. It holds that either $\frac{Au}{n}\geq \frac{t}{2}$ or $A\sqrt{\frac{u}{n}}\geq \frac{t}{2}$. It follows that $u \geq \min\left\{\frac{tn}{2A},\frac{t^2n}{4A^2}\right\}$. The above yields that
\begin{equation}\label{eq:sub gauss prob bound}
    \mathbb{P}\left(\|\hat{\Sigma}-\Sigma\|_{\mathrm{op}}-Ab \geq t\right) \leq 2e^{-\frac{n}{4}\min\left\{\frac{t}{A},\frac{t^2}{A^2}\right\}}.
\end{equation}
The final part of the Lemma follows immediately.
We have that
\begin{equation}\label{eq: subgauss decomposition}
    \|\hat{\Sigma}-\Sigma\|_{\mathrm{op}}\leq Ab + \left(\|\hat{\Sigma}-\Sigma\|_{\mathrm{op}}-Ab\right)_{+} \equiv Ab+ Z.
\end{equation}
We also have that
\begin{align}
    \mathbb{E}\left[Z^{a}\right] &= \int_{0}^{\infty}\mathbb{P}\left(Z \geq t^{\frac{1}{a}}\right)dt \nonumber\\
    & = a \int_{0}^{\infty}\mathbb{P}\left(Z \geq u\right)u^{a-1}du \nonumber\\
    &\leq 2a\int_{0}^{\infty}e^{-\frac{n}{4}\min\{\frac{u}{A},\frac{u^2}{A^2}\}}u^{a-1}du \nonumber\\
    &=2a\int_{A}^{\infty}e^{-\frac{nu}{4A}}u^{a-1}du+2a\int_{0}^{A}e^{-\frac{nu^2}{4A^2}}u^{a-1}du\nonumber\\
    &\leq 2a\left(\frac{4A}{n}\right)^a\int_{0}^{\infty}e^{-z}z^{a-1}dz +2a\left(\frac{2A}{\sqrt{n}}\right)^{a}\int_{0}^{\infty}e^{-z}z^{\frac{a}{2}-1}dz\nonumber\\
    &\leq \frac{C(C_{X},a,\lambda_{\min}(\Sigma),\lambda_{\max}(\Sigma))}{n^{\frac{a}{2}}},\nonumber
\end{align}
where the third line follows from \eqref{eq:sub gauss prob bound}. Applying the previous display and the fact that for $a,b,r>0$, $(a+b)^{r}\leq 2^{r-1}(a^r+b^r)$ to \eqref{eq: subgauss decomposition} yields the first statement. In particular, we have that
\begin{align*}
    \mathbb{E}\left[\|\hat{\Sigma}-\Sigma\|_{\mathrm{op}}^{a}\right] &\leq 2^{a-1}\left((Ab)^{a}+\mathbb{E}\left[Z^a\right]\right)\\
    &\leq C(C_{X},a,\lambda_{\min}(\Sigma),\lambda_{\max}(\Sigma))\left(\frac{p}{n}\right)^{\frac{a}{2}}.
\end{align*}
 The second statement follows straightforwardly from the first. By the triangle inequality, the previous display and since $(a+b)^{r}\leq 2^{r-1}(a^r+b^r)$, it holds that
\begin{equation*}
    \mathbb{E}\left[\lambda_{\max}(\hat{\Sigma})^{a}\right]=\mathbb{E}\left[\|\hat{\Sigma}\|_{\mathrm{op}}^{a}\right] \leq 2^{a-1}\left(\mathbb{E}\left[\|\hat{\Sigma}-\Sigma\|_{\mathrm{op}}^{a}\right]+\|\Sigma\|_{\mathrm{op}}^{a}\right)\leq C.
\end{equation*}
\end{proof}

\subsection{Proofs of blockwise-missing OSS upper bounds}\label{A:Proofs of blockwise-missing OSS upper bounds}
The main goal of this section is to prove an upper bound for the estimator in the blockwise-missing setting (Theorem \ref{thm:Master OSS with estimating weights}). The key result to show this is Lemma \ref{lemma:squared matrix bound}.  Towards the end of this section we prove Theorem \ref{thm:weights initialisation guarantee}, regarding how to approximate the optimal weights. See the results after Lemma \ref{lemma:supervised Covariance Matrix Estimation} for more details on how to estimate the covariance matrix at the desired rate.
\begin{lemma}\label{lemma:squared matrix bound}
    Let $K \in \mathbb{N}$, $0<t_{1}\leq t_{2} \leq \cdots \leq t_{K}$, $r_{1},\ldots,r_{K} \in \mathbb{N}$ and for $i \in [K]$, let $C_{i} \in \mathbb{R}^{r_{i}\times r_{i}}$ be a symmetric positive-definite matrix satisfying $\mu_{C}I_{r_{i}} \preceq C_{i} \preceq \Omega_{C}I_{r_{i}}$ for $\mu_{C},\Omega_{C}>0$. Assume there exists a symmetric positive-definite matrix $\Sigma \in \mathbb{R}^{p\times p}$ such that $R_{i} = \Sigma_{O_{i}}^T\in\mathbb{R}^{p\times r_{i}}$ for some set $O_{i}\subseteq [p]$. Furthermore, we assume that $\cup_{i=1}^K O_{i} = [p]$. Then, letting $S = t_{1}R_{1}C_{1}R_{1}^T+\cdots+t_{K}R_{K}C_{K}R_{K}^T$, for $l \in [K]$ we have
    \begin{equation*}
    \|R_{l}^TS^{-2}R_{l}\|_\mathrm{op}=\|S^{-1}R_{l}\|_\mathrm{op}^2 
    \leq F(K)\frac{\lambda^6_{\max}(\Sigma)}{\lambda^8_{\min}(\Sigma)}\frac{\max\{1,\Omega_{C}^{2K}\}}{t_{l}^2\min\{\mu_{C}^{2K},1\}},
    \end{equation*}
    for some constant $F=F(K)$.
\end{lemma}

\begin{proof}
    We induct on $K$. For $K=1$, we necessarily have $R_{1}=\Sigma$ and the statement follows, since
    \begin{align*}
        \|(t_{1}\Sigma C_{1}\Sigma)^{-1}\Sigma\|_\mathrm{op}^2 & = \frac{1}{t_{1}^2}\|\Sigma^{-1}C_{1}^{-1}\|_\mathrm{op}^2\\
        & \leq \frac{1}{t_{1}^2}\|\Sigma^{-1}\|_\mathrm{op}^2\|C_{1}^{-1}\|_\mathrm{op}^2\\
        & \leq \frac{1}{\mu_{C}^2t_{1}^2\lambda^2_{\min}(\Sigma)}.
    \end{align*}
    We now assume that the statement holds up to $K-1$ for all appropriate $\Sigma$ and consider the statement for $K$. For $i \in [K]$, define $E_{i} \in \mathbb{R}^{p \times r_{i}}$ as $E_{i}=\Sigma^{-1}R_{i}$, i.e., $E_{i}$ `selects' certain columns of $\Sigma$. This ensures that $E_iE_i^T$ is a diagonal matrix with ones on the diagonal entries corresponding to $O_i$ and zeroes elsewhere. Let $t_{M}$ be the largest element of $\{t_{i}\}_{i=1}^K$ such that $2t_{1} > t_{M}$. Necessarily, it holds that $M \geq 1$. For $l \in \{1,\ldots,M\}$, note that \begin{align*}
        S \succeq t_{1}\Sigma\left(E_{1}C_{1}E_{1}^T+\cdots+ E_{K}C_{K}E_{K}^T\right)\Sigma \succeq t_{1}\mu_{C}\Sigma\left(E_{1}E_{1}^T+\cdots+ E_{K}E_{K}^T\right)\Sigma \succeq \frac{t_{l}}{2}\mu_{C}\lambda_{\min}^2(\Sigma)I_{p},
    \end{align*} 
    where the final bound follows since $\cup_{i=1}^K O_i = [p]$. Therefore, the appropriate bound on $\|S^{-1}R_{l}\|_\mathrm{op}^2$ holds for these $l$. If $M = K$ then the result holds for all $l$. Therefore, it suffices to consider cases for which $M < K$. 
    We rewrite $S$ as 
    \begin{align*}
        S &= \sum_{l=1}^M t_{l}R_{l}C_{l}R_{l}^T+\sum_{l=M+1}^Kt_{1}R_{l}C_{l}R_{l}^T+\sum_{l=M+1}^KR_{l}C_{l}R_{l}^T(t_{l}-t_{1})\\
        & \equiv A + \sum_{l=M+1}^KR_{l}C_{l}R_{l}^T(t_{l}-t_{1}).
    \end{align*}
    For a set $Q\subseteq [p]$ and $k \in [|Q|]$ define $(Q)_{(k)}$ as the $k^{th}$ largest element of $Q$. Let $\cup_{l=M+1}^KO_{l} = N$ and $p' = |N|$ and for $l \in [K]\setminus [M]$ define $E_{l}'\in \mathbb{R}^{p'\times r_{l}}$ via $(E_{l}')_{jk}=\mathbbm{1}_{\{(N)_{(j)}=(O_{l})_{(k)}\}}$ for $j \in [p'], k \in [r_{l}]$. It holds that $\Sigma_{N}^TE_{l}'=R_{l}$ for $l \in \{M+1,\ldots,K\}$, i.e., $E_{l}'$ again `selects' the relevant columns of $\Sigma$. Additionally, it holds that $(K-M)I_{p'}\succeq \sum_{l=M+1}^KE_{l}'(E_{l}')^T \succeq I_{p'}$. In this notation, we rewrite $S$ as 
    \begin{align*}
        S & = A + \Sigma_{N}^T\left(\sum_{l=M+1}^K(t_{l}-t_{1})E_{l}'C_{l}(E_{l}')^T\right)\Sigma_{N}\\
        & \equiv A+\Sigma_{N}^TT\Sigma_{N}.
    \end{align*}
    By the Woodbury matrix identity, it holds that 
    \begin{equation*}
        S^{-1}=A^{-1}-A^{-1}\Sigma_{N}^T\left(T^{-1}+\Sigma_{N}A^{-1}\Sigma_{N}^T\right)^{-1}\Sigma_{N}A^{-1}.
    \end{equation*}
    Multiplying by $R_{l}$ for $l \in [K]\setminus [M]$ yields
    \begin{align*}
        S^{-1}R_{l} &= A^{-1}R_{l}-A^{-1}\Sigma_{N}^T\left(T^{-1}+\Sigma_{N}A^{-1}\Sigma_{N}^T\right)^{-1}\Sigma_{N}A^{-1}\Sigma_{N}^TE_{l}'\\
        &=A^{-1}\Sigma_{N}^T\left(T^{-1}+\Sigma_{N}A^{-1}\Sigma_{N}^T\right)^{-1}T^{-1}E_{l}'.
    \end{align*}
    By the submultiplicativity of the operator norm and since $\left(T^{-1}+\Sigma_{N}A^{-1}\Sigma_{N}^T\right)^{-1} \preceq \left(\Sigma_{N}A^{-1}\Sigma_{N}^T\right)^{-1}$, we have that
    \begin{align}\label{eq:squared operator penultimate bound}
        \|S^{-1}R_{l}\|_\mathrm{op} \leq \|A^{-1}\|_\mathrm{op}\|\Sigma_{N}^T\|_\mathrm{op}\|\left(\Sigma_{N}A^{-1}\Sigma_{N}^T\right)^{-1}\|_\mathrm{op}\|T^{-1}E_{l}'\|_\mathrm{op}.
    \end{align}
    We bound each of these terms in turn. Clearly, it holds that 
    \begin{equation}\label{eq:squared A bound}
        t_{1}\lambda_{\min}^2(\Sigma)\mu_{C}I_{p}\preceq A\preceq 2Kt_{1}\lambda_{\max}^2(\Sigma)\Omega_{C}I_{p},
    \end{equation}
    and that 
    \begin{equation}\label{eq: R^2 bound}
        \lambda_{\min}(\Sigma_{N}\Sigma_{N}^T) = \inf_{\substack{x \in \mathcal{S}^{p-1}\\ x_{N^c}=0}} x^T \Sigma \Sigma^T x \geq \lambda_{\min}^2(\Sigma).
    \end{equation}
    By the inductive hypothesis (noting that each $E_{l}'$ is formed from columns of $I_{p'}$ which has unit eigenvalues), we have that
    \begin{equation}\label{eq:squared inductive hypothesis}
        \|T^{-1}E_{l}'\|_\mathrm{op} \leq F(K)\frac{\max\{1,\Omega_{C}^{(K-1)}\}}{(t_{l}-t_{1})\min\{1,\mu_{C}^{(K-1)}\}}.
    \end{equation}
    Combining \eqref{eq:squared inductive hypothesis}, \eqref{eq: R^2 bound} and \eqref{eq:squared A bound} into \eqref{eq:squared operator penultimate bound} yields
    \begin{align*}
        \|S^{-1}R_{l}\|_\mathrm{op} &\leq \frac{1}{t_{1}\lambda_{\min}^2(\Sigma)\mu_{C}}\times\lambda_{\max}(\Sigma)\times\frac{2t_{1}K\lambda^2_{\max}(\Sigma)\Omega_{C}}{\lambda^2_{\min}(\Sigma)}\times F(K)\frac{\max\{1,\Omega_{C}^{(K-1)}\}}{(t_{l}-t_{1})\min\{1,\mu_{C}^{(K-1)}\}}\\
        &\leq F(K)\frac{\lambda^3_{\max}(\Sigma)}{\lambda_{\min}^4(\Sigma)}\frac{\max\{1,\Omega_{C}^{K}\}}{t_{l}\min\{1,\mu_{C}^{K}\}}.
    \end{align*}
    The result follows via induction. 
\end{proof}

Define
\begin{align*}
\hat{\lambda}_{\min} \equiv \min_{j\in[K]} \lambda_{\min} \left( \frac{X_{O_{j}}^T X_{O_{j}}}{n_{j}}\right), \quad\quad\quad \hat{\lambda}_{\max} \equiv \max_{j\in[K]} \lambda_{\max} \left( \frac{X_{O_{j}}^T X_{O_{j}}}{n_{j}}\right),
\end{align*} 
where $X_{O_{j}}$ is the observed design matrix in the $j^{th}$ missingness pattern.

In order to prove Theorem \ref{thm:Master OSS with estimating weights}, we prove a result with slightly weaker conditions on $\hat{\Sigma},\hat{\lambda}_{\min}, \hat{\lambda}_{\max}$. These are implied by the conditions of Theorem \ref{thm:Master OSS with estimating weights}.

\begin{theorem}\label{thm:Complicated Master OSS with estimating weights}
Assume the noise distribution satisfies $\mathbb{E}\left[\epsilon^8\right]^{\frac{1}{4}}\leq\kappa_{\epsilon}\sigma^2$, for some $\kappa_{\epsilon}\geq 1$. Further assume the distribution of covariates satisfies Assumptions \ref{assump:small-ball}, \ref{assump:subgauss} and \ref{assump:eigen} and that for some constant $F=F(\lambda_{+},\lambda_{-},C_{X},\chi,C,K,\kappa_{\epsilon})$ 
\begin{equation}\label{eq:Complicated OSS master covariance evals bound}
    \mathbb{E}\left[\frac{\max\{1,\lambda_{\max}(\hat{\Sigma})\}^{24}}{\min\{1,\lambda_{\min}(\hat{\Sigma})\}^{32}}\left(\max\left\{1,\frac{\hat{\lambda}_{\max}}{\lambda_{\min}^2(\hat{\Sigma})}\right\}\max\left\{1,\frac{\lambda_{\max}^2(\hat{\Sigma})}{\hat{\lambda}_{\min}}\right\}\right)^{8K}\right]\leq F.
\end{equation}
Also assume that the estimate of the covariance $\hat{\Sigma}$ is symmetric, positive-definite and independent of the labelled data used in \eqref{eq:OSS estimator definition initial}. Additionally, assume that for some $\zeta \!=\! \zeta(\{n_{k}\}_{k=1}^K,\{|L_{k}|\}_{k=1}^L,p,N)\!\leq\! 1$ 
\begin{equation}\label{eq:Complicated OSS master covariance deviation bound}
    \mathbb{E}\left[\frac{\lambda_{\max}^{8}(\hat{\Sigma})}{\min\{\lambda_{\min}^{12}(\hat{\Sigma}),1\}}\left(\max\left\{1,\frac{\hat{\lambda}_{\max}}{\lambda_{\min}^2(\hat{\Sigma})}\right\}\max\left\{1,\frac{\lambda_{\max}^2(\hat{\Sigma})}{\hat{\lambda}_{\min}}\right\}\right)^{2K}\left(\frac{\|\hat{\Sigma}-\Sigma\|_\mathrm{op}^{2}}{\zeta}+\frac{\|\hat{\Sigma}-\Sigma\|_\mathrm{op}^{8}}{\zeta^4}\right)\right] \leq F
\end{equation}
 and  
\begin{equation*}\label{eq:Complicated OSS master weights deviation bound}
    \max_{k \in [K]}\left\{\mathbb{E}\left[\left(\frac{D_{k}^*}{\hat{D}_{k}}\right)^{16}\right],\mathbb{E}\left[\left(\frac{\hat{D}_{k}}{D_{k}^*}\right)^{16}\right]\right\}\leq F.
\end{equation*}
For some $G=G(\lambda_{+},\lambda_{-},C_{X},\chi,C,K,\kappa_{\epsilon})$, it holds that
    \begin{equation*}
        \mathbb{E}\left[\|\hat{\beta}-\beta^*\|_{2}^2\right] \leq G\left(\sigma^2\sum_{i=1}^p\frac{1}{\alpha_{i}}+\|\beta^*\|_{2}^2\zeta\right).
    \end{equation*}
\end{theorem}
\begin{proof}
We define some additional notation.
For some $P=P_{j}$, $O_{j}=O\subseteq [p]$ and samples $(X_{O},Y)$, which are independent of $\hat{\Sigma}$, define $\gamma = X_{O}^T(P-\hat{P})\beta^*$ and $\hat{V} = \hat{P}^TX_{O}$. It holds  that
\begin{align*}
    Y &= X^T\beta^*+\epsilon = X_{O}^TP\beta^*+X^T\beta^*-X_{O}^TP\beta^*+\epsilon\\
    &= X_{O}^T\hat{P}\beta^*+X_{O}^T(P-\hat{P})\beta^*+X^T\beta^*-X_{O}^TP\beta^*+\epsilon\\
    &\equiv \hat{V}^T\beta^* + \gamma + \eta.
\end{align*}
 For $i \in [n_{\mathcal{L}}]$, the above yields 
\begin{align*}
    \hat{V}_{i}Y_{i} &= \hat{V}_{i}\hat{V}_{i}^T\beta^*+\hat{V}_{i}\gamma_{i}+\hat{V}_{i}\eta_{i}\\
    & = \hat{V}_{i}\hat{V}_{i}^T\beta^*+\mathbb{E}\left[\hat{V}_{i}\gamma_{i}|\hat{\Sigma}\right]+\hat{V}_{i}\gamma_{i}-\mathbb{E}\left[\hat{V}_{i}\gamma_{i}|\hat{\Sigma}\right]+\hat{V}_{i}\eta_{i}\\
    &\equiv \hat{V}_{i}\hat{V}_{i}^T\beta^*+\mathbb{E}\left[\hat{V}_{i}\gamma_{i}|\hat{\Sigma}\right] + \delta_{i}.
\end{align*}
Note that by the triangle inequality and the submultiplicativity of the operator norm, we have that
\begin{align}\label{eq:P hat P upper bound}
    \|\hat{P}-P\|_\mathrm{op} & = \|\left(\hat{\Sigma}_{OO}\right)^{-1}\hat{\Sigma}_{O}-\Sigma_{OO}^{-1}\Sigma_{O}\|_\mathrm{op}\nonumber\\
    & = \|\left(\hat{\Sigma}_{OO}\right)^{-1}\hat{\Sigma}_{O}-\Sigma_{OO}^{-1}\hat{\Sigma}_{O}+\Sigma_{OO}^{-1}\hat{\Sigma}_{O}-\Sigma_{OO}^{-1}\Sigma_{O}\|_\mathrm{op}\nonumber\\
    &\leq \|\left(\hat{\Sigma}_{OO}\right)^{-1}\hat{\Sigma}_{O}-\Sigma_{OO}^{-1}\hat{\Sigma}_{O}\|_\mathrm{op}+\|\Sigma_{OO}^{-1}\hat{\Sigma}_{O}-\Sigma_{OO}^{-1}\Sigma_{O}\|_\mathrm{op}\nonumber\\
    &\leq \|\left(\hat{\Sigma}_{OO}\right)^{-1}-\Sigma_{OO}^{-1}\|_\mathrm{op}\|\hat{\Sigma}_{O}\|_\mathrm{op}+\|\Sigma_{OO}^{-1}\|_\mathrm{op}\|\hat{\Sigma}_{O}-\Sigma_{O}\|_\mathrm{op}\nonumber\\
    &=  \|\Sigma_{OO}^{-1}\left(\Sigma_{OO}-\hat{\Sigma}_{OO}\right)\left(\hat{\Sigma}_{OO}\right)^{-1}\|_\mathrm{op}\|\hat{\Sigma}_{O}\|_\mathrm{op}+\|\Sigma_{OO}^{-1}\|_\mathrm{op}\|\hat{\Sigma}_{O}-\Sigma_{O}\|_\mathrm{op}\nonumber\\
    &\leq \|\Sigma_{OO}^{-1}\|_\mathrm{op}\|\hat{\Sigma}_{O}\|_\mathrm{op}\|\left(\hat{\Sigma}_{OO}\right)^{-1}\|_\mathrm{op}\|\Sigma_{OO}-\hat{\Sigma}_{OO}\|_\mathrm{op}+\|\Sigma_{OO}^{-1}\|_\mathrm{op}\|\hat{\Sigma}_{O}-\Sigma_{O}\|_\mathrm{op}\nonumber\\
    &\leq \frac{2}{\lambda_{\min}(\Sigma)}\|\hat{\Sigma}-\Sigma\|_\mathrm{op}\frac{\lambda_{\max}(\hat{\Sigma})}{\lambda_{\min}(\hat{\Sigma})}.
\end{align}
Defining
\begin{equation*}
    \hat{B}=\sum_{i=1}^{n_{\mathcal{L}}}\hat{w}_{i}\hat{V}_{i}\hat{V}_{i}^T,
\end{equation*}
we decompose the error as follows 
\begin{align}\label{eq:OSS structured error decomposition}
    \|\hat{\beta}-\beta^*\|_{2} &= \left\|\hat{B}^{-1}\sum_{i=1}^{n_{\mathcal{L}}}\hat{w}_{i}\mathbb{E}\left[\hat{V}_{i}\gamma_{i}|\hat{\Sigma}\right]+\hat{B}^{-1}\sum_{i=1}^{n_{\mathcal{L}}}\hat{w}_{i}\delta_{i}\right\|_{2}\nonumber\\
    & \leq \left\|\hat{B}^{-1}\sum_{i=1}^{n_{\mathcal{L}}}\hat{w}_{i}\mathbb{E}\left[\hat{V}_{i}\gamma_{i}|\hat{\Sigma}\right]\right\|_{2}+\left\|\hat{B}^{-1}\sum_{i=1}^{n_{\mathcal{L}}}\hat{w}_{i}\delta_{i}\right\|_{2}\nonumber\\
    & \equiv \circled{1}+\circled{2}.
\end{align}
We control both of these terms in turn. Throughout $F=F(\lambda_{+},\lambda_{-},C_{X},\chi,C,\kappa_{\epsilon},K)$ denotes a constant depending only on $\lambda_{+},\lambda_{-},C_{X},\chi,C,\kappa_{\epsilon}$ and $K$ whose value may change from line to line. Since our covariates are sub-Gaussian, they are satisfy $(8,2,C_{8})-$hypercontractivity for some $C_{8}=C_{8}(C_{X},\lambda_{-})$
\subsubsection*{Controlling $\circled{1}$}
Applying Lemma \ref{lemma:squared matrix bound} with $C_{j} = \left(\hat{\Sigma}_{O_{j}O_{j}}\right)^{-1}\frac{X_{O_{j}}^TX_{O_{j}}}{n_{j}}\left(\hat{\Sigma}_{O_{j}O_{j}}\right)^{-1}$, $t_{j}=n_{j}\hat{D}_{j}$, $R_{j}=\hat{\Sigma}_{O_{j}}^T$, $ \mu_{C}=\frac{\hat{\lambda}_{\min}}{\lambda_{\max}^2(\hat{\Sigma})}$ and $  \Omega_{C} = \frac{\hat{\lambda}_{\max}}{\lambda_{\min}^2(\hat{\Sigma})}$ yields that for $j \in [K]$
\begin{equation}\label{eq:applying squared matrix bound 1}
    \|\hat{B}^{-1}\hat{\Sigma}_{O_{j}}^T\|_\mathrm{op} \leq F\frac{\lambda_{\max}^3(\hat{\Sigma})}{\lambda_{\min}^4(\hat{\Sigma})}\frac{1}{n_{j}\hat{D}_{j}}\frac{\max\{1,\Omega_{C}^{K}\}}{\min\{1,\mu_{C}^{K}\}}.
\end{equation}
We bound the first term in the \eqref{eq:OSS structured error decomposition} as follows
\begin{align*}
    &\left\|\hat{B}^{-1}\sum_{i=1}^{n_{\mathcal{L}}}\hat{w}_{i}\mathbb{E}\left[\hat{V}_{i}\gamma_{i}|\hat{\Sigma}\right]\right\|_{2}\\
    &\leq \sum_{k=1}^K\left\|\hat{B}^{-1}\sum_{i\in\mathcal{I}_{k}}\hat{w}_{i}\mathbb{E}\left[\hat{V}_{i}\gamma_{i}|\hat{\Sigma}\right]\right\|_{2}\\
    &\leq \sum_{k=1}^K\left\|\hat{D}_{k}\hat{B}^{-1}\hat{P}_{k}^T\sum_{i\in\mathcal{I}_{k}}\Sigma_{O_{k}O_{k}}(P_{k}-\hat{P}_{k})\beta^*\right\|_{2}\\
    &\leq \sum_{k=1}^K \hat{D}_{k}\|\hat{B}^{-1}\hat{P}_{k}^T\|_\mathrm{op}\sum_{i \in \mathcal{I}_{k}}\|\Sigma_{O_{k}O_{k}}(P_{k}-\hat{P}_{k})\beta^*\|_{2}\\
    &\leq \sum_{k=1}^K \hat{D}_{k}\|\hat{B}^{-1}\hat{\Sigma}_{O_{k}}^T\|_\mathrm{op}\|\hat{\Sigma}_{O_{k}O_{k}}^{-1}\|_\mathrm{op}\sum_{i \in \mathcal{I}_{k}}\|\Sigma_{O_{k}O_{k}}\|_\mathrm{op}\|P_{k}-\hat{P}_{k}\|_\mathrm{op}\|\beta^*\|_{2}\\
    &\leq \sum_{k=1}^K \hat{D}_{k}F\frac{\lambda_{\max}^3(\hat{\Sigma})}{\lambda_{\min}^5(\hat{\Sigma})n_{k}\hat{D}_{k}}\left(\max\left\{1,\frac{\hat{\lambda}_{\max}}{\lambda_{\min}^2(\hat{\Sigma})}\right\}\max\left\{1,\frac{\lambda_{\max}^2(\hat{\Sigma})}{\hat{\lambda}_{\min}}\right\}\right)^{K}\sum_{i \in \mathcal{I}_{k}}\lambda_{\max}(\Sigma)\|P_{k}-\hat{P}_{k}\|_\mathrm{op}\|\beta^*\|_{2}\\
    &\leq \sum_{k=1}^K F\frac{\lambda_{\max}^3(\hat{\Sigma})}{\lambda_{\min}^5(\hat{\Sigma})n_{k}}\left(\max\left\{1,\frac{\hat{\lambda}_{\max}}{\lambda_{\min}^2(\hat{\Sigma})}\right\}\max\left\{1,\frac{\lambda_{\max}^2(\hat{\Sigma})}{\hat{\lambda}_{\min}}\right\}\right)^{K}n_{k}\|\hat{\Sigma}-\Sigma\|_\mathrm{op}\frac{\lambda_{\max}(\hat{\Sigma})}{\lambda_{\min}(\hat{\Sigma})}\|\beta^*\|_{2}\\
    &= F\frac{\lambda_{\max}^4(\hat{\Sigma})}{\lambda_{\min}^6(\hat{\Sigma})}\left(\max\left\{1,\frac{\hat{\lambda}_{\max}}{\lambda_{\min}^2(\hat{\Sigma})}\right\}\max\left\{1,\frac{\lambda_{\max}^2(\hat{\Sigma})}{\hat{\lambda}_{\min}}\right\}\right)^{K}\|\hat{\Sigma}-\Sigma\|_\mathrm{op}\|\beta^*\|_{2},
\end{align*}
where the first and third inequalities follow from the triangle inequality; the fourth inequality follows from the submultiplicativity of the operator norm; the fifth inequality follows from \eqref{eq:applying squared matrix bound 1}; the sixth inequality follows from \eqref{eq:P hat P upper bound}. Squaring the above display and taking expectations yields
\begin{align*}
    &\mathbb{E}\left[\left\|\hat{B}^{-1}\sum_{i=1}^{n_{\mathcal{L}}}\hat{w}_{i}\mathbb{E}\left[\hat{V}_{i}\gamma_{i}|\hat{\Sigma}\right]\right\|_{2}^2\right]\\
    &\leq F\mathbb{E}\left[\frac{\lambda_{\max}^{8}(\hat{\Sigma})}{\lambda_{\min}^{12}(\hat{\Sigma})}\left(\max\left\{1,\frac{\hat{\lambda}_{\max}}{\lambda_{\min}^2(\hat{\Sigma})}\right\}\max\left\{1,\frac{\lambda_{\max}^2(\hat{\Sigma})}{\hat{\lambda}_{\min}}\right\}\right)^{2K}\|\hat{\Sigma}-\Sigma\|_\mathrm{op}^2\right]\|\beta^*\|_{2}^2\\
    &\leq F\zeta\|\beta^*\|_{2}^2,
\end{align*}
where the final line follows from \eqref{eq:Complicated OSS master covariance deviation bound}. 

\subsubsection*{Controlling $\circled{2}$}
We first define, for $l \in [L]$, the random variable $\hat{\alpha}^{(L_{l})} = \max_{k \in [K]}\{\hat{D}_{k}n_{k}\mathbbm{1}_{\{L_{l}\cap O_{k}\neq \emptyset\}}\}$ and its population counterpart $\alpha^{(L_{l})} = \max_{k \in [K]}\{D_{k}^*n_{k}\mathbbm{1}_{\{L_{l}\cap O_{k}\neq \emptyset}\}\}$. Note that \eqref{eq:applying squared matrix bound 1} implies
\begin{equation}\label{eq:applying squared matrix bound to modalities}
    \|\hat{B}^{-1}\hat{\Sigma}^T_{L_{l}}\|_\mathrm{op} \leq F\frac{\lambda_{\max}^3(\hat{\Sigma})}{\lambda_{\min}^4(\hat{\Sigma})}\frac{1}{\hat{\alpha}^{(L_{l})}}\frac{\max\{1,\Omega_{C}^{K}\}}{\min\{1,\mu_{C}^{K}\}}.
\end{equation}
Recall that $\delta_{i}=\hat{V}_{i}\gamma_{i}-\mathbb{E}\left[\hat{V}_{i}\gamma_{i}|\hat{\Sigma}\right]+\hat{V}_{i}\eta_{i}$. For $i \in [n_{\mathcal{L}}]$, define 
\begin{equation*}
    \tau_{i}=(\hat{\Sigma}_{O_{\Xi(i)}O_{\Xi(i)}})^{-1}\left((X_{i})_{O_{\Xi(i)}}\gamma_{i}-\mathbb{E}\left[\gamma_{i}(X_{i})_{O_{\Xi(i)}}|\hat{\Sigma}\right]+(X_{i})_{O_{\Xi(i)}}\eta_{i}\right) \in \mathbb{R}^{|O_{\Xi(i)}|},
\end{equation*} such that $\delta_{i} = \hat{\Sigma}_{O_{\Xi(i)}}^T\tau_{i}.$ Additionally, for $l \in [L]$, define $\tau_{i}^{(l)} \in \mathbb{R}^{|L_{l}| }$ such that $\tau_{i}^{(l)}= 0$ if $O_{\Xi(i)}\cap L_{l}=\emptyset$ and $\delta_{i}=\sum_{l=1}^L\hat{\Sigma}^T_{L_{l}}\tau_{i}^{(l)}$ (where $\tau_{i}^{(l)}$ is the component of $\tau_{i}$ corresponding to modality $l$). 
It follows that
\begin{align*}
    &\left\|\hat{B}^{-1}\sum_{i=1}^{n_{\mathcal{L}}}\hat{w}_{i}\delta_{i}\right\|_{2}^2 \\
    &\leq F\sum_{k=1}^K\hat{D}_{k}^2\left\|\sum_{i \in \mathcal{I}_{k}}\hat{B}^{-1}\delta_{i}\right\|_{2}^2\\
    &\leq F\sum_{k=1}^K\hat{D}_{k}^2\left\|\sum_{l=1}^L\sum_{i \in \mathcal{I}_{k}}\hat{B}^{-1}\hat{\Sigma}^T_{L_{l}}\tau_{i}^{(l)}\right\|_{2}^2\\
    &\leq F\sum_{k=1}^K\hat{D}_{k}^2\sum_{l=1}^L\left\|\sum_{i \in \mathcal{I}_{k}}\hat{B}^{-1}\hat{\Sigma}^T_{L_{l}}\tau_{i}^{(l)}\right\|_{2}^2\\
    &\leq F\sum_{k=1}^K\hat{D}_{k}^2\sum_{l=1}^L\lambda_{\max}\left(\hat{\Sigma}_{L_{l}}\hat{B}^{-2}\hat{\Sigma}^T_{L_{l}}\right)\left\|\sum_{i \in \mathcal{I}_{k}}\tau_{i}^{(l)}\right\|_{2}^2\\
    &\leq F\sum_{k=1}^K\hat{D}_{k}^2\sum_{l=1}^L\frac{\lambda_{\max}^6(\hat{\Sigma})}{\lambda_{\min}^8(\hat{\Sigma})(\hat{\alpha}^{(L_{l})})^2}\left(\max\left\{1,\frac{\hat{\lambda}_{\max}}{\lambda_{\min}^2(\hat{\Sigma})}\right\}\max\left\{1,\frac{\lambda_{\max}^2(\hat{\Sigma})}{\hat{\lambda}_{\min}}\right\}\right)^{2K}\left\|\sum_{i \in \mathcal{I}_{k}}\tau_{i}^{(l)}\right\|_{2}^2,
\end{align*}
where the first and third inequalities follow by the triangle inequality and the fact that for $q \in \mathcal{N}$ and $a_{1},\cdots, a_{q} \in \mathbb{R}, (a_{1}+\ldots + a_{q})^2 \leq q(a_{1}^2+\ldots+ a_{q}^2)$; the final line follows from \eqref{eq:applying squared matrix bound to modalities}. Therefore, by Cauchy-Schwarz, it holds that
\begin{align}\label{eq:pre tau and cov bounds}
&\mathbb{E}\left[\left\|\hat{B}^{-1}\sum_{i=1}^{n_{\mathcal{L}}}\hat{w}_{i}\delta_{i}\right\|_{2}^2\right]\nonumber\\&\leq F\sum_{k=1}^K\sum_{l=1}^L\mathbb{E}\left[\hat{D}_{k}^4\frac{\lambda_{\max}^{12}(\hat{\Sigma})}{\lambda_{\min}^{16}(\hat{\Sigma})(\hat{\alpha}^{(L_{l})})^4}\left(\max\left\{1,\frac{\hat{\lambda}_{\max}}{\lambda_{\min}^2(\hat{\Sigma})}\right\}\max\left\{1,\frac{\lambda_{\max}^2(\hat{\Sigma})}{\hat{\lambda}_{\min}}\right\}\right)^{4K}\right]^{\frac{1}{2}}\mathbb{E}\left[\left\|\sum_{i \in \mathcal{I}_{k}}\tau_{i}^{(l)}\right\|_{2}^4\right]^{\frac{1}{2}}.
\end{align}
We now control the second factor in the above display. For some $t \in \mathcal{I}_{k}$, it holds that
\begin{align}\label{eq:tau fourth moment bound}
    \mathbb{E}\left[\left\|\sum_{i \in \mathcal{I}_{k}}\tau_{i}^{(l)}\right\|_{2}^4\right]&=\sum_{i,j,k,r}\mathbb{E}\left[\mathbb{E}\left[\langle\tau_{i}^{(l)},\tau_{j}^{(l)}\rangle\langle\tau_{k}^{(l)},\tau_{r}^{(l)}\rangle\big|\hat{\Sigma}\right]\right]\nonumber\\
    &=\sum_{i}\mathbb{E}\left[\|\tau_{i}^{(l)}\|_{2}^4\right]+4\sum_{i<j}\mathbb{E}\left[\langle\tau_{i}^{(l)},\tau_{j}^{(l)}\rangle^2\right]+2\sum_{i<j}\mathbb{E}\left[\|\tau_{i}^{(l)}\|_{2}^2\|\tau_{j}^{(l)}\|_{2}^2\right]\nonumber\\
    &\leq \sum_{i}\mathbb{E}\left[\|\tau_{i}^{(l)}\|_{2}^4\right]+6\sum_{i<j}\mathbb{E}\left[\|\tau_{i}^{(l)}\|_{2}^2\|\tau_{j}^{(l)}\|_{2}^2\right]\nonumber\\
    &\leq \sum_{i}\mathbb{E}\left[\|\tau_{i}^{(l)}\|_{2}^4\right]+6\sum_{i<j}\mathbb{E}\left[\|\tau_{i}^{(l)}\|_{2}^4\right]^{\frac{1}{2}}\mathbb{E}\left[\|\tau_{j}^{(l)}\|_{2}^4\right]^{\frac{1}{2}}\nonumber\\
    &\leq 7n_{k}^2\mathbb{E}\left[\|\tau_{t}^{(l)}\|_{2}^4\right],
\end{align}
where the first line follows by the tower law; the second line follows from the fact that the $\tau_i$ are conditionally independent given $\hat{\Sigma}$ and $\mathbb{E}(\tau_i | \hat{\Sigma})=0$ and the third and fourth lines follow from Cauchy-Schwarz. Since for real $a,b,c$ we have $(a+b+c)^4 \leq 3^3(a^4+b^4+c^4)$, it holds for any $r \in [p]$ that
\begin{align}\label{eq:overall tau bound}
     &\mathbb{E}\left[\langle  e_{r},\tau_{i}\rangle^4\right]\nonumber\\
     & \leq 3^3\left(\mathbb{E}\left[\langle  e_{r},(\hat{\Sigma}_{O_{\Xi(i)}O_{\Xi(i)}})^{-1}(X_{i})_{O_{\Xi(i)}}\gamma_{i}\rangle^4\right]+\mathbb{E}\left[\langle  e_{r},(\hat{\Sigma}_{O_{\Xi(i)}O_{\Xi(i)}})^{-1}\mathbb{E}\left[(X_{i})_{O_{\Xi(i)}}\gamma_{i}|\hat{\Sigma}\right]\rangle^4\right]\right.\nonumber\\
     & \hspace{100pt} \left.+\mathbb{E}\left[\langle  e_{r},(\hat{\Sigma}_{O_{\Xi(i)}O_{\Xi(i)}})^{-1}(X_{i})_{O_{\Xi(i)}}\eta_{i}\rangle^4\right]\right).
\end{align}
We now focus on controlling the contribution to the error for each of the terms in the above display. Recalling that $\eta_{i} = \epsilon_{i}+m_{i}$, we further decompose the final term of \eqref{eq:overall tau bound} as follows
\begin{align}\label{eq:eta decomposed bound}
    &\mathbb{E}\left[\langle  e_{r},(\hat{\Sigma}_{O_{\Xi(i)}O_{\Xi(i)}})^{-1}(X_{i})_{O_{\Xi(i)}}\eta_{i}\rangle^4\right]\nonumber\\
    &\leq 2^3\left(\mathbb{E}\left[\langle  e_{r},(\hat{\Sigma}_{O_{\Xi(i)}O_{\Xi(i)}})^{-1}(X_{i})_{O_{\Xi(i)}}\epsilon_{i}\rangle^4\right]+\mathbb{E}\left[\langle  e_{r},(\hat{\Sigma}_{O_{\Xi(i)}O_{\Xi(i)}})^{-1}(X_{i})_{O_{\Xi(i)}}m_{i}\rangle^4\right]\right).
\end{align}
We control each of the terms in \eqref{eq:eta decomposed bound} in turn. The first term can be bounded by
\begin{align}\label{eq:epsilon decomposed bound}
    \mathbb{E}\left[\langle  e_{r},(\hat{\Sigma}_{O_{\Xi(i)}O_{\Xi(i)}})^{-1}(X_{i})_{O_{\Xi(i)}}\epsilon_{i}\rangle^4\right]&\leq \mathbb{E}\left[\mathbb{E}\left[\langle  (\hat{\Sigma}_{O_{\Xi(i)}O_{\Xi(i)}})^{-1}e_{r},(X_{i})_{O_{\Xi(i)}}\rangle^8\big|\hat{\Sigma}\right]\right]^{\frac{1}{2}}\mathbb{E}\left[\epsilon_{i}^8\right]^{\frac{1}{2}}\nonumber\\
    &\leq F\sigma^4\mathbb{E}\left[\mathbb{E}\left[\langle  (\hat{\Sigma}_{O_{\Xi(i)}O_{\Xi(i)}})^{-1}e_{r},(X_{i})_{O_{\Xi(i)}}\rangle^2\big|\hat{\Sigma}\right]^4\right]^{\frac{1}{2}}\nonumber\\
    &\leq F\sigma^4\mathbb{E}\left[\|(\hat{\Sigma}_{O_{\Xi(i)}O_{\Xi(i)}})^{-1}\|_{\mathrm{op}}^{8}\right]\nonumber\\
    &\leq F\sigma^4,
\end{align}
where the first line follows from Cauchy-Schwarz and the tower law, the second line follows from hypercontractivity and the assumptions of the Theorem; the final line follows from \eqref{eq:Complicated OSS master covariance evals bound}. We bound the second term as follows
\begin{align}\label{eq:m decomposed bound}
    \mathbb{E}\left[\langle  e_{r},(\hat{\Sigma}_{O_{\Xi(i)}O_{\Xi(i)}})^{-1}(X_{i})_{O_{\Xi(i)}}m_{i}\rangle^4\right] &\leq \mathbb{E}\left[\mathbb{E}\left[\langle  (\hat{\Sigma}_{O_{\Xi(i)}O_{\Xi(i)}})^{-1}e_{r},(X_{i})_{O_{\Xi(i)}}\rangle^8\big|\hat{\Sigma}\right]\right]^{\frac{1}{2}}\mathbb{E}\left[m_{i}^8\right]^{\frac{1}{2}}\nonumber\\
    &\leq F\mathbb{E}\left[\mathbb{E}\left[\langle  (\hat{\Sigma}_{O_{\Xi(i)}O_{\Xi(i)}})^{-1}e_{r},(X_{i})_{O_{\Xi(i)}}\rangle^2\big|\hat{\Sigma}\right]^4\right]^{\frac{1}{2}}\mathbb{E}\left[m_{i}^2\right]^{2}\nonumber\\
    &\leq F\mathbb{E}\left[\|  (\hat{\Sigma}_{O_{\Xi(i)}O_{\Xi(i)}})^{-1}\|_{\mathrm{op}}^8\right]^{\frac{1}{2}}\mathbb{E}\left[m_{i}^2\right]^{2}\nonumber\\
    &\leq F((\beta^{*}_{M_{\Xi(i)}})^TS_{M_{\Xi(i)}}(\beta^{*}_{M_{\Xi(i)}}))^2,
\end{align}
where the first line follows from Cauchy-Schwarz and the tower law, the second line follows from hypercontractivity and the final line follows from \eqref{eq:Complicated OSS master covariance evals bound}. We bound the first term of \eqref{eq:overall tau bound} as follows
\begin{align}\label{eq:gamma decomposed bound}
    &\mathbb{E}\left[\langle  e_{r},(\hat{\Sigma}_{O_{\Xi(i)}O_{\Xi(i)}})^{-1}(X_{i})_{O_{\Xi(i)}}\gamma_{i}\rangle^4\right]\nonumber\\
    &\leq \mathbb{E}\left[\langle  e_{r},(\hat{\Sigma}_{O_{\Xi(i)}O_{\Xi(i)}})^{-1}(X_{i})_{O_{\Xi(i)}}\rangle^8\right]^{\frac{1}{2}}\mathbb{E}\left[\gamma_{i}^8\right]^{\frac{1}{2}}\nonumber\\
    &= \mathbb{E}\left[\mathbb{E}\left[\langle  e_{r},(\hat{\Sigma}_{O_{\Xi(i)}O_{\Xi(i)}})^{-1}(X_{i})_{O_{\Xi(i)}}\rangle^8\big|\hat{\Sigma}\right]\right]^{\frac{1}{2}}\mathbb{E}\left[\mathbb{E}\left[((X_{i})_{O_{\Xi(i)}}^T(P_{O_{\Xi(i)}}-\hat{P}_{O_{\Xi(i)}})\beta^*)^8\big|\hat{\Sigma}\right]\right]^{\frac{1}{2}}\nonumber\\
    &\leq F\mathbb{E}\left[\mathbb{E}\left[\langle  e_{r},(\hat{\Sigma}_{O_{\Xi(i)}O_{\Xi(i)}})^{-1}(X_{i})_{O_{\Xi(i)}}\rangle^2\big|\hat{\Sigma}\right]^4\right]^{\frac{1}{2}}\mathbb{E}\left[\mathbb{E}\left[((X_{i})_{O_{\Xi(i)}}^T(P_{O_{\Xi(i)}}-\hat{P}_{O_{\Xi(i)}})\beta^*)^2\big|\hat{\Sigma}\right]^4\right]^{\frac{1}{2}}\nonumber\\
    &\leq F\mathbb{E}\left[\|(\hat{\Sigma}_{O_{\Xi(i)}O_{\Xi(i)}})^{-1}\|_{\mathrm{op}}^8\right]^{\frac{1}{2}}\mathbb{E}\left[\|P_{O_{\Xi(i)}}-\hat{P}_{O_{\Xi(i)}}\|_{\mathrm{op}}^8\|\beta^*_{M_{\Xi(i)}}\|_{2}^8\right]^{\frac{1}{2}}\nonumber\\
    &\leq F\mathbb{E}\left[\|(\hat{\Sigma}_{O_{\Xi(i)}O_{\Xi(i)}})^{-1}\|_{\mathrm{op}}^8\right]^{\frac{1}{2}}\mathbb{E}\left[\left(\frac{2}{{\lambda_{\min}}(\Sigma)}\|\hat{\Sigma}-\Sigma\|_\mathrm{op}\frac{\lambda_{\max}(\hat{\Sigma})}{\lambda_{\min}(\hat{\Sigma})}\right)^8\|\beta^*_{M_{\Xi(i)}}\|_{2}^8\right]^{\frac{1}{2}}\nonumber\\
    &\leq F\|\beta^*_{M_{\Xi(i)}}\|^4\zeta^2\nonumber\\
    &\leq F((\beta^{*}_{M_{\Xi(i)}})^TS_{M_{\Xi(i)}}(\beta^{*}_{M_{\Xi(i)}}))^2,
\end{align}
where the second line follows from Cauchy-Schwarz, the third line follows from the tower law, the fourth line follows from hypercontractivity and the sixth line follows from \eqref{eq:P hat P upper bound}; the penultimate line follows from \eqref{eq:Complicated OSS master covariance deviation bound} and \eqref{eq:Complicated OSS master covariance evals bound}; the final line follows since $\zeta \leq 1$ and the fact that the eigenvalues of $S_{M_{\Xi(i)}}$ can be bounded below by $\lambda_-$. Finally, we bound the second term of \eqref{eq:overall tau bound} as follows
\begin{align}\label{eq:decomposed condtional expecation gamma bound}
    &\mathbb{E}\left[\langle  e_{r},(\hat{\Sigma}_{O_{\Xi(i)}O_{\Xi(i)}})^{-1}\mathbb{E}\left[\gamma_{i}(X_{i})_{O_{\Xi(i)}}|\hat{\Sigma}\right]\rangle^4\right]\nonumber\\
    &=\mathbb{E}\left[\langle  e_{r},(\hat{\Sigma}_{O_{\Xi(i)}O_{\Xi(i)}})^{-1}\Sigma_{O_{\Xi(i)}O_{\Xi(i)}}\left(P_{O_{\Xi(i)}}-\hat{P}_{O_{\Xi(i)}}\right)\beta^*\rangle^4\right]\nonumber\\
    &\leq F\mathbb{E}\left[\|(\hat{\Sigma}_{O_{\Xi(i)}O_{\Xi(i)}})^{-1}\|_{\mathrm{op}}^8\right]^{\frac{1}{2}}\mathbb{E}\left[\|P_{O_{\Xi(i)}}-\hat{P}_{O_{\Xi(i)}}\|_{\mathrm{op}}^8\|\beta^*_{M_{\Xi(i)}}\|_{2}^8\right]^{\frac{1}{2}}\nonumber\\
    &\leq F\mathbb{E}\left[\|(\hat{\Sigma}_{O_{\Xi(i)}O_{\Xi(i)}})^{-1}\|_{\mathrm{op}}^8\right]^{\frac{1}{2}}\mathbb{E}\left[\left(\frac{2}{{\lambda_{\min}(\Sigma)}}\|\hat{\Sigma}-\Sigma\|_\mathrm{op}\frac{\lambda_{\max}(\hat{\Sigma})}{\lambda_{\min}(\hat{\Sigma})}\right)^8\|\beta^*_{M_{\Xi(i)}}\|_{2}^8\right]^{\frac{1}{2}}\nonumber\\
    &\leq F\|\beta^*_{M_{\Xi(i)}}\|^4\zeta^2\nonumber\\
    &\leq F((\beta^{*}_{M_{\Xi(i)}})^TS_{M_{\Xi(i)}}(\beta^{*}_{M_{\Xi(i)}}))^2,
\end{align}
where the third line follows from Cauchy-Schwarz and the fourth line follows from \eqref{eq:P hat P upper bound}; the penultimate line follows from \eqref{eq:Complicated OSS master covariance deviation bound} and \eqref{eq:Complicated OSS master covariance evals bound}; the final line follows since $\zeta \leq 1$ and the fact that the eigenvalues of $S_{M_{\Xi(i)}}$ can be bounded below by $\lambda_-$.
Combining \eqref{eq:decomposed condtional expecation gamma bound}, \eqref{eq:gamma decomposed bound}, \eqref{eq:m decomposed bound}, \eqref{eq:epsilon decomposed bound} and \eqref{eq:eta decomposed bound} into \eqref{eq:overall tau bound} yields that 
\begin{align}\label{eq:tau inner product bound}
    \mathbb{E}\left[\langle  e_{r},\tau_{i}\rangle^4\right] 
    \leq F\left(\sigma^4+((\beta^{*}_{M_{\Xi(i)}})^TS_{M_{\Xi(i)}}(\beta^{*}_{M_{\Xi(i)}}))^2\right)  
    \leq F\left(\frac{\sigma^4}{(D_{\Xi(i)}^*)^2}\right).
\end{align}
For $l \in [L],$  when $ L_{l}\subseteq O_{\Xi(i)}$, we have that 
\begin{equation}\label{eq: tau ell decomposition}
    \|\tau_{i}^{(l)}\|_{2}^2=\sum_{\substack{r \in [|O_{\Xi(i)}|]\\ \text{corresponding to } L_{l}}}\langle  e_{r},\tau_{i}\rangle^2,
\end{equation} 
where the sum is taken over the elements of $[|O_{\Xi(i)}|]$ that correspond to elements of $L_{l}$ based on the decomposition $\delta_{i} = \hat{\Sigma}_{O_{\Xi(i)}}^T\tau_{i} =\sum_{l=1}^L\hat{\Sigma}^T_{L_{l}}\tau_{i}^{(l)}$. We have that
\begin{align*}
    \mathbb{E}\left[\|\tau_{i}^{(l)}\|_{2}^4\right]&=\mathbb{E}\left[\left(\sum_{\substack{r \in [|O_{\Xi(i)}|]\\ \text{corresponding to } L_{l}}}\langle  e_{r},\tau_{i}\rangle^2\right)^2\right]
    \\
    &=\sum_{\substack{r \in [|O_{\Xi(i)}|]\\ \text{corresponding to } L_{l}}}\sum_{\substack{r' \in [|O_{\Xi(i)}|]\\ \text{corresponding to } L_{l}}}\mathbb{E}\left[\langle  e_{r},\tau_{i}\rangle^2\langle  e_{r'},\tau_{i}\rangle^2\right]
    \\
    &=\sum_{\substack{r \in [|O_{\Xi(i)}|]\\ \text{corresponding to } L_{l}}}\sum_{\substack{r' \in [|O_{\Xi(i)}|]\\ \text{corresponding to } L_{l}}}\mathbb{E}\left[\langle  e_{r},\tau_{i}\rangle^4\right]^{\frac{1}{2}}\mathbb{E}\left[\langle  e_{r'},\tau_{i}\rangle^4\right]^{\frac{1}{2}}
    \\
    &\leq \left(\frac{\sigma^4}{(D_{\Xi(i)}^*)^2}\right)|L_{l}|^2,
\end{align*}
where the first line follows from \eqref{eq: tau ell decomposition}, the penultimate line follows from Cauchy-Schwarz and the final bound follows from \eqref{eq:tau inner product bound}. Plugging the above display into \eqref{eq:tau fourth moment bound} yields
\begin{equation}\label{eq:tau sum final bound}
    \mathbb{E}\left[\left\|\sum_{i \in \mathcal{I}_{k}}\tau_{i}^{(l)}\right\|_{2}^4\right] \leq Fn_{k}^2\left(\frac{\sigma^2}{D_{k}^*}|L_{l}|\right)^2\mathbbm{1}_{\{L_{l}\cap O_{k}\neq \emptyset\}}.
\end{equation}
We now return to control the first factor in \eqref{eq:pre tau and cov bounds}. By Generalised H\"older, it holds that 
\begin{align}\label{eq:pre-generalised holder hat alpha,hat D control}
    &\mathbb{E}\left[\hat{D}_{k}^4\frac{\lambda_{\max}^{12}(\hat{\Sigma})}{\lambda_{\min}^{16}(\hat{\Sigma})(\hat{\alpha}^{(L_{l})})^4}\left(\max\left\{1,\frac{\hat{\lambda}_{\max}}{\lambda_{\min}^2(\hat{\Sigma})}\right\}\max\left\{1,\frac{\lambda_{\max}^2(\hat{\Sigma})}{\hat{\lambda}_{\min}}\right\}\right)^{4K}\right] \nonumber\\&\leq \mathbb{E}\left[\hat{D}_{k}^{16}\right]^{\frac{1}{4}}\mathbb{E}\left[\frac{1}{(\hat{\alpha}^{(L_{l})})^{16}}\right]^{\frac{1}{4}}\mathbb{E}\left[\frac{\lambda_{\max}^{24}(\hat{\Sigma})}{\lambda_{\min}^{32}(\hat{\Sigma})}\left(\max\left\{1,\frac{\hat{\lambda}_{\max}}{\lambda_{\min}^2(\hat{\Sigma})}\right\}\max\left\{1,\frac{\lambda_{\max}^2(\hat{\Sigma})}{\hat{\lambda}_{\min}}\right\}\right)^{8K}\right]^{\frac{1}{2}}.
\end{align}
We control each of these factors in turn. By assumption, it holds that
\begin{equation}\label{eq:hat D control}
    \mathbb{E}\left[\hat{D}_{k}^{16}\right]=\mathbb{E}\left[\frac{\hat{D}_{k}^{16}}{(D_{k}^*)^{16}}\right](D_{k}^*)^{16} \leq F(D_{k}^*)^{16}.
\end{equation}
Additionally, we have that
\begin{align*}
    \hat{\alpha}^{(L_{l})} \geq \frac{1}{K}\sum_{\substack{k\in[K]:\\L_{l}\cap O_{k} \neq \emptyset}}\hat{D}_{k}n_{k}\geq \frac{1}{K}\sum_{\substack{k\in[K]:\\L_{l}\cap O_{k} \neq \emptyset}}\frac{\hat{D}_{k}}{D_{k}^*}D_{k}^*n_{k}
    \geq \frac{1}{K}\min_{k \in [K]}\left\{\frac{\hat{D}_{k}}{D_{k}^*}\right\}\alpha^{(L_{l})}.
\end{align*}
By inverting the above inequality, it holds that
\begin{align}\label{eq:hat alpha control}
    \mathbb{E}\left[\frac{1}{(\hat{\alpha}^{(L_{l})})^{16}}\right] \leq \left(\frac{K}{\alpha^{(L_{l})}}\right)^{16}\mathbb{E}\left[\max_{k \in [K]}\left\{\frac{D_{k}^*}{\hat{D}_{k}}\right\}^{16}\right]
    \leq \left(\frac{K}{\alpha^{(L_{l})}}\right)^{16}\mathbb{E}\left[\sum_{k=1}^{K}\left(\frac{D_{k}^*}{\hat{D}_{k}}\right)^{16}\right]
    \leq \frac{F}{(\alpha^{(L_{l})})^{16}}.
\end{align}
The final factor in \eqref{eq:pre-generalised holder hat alpha,hat D control} is controlled by assumption. Plugging \eqref{eq:hat alpha control} and \eqref{eq:hat D control} into \eqref{eq:pre-generalised holder hat alpha,hat D control} yields that
\begin{equation}\label{eq:weights final bound}
    \mathbb{E}\left[\hat{D}_{k}^4\frac{\lambda_{\max}^{12}(\hat{\Sigma})}{\lambda_{\min}^{16}(\hat{\Sigma})(\hat{\alpha}^{(L_{l})})^4}\left(\max\left\{1,\frac{\hat{\lambda}_{\max}}{\lambda_{\min}^2(\hat{\Sigma})}\right\}\max\left\{1,\frac{\lambda_{\max}^2(\hat{\Sigma})}{\hat{\lambda}_{\min}}\right\}\right)^{4K}\right] \leq F\frac{(D_{k}^*)^{4}}{(\alpha^{(L_{l})})^{4}}.
\end{equation}
Combining \eqref{eq:tau sum final bound}  and \eqref{eq:weights final bound} into \eqref{eq:pre tau and cov bounds} yields that $\circled{2}$ can be controlled as
\begin{align*}
\mathbb{E}\left[\left\|\hat{B}^{-1}\sum_{i=1}^{n_{\mathcal{L}}}\hat{w}_{i}\delta_{i}\right\|_{2}^2\right] &\leq F\sum_{k=1}^K\sum_{l=1}^L\frac{(D_{k}^*)^2}{(\alpha^{(L_{l})})^2}\frac{\sigma^2|L_{l}|}{D_{k}^*}n_{k}\mathbbm{1}_{\{L_{l}\cap O_{k}\neq \emptyset\}}\\
&\leq F\sigma^2\sum_{l=1}^L\frac{|L_{l}|}{(\alpha^{(L_{l})})^2}\sum_{k=1}^Kn_{k}D_{k}^*\mathbbm{1}_{\{L_{l}\cap O_{k}\neq \emptyset\}}\\
&\leq F\sigma^2\sum_{l=1}^L\frac{|L_{l}|}{\alpha^{(L_{l})}} \\
&\leq F\sigma^2\sum_{i=1}^p\frac{1}{\alpha_{i}},
\end{align*}
where the final line follows since $\alpha^{(L_{l})}=\max_{k \in [K]}\{D_{k}^*n_{k}\mathbbm{1}_{\{L_{l}\cap O_{k}\neq \emptyset\}}\}\geq \frac{1}{K}\sum_{k=1}^KD_{k}^*n_{k}\mathbbm{1}_{\{L_{l}\cap O_{k}\neq \emptyset\}}=\frac{\alpha_{i}}{K}$ for $i \in L_{l}$. Combining both our bounds for terms $\circled{1}$ and $\circled{2}$ gives the result. 
\end{proof}

We recall the following definitions 
\begin{align*}
\hat{\lambda}_{\min} \equiv \min_{j\in[K]} \lambda_{\min} \left( \frac{X_{O_{j}}^T X_{O_{j}}}{n_{j}}\right), \quad\quad\quad \hat{\lambda}_{\max} \equiv \max_{j\in[K]} \lambda_{\max} \left( \frac{X_{O_{j}}^T X_{O_{j}}}{n_{j}}\right),
\end{align*} 
where $X_{O_{j}}$ is the observed design matrix in the $j^{th}$ missingness pattern.

\begin{proofof}{thm:Master OSS with estimating weights}
    All that remains be shown is that condition \eqref{eq:OSS master covariance evals bound} implies conditions \eqref{eq:Complicated OSS master covariance evals bound} and \eqref{eq:Complicated OSS master covariance deviation bound}. Applying, the independence of $\hat{\Sigma}$ and the labelled data, Cauchy-Schwarz and Jensen's inequality, we have that \eqref{eq:Complicated OSS master covariance evals bound} and \eqref{eq:Complicated OSS master covariance deviation bound} are satisfied provided
    \begin{align*}
    &\mathbb{E}\left[\frac{\max\{1,\lambda_{\max}(\hat{\Sigma})\}^{24+16K}}{\min\{1,\lambda_{\min}(\hat{\Sigma})\}^{32+16K}}\right]\leq F, \quad 
    &\mathbb{E}\left[\|\hat{\Sigma}-\Sigma\|_{\mathrm{op}}^{16}\right]^{\frac{1}{8}} \leq F\zeta,
    \end{align*}
    and
    \begin{align*}
    \mathbb{E}\left[\left(\frac{\max\{1,\hat{\lambda}_{\max}\}}{\min\{1,\hat{\lambda}_{\min}\}}\right)^{8K}\right]\leq F,
    \end{align*}
    for some $F=F(\lambda_{+},\lambda_{-},C_{X},\chi,C,K,\kappa_{\epsilon})$. The first two of the above equations are implied by the assumptions of the Theorem. 
    By Cauchy-Schwarz, the final of the above equations are implied by
    \begin{align}\label{eq:needed eval bounds}
        \mathbb{E}\left[(\max\{1,\hat{\lambda}_{\max}\})^{16K}\right] \leq F  &\quad\text{ and }\quad \mathbb{E}\left[(\max\{1,\hat{\lambda}_{\min}^{-1}\})^{16K}\right] \leq F.
    \end{align}
    The second equation in \eqref{eq:needed eval bounds} is holds by Lemma \ref{lemma:eigenval control} and the fact that $n_{\min} \geq H p$ for some $H=H(\lambda_{+},\lambda_{-},C_{X},\chi,C,K,\kappa_{\epsilon})$. The first equation in \eqref{eq:needed eval bounds} is implied by Lemma \ref{lemma:subgaussian tail control}
\end{proofof}

\begin{proofof}{thm:weights initialisation guarantee}
    Since the distribution of the covariates satisfies Assumption \ref{assump:subgauss}, it also satisfies $(8q,2,C_{8q})$--hypercontractivity for some constant $C_{8q}=C_{8q}(C_{X},\lambda_{-},q)$. Throughout $F=F(\lambda_{+},\lambda_{-},C,\chi,q,C_{8q},\kappa_{\epsilon})$ denotes a constant depending only on $\lambda_{+},\lambda_{-},C,\chi,q,C_{8q}$ and $\kappa_{\epsilon}$ that may change from line to line. It suffices to prove \eqref{eq:weights initialisation conclusion} for one missingness pattern, so we return to the notation before the statement of the theorem. Let $O \subseteq[p]$ and $M = [p]\setminus O$, $d = \sigma^2+(\beta^*_{M})^TS_{M}\beta^*_{M}$ and $D= \frac{\sigma^2}{d}$. We recall that at the population level, 
    \begin{align*}
    Y=X^T\beta^*+\epsilon = X_{O}^TP\beta^*+X^T\beta^*-X_{O}^TP\beta^*+\epsilon \equiv X_{O}^T\phi^* + \eta,
    \end{align*}
    where $\phi^*=P\beta^*$ and $\eta=X^T\beta^*-X_{O}^TP\beta^*+\epsilon\equiv m+\epsilon$. Note that  $\mathbb{E}\left[\eta X_{O}\right]=0$ and $\mathbb{E}\left[\eta^2\right]=\sigma^2+(\beta^*_{M})^TS_{M}\beta^*_{M}$. For $1\leq t \leq 4q$, by the triangle inequality, hypercontractivity and the condition on $\epsilon$ in the Theorem, we have that
    \begin{align}\label{eq:eta moment bound}
        \mathbb{E}\left[\eta^{2t}\right] &\leq F\left(\mathbb{E}\left[\epsilon^{2t}\right]+\mathbb{E}\left[m^{2t}\right]\right)
        \leq F\left(\sigma^{2t}+\|\beta^*_{M}\|_{2}^{2t}\right)\leq Fd^{t}.
    \end{align}
    Additionally, by Cauchy-Schwarz, Lemma \ref{lemma:hyper norm control} and hypercontractivity, we have that, for $1\leq t\leq 4q$,
    \begin{align}\label{eq: Xo eta moment bound}
        \mathbb{E}\left[(\|X_{O}\|_{2}\eta)^t\right] & \leq \mathbb{E}\left[\|X_{O}\|_{2}^{2t}\right]^{\frac{1}{2}}\mathbb{E}\left[\eta^{2t}\right]^\frac{1}{2}  \leq F|O|^{\frac{t}{2}}d^{\frac{t}{2}}.
    \end{align}
    
    Returning to the sample level, denote by $X_{O}$ the portion of the design matrix observed in this missingness pattern, $Y_{O}$ the corresponding responses and similarly define $\epsilon_{O},\eta_{O},m_{O}$ and let $n$ be the number of samples in this pattern. Set $\hat{\phi}=(X_{O}^TX_{O})^{-1}X_{O}^TY_{O}$ and $\hat{d} = \frac{1}{n}\|Y_{O}-X_{O}\hat{\phi}\|_{2}^2$. Our final estimate of $d$ is given by 
    \begin{equation*}
        \hat{d}_{F} =
    \begin{cases}
    \kappa_{L} &\text{if } \hat{d}\leq \frac{\kappa_{L}}{2},\\
    \hat{d} & \text{if } \frac{\kappa_{L}}{2}<\hat{d}< 2\kappa_{U}, \\
    \kappa_{U} & \text{if } \hat{d} \geq 2\kappa_{U}.
    \end{cases}
    \end{equation*}
    We then set $\hat{D} = \frac{\sigma^2}{\hat{d}_{F}}$. It suffices to bound $\mathbb{E}\left[(\hat{D}/D)^{16}+(D/\hat{D})^{16}\right]$. We introduce the `hat' matrix $H=X_{O}(X_{O}^TX_{O})^{-1}X_{O}^T$ which allows us to write
    \begin{equation}\label{eq:estimating weights: hat decomposition}
        \hat{d}=\frac{1}{n}\|(I-H)Y_{O}\|_{2}^2=\frac{1}{n}\|(I-H)\eta_{O}\|_{2}^2=\frac{1}{n}\|\eta_{O}\|_{2}^2-\frac{1}{n}\eta_{O}^TH\eta_{O}.
    \end{equation}  
    We control the first term in the above decomposition as follows
    \begin{align}\label{eq:eta first concentation bound}
        \mathbb{P}\left(\left|\frac{\|\eta_{O}\|_{2}^2}{nd}-1\right|\geq t\right) &\leq \frac{\mathbb{E}\left[\left|\frac{\|\eta_{O}\|_{2}^2}{nd}-1\right|^{4q}\right]}{t^{4q}}\nonumber\\
        & = \frac{\mathbb{E}\left[\left|\sum_{i=1}^n\left(\frac{(\eta_{O})_{i}^2}{d}-1\right)\right|^{4q}\right]}{t^{4q}n^{4q}}\nonumber\\
        &\leq \frac{F}{t^{4q}n^{4q}}\left(n\mathbb{E}\left[\left(\frac{\eta^2}{d}\right)^{4q}\right]+n^{2q}\mathbb{E}\left[\frac{\eta^4}{d^2}\right]^{2q}\right)\nonumber\\
        &\leq \frac{F}{t^{4q}n^{2q}},
    \end{align}
    where the first line follows from Markov's inequality; the third line follows from Rosenthal's inequality \citep[]{boucheron2003concentration} and the final line follows from \eqref{eq:eta moment bound}. 
    
    We bound the second term in \eqref{eq:estimating weights: hat decomposition} as follows
    \begin{align}\label{eq:second eta concentration bound}
        \mathbb{P}\left(\left|\frac{\eta_{O}^TH\eta_{O}}{nd}\right|\geq t\right) &\leq \frac{1}{t^q}\mathbb{E}\left[\left|\frac{\eta_{O}^TH\eta_{O}}{nd}\right|^q\right]\nonumber\\
        & = \frac{\mathbb{E}\left[\left|\frac{\eta_{O}^TX_{O}(X_{O}^TX_{O})^{-1}X_{O}^T\eta_{O}}{nd}\right|^q\right]}{t^q}\nonumber\\
        &\leq \frac{1}{t^q}\mathbb{E}\left[\lambda_{\min}\left(\frac{X_{O}^TX_{O}}{n}\right)^{-q}\frac{\|X_{O}^T\eta_{O}\|_{2}^{2q}}{(n^2d)^{q}}\right]\nonumber\\
        &\leq \frac{1}{t^q}\mathbb{E}\left[\lambda_{\min}\left(\frac{X_{O}^TX_{O}}{n}\right)^{-2q}\right]^{\frac{1}{2}}\mathbb{E}\left[\frac{\|X_{O}^T\eta_{O}\|_{2}^{4q}}{(n^2d)^{2q}}\right]^{\frac{1}{2}}\nonumber\\
        &\leq \frac{F}{t^q}\mathbb{E}\left[\frac{\|X_{O}^T\eta_{O}\|_{2}^{4q}}{(n^2d)^{2q}}\right]^{\frac{1}{2}}\nonumber\\
        &\leq \frac{F}{t^q}\mathbb{E}\left[\frac{\left(\|X_{O}^T\eta_{O}\|_{2}-\mathbb{E}\left[\|X_{O}^T\eta_{O}\|_{2}\right]\right)^{4q}}{(n^2d)^{2q}}\right]^{\frac{1}{2}}+\frac{F}{t^q}\mathbb{E}\left[\frac{\mathbb{E}\left[\|X_{O}^T\eta_{O}\|_{2}\right]^{4q}}{(n^2d)^{2q}}\right]^{\frac{1}{2}}\nonumber\\
        &\leq \frac{F}{t^q}
        \left(
          \frac{n|O|^{2q}d^{2q} + n^{2q}|O|^{2q}d^{2q}}{(n^2 d)^{2q}}
        \right)^{\frac{1}{2}} + \frac{F}{t^q}\left(\frac{|O|n}{n^2}\right)^q\nonumber\\
        &\leq \frac{F|O|^q}{t^qn^q},
    \end{align}
    where the first line is due to Markov's inequality; the fourth line is due to Cauchy-Schwarz; the fifth line is due to~\citet[ Corollary 4]{mourtada2022exact}; the sixth line holds by the fact that for $a,b \in \mathbb{R}$ and $r \in \mathbb{N}$ we have $(a+b)^{2r}\leq 2^{2r}(a^{2r}+b^{2r})$; the penultimate line follows from the generalisation of Rosenthal's inequality in~\citet{de1981inequalities} and \eqref{eq: Xo eta moment bound}.
    Setting $t = \frac{1}{4}$ in both \eqref{eq:eta first concentation bound} and \eqref{eq:second eta concentration bound} yields an event $A$ such that $\mathbb{P}\left(A\right) \geq 1-F\frac{|O|^q}{n^q}$ and on $A$
    \begin{equation*}
        \left|\frac{\eta_{O}^TH\eta_{O}}{nd}\right| \leq \frac{1}{4} \text{  and  } \left|\frac{\|\eta_{O}\|_{2}^2}{nd}-1\right| \leq \frac{1}{4}.
    \end{equation*}
    Therefore, on this event, our initial estimator $\hat{d}$ satisfies 
    \begin{equation*}
        \frac{1}{2}\leq \frac{\hat{d}}{d} \leq \frac{3}{2}.
    \end{equation*}
    Hence, the final estimate of $D$ satisfies 
    \begin{align*}
        \mathbb{E}\left[\left(\frac{\hat{D}}{D}\right)^{16}+\left(\frac{D}{\hat{D}}\right)^{16}\right] & = \mathbb{E}\left[\left(\frac{d}{\hat{d}_{F}}\right)^{16}+\left(\frac{\hat{d}_{F}}{d}\right)^{16}\right]\\
        &= \mathbb{E}\left[\left(\frac{d}{\hat{d}_{F}}\right)^{16}+\left(\frac{\hat{d}_{F}}{d}\right)^{16}\bigg| A\right]\mathbb{P}(A)+\mathbb{E}\left[\left(\frac{d}{\hat{d}_{F}}\right)^{16}+\left(\frac{\hat{d}_{F}}{d}\right)^{16}\bigg|A^c\right]\mathbb{P}(A^c)\\
        &\leq 2\times 2^{16} + F\frac{|O|^q}{n^q}\left(\left(\frac{d}{\kappa_{L}}\right)^{16}+\left(\frac{\kappa_{U}}{d}\right)^{16}\right)\\
        &\leq H,
    \end{align*}
    where the second line follows from the law of total expectation; the penultimate line follows from properties of $A$ and the final line follows from condition \eqref{eq:weights tuning sample size}.
    \end{proofof}

\begin{lemma}\label{lemma:supervised Covariance Matrix Estimation}
    Fix $A \in \mathbb{R}^{p\times p}$ assumed to be symmetric. For $r,l \in [L]$, let $E_{L_{r}L_{l}} \in \mathbb{R}^{|L_{r}|\times |L_{l}|}$ denote the block of $A$ corresponding to the relationship between modalities $r$ and $l$. Then
    \begin{equation*}
        \|A\|_\mathrm{op} \leq \sum_{r,l=1}^L\|E_{L_{r}L_{l}}\|_\mathrm{op}.
    \end{equation*}
\end{lemma}
\begin{proof}
    Properties of the operator norm yields that 
    \begin{align*}
        \|A\|_\mathrm{op} &= \sup_{u \in \mathbb{R}^p:\|u\|_{2}=1}\{|u^TAu|\}\leq \sup_{u \in \mathbb{R}^p:\|u\|_{2}=1}\left\{\sum_{r,l=1}^L|u_{L_{r}}^TE_{L_{r}L_{l}}u_{L_{l}}|\right\}\leq \sum_{r,l=1}^L\|E_{L_{r}L_{l}}\|_\mathrm{op}.
    \end{align*}
\end{proof}

We now explain how to construct an estimate of the covariance in the supervised case that satisfies Theorem \ref{thm:Master OSS with estimating weights} and Theorem \ref{thm:OSS balanced}. Recall $\xi : [p] \rightarrow [L], \text{ where for } j \in L_{i}, \space \xi(j) = i$. Define $\hat{\Sigma}\in \mathbb{R}^{p\times p}$ by, for $(i,j) \in [p]^2$,
\begin{equation}\label{eq:LD structured inital cov}
    \hat{\Sigma}_{ij} = \frac{1}{n_{\xi(i),\xi(j)}}\left(\sum_{k \in \mathcal{I}_{\mathcal{L}}}(X_{k})_{i}(X_{k})_{j}\mathbbm{1}_{\{\{i,j\}\subseteq O_{\Xi(k)}\}}\right).
\end{equation}
We form a clipped version of the above estimator
\begin{equation}\label{eq:LD clippled covariance estimate}
    \hat{\Sigma}^{C}\equiv \begin{cases}
        \hat{\Sigma} & \text{ if } \frac{\lambda_{-}}{2} \leq \lambda_{\min}(\hat{\Sigma}) \leq \lambda_{\max}(\hat{\Sigma}) \leq 2 \lambda_{+},\\
        \lambda_{-} I_{p} & \text{ otherwise.}
    \end{cases}
\end{equation}
The following proposition establishes that \eqref{eq:LD clippled covariance estimate} satisfies the assumptions of Theorem \ref{thm:Master OSS with estimating weights} in the blockwise missing supervised case. 
\begin{proposition}\label{Prop:Supervised covariance blocks}
    Consider the estimator \eqref{eq:LD clippled covariance estimate}. Further assume that the distribution of the covariates satisfies Assumptions \ref{assump:subgauss} and \ref{assump:eigen}. For all $a\geq 1$, there exists $C=C(\lambda_{-},\lambda_{+},C_{X},a,L)$ and $H=H(\lambda_{-},\lambda_{+},C_{X},L,a)$ such that, provided $\min_{(g,h)\in[L]^2}n_{g,h}\geq Hp$, we have
    \begin{align*}
        \mathbb{E}\left[\left\|\hat{\Sigma}^{C}-\Sigma\right\|_{\mathrm{op}}^{a}\right] \leq C\max_{(g,h)\in[L]^2}\left\{\left(\frac{|L_{g}|}{n_{g,h}}\right)^{\frac{a}{2}}\right\}, \quad \text{and} \quad \mathbb{E}\left[\left(\frac{\max\{1,\lambda_{\max}(\hat{\Sigma}^{C})\}}{\min\{1,\lambda_{\min}(\hat{\Sigma}^{C})\}}\right)^{a}\right] \leq C.
    \end{align*}
\end{proposition}
\begin{proof}
    Define $\hat{E}=\hat{\Sigma}-\Sigma$, where $\hat{\Sigma}$ is defined in \eqref{eq:LD structured inital cov}. Note that, for $(r,l)\in[L]^2$, $\hat{E}_{L_{r}L_{l}}$ is a sub-matrix of a regular sample second moment matrix. By Lemma \ref{lemma:subgaussian tail control}, it holds for $n_{l,r} \geq H p$ that, for some  $F=F(\lambda_{-},\lambda_{+},C_{X},L)>0$ and $H=H(\lambda_{-},\lambda_{+},C_{X},L)$,
    \begin{equation*}
     \mathbb{P}\left(\|\hat{E}_{L_{r}L_{l}}\|_{\mathrm{op}}\geq \frac{\lambda_{-}}{2L^2}\right) \leq 2e^{-n_{l,r}F}.
    \end{equation*}
    Define the intersection of the complement of all these events to be $G$. On $G$, we have, by Lemma \ref{lemma:supervised Covariance Matrix Estimation}, that $\|\hat{\Sigma}-\Sigma\|_{\mathrm{op}} \leq \frac{\lambda_{-}}{2}$. Thus, on $G$, we have $\hat{\Sigma}=\hat{\Sigma}^{C}$. For $C=C(\lambda_{-},\lambda_{+},L,a)$, it follows that 
    \begin{align*}
        \mathbb{E}\left[\|\hat{\Sigma}^{C}-\Sigma\|_{\mathrm{op}}^{a}\right] &=\mathbb{E}\left[\|\hat{\Sigma}^{C}-\Sigma\|_{\mathrm{op}}^{a}\mathbbm{1}_{G}\right] +\mathbb{E}\left[\|\hat{\Sigma}^{C}-\Sigma\|_{\mathrm{op}}^{a}\mathbbm{1}_{G^c}\right]\\
        & =\mathbb{E}\left[\|\hat{\Sigma}-\Sigma\|_{\mathrm{op}}^{a}\mathbbm{1}_{G}\right] +C\mathbb{E}\left[\mathbbm{1}_{G^c}\right]\\
        &\leq \mathbb{E}\left[\|\hat{\Sigma}-\Sigma\|_{\mathrm{op}}^{a}\right] +C\mathbb{P}\left(G^c\right)\\
        &\leq \sum_{(r,l)\in[L]^2}\mathbb{E}\left[\|\hat{E}_{L_{r}L_{l}}\|_{\mathrm{op}}^{a}\right]+\sum_{(r,l)\in [L]^2}2e^{-n_{l,r}F}\\
        &\leq C\max_{(g,h)\in [L]^2}\left\{\left(\frac{|L_{g}|}{n_{g,h}}\right)^{\frac{a}{2}}\right\},
    \end{align*}
    where the final line follows from Lemma \ref{lemma:subgaussian tail control}. The second statement is now immediate by the construction of $\hat{\Sigma}^{C}$.
\end{proof}

We give an additional analysis of this estimator in the unstructured setting of Example \ref{examp:unstructured pattern}.

\begin{proposition}\label{Prop: unstructured cov estimation theorem}
    Suppose our distribution of covariates is formed via $X=\Sigma^{\frac{1}{2}}Z$, where the coordinates of $Z$ are independent, mean-zero, have unit variance and are sub-Gaussian with parameter $C_{X}$ and $\Sigma$ satisfies Assumption \ref{assump:eigen}. Consider \eqref{eq:LD clippled covariance estimate} defined above and assume that for all distinct $(g,h) \in [L]^2$,  $\rho^2n_{\mathcal{L}}\leq n_{g,h} \leq C_{\rho}\rho^2n_{\mathcal{L}}$. Provided $\rho^2n_{\mathcal{L}} \geq H_{1}p$, then 
    \begin{align*}
        \mathbb{P}\left(\left\|\hat{\Sigma}^{C}-\Sigma\right\|_{\mathrm{op}} \geq \min\left\{1,\frac{\lambda_{-}}{2}\right\}\right)\leq e^{-F_{1}\rho^2n_{\mathcal{L}}} \quad \text{and} \quad 
        \mathbb{E}\left[\left\|\hat{\Sigma}^{C}-\Sigma\right\|_{\mathrm{op}}^{a}\right]\leq F_{2}\left(\frac{p}{\rho^2n_{\mathcal{L}}}\right)^{\frac{a}{2}},
    \end{align*}
    where $F_{1}=F_{1}(C_{X},\lambda_{+},\lambda_{-})>0$, $F_{2}=F_{2}(C_{X},\lambda_{+},\lambda_{-},a)$ and $H_{1}=H_{1}(C_{X},\lambda_{+},\lambda_{-})$.
\end{proposition}

\begin{proof}
A standard argument \citep[Chapter 6]{wainwright2019high} yields that
\begin{equation}\label{eq:cov est Wainwright}
    \mathbb{P}\left(\left\|\hat{\Sigma}-\Sigma\right\|_{\mathrm{op}}\geq t\right) \leq 17^{p}\sup_{\theta \in \mathcal{S}^{p-1}}\mathbb{P}\left(\theta^T(\hat{\Sigma}-\Sigma)\theta\geq \frac{t}{2}\right).
\end{equation}
Fixing $\theta \in \mathcal{S}^{p-1}$, we focus on controlling $\mathbb{P}(\theta^T(\hat{\Sigma}-\Sigma)\theta \geq \frac{t}{2})$. We have that 
\begin{equation*}
    \theta^T\hat{\Sigma}\theta = \sum_{k \in \mathcal{I}_\mathcal{L}}Z_{k}^T\Sigma^{\frac{1}{2}}B_{k}(\theta)\Sigma^{\frac{1}{2}}Z_{k}, 
\end{equation*}
where $B_{k}(\theta)\in \mathbb{R}^{p \times p}$ satisfies for $(r,l)\in [p]^2$
\begin{equation*}
    \{B_{k}(\theta)\}_{r,l}=\frac{\theta_{r}\theta_{l}\mathbbm{1}_{\left\{\{r,l\}\subseteq O_{\Xi(k)}\right\}}}{n_{\xi(r)\xi(l)}},
\end{equation*}
and for $k \in [n_{\mathcal{L}}]$, we have $Z_{k}=\Sigma^{-\frac{1}{2}}X_{k}$. Define $Z \equiv (Z_{1}^T,\cdots,Z_{n_{\mathcal{L}}}^T)^T\in \mathbb{R}^{n_{\mathcal{L}}p}$ and the block-diagonal matrix $A \in \mathbb{R}^{n_{\mathcal{L}}p\times n_{\mathcal{L}}p}$ as

\begin{equation*}
A =
\begin{pmatrix}
\Sigma^{1/2} B_{1}(\theta)\Sigma^{1/2} &  & \\
&  \ddots & \\
& & \Sigma^{1/2} B_{n_{\mathcal{L}}}(\theta)\Sigma^{1/2}
\end{pmatrix}.
\end{equation*}
It therefore holds that $\theta^T\hat{\Sigma} \theta=Z^TAZ$.We have that
\begin{align*}
    \left\|A\right\|_{F}^2 = \Tr(AA^T) &= \sum_{k=1}^{n_{\mathcal{L}}}\Tr(\Sigma^{\frac{1}{2}}B_{k}(\theta)\Sigma B_{k}(\theta)\Sigma^{\frac{1}{2}})
    \leq \lambda_{+}^2\sum_{k=1}^{n_{\mathcal{L}}}\Tr(B_{k}(\theta)^2)\\
    & = \lambda_{+}^2\sum_{k=1}^{n_{\mathcal{L}}}\sum_{l=1}^p\sum_{r=1}^p\frac{\theta_{r}^2\theta_{l}^2\mathbbm{1}_{\left\{\{r,l\}\subseteq O_{\Xi(k)}\right\}}}{n_{\xi(r)\xi(l)}^2}\leq \frac{\lambda_{+}^2}{\rho^2n_{\mathcal{L}}}.
\end{align*}
Additionally, since $A$ is symmetric and block-diagonal, it holds that
\begin{equation*}
    \|A\|_{\mathrm{op}}\leq \max_{k \in [n_{\mathcal{L}}]}\{\|\Sigma^{\frac{1}{2}}B_{k}(\theta)\Sigma^{\frac{1}{2}}\|_{\mathrm{op}}\} \leq \lambda_{+}\max_{k \in [n_{\mathcal{L}}]}\{\|B_{k}(\theta)\|_{\mathrm{op}}\}\leq \lambda_{+}\max_{k \in [n_{\mathcal{L}}]}\{\|B_{k}(\theta)\|_{F}\} \leq \frac{\lambda_{+}}{\rho^2n_{\mathcal{L}}},
\end{equation*}
where $\|\cdot\|_{F}$ denotes the Frobenius norm of a matrix. 
By the Hanson-Wright inequality \citep[Theorem 6.2.2]{vershynin2025HDP}, there exists a universal constant $c$ such that
\begin{equation*}
    \mathbb{P}\left(|\theta^T\hat{\Sigma}\theta-\theta^T\Sigma\theta| \geq t\right) \leq 2e^{-c \min\left\{\frac{\rho^2n_{\mathcal{L}}t^2}{C_{X}^4\lambda_{+}^2}, \frac{\rho^2n_{\mathcal{L}}t}{C_{X}^2\lambda_{+}}\right\}}.
\end{equation*}
It follows, by \eqref{eq:cov est Wainwright}, that
\begin{equation*}
    \mathbb{P}\left(\left\|\hat{\Sigma}-\Sigma\right\|_{\mathrm{op}}\geq t\right) \leq 17^{p}\cdot 2e^{-c \min\left\{\frac{\rho^2n_{\mathcal{L}}t^2}{C_{X}^4\lambda_{+}^2}, \frac{\rho^2n_{\mathcal{L}}t}{C_{X}^2\lambda_{+}}\right\}}
\end{equation*}
Setting $t = FC_{X}^2\lambda_{+}\left(\frac{p}{\rho^2n_{\mathcal{L}}}+\sqrt{\frac{p}{\rho^2n_{\mathcal{L}}}}+u\right)$ for a sufficiently large universal constant $F$ yields that
\begin{equation}\label{eq:Cov norm high prob bound}
    \mathbb{P}\left(\left\|\hat{\Sigma}-\Sigma\right\|_{\mathrm{op}}\geq FC_{X}^2\lambda_{+}\left(\frac{p}{\rho^2n_{\mathcal{L}}}+\sqrt{\frac{p}{\rho^2n_{\mathcal{L}}}}+u\right) \right) \leq 2e^{-\rho^2n_{\mathcal{L}}\min\{u,u^2\}}.
\end{equation}
The first statement of the proposition follows immediately. In particular, we define the event $G=\{\|\hat{\Sigma}-\Sigma\|_{\mathrm{op}}\leq \min \{1,\frac{\lambda_{-}}{2}\}\}$ which satisfies $\mathbb{P}(G)\geq 1 - e^{-F_{1}\rho^2n_{\mathcal{L}}}$ where $F_{1}=F_{1}(C_{X},\lambda_{+},\lambda_{-})>0$. On $G$ we have that $\hat{\Sigma}=\hat{\Sigma}^{C}$. Furthermore, we define $R=FC_{X}^2\lambda_{+}\left(\frac{p}{\rho^2n_{\mathcal{L}}}+\sqrt{\frac{p}{\rho^2n_{\mathcal{L}}}}\right)$ and the random variable $W=(\|\hat{\Sigma}-\Sigma\|_{\mathrm{op}}-R)_{+}$ such that $\|\hat{\Sigma}-\Sigma\|_{\mathrm{op}}\leq R+W$. It follows that
\begin{equation}\label{eq: cov est decomp}
    \|\hat{\Sigma}-\Sigma\|_{\mathrm{op}}^{a}\leq 2^{a-1}\left(R^{a}+W^{a}\right).
\end{equation}
Additionally, for some $H=H(C_{X},\lambda_{+},\lambda_{-},a)$ and $F_{2}=F_{2}(C_{X},\lambda_{+},\lambda_{-},a)>0$, 
\begin{align*}
    \mathbb{E}\left[W^{a}\right] &= \int_{0}^{\infty}\mathbb{P}\left(W \geq t^{\frac{1}{a}}\right) dt\\
    & = a \int_{0}^{\infty}\mathbb{P}\left(W \geq u\right)u^{a-1}du\\
    &\leq 2a\int_{0}^{\infty}e^{-F_{2}\rho^2n_{\mathcal{L}}\min\{u,u^2\}}u^{a-1}du\\
    &\leq 2a\int_{0}^{\infty}e^{-F_{2}\rho^2n_{\mathcal{L}}u}u^{a-1}du+2a\int_{0}^{\infty}e^{-F_{2}\rho^2n_{\mathcal{L}}u^2}u^{a-1}du\\
    &\leq F_{2}\left(\frac{1}{\rho^2n_{\mathcal{L}}}\right)^{\frac{a}{2}},
\end{align*}
where the third line follows from \eqref{eq:Cov norm high prob bound} and the final line follows by integrating by substitution and properties of the $\Gamma$ function. It follows from the previous display and \eqref{eq: cov est decomp} that
\begin{equation*}
    \mathbb{E}\left[\|\hat{\Sigma}-\Sigma\|_{\mathrm{op}}^a\right] \leq F_{2}\left(\frac{p}{\rho^2n_{\mathcal{L}}}\right)^{\frac{a}{2}}.
\end{equation*}
Using properties of $G$, Cauchy-Schwarz and the previous display, the second claim of the Theorem follows
\begin{align*}
    \mathbb{E}\left[\|\hat{\Sigma}^{C}-\Sigma\|_{\mathrm{op}}^a\right] &= \mathbb{E}\left[\|\hat{\Sigma}^{C}-\Sigma\|_{\mathrm{op}}^a\mathbbm{1}_{G}\right] +\mathbb{E}\left[\|\hat{\Sigma}^{C}-\Sigma\|_{\mathrm{op}}^a\mathbbm{1}_{G^c}\right]\\
    &\leq \mathbb{E}\left[\|\hat{\Sigma}-\Sigma\|_{\mathrm{op}}^a\right]+\mathbb{P}\left(G^{c}\right)^{\frac{1}{2}}\mathbb{E}\left[\|\hat{\Sigma}^{C}-\Sigma\|_{\mathrm{op}}^{2a}\right]^{\frac{1}{2}}\\
    &\leq F_{2}\left(\frac{p}{\rho^2n_{\mathcal{L}}}\right)^{\frac{a}{2}} + F_{2}e^{-F_{1}\rho^2n_{\mathcal{L}}}\\
    &\leq F_{2}\left(\frac{p}{\rho^2n_{\mathcal{L}}}\right)^{\frac{a}{2}}.
\end{align*}
\end{proof}
We now explain how to construct an estimate of the covariance in the OSS case that satisfies the conditions of Theorem \ref{thm:OSS balanced}. Define $\hat{\Sigma}\in \mathbb{R}^{p\times p}$ by, for $(i,j) \in [p]^2$,
\begin{equation*}
    \hat{\Sigma}_{ij} = \frac{1}{N}\sum_{k\in \mathcal{I}_{\mathcal{U}}}(X_{k})_{i}(X_{k})_{j}.
\end{equation*}
We form a clipped version of the above estimator
\begin{equation}\label{eq:LD OSS clippled covariance estimate}
    \hat{\Sigma}^{C}\equiv \begin{cases}
        \hat{\Sigma} & \text{ if } \frac{\lambda_{-}}{2} \leq \lambda_{\min}(\hat{\Sigma}) \leq \lambda_{\max}(\hat{\Sigma}) \leq 2 \lambda_{+},\\
        \lambda_{-} I_{p} & \text{ otherwise.}
    \end{cases}
\end{equation}
The following proposition establishes that \eqref{eq:LD OSS clippled covariance estimate} satisfies the assumptions of Theorem \ref{thm:OSS balanced} in the supervised case. 
\begin{proposition}\label{Prop:OSS covariance blocks}
    Consider the estimator \eqref{eq:LD OSS clippled covariance estimate}. Further assume that the distribution of the covariates satisfies Assumptions \ref{assump:subgauss} and \ref{assump:eigen}. For all $a\geq 1$, there exists $C=C(\lambda_{-},\lambda_{+},C_{X},a)$ and $H=H(\lambda_{-},\lambda_{+},C_{X},a)$ such that, provided $N\geq Hp$, we have
    \begin{align*}
        \mathbb{E}\left[\left\|\hat{\Sigma}^{C}-\Sigma\right\|_{\mathrm{op}}^{a}\right] \leq C\max_{(g,h)\in[L]^2}\left\{\left(\frac{p}{N}\right)^{\frac{a}{2}}\right\} \quad \text{and} \quad \mathbb{E}\left[\left(\frac{\max\{1,\lambda_{\max}(\hat{\Sigma}^{C})\}}{\min\{1,\lambda_{\min}(\hat{\Sigma}^{C})\}}\right)^{a}\right] \leq C.
    \end{align*}
\end{proposition}
\begin{proof}
    The proof follows exactly the same argument as Proposition \ref{Prop:Supervised covariance blocks}.
\end{proof}

\subsection{OSS proofs under a balancing assumption}\label{A:OSS Proofs under a balancing assumption}
This section consists solely of proving an upper bound for our estimator in the unstructured setting (Theorem \ref{thm:OSS balanced}).

\begin{proofof}{thm:OSS balanced}
    We first note that the sub-Gaussian assumption on the covariates allows us to assume that the distribution satisfies $(4,2,C_{4})$--hypercontractivity for some constant $C_{4}=C_{4}(C_{X},\lambda_{-})$. The first half of the proof involves showing that $\hat{B}^{C}$ is equal to $\hat{B}$ with high probability, allowing us to control the effect of replacing the latter with the former in the estimator.
    We define the event $E_{1}$ to be $\|\hat{\Sigma}-\Sigma\|_\mathrm{op}\leq \mu\equiv \min\{1,\frac{\lambda_{-}}{2}\}$. Assumption \eqref{eq:complicated eval assumption 2} yields that for a constant $F_{1}=F_{1}(C_{X},C_{\rho},\lambda_{+},\lambda_{-})$,
    \begin{equation*}
        \mathbb{P}\left(E_{1}\right) \geq 1-\frac{F_{1}}{(\rho^2n_{\mathcal{L}}+N)^{8}}.
    \end{equation*}
    We additionally have $\frac{\lambda_{-}}{2}I_{p}\preceq \hat{\Sigma}\preceq 2\lambda_{+}I_{p}$ on $E_{1}$. The next step of the proof is to define a high probability event $E_{2}$ such that we have $\hat{B}=\hat{B}^{C}$ on $E_{2}$ provided $\hat{\Sigma}$ is such that $E_{1}$ holds. We begin by controlling $\mathbb{E}[\hat{B}\big|\hat{\Sigma}]$ for $\hat{\Sigma}$ such that $E_{1}$ holds. We have that 
    \begin{align*}
        \mathbb{E}\left[\hat{\Sigma}^{-1}\hat{B}\hat{\Sigma}^{-1}\big|\hat{\Sigma}\right] & = \hat{\Sigma}^{-1}\sum_{i=1}^{n_{\mathcal{L}}}\hat{\Sigma}_{O_{\Xi(i)}}^T\hat{\Sigma}_{O_{\Xi(i)}O_{\Xi(i)}}^{-1}\Sigma_{O_{\Xi(i)O_{\Xi(i)}}}\hat{\Sigma}_{O_{\Xi(i)}O_{\Xi(i)}}^{-1}\hat{\Sigma}_{O_{\Xi(i)}}\hat{\Sigma}^{-1}.
    \end{align*}
    We therefore have on the event $E_1$ that 
    \begin{align*}
        \frac{\lambda_{-}}{4\lambda_{+}^2}\hat{\Sigma}^{-1}\sum_{i=1}^{n_{\mathcal{L}}}\hat{\Sigma}_{O_{\Xi(i)}}^T\hat{\Sigma}_{O_{\Xi(i)}}\hat{\Sigma}^{-1}\preceq \mathbb{E}\left[\hat{\Sigma}^{-1}\hat{B}\hat{\Sigma}^{-1}\big|\hat{\Sigma}\right] \preceq \frac{4\lambda_{+}}{\lambda_{-}^2}\hat{\Sigma}^{-1}\sum_{i=1}^{n_{\mathcal{L}}}\hat{\Sigma}_{O_{\Xi(i)}}^T\hat{\Sigma}_{O_{\Xi(i)}}\hat{\Sigma}^{-1}.
    \end{align*}
    The above display and the conditions on the set $\{h_{g}\}_{g=1}^L$ in the Theorem yield that 
    \begin{align*}
        \frac{\lambda_{-}}{4\lambda_{+}^2}\rho n_{\mathcal{L}}I_{p}\preceq \mathbb{E}\left[\hat{\Sigma}^{-1}\hat{B}\hat{\Sigma}^{-1}\big|\hat{\Sigma}\right] \preceq \frac{4\lambda_{+}C_{\rho}}{\lambda_{-}^2}\rho n_{\mathcal{L}}I_{p}.
    \end{align*}
    Pre- and post-multiplying by $\hat{\Sigma}$ yields that, on the event $E_1$,
    \begin{align}\label{eq:E1 event properties}
          \frac{\lambda_{-}^3}{16\lambda_{+}^2}\rho n_{\mathcal{L}}I_{p}\preceq \mathbb{E}\left[\hat{B}\big|\hat{\Sigma}\right] \preceq \frac{16\lambda_{+}^3C_{\rho}}{\lambda_{-}^2}\rho n_{\mathcal{L}}I_{p}.
    \end{align}
    To simplify notation, define $\mu_{+}=\frac{16\lambda_{+}^3C_{\rho}}{\lambda_{-}^2}$ and $\mu_{-}=\frac{\lambda_{-}^3}{16\lambda_{+}^2}$ and set $\omega = \min\{1,\frac{\mu_{-}}{2},\mu_{+}\}$. Additionally, define for $i \in [n_{\mathcal{L}}],      W_{i}=\hat{P}_{\Xi(i)}^T(X_{i})_{O_{\Xi(i)}}(X_{i})_{O_{\Xi(i)}}^T\hat{P}_{\Xi(i)} $ such that $\hat{B}=\sum_{i=1}^{n_{\mathcal{L}}}W_{i}$. For $i \in [n_{\mathcal{L}}]$, we also define $A_{i}^2=p\hat{\Sigma}_{O_{\Xi(i)}}^T
    \hat{\Sigma}_{O_{\Xi(i)}}$. For the subsequent result, we restrict to the $\hat{\Sigma}$ such that $E_{1}$ holds. For some universal constant $H_{1}>0$ and $R_{X}= R_{X}(\lambda_{-},\lambda_{+},C_{X})$ some constant depending only on $\lambda_{-},\lambda_{+}$ and $C_{X}$, we have that for $\theta \in \mathbb{R}^{p}$ and for $q = 2,3,\ldots$
    \begin{align*}
    &\theta^T\mathbb{E}\!\left[W_{i}^q|\hat{\Sigma}\right]\theta\\
    & =\mathbb{E}\left[\|\hat{P}_{\Xi(i)}^T(X_{i})_{O_{\Xi(i)}}\|^{2q-2}\left(\theta^T\hat{P}_{\Xi(i)}^T(X_{i})_{O_{\Xi(i)}}\right)^2\big|\hat{\Sigma}\right]\\
    &\leq 
    \mathbb{E}\!\left[
        \|\hat{P}_{\Xi(i)}\|_\mathrm{op}^{4q-4}
        \|(X_{i})_{O_{\Xi(i)}}\|_{2}^{4q-4}
        \,\big|\hat{\Sigma}
    \right]^{\frac{1}{2}}
    \mathbb{E}\!\left[
        \left(\theta^T\hat{P}_{\Xi(i)}(X_{i})_{O_{\Xi(i)}}\right)^4
        \,\big|\hat{\Sigma}
    \right]^{\frac{1}{2}}\\[2mm]
    &\leq 
    \left(\frac{4\lambda_{+}}{\lambda_{-}}\right)^{2q-2}
    \mathbb{E}\!\left[
        \|(X_{i})_{O_{\Xi(i)}}\|_{2}^{4q-4}
    \right]^{\frac{1}{2}}
    C_{4}^2
    \mathbb{E}\!\left[
        \left(\theta^T\hat{P}_{\Xi(i)}(X_{i})_{O_{\Xi(i)}}\right)^2
        \,\big|\hat{\Sigma}
    \right]\\[2mm]
    &\leq 
    \left(\frac{4\lambda_{+}}{\lambda_{-}}\right)^{2q-2}\frac{4\lambda_{+}}{\lambda_{-}^2}
    C_{4}^2(4p)^{q-1}
    \left((H_{1}C_{X})^{(4q-4)}\sqrt{4q-4}^{(4q-4)}+(H_{1}\lambda_{+})^{2q-2}\right)^{\frac{1}{2}}
    \theta^T
    \hat{\Sigma}_{O_{\Xi(i)}}^T
    \hat{\Sigma}_{O_{\Xi(i)}}\theta\\
    &\leq \frac{q!}{2}(R_{X}p)^{q-2}\theta^TA_{i}^2\theta,
\end{align*}
 where the first inequality holds by Cauchy-Schwarz and the submultiplicativity of the operator norm; the second inequality holds by hypercontractivity and properties of $E_{1}$; the third inequality holds via a standard bound on the norm of a sub-Gaussian random vector e.g., \cite[Lemmas~1 and~2]{jin2019short}. Define $\sigma_{X}^2 = \left\|\sum_{i=1}^{n_{\mathcal{L}}}A_{i}^2\right\|_\mathrm{op}$, so that on $E_{1}$ we have 
\begin{align*}
    \sigma_{X}^2&= p\left\|\hat{\Sigma}\hat{\Sigma}^{-1}\left(\sum_{i=1}^{n_{\mathcal{L}}}\hat{\Sigma}_{O_{\Xi(i)}}^T\hat{\Sigma}_{O_{\Xi(i)}}\right)\hat{\Sigma}^{-1}\hat{\Sigma}\right\|_\mathrm{op}\leq C_{\rho}\rho n_{\mathcal{L}}p\left\|\hat{\Sigma}^2\right\|_\mathrm{op}\leq 4C_{\rho}\lambda_{+}^2p\rho n_{\mathcal{L}}.
\end{align*}
Therefore, by the matrix Bernstein inequality~\citep[Lemma 3]{zhu2022high}, we have  for some constant $F_{3}=F_{3}(C_{X},\lambda_{+},\lambda_{-},C_{\rho})$ that 
\begin{align}\label{eq:E2 event definition}
        \mathbb{P}\left(\left\|\frac{1}{\rho n_{\mathcal{L}}}\left(\hat{B}-\mathbb{E}\left[\hat{B}\big|\hat{\Sigma}\right]\right)\right\|_\mathrm{op}\geq \omega\bigg|\hat{\Sigma}\right) \leq 8pe^{-\frac{F_{3}\rho n_{\mathcal{L}}\omega^2}{p}}.
\end{align}
    Define the complement of this event to be $E_{2}$ such that, on $E_{1}$, we have $\mathbb{P}(E_{2}^c|\hat{\Sigma})\leq 8pe^{-\frac{F_{3}\rho n_{\mathcal{L}}\omega^2}{p}}$. Further, define $A$ to be the event $E_{1}\cap E_{2}$. We have by the above results that
    \begin{equation*}
        \mathbb{P}\left(A\right) = \mathbb{P}(E_1)-\mathbb{P}(E_1 \cap E_2^c) = 1 - \mathbb{P}(E_1^c) - \mathbb{E}\{\mathbb{P}(E_2^c | \hat{\Sigma}) \mathbbm{1}_{E_1}\} \geq 1-8pe^{-\frac{F_{3}\rho n_{\mathcal{L}}\omega^2}{p}}-\frac{F_{1}}{(\rho^2n_{\mathcal{L}}+N)^{8}}.
    \end{equation*}
    Combining \eqref{eq:E1 event properties} and \eqref{eq:E2 event definition} yields that on $A$
    \begin{align*}
        \frac{\lambda_{-}^3}{32\lambda_{+}^2}\rho n_{\mathcal{L}} \leq \lambda_{\min}(\hat{B}) \leq \lambda_{\max}(\hat{B}) \leq \frac{32\lambda_{+}^3C_{\rho}}{\lambda_{-}^2}\rho n_{\mathcal{L}}.
    \end{align*}
    It follows that $\hat{B}=\hat{B}^{C}$ on $A$. We additionally have that
    \begin{align}\label{eq:B hat crude upper bound}
        \|\hat{B}\|_\mathrm{op} \leq \frac{\lambda_{\max}^2(\hat{\Sigma})}{\lambda_{\min}^2(\hat{\Sigma})}\sum_{i=1}^{n_{\mathcal{L}}}\|X_{O_{\Xi(i)}}\|_{2}^2.
    \end{align}
    By Jensen's inequality, it follows that
    \begin{align*}
        \|\hat{B}\|_\mathrm{op}^4\leq \frac{\lambda_{\max}^8(\hat{\Sigma})}{\lambda_{\min}^8(\hat{\Sigma})}\left(\frac{\sum_{i=1}^{n_{\mathcal{L}}}\|X_{O_{\Xi(i)}}\|_{2}^2}{n_{\mathcal{L}}}\right)^4n_{\mathcal{L}}^4 \leq \frac{\lambda_{\max}^8(\hat{\Sigma})}{\lambda_{\min}^8(\hat{\Sigma})}\left(\sum_{i=1}^{n_{\mathcal{L}}}\|X_{i}\|_{2}^8\right)n_{\mathcal{L}}^3.
    \end{align*}
    Throughout the rest of the proof $F=F(C_{X},C_{\rho},C_{*},\lambda_{-},\lambda_{+})$ denotes a constant depending only on $C_{X},C_{\rho},C_{*},\lambda_{-}$ and $\lambda_{+}$.
    For later use, we bound the following quantity 
    \begin{align}\label{eq:hat B minus hat BF bound}
        &\mathbb{E}\left[\left\|\frac{\hat{B}-\hat{B}^{C}}{\rho n_{\mathcal{L}}}\right\|_\mathrm{op}^2\left\|\left(\frac{\hat{B}^{C}}{\rho n_{\mathcal{L}}}\right)^{-1}\right\|_\mathrm{op}^2\right]\nonumber\\ & = \mathbb{E}\left[\left\|\frac{\hat{B}-\hat{B}^{C}}{\rho n_{\mathcal{L}}}\right\|_\mathrm{op}^2\left\|\left(\frac{\hat{B}^{C}}{\rho n_{\mathcal{L}}}\right)^{-1}\right\|_\mathrm{op}^2\mathbbm{1}_{A}\right]+\mathbb{E}\left[\left\|\frac{\hat{B}-\hat{B}^{C}}{\rho n_{\mathcal{L}}}\right\|_\mathrm{op}^2\left\|\left(\frac{\hat{B}^{C}}{\rho n_{\mathcal{L}}}\right)^{-1}\right\|_\mathrm{op}^2\mathbbm{1}_{A^c}\right]\nonumber\\
        &\leq 0 + F\mathbb{P}\left(A^c\right)^{\frac{1}{2}}\mathbb{E}\left[\left\|\frac{\hat{B}-\hat{B}^{C}}{\rho n_{\mathcal{L}}}\right\|_\mathrm{op}^4\right]^{\frac{1}{2}}\nonumber\\
        &\leq F\left(pe^{-\frac{F_{3}\rho n_{\mathcal{L}}\omega^2}{p}}+\frac{1}{(\rho^2n_{\mathcal{L}}+N)^{8}}\right)^{\frac{1}{2}}\left(\mathbb{E}\left[\left\|\frac{\hat{B}}{\rho n_{\mathcal{L}}}\right\|_\mathrm{op}^4\right]+\mathbb{E}\left[\left\|\frac{\hat{B}^{C}}{\rho n_{\mathcal{L}}}\right\|_\mathrm{op}^4\right]\right)^{\frac{1}{2}}\nonumber\\
        &\leq F\max\left\{p^{\frac{1}{2}}e^{-\frac{F_{3}\rho n_{\mathcal{L}}\omega^2}{2p}},\frac{1}{(\rho^2n_{\mathcal{L}}+N)^{4}}\right\}\left(\mathbb{E}\left[\frac{\lambda_{\max}^8(\hat{\Sigma})}{\lambda_{\min}^8(\hat{\Sigma})}\left(\sum_{i=1}^{n_{\mathcal{L}}}\|X_{i}\|_{2}^8\right)\frac{n_{\mathcal{L}}^3}{(\rho n_{\mathcal{L}})^4}\right]+1\right)^{\frac{1}{2}}\nonumber\\
        &\leq F\max\left\{p^{\frac{1}{2}}e^{-\frac{F_{3}\rho n_{\mathcal{L}}\omega^2}{2p}},\frac{1}{(\rho^2n_{\mathcal{L}}+N)^{4}}\right\}\left(\frac{p^4}{\rho^4}+1\right)^{\frac{1}{2}}\nonumber\\
        &\leq F\max\left\{\max\left\{\frac{p^{\frac{5}{2}}}{\rho^2},p^{\frac{1}{2}}\right\}e^{-\frac{F_{3}\rho n_{\mathcal{L}}\omega^2}{2p}},\max\left\{\frac{p^2}{\rho^2},1\right\}\frac{1}{(\rho^2n_{\mathcal{L}}+N)^{4}}\right\}\nonumber\\
        &\leq F\left(\frac{p}{\rho n_{\mathcal{L}}}+\frac{p}{\rho^2n_{\mathcal{L}}+N}\right),
    \end{align}
    where the first inequality uses properties of $A$ and Cauchy-Schwarz; the second inequality uses the fact that $(a+b)^4 \leq 8(a^4+b^4)$; the third inequality uses the fact that for $a,b>0$, $\sqrt{a+b}\leq \sqrt{a}+\sqrt{b}$ and \eqref{eq:B hat crude upper bound}; the fourth inequality uses a standard bound on the norm of a sub-Gaussian random vector e.g., \cite[Lemmas~1 and~2]{jin2019short}; the final line follows since $\min\{\rho^2n_{\mathcal{L}}+N,\rho n_{\mathcal{L}}\}\geq G\log(\frac{p}{\rho})\max\{p,\frac{1}{\rho}\}$ for some $G=G(C_{X},C_{\rho},C_{*},\lambda_{-},\lambda_{+})$.

We note that by construction 
\begin{align*}
    \hat{B}^C \succeq \frac{\lambda_{-}^3}{32\lambda_{+}^2}\rho n_{\mathcal{L}}I_{p}.
\end{align*}
It follows for $k \in \mathbb{N}$, that we have 
\begin{equation}\label{eq:hat B F bound}
    \frac{1}{\lambda_{\min}^{k}(\hat{B}^C)} \leq \frac{F^k}{(\rho n_{\mathcal{L}})^{k}}.
\end{equation}

Having established important properties of $\hat{B}^\mathrm{C}$, we may now turn to decomposing the error. We now recall some additional notation also used in the proof of Theorem \ref{thm:Complicated Master OSS with estimating weights}.
For some $P=P_{j}$ and $(X,Y)$ independent of $\hat{\Sigma}$, it holds that
\begin{align*}
    Y &= X^T\beta^*+\epsilon = X_{O}^TP\beta^*+X^T\beta^*-X_{O}^TP\beta^*+\epsilon\\
    &= X_{O}^T\hat{P}\beta^*+X_{O}^T(P-\hat{P})\beta^*+X^T\beta^*-X_{O}^TP\beta^*+\epsilon\\
    &\equiv \hat{V}^T\beta^* + \gamma + \eta,
\end{align*}
where $\gamma = X_{O}^T(P-\hat{P})\beta^*$. Therefore, we have that
\begin{align*}
    \hat{V}_{i}Y_{i} &= \hat{V}_{i}\hat{V}_{i}^T\beta^*+\hat{V}_{i}\gamma_{i}+\hat{V}_{i}\eta_{i}\\
    & = \hat{V}_{i}\hat{V}_{i}^T\beta^*+\mathbb{E}\left[\hat{V}_{i}\gamma_{i}|\hat{\Sigma}\right]+\hat{V}_{i}\gamma_{i}-\mathbb{E}\left[\hat{V}_{i}\gamma_{i}|\hat{\Sigma}\right]+\hat{V}_{i}\eta_{i}\\
    &\equiv \hat{V}_{i}\hat{V}_{i}^T\beta^*+\mathbb{E}\left[\hat{V}_{i}\gamma_{i}|\hat{\Sigma}\right] + \delta_{i}.
\end{align*}
Similar to \eqref{eq:OSS structured error decomposition}, we can decompose the error in general as follows 
\begin{align}
    \|\hat{\beta}-\beta^*\|_{2}^2 &= \left\|(\hat{B}^C)^{-1}\sum_{i=1}^{n_{\mathcal{L}}}\hat{V}_{i}Y_{i}-\beta^*\right\|_{2}^2\nonumber\\
    &= \left\|(\hat{B}^C)^{-1}\sum_{i=1}^{n_{\mathcal{L}}}\left(\mathbb{E}\left[\hat{V}_{i}\gamma_{i}|\hat{\Sigma}\right] + \delta_{i}\right)+(\hat{B}^C)^{-1}\hat{B}\beta^*-\beta^*\right\|_{2}^2\nonumber\\
    &\leq 2\left\|(\hat{B}^{C})^{-1}\sum_{i=1}^{n_{\mathcal{L}}}\mathbb{E}\left[\hat{V}_{i}\gamma_{i}|\hat{\Sigma}\right]+(\hat{B}^{C})^{-1}\sum_{i=1}^{n_{\mathcal{L}}}\delta_{i}\right\|_{2}^2+\left\|\left(\hat{B}^{C}\right)^{-1}\left(\hat{B}-\hat{B}^{C}\right)\right\|_\mathrm{op}^2\|\beta^*\|_{2}^2\nonumber\\
    & = 2\left\|(\hat{B}^{C})^{-1}\sum_{i=1}^{n_{\mathcal{L}}}\mathbb{E}\left[\hat{V}_{i}\gamma_{i}|\hat{\Sigma}\right]+(\hat{B}^{C})^{-1}\sum_{i=1}^{n_{\mathcal{L}}}\left(\hat{V}_{i}\gamma_{i}-\mathbb{E}\left[\hat{V}_{i}\gamma_{i}|\hat{\Sigma}\right]+\hat{V}_{i}\eta_{i}\right)\right\|_{2}^2 \nonumber\\
    &\hspace{4mm}+2\left\|\left(\hat{B}^{C}\right)^{-1}\left(\hat{B}-\hat{B}^{C}\right)\right\|_\mathrm{op}^2\|\beta^*\|_{2}^2\nonumber\\
    & \leq 6\left\|(\hat{B}^{C})^{-1}\sum_{i=1}^{n_{\mathcal{L}}}\mathbb{E}\left[\hat{V}_{i}\gamma_{i}|\hat{\Sigma}\right]\right\|_{2}^2+6\left\|(\hat{B}^{C})^{-1}\sum_{i=1}^{n_{\mathcal{L}}}\left(\hat{V}_{i}\gamma_{i}-\mathbb{E}\left[\hat{V}_{i}\gamma_{i}|\hat{\Sigma}\right]\right)\right\|_{2}^2\nonumber\\
    \hspace{4mm}&+6\left\|(\hat{B}^{C})^{-1}\sum_{i=1}^{n_{\mathcal{L}}}\hat{V}_{i}\eta_{i}\right\|_{2}^2+2\left\|\left(\hat{B}^{C}\right)^{-1}\left(\hat{B}-\hat{B}^{C}\right)\right\|_\mathrm{op}^2\|\beta^*\|_{2}^2\nonumber\\
    & \equiv \circled{1}+\circled{2}+\circled{3}+\circled{4}\label{eq:balanced error decomposition},
\end{align}
where the final inequality follows from the fact that for $a,b,c\in \mathbb{R}$ we have $(a+b+c)^2\leq 3(a^2+b^2+c^2)$. We now begin bounding each term in the decomposition \eqref{eq:balanced error decomposition} in turn. 
\subsubsection*{Controlling $\circled{1}$}
We control $\circled{1}$ as follows 
    \begin{align}\label{eq: circled 1 preliminary bound}
    \mathbb{E}\left[\left\|(\hat{B}^{C})^{-1} \sum_{i=1}^{n_{\mathcal{L}}}  \mathbb{E}\left[\hat{V}_{i} \gamma_{i} \mid \hat{\Sigma}\right]\right\|_{2}^2\right]\nonumber
    &\leq\frac{F}{n_{\mathcal{L}}^2\rho^2}
    \left( \mathbb{E}\left[ \left\| \sum_{i=1}^{n_{\mathcal{L}}}\mathbb{E}\left[\hat{V}_{i}(X_{i})_{O_{\Xi(i)}}^T(P_{\Xi(i)}-\hat{P}_{\Xi(i)})\beta^*|\hat{\Sigma}\right] \right\|_{2}^2 \right] \right)\nonumber\\
    &=\frac{F}{n_{\mathcal{L}}^2\rho^2}
    \left( \mathbb{E}\left[ \left\| \sum_{i=1}^{n_{\mathcal{L}}}\hat{P}_{\Xi(i)}^T\Sigma_{O_{\Xi(i)}O_{\Xi(i)}}(P_{\Xi(i)}-\hat{P}_{\Xi(i)})\beta^* \right\|_{2}^2 \right]\right),
\end{align}
where the first line follows from \eqref{eq:hat B F bound}.
For $i \in [n_{\mathcal{L}}]$, define $S_{\Xi(i)}\in \mathbb{R}^{ p\times|O_{\Xi(i)}|}$ via $S_{\Xi(i)}\equiv\hat{\Sigma}^{-1}\hat{\Sigma}_{O_{\Xi(i)}}^T=\Sigma^{-1}\Sigma_{O_{\Xi(i)}}^T$. Here, $S_{\Xi(i)}$  `selects' certain columns of $\Sigma$ to produce $\Sigma_{O_{\Xi(i)}}^T$. We note the following decomposition
\begin{align*}
    &\hat{P}_{\Xi(i)}^T\Sigma_{O_{\Xi(i)}O_{\Xi(i)}}(P_{\Xi(i)}-\hat{P}_{\Xi(i)})\beta^* \\& = \hat{\Sigma}S_{\Xi(i)}\hat{\Sigma}_{O_{\Xi(i)}O_{\Xi(i)}}^{-1}\Sigma_{O_{\Xi(i)}O_{\Xi(i)}}\Sigma_{O_{\Xi(i)}O_{\Xi(i)}}^{-1}S_{\Xi(i)}^T\Sigma\beta^* - \hat{\Sigma}S_{\Xi(i)}\hat{\Sigma}_{O_{\Xi(i)}O_{\Xi(i)}}^{-1}\Sigma_{O_{\Xi(i)}O_{\Xi(i)}}\hat{\Sigma}_{O_{\Xi(i)}O_{\Xi(i)}}^{-1}S_{\Xi(i)}^T\hat{\Sigma}\beta^* \\
    & = \hat{\Sigma}S_{\Xi(i)}\hat{\Sigma}_{O_{\Xi(i)}O_{\Xi(i)}}^{-1}S_{\Xi(i)}^T(\Sigma-\hat{\Sigma})\beta^*+\hat{\Sigma}S_{\Xi(i)}\left(\hat{\Sigma}_{O_{\Xi(i)}O_{\Xi(i)}}^{-1}-\hat{\Sigma}_{O_{\Xi(i)}O_{\Xi(i)}}^{-1}\Sigma_{O_{\Xi(i)}O_{\Xi(i)}}\hat{\Sigma}_{O_{\Xi(i)}O_{\Xi(i)}}^{-1}\right)S_{\Xi(i)}^T\hat{\Sigma}\beta^*.
\end{align*}
By the previous identity, it holds that
\begin{align*}
    &\left\| \sum_{i=1}^{n_{\mathcal{L}}}\hat{P}_{\Xi(i)}^T\Sigma_{O_{\Xi(i)}O_{\Xi(i)}}(P_{\Xi(i)}-\hat{P}_{\Xi(i)})\beta^* \right\|_{2}^2\\
    & \leq 2\left\|\hat{\Sigma}\left(\sum_{i=1}^{n_{\mathcal{L}}}S_{\Xi(i)}\left(\hat{\Sigma}_{O_{\Xi(i)}O_{\Xi(i)}}^{-1}-\hat{\Sigma}_{O_{\Xi(i)}O_{\Xi(i)}}^{-1}\Sigma_{O_{\Xi(i)}O_{\Xi(i)}}\hat{\Sigma}_{O_{\Xi(i)}O_{\Xi(i)}}^{-1}\right)S_{\Xi(i)}^T\right)\hat{\Sigma}\beta^*\right\|_{2}^2\\
    &\hspace{4mm}+2\left\|\hat{\Sigma}\left(\sum_{i=1}^{n_{\mathcal{L}}}S_{\Xi(i)}\hat{\Sigma}_{O_{\Xi(i)}O_{\Xi(i)}}^{-1}S_{\Xi(i)}^T\right)(\Sigma-\hat{\Sigma})\beta^*\right\|_{2}^2\\
    &\leq 2\lambda_{\max}^4(\hat{\Sigma})\|\beta^*\|_{2}^2\left\|\sum_{i=1}^{n_{\mathcal{L}}}S_{\Xi(i)}\left(\hat{\Sigma}_{O_{\Xi(i)}O_{\Xi(i)}}^{-1}-\hat{\Sigma}_{O_{\Xi(i)}O_{\Xi(i)}}^{-1}\Sigma_{O_{\Xi(i)}O_{\Xi(i)}}\hat{\Sigma}_{O_{\Xi(i)}O_{\Xi(i)}}^{-1}\right)S_{\Xi(i)}^T\right\|_\mathrm{op}^2\\
    &\hspace{4mm}+2\lambda_{\max}^2(\hat{\Sigma})\|\Sigma-\hat{\Sigma}\|_\mathrm{op}^2\|\beta^*\|_{2}^2\left\|\sum_{i=1}^{n_{\mathcal{L}}}S_{\Xi(i)}\hat{\Sigma}_{O_{\Xi(i)}O_{\Xi(i)}}^{-1}S_{\Xi(i)}^T\right\|_\mathrm{op}^2\\
    &= 2\lambda_{\max}^4(\hat{\Sigma})\|\beta^*\|_{2}^2\sup_{\theta \in \mathcal{S}^{p-1}}\left\{\left|\sum_{i=1}^{n_{\mathcal{L}}}\theta^TS_{\Xi(i)}\left(\hat{\Sigma}_{O_{\Xi(i)}O_{\Xi(i)}}^{-1}-\hat{\Sigma}_{O_{\Xi(i)}O_{\Xi(i)}}^{-1}\Sigma_{O_{\Xi(i)}O_{\Xi(i)}}\hat{\Sigma}_{O_{\Xi(i)}O_{\Xi(i)}}^{-1}\right)S_{\Xi(i)}^T\theta\right|\right\}^2\\
    &\hspace{4mm}+2\lambda_{\max}^2(\hat{\Sigma})\|\Sigma-\hat{\Sigma}\|_\mathrm{op}^2\|\beta^*\|_{2}^2\sup_{\theta \in \mathcal{S}^{p-1}}\left\{\sum_{i=1}^{n_{\mathcal{L}}}\theta^TS_{\Xi(i)}\hat{\Sigma}_{O_{\Xi(i)}O_{\Xi(i)}}^{-1}S_{\Xi(i)}^T\theta\right\}^2\\
    &\leq 2\lambda_{\max}^4(\hat{\Sigma})\|\beta^*\|_{2}^2\sup_{\theta \in \mathcal{S}^{p-1}}\left\{\sum_{i=1}^{n_{\mathcal{L}}}\left|\theta^TS_{\Xi(i)}\left(\hat{\Sigma}_{O_{\Xi(i)}O_{\Xi(i)}}^{-1}-\hat{\Sigma}_{O_{\Xi(i)}O_{\Xi(i)}}^{-1}\Sigma_{O_{\Xi(i)}O_{\Xi(i)}}\hat{\Sigma}_{O_{\Xi(i)}O_{\Xi(i)}}^{-1}\right)S_{\Xi(i)}^T\theta\right|\right\}^2\\
    &\hspace{4mm}+2\lambda_{\max}^2(\hat{\Sigma})\|\Sigma-\hat{\Sigma}\|_\mathrm{op}^2\|\beta^*\|_{2}^2\sup_{\theta \in \mathcal{S}^{p-1}}\left\{\sum_{i=1}^{n_{\mathcal{L}}}\theta^TS_{\Xi(i)}\hat{\Sigma}_{O_{\Xi(i)}O_{\Xi(i)}}^{-1}S_{\Xi(i)}^T\theta\right\}^2\\
    &\leq 2\lambda_{\max}^4(\hat{\Sigma})\|\beta^*\|_{2}^2\sup_{\theta \in \mathcal{S}^{p-1}}\left\{\sum_{i=1}^{n_{\mathcal{L}}}\theta^TS_{\Xi(i)}\left(\|\hat{\Sigma}_{O_{\Xi(i)}O_{\Xi(i)}}^{-1}\left(\hat{\Sigma}_{O_{\Xi(i)}O_{\Xi(i)}}-\Sigma_{O_{\Xi(i)}O_{\Xi(i)}}\right)\hat{\Sigma}_{O_{\Xi(i)}O_{\Xi(i)}}^{-1}\|_\mathrm{op}\right)S_{\Xi(i)}^T\theta\right\}^2\\
    &\hspace{4mm}+2\lambda_{\max}^2(\hat{\Sigma})\|\Sigma-\hat{\Sigma}\|_\mathrm{op}^2\|\beta^*\|_{2}^2\sup_{\theta \in \mathcal{S}^{p-1}}\left\{\sum_{i=1}^{n_{\mathcal{L}}}\theta^TS_{\Xi(i)}\|\hat{\Sigma}_{O_{\Xi(i)}O_{\Xi(i)}}^{-1}\|_\mathrm{op}S_{\Xi(i)}^T\theta\right\}^2\\
    &\leq 2\frac{\lambda_{\max}^4(\hat{\Sigma})}{\lambda_{\min}^4(\hat{\Sigma})}\|\hat{\Sigma}-\Sigma\|_\mathrm{op}^2\|\beta^*\|_{2}^2\sup_{\theta \in \mathcal{S}^{p-1}}\left\{\sum_{i=1}^{n_{\mathcal{L}}}\theta^TS_{\Xi(i)}S_{\Xi(i)}^T\theta\right\}^2\\
    &\hspace{4mm}+2\frac{\lambda_{\max}^2(\hat{\Sigma})}{\lambda_{\min}^2(\hat{\Sigma})}\|\Sigma-\hat{\Sigma}\|_\mathrm{op}^2\|\beta^*\|_{2}^2\sup_{\theta \in \mathcal{S}^{p-1}}\left\{\sum_{i=1}^{n_{\mathcal{L}}}\theta^TS_{\Xi(i)}S_{\Xi(i)}^T\theta\right\}^2\\
    & = \left(2\frac{\lambda_{\max}^4(\hat{\Sigma})}{\lambda_{\min}^4(\hat{\Sigma})}\|\hat{\Sigma}-\Sigma\|_\mathrm{op}^2\|\beta^*\|_{2}^2+2\frac{\lambda_{\max}^2(\hat{\Sigma})}{\lambda_{\min}^2(\hat{\Sigma})}\|\Sigma-\hat{\Sigma}\|_\mathrm{op}^2\|\beta^*\|_{2}^2\right)\left\|\sum_{i=1}^{n_{\mathcal{L}}}S_{\Xi(i)}S_{\Xi(i)}^T\right\|_\mathrm{op}^2\\
    & = \left(2\frac{\lambda_{\max}^4(\hat{\Sigma})}{\lambda_{\min}^4(\hat{\Sigma})}\|\hat{\Sigma}-\Sigma\|_\mathrm{op}^2\|\beta^*\|_{2}^2+2\frac{\lambda_{\max}^2(\hat{\Sigma})}{\lambda_{\min}^2(\hat{\Sigma})}\|\Sigma-\hat{\Sigma}\|_\mathrm{op}^2\|\beta^*\|_{2}^2\right)\left\|\sum_{i=1}^{n_{\mathcal{L}}}\sum_{r=1}^p\mathbbm{1}_{\{r \in O_{\Xi(i)}\}}e_{r}e_{r}^T\right\|_\mathrm{op}^2\\
    &\leq 4C_{\rho}^2\frac{\lambda_{\max}^4(\hat{\Sigma})}{\lambda_{\min}^4(\hat{\Sigma})}\|\hat{\Sigma}-\Sigma\|_\mathrm{op}^2\|\beta^*\|_{2}^2\rho^2n_{\mathcal{L}}^2,
\end{align*}
where the first and third inequalities follow from the triangle inequality; the second and fifth inequalities follow from the submultiplicativity of the operator norm and the final line follows from the balancing condition on the modalities. 

Applying the previous bound to \eqref{eq: circled 1 preliminary bound} and using the conditions of the theorem yields
\begin{align}\label{eq:balanced circled 1 final bound}
    \mathbb{E}\left[\left\|(\hat{B}^C)^{-1} \sum_{i=1}^{n_{\mathcal{L}}}  \mathbb{E}\left[\hat{V}_{i} \gamma_{i} \mid \hat{\Sigma}\right]\right\|_{2}^2\right] &\leq \frac{F}{\rho^2n_{\mathcal{L}}^2}\mathbb{E}\left[\frac{\lambda_{\max}^4(\hat{\Sigma})}{\lambda_{\min}^4(\hat{\Sigma})}\|\hat{\Sigma}-\Sigma\|_\mathrm{op}^2\|\beta^*\|_{2}^2\rho^2n_{\mathcal{L}}^2\right]\nonumber\\
    &\leq F\frac{p}{\rho^2n_{\mathcal{L}}+N}\|\beta^*\|_{2}^2.
\end{align}

\subsubsection*{Controlling $\circled{2}$}
We now bound $\circled{2}$ as follows
\begin{align}\label{eq:balanced circled 2 final bound}
    &\mathbb{E}\left[\left\|(\hat{B}^{C})^{-1}\sum_{i=1}^{n_{\mathcal{L}}}\hat{V}_{i}\gamma_{i}-\mathbb{E}\left[\hat{V}_{i}\gamma_{i}|\hat{\Sigma}\right]\right\|_{2}^2\right]\nonumber\\
    &\leq \mathbb{E}\left[\lambda_{\max}^2\left((\hat{B}^{C})^{-1}\right)\left\|\sum_{i=1}^{n_{\mathcal{L}}}\hat{V}_{i}\gamma_{i}-\mathbb{E}\left[\hat{V}_{i}\gamma_{i}|\hat{\Sigma}\right]\right\|_{2}^2\right]\nonumber\\
    &\leq \frac{F}{\rho^2n_{\mathcal{L}}^2}\mathbb{E}\left[\left\|\sum_{i=1}^{n_{\mathcal{L}}}\hat{V}_{i}\gamma_{i}-\mathbb{E}\left[\hat{V}_{i}\gamma_{i}|\hat{\Sigma}\right]\right\|_{2}^2\right]\nonumber\\
    & =\frac{F}{\rho^2n_{\mathcal{L}}^2}\mathbb{E}\left[\sum_{i=1}^{n_{\mathcal{L}}}\left\|\hat{V}_{i}\gamma_{i}-\mathbb{E}\left[\hat{V}_{i}\gamma_{i}|\hat{\Sigma}\right]\right\|_{2}^2\right]\nonumber\\
    &\leq \frac{F}{\rho^2n_{\mathcal{L}}^2}\sum_{i=1}^{n_{\mathcal{L}}}\mathbb{E}\left[\left\|\hat{V}_{i}\gamma_{i}\right\|_{2}^2+\left\|\mathbb{E}\left[\hat{V}_{i}\gamma_{i}\big|\hat{\Sigma}\right]\right\|_{2}^2\right]\nonumber\\
    &\leq \frac{F}{\rho^2n_{\mathcal{L}}^2}\sum_{i=1}^{n_{\mathcal{L}}}\mathbb{E}\left[\left\|\hat{V}_{i}\gamma_{i}\right\|_{2}^2+\mathbb{E}\left[\left\|\hat{V}_{i}\gamma_{i}\right\|_{2}^2\big|\hat{\Sigma}\right]\right]\nonumber\\
    &\leq \frac{F}{\rho^2n_{\mathcal{L}}^2}\sum_{i=1}^{n_{\mathcal{L}}}\mathbb{E}\left[\left\|\hat{V}_{i}\gamma_{i}\right\|_{2}^2\right]\nonumber\\
    &\leq \frac{F}{\rho^2n_{\mathcal{L}}^2}\sum_{i=1}^{n_{\mathcal{L}}}\mathbb{E}\left[\|\hat{V}_{i}\|_{2}^4\right]^{\frac{1}{2}}\mathbb{E}\left[\gamma_{i}^4\right]^{\frac{1}{2}}\nonumber\\
    &= \frac{F}{\rho^2n_{\mathcal{L}}^2}\sum_{i=1}^{n_{\mathcal{L}}}\mathbb{E}\left[\|\hat{P}_{O_{\Xi(i)}}^T (X_{i})_{O_{\Xi(i)}}\|_{2}^4\right]^{\frac{1}{2}}\mathbb{E}\left[\mathbb{E}\left[\left( (X_{i})_{O_{\Xi(i)}}^{\top} 
        \bigl(P_{O_{\Xi(i)}} - \hat{P}_{O_{\Xi(i)}}\bigr)\beta^* \right)^4\big|\hat{\Sigma}\right]\right]^{\frac{1}{2}}\nonumber\\
    &\leq \frac{F}{\rho^2n_{\mathcal{L}}^2}\sum_{i=1}^{n_{\mathcal{L}}}\mathbb{E}\left[\frac{\lambda_{\max}^4(\hat{\Sigma})}{\lambda_{\min}^4(\hat{\Sigma})}\|(X_{i})_{O_{\Xi(i)}}\|_{2}^4\right]^{\frac{1}{2}}\mathbb{E}\left[\left((\beta^*)^T  
    \bigl(P_{O_{\Xi(i)}} - \hat{P}_{O_{\Xi(i)}}\bigr)^T\Sigma_{O_{\Xi(i)\Xi(i)}}\bigl(P_{O_{\Xi(i)}} - \hat{P}_{O_{\Xi(i)}}\bigr)\beta^* \right)^2\right]^{\frac{1}{2}}\nonumber\\
    &\leq \frac{F}{\rho^2n_{\mathcal{L}}^2}\sum_{i=1}^{n_{\mathcal{L}}}|O_{\Xi(i)}|\|\beta^*\|_{2}^2\lambda_{\max}^2(\Sigma)\mathbb{E}\left[\frac{16}{\lambda_{\min}(\Sigma)^4}\|\hat{\Sigma}-\Sigma\|_{\mathrm{op}}^4\frac{\lambda_{\max}(\hat{\Sigma})^4}{\lambda_{\min}(\hat{\Sigma})^4}\right]^{\frac{1}{2}}\nonumber\\
    &\leq \frac{p}{\rho^2n_{\mathcal{L}}+N}\frac{F\|\beta^*\|_{2}^2}{\rho^2n_{\mathcal{L}}^2}\sum_{i=1}^{n_{\mathcal{L}}}|O_{\Xi(i)}|\nonumber\\
    &\leq F\frac{p\|\beta^*\|_{2}^2}{\rho^2n_{\mathcal{L}}+N},
\end{align}
where the second inequality follows from \eqref{eq:hat B F bound}; the fourth line follows from the conditional orthogonality of $\{\hat{V}_{i}\gamma_{i}-\mathbb{E}[\hat{V}_{i}\gamma_{i}|\hat{\Sigma}]\}_{i=1}^{n_{\mathcal{L}}}$; the fifth line follows from the triangle inequality and the fact that for $a,b \in \mathbb{R}$, $(a+b)^2\leq2(a^2+b^2)$; the sixth line follows from the conditional version of Jensen's inequality; the seventh line follows from the tower law; the eighth line follows by Cauchy-Schwarz; the ninth line follows by the tower law; the tenth line follows from hypercontractivity; the eleventh line follows from Lemma~\ref{lemma:hyper norm control} and \eqref{eq:P hat P upper bound}; the penultimate line follows from \eqref{eq:complicated eval assumption 1} and the final line follows from the balancing condition on the sizes of the modalities and the fact that $\rho n_{\mathcal{L}} \geq p$.

\subsubsection*{Controlling $\circled{3}$}
We control $\circled{3}$, beginning with the following decomposition
\begin{align}\label{eq:balanced circled 3 decomposition}
    \mathbb{E}\left[\left\|(\hat{B}^{C})^{-1}\sum_{i=1}^{n_{\mathcal{L}}}\hat{V}_{i}\eta_{i}\right\|_{2}^2\right]& = \mathbb{E}\left[\left\|(\hat{B}^{C})^{-1}\sum_{i=1}^{n_{\mathcal{L}}}\hat{V}_{i}\epsilon_{i}\right\|_{2}^2\right]+\mathbb{E}\left[\left\|(\hat{B}^{C})^{-1}\sum_{i=1}^{n_{\mathcal{L}}}\hat{V}_{i}m_{i}\right\|_{2}^2\right].
\end{align}
The first term satisfies 
\begin{align}
\label{eq:balanced B and epsilon bound}
    \mathbb{E}\left[\left\|(\hat{B}^{C})^{-1}\sum_{i=1}^{n_{\mathcal{L}}}\hat{V}_{i}\epsilon_{i}\right\|_{2}^2\right] & =\mathbb{E}\left[\left(\sum_{i=1}^{n_{\mathcal{L}}}\hat{V}_{i}\epsilon_{i}\right)^T(\hat{B}^{C})^{-2}\left(\sum_{i=1}^{n_{\mathcal{L}}}\hat{V}_{i}\epsilon_{i}\right)\right] \nonumber\\
    &\leq \frac{F}{\rho^2n_{\mathcal{L}}^2}\mathbb{E}\left[\left\|\sum_{i=1}^{n_{\mathcal{L}}}\hat{V}_{i}\epsilon_{i}\right\|_{2}^2\right]\nonumber\\ 
    & \leq \frac{F\sigma^2}{\rho^2n_{\mathcal{L}}^2}\mathbb{E}\left[\sum_{i=1}^{n_{\mathcal{L}}}\|\hat{V}_{i}\|_{2}^2\right]\nonumber\\
    & = \frac{F\sigma^2}{\rho^2n_{\mathcal{L}}^2}\sum_{i=1}^{n_{\mathcal{L}}}\mathbb{E}\left[\|\hat{P}_{\Xi(i)}^TX_{O_{\Xi(i)}}\|_{2}^2\right]\nonumber\\
    & \leq \frac{F\sigma^2}{\rho^2n_{\mathcal{L}}^2}\sum_{i=1}^{n_{\mathcal{L}}}\mathbb{E}\left[\frac{\lambda_{\max}^2(\hat{\Sigma})}{\lambda_{\min}^2(\hat{\Sigma})}\|X_{O_{\Xi(i)}}\|_{2}^2\right]\nonumber\\
    & \leq \frac{F\sigma^2}{\rho^2n_{\mathcal{L}}^2}\sum_{i=1}^{n_{\mathcal{L}}}|O_{\Xi(i)}|\nonumber\\
    &\leq \frac{F\sigma^2 p}{\rho n_{\mathcal{L}}},
\end{align}
where the second line follows from \eqref{eq:hat B F bound}; the penultimate line follows from Lemma \ref{lemma:hyper norm control} and \eqref{eq:complicated eval assumption 1}; the final line follows from the conditions of the Theorem.  
The second term in \eqref{eq:balanced circled 3 decomposition} satisfies 
\begin{align}
    \label{eq:balanced circled 3 middle bound}
    \mathbb{E}\left[\left\|(\hat{B}^C)^{-1}\sum_{i=1}^{n_{\mathcal{L}}}\hat{V}_{i}m_{i}\right\|_{2}^2\right] & = \mathbb{E}\left[\left(\sum_{i={1}}^{n_{\mathcal{L}}}\hat{V}_{i}m_{i}\right)^T(\hat{B}^C)^{-2}\left(\sum_{i={1}}^{n_{\mathcal{L}}}\hat{V}_{i}m_{i}\right)\right]\nonumber\\
    &\leq \frac{F}{\rho^2n_{\mathcal{L}}^2}\mathbb{E}\left[\left\|\sum_{i={1}}^{n_{\mathcal{L}}}\hat{V}_{i}m_{i}\right\|^2_{2}\right]\nonumber\\
    & = \frac{F}{\rho^2n_{\mathcal{L}}^2}\sum_{i={1}}^{n_{\mathcal{L}}}\mathbb{E}\left[\left\|\hat{V}_{i}m_{i}\right\|^2_{2}\right]\nonumber\\
    &\leq \frac{F}{\rho^2n_{\mathcal{L}}^2}\sum_{i={1}}^{n_{\mathcal{L}}}\mathbb{E}\left[\left\|\hat{V}_{i}\right\|_{2}^4\right]^{\frac{1}{2}}\mathbb{E}\left[m_{i}^4\right]^{\frac{1}{2}}\nonumber\\
    &\leq \frac{F}{\rho^2n_{\mathcal{L}}^2}\sum_{i={1}}^{n_{\mathcal{L}}}\mathbb{E}\left[\frac{\lambda_{\max}^4(\hat{\Sigma})}{\lambda_{\min}^4(\hat{\Sigma})}\left\|X_{O_{\Xi(i)}}\right\|_{2}^4\right]^{\frac{1}{2}}\mathbb{E}\left[m_{i}^2\right]\nonumber\\
    &\leq \frac{F}{\rho^2n_{\mathcal{L}}^2}\sum_{i={1}}^{n_{\mathcal{L}}}|O_{\Xi(i)}|(\beta^*_{M_{\Xi(i)}})^TS_{M_{\Xi(i)}}(\beta^*_{M_{\Xi(i)}})\nonumber\\
    &\leq \frac{Fp\sigma^2}{\rho n_{\mathcal{L}}},
\end{align}
where the second line follows from \eqref{eq:hat B F bound}; the third line follows from the conditional orthogonality of $\{m_{i}\hat{V}_{i}\}_{i=1}^{n_{\mathcal{L}}}$; the fourth line follows from Cauchy-Schwarz; the fifth line follows from hypercontractivity; the penultimate line follows from Lemma \ref{lemma:hyper norm control} and the final line follows from the conditions of the Theorem. 
Combining \eqref{eq:balanced circled 3 middle bound} with \eqref{eq:balanced B and epsilon bound} shows that $\circled{3}$ can be bound as 
\begin{equation}\label{eq:balanced circled 3 final bound}
    \mathbb{E}\left[\left\|(\hat{B}^C)^{-1}\sum_{i=1}^{n_{\mathcal{L}}}\hat{V}_{i}\eta_{i}\right\|_{2}^2\right] \leq \frac{Fp\sigma^2}{\rho n_{\mathcal{L}}}.
\end{equation}
Finally, $\circled{4}$ can be bounded from \eqref{eq:hat B minus hat BF bound} to give
\begin{equation}
    \label{eq:balanced circled 4 bound}\mathbb{E}\left[\left\|\left(\hat{B}^{C}\right)^{-1}\left(\hat{B}-\hat{B}^{C}\right)\right\|_\mathrm{op}^2\|\beta^*\|_{2}^2\right]\leq F\|\beta^*\|_{2}^2\left(\frac{p}{\rho n_{\mathcal{L}}}+\frac{p}{\rho^2n_{\mathcal{L}}+N}\right).
\end{equation}
 Combining the bounds \eqref{eq:balanced circled 1 final bound}, \eqref{eq:balanced circled 2 final bound}, \eqref{eq:balanced circled 3 final bound} and \eqref{eq:balanced circled 4 bound} into the decomposition  \eqref{eq:balanced error decomposition} yields that, 
\begin{align*}
    \mathbb{E}\left[\|\hat{\beta}-\beta^*\|_{2}^2\right] &\leq F\left(\frac{p\sigma^2}{\rho n_{\mathcal{L}}}+\frac{p\|\beta^*\|_{2}^2}{\rho^2n_{\mathcal{L}}+N}\right). 
\end{align*}

\end{proofof}

\subsection{Proofs of lower bounds}\label{A:Lower Bounds}
The goal of this section is to prove Theorems \ref{thm:LD structured lower bound} and \ref{thm:LD unstruc lower}. Both of these minimax lower bounds decouple into two separate terms. These parts follow straightforwardly from Propositions \ref{prop:LD ISS lower bound} and \ref{prop:LD OSS lower bound}. We begin by proving the first terms in each of these lower bounds. This follows by constructing a finite packing of distributions and applying Assouad's lemma. We begin by describing the form these distributions can take.

We define for $\beta \in \mathbb{R}^p, \sigma^2_{\beta}>0$ and $O \subseteq [p]$, 
\begin{align*}
    \Omega_{\beta}^{O} = \begin{pmatrix}
        \Sigma_{OO} & \Sigma_{O}\beta\\
        \beta^T\Sigma_{O}^T & \beta^T\Sigma\beta+\sigma_{\beta}^2
        \end{pmatrix}.
\end{align*}
With parameters $\beta$ and $\sigma_{\beta}$, when $\epsilon \sim N(0,\sigma^2_{\beta})$, $X \sim N(0,\Sigma)$ and $\epsilon \indep X$, it holds that for $O \subseteq [p]$
\begin{equation*}
    \begin{pmatrix}
        X_{O}\\
        Y
    \end{pmatrix} \sim N(0,\Omega_{\beta}^{O}).
\end{equation*}
We require the following lemma to control the KL divergence between two distributions in our class. 
\begin{lemma}\label{Lemma:First KL lemma}
    Recall that for $O \subseteq [p]$ and $M = [p]\setminus O$, $S_{M} = \Sigma_{MM}-\Sigma_{MO}\Sigma_{OO}^{-1}\Sigma_{OM}$. For $\beta, \gamma \in \mathbb{R}^p $ and $\sigma^2_{\gamma},\sigma^2_{\beta} > 0$,
\begin{align*}
    & \KL(N(0,\Omega^O_{\beta}),N(0,\Omega^O_{\gamma})) \\&\leq \frac{1}{2}\left\{\frac{{\sigma_{\gamma}^2+\gamma_M^TS_M\gamma_M}}{{\sigma_{\beta}^2+\beta_M^TS_M\beta_M}}+\frac{{\sigma_{\beta}^2+\beta_M^TS_M\beta_M}}{\sigma_{\gamma}^2+\gamma_M^TS_M\gamma_M}-2 + \frac{(\beta-\gamma)^T\Sigma_O^T\Sigma_{OO}^{-1}\Sigma_O(\beta-\gamma)}{\sigma_{\gamma}^2+\gamma_M^TS_M\gamma_M}\right\}.
\end{align*}
 
\end{lemma}
\begin{proof}
Applying Exercise 15.13 of~\citet{wainwright2019high} yields that
    \begin{equation}\label{eq:KL1 defn}
    \KL (N(0,\Omega^O_{\beta}),N(0,\Omega^O_{\gamma})) = \frac{1}{2}\left\{ \log\left(\frac{\det(\Omega^O_{\gamma})}{\det(\Omega^O_{\beta})}\right)-(|O|+1)+\Tr((\Omega^O_{\gamma})^{-1}\Omega^O_{\beta})\right\}.
\end{equation}
Without loss of generality, we set $O=\{1,\ldots,p-p_{0}\}$ for some $p_{0}<p$. It holds that
\begin{align}
    \Sigma_{O}^T\Sigma_{OO}^{-1}\Sigma_{O} &= \Sigma_{O}^T\Sigma_{OO}^{-1}\begin{pmatrix}
        \Sigma_{OO} & \Sigma_{OM}
    \end{pmatrix}\nonumber\\
    &=\begin{pmatrix}
        \Sigma_{OO}\\
        \Sigma_{MO}
    \end{pmatrix}\begin{pmatrix}
        I_{|O|} & \Sigma_{OO}^{-1}\Sigma_{OM}
    \end{pmatrix}\nonumber\\
    &= \Sigma - \begin{pmatrix}
        0 & 0 \\
        0 & S_M
    \end{pmatrix}. \label{eq:KL1 Schur}
\end{align}
Therefore, by 9.1.2 of~\citet{petersen2008matrix}, we have that
\begin{align*}
    \det(\Omega^O_{\beta}) &= \det(\Sigma_{OO})(\beta^T\Sigma\beta + \sigma_{\beta}^2 - \beta^T\Sigma_{O}^T\Sigma_{OO}^{-1}\Sigma_{O}\beta)\\
    &= \det(\Sigma_{OO})(\sigma^2_{\beta}+\beta_{M}^TS_M\beta_{M}).
\end{align*}
Using a very similar formula for $\det(\Omega^O_{\gamma})$, it follows that
\begin{equation}\label{eq:KL1 log term}
    \log \left(\frac{\det(\Omega^O_{\gamma})}{\det(\Omega^O_{\beta})}\right) = \log\left(\frac{\sigma^2_{\gamma}+\gamma_{M}^TS_M\gamma_{M}}{\sigma^2_{\beta}+\beta_{M}^TS_M\beta_{M}}\right).
\end{equation}
Observe that 
\begin{equation*} 
    \Omega^O_{\beta} = \Omega^O_{\gamma} + \begin{pmatrix}
        0 & \Sigma_{O}(\beta-\gamma)\\
        (\beta-\gamma)^T\Sigma_{O}^T & \Delta
    \end{pmatrix},
\end{equation*}
where $\Delta = \beta^T\Sigma\beta+\sigma_{\beta}^2-\gamma^T\Sigma\gamma-\sigma_{\gamma}^2$. Using 9.1.3 of~\citet{petersen2008matrix}, there exists $A_\gamma\in\mathbb{R}^{|O|\times |O|}$ such that
\begin{equation*}
    (\Omega_{\gamma}^O)^{-1} = \begin{pmatrix}
        A_{\gamma}& \frac{-\Sigma_{OO}^{-1}\Sigma_{O}\gamma}{\sigma_{\gamma}^2+\gamma_{M}^TS_M\gamma_M}\\
        \frac{-\gamma^T\Sigma_{O}^T\Sigma_{OO}^{-1}}{\sigma_{\gamma}^2+\gamma_{M}^TS_M\gamma_M} & \frac{1}{\sigma_{\gamma}^2+\gamma_{M}^TS_M\gamma_M}
    \end{pmatrix}.
\end{equation*}
It follows that
\begin{align}\label{eq:KL1 Trace term}
    &\Tr((\Omega^O_\gamma)^{-1}\Omega^O_{\beta}) -(|O|+1)\nonumber\\ 
    &= \Tr\left(\begin{pmatrix}
        A_{\gamma} & \frac{-\Sigma_{OO}^{-1}\Sigma_{O}\gamma}{\sigma_{\gamma}^2+\gamma_{M}^TS_M\gamma_M}\\
        \frac{-\gamma^T\Sigma_{O}^T\Sigma_{OO}^{-1}}{\sigma_{\gamma}^2+\gamma_{M}^TS_M\gamma_M} & \frac{1}{\sigma_{\gamma}^2+\gamma_{M}^TS_M\gamma_M}
    \end{pmatrix}\begin{pmatrix}
        0 & \Sigma_{O}(\beta-\gamma)\\
        (\beta-\gamma)^T\Sigma_{O}^T & \Delta
    \end{pmatrix}\right)\nonumber \\
    & =  \frac{\Delta-2\gamma^T\Sigma_{O}^T\Sigma_{OO}^{-1}\Sigma_{O}(\beta-\gamma)}{\sigma_{\gamma}^2+\gamma_M^TS_M\gamma_M}\nonumber\\
    &= \frac{\sigma_\beta^2 - \sigma_\gamma^2 +\beta^T \Sigma \beta - \gamma^T \Sigma \gamma - 2\gamma^T\Sigma_{O}^T\Sigma_{OO}^{-1}\Sigma_{O}(\beta-\gamma)}{\sigma_{\gamma}^2+\gamma_M^TS_M\gamma_M}\nonumber\\
    &=\frac{\sigma_\beta^2 - \sigma_\gamma^2 + \beta_M^T S_M \beta_M - \gamma_M^T S_M \gamma_M + \beta^T \Sigma_O^T \Sigma_{OO}^{-1} \Sigma_O \beta - \gamma^T \Sigma_O^T \Sigma_{OO}^{-1} \Sigma_O \gamma - 2\gamma^T\Sigma_{O}^T\Sigma_{OO}^{-1}\Sigma_{O}(\beta-\gamma)}{\sigma_{\gamma}^2+\gamma_M^TS_M\gamma_M}\nonumber\\
    &=\frac{\sigma_\beta^2 - \sigma_\gamma^2 + \beta_M^T S_M \beta_M - \gamma_M^T S_M \gamma_M + (\beta-\gamma)^T\Sigma_O^T\Sigma_{OO}^{-1}\Sigma_O(\beta-\gamma)}{\sigma_{\gamma}^2+\gamma_M^TS_M\gamma_M},
\end{align}
where the penultimate line follows from \eqref{eq:KL1 Schur}. Therefore, combining \eqref{eq:KL1 log term}  and \eqref{eq:KL1 Trace term} into \eqref{eq:KL1 defn} yields
\begin{align*}
    &\KL(N(0,\Omega^O_{\beta}),N(0,\Omega^O_{\gamma})) \\
    &= \frac{1}{2}\left\{ \log\left(\frac{\sigma_{\gamma}^2+\gamma_{M}^TS_M\gamma_M}{\sigma_{\beta}^2+\beta_{M}^TS_M\beta_M}\right)+\frac{\sigma_{\beta}^2+\beta_M^TS_M\beta_M + (\beta-\gamma)^T\Sigma_O^T\Sigma_{OO}^{-1}\Sigma_O(\beta-\gamma)-\sigma_{\gamma}^2-\gamma_M^TS_M\gamma_M}{\sigma_{\gamma}^2+\gamma_M^TS_M\gamma_M}\right\}\\
    &\leq \frac{1}{2}\left\{\frac{{\sigma_{\gamma}^2+\gamma_M^TS_M\gamma_M}}{{\sigma_{\beta}^2+\beta_M^TS_M\beta_M}}+\frac{{\sigma_{\beta}^2+\beta_M^TS_M\beta_M}}{\sigma_{\gamma}^2+\gamma_M^TS_M\gamma_M}-2 + \frac{(\beta-\gamma)^T\Sigma_O^T\Sigma_{OO}^{-1}\Sigma_O(\beta-\gamma)}{\sigma_{\gamma}^2+\gamma_M^TS_M\gamma_M}\right\},
\end{align*}
where the final line follows since $\log(x) \leq x-1$ for $x > 0$.

\end{proof}

\begin{proposition}\label{prop:LD ISS lower bound}
    Suppose we have data drawn from \eqref{eq:datamech1} with distribution lying in $\mathcal{P}_{OSS}^{Gauss}$ satisfying $C_{X}\geq \sqrt{\frac{8\lambda_{+}}{3}}$, $\kappa_{\epsilon}\geq 105^{\frac{1}{4}}$. 
    If \begin{equation*}
        \sum_{i=1}^p\frac{1}{h_{i}} \leq \min\left\{2,\frac{4B^2\lambda_{-}}{R^2}\right\},
    \end{equation*} 
    then it holds that
    \begin{equation*}
        \inf_{\hat{\beta}\in \hat{\Theta}_{OSS}}\sup_{P \in \mathcal{P}_{OSS}^{Gauss}}\mathbb{E}\left[\|\hat{\beta}-\beta^*(P)\|^2_{2}\right] \geq \frac{R^2}{32\lambda_{-}}\sum_{i=1}^p\frac{1}{h_{i}}.
    \end{equation*}
    For each $z \in \mathbb{R}^p$ and $j \in [K]$ define $D_{j}(z) = \frac{R^2}{R^2+\|z_{M_{j}}\|_{2}^2}$ and $\alpha_{i}(z) = \sum_{j=1}^Kn_{j}D_{j}(z)\mathbbm{1}_{\left\{i \in O_{j}\right\}}$. Alternatively, if
    \begin{equation}\label{eq:min alpha size ISS lower bound}
        \inf_{\substack{z \in \mathbb{R}^p\\
        \|z\|_{2}\leq B}}\left\{\alpha_{i}(z)\right\} \geq \frac{1}{2}\max\left\{\frac{1}{\lambda_{-}},\lambda_{-}\right\}\max\left\{\frac{R^2}{B^2}, \frac{B^2}{R^2}\right\}p,
    \end{equation} for all $i \in [p]$, then provided $|L_{i}| \geq 2$ for $i \in [L]$, it holds that
    \begin{equation*}
        \inf_{\hat{\beta}\in \hat{\Theta}_{OSS}}\sup_{P \in \mathcal{P}_{OSS}^{Gauss}}\mathbb{E}\left[\|\hat{\beta}-\beta^*(P)\|^2_{2}\right] \gtrsim\sup_{\substack{z \in \mathbb{R}^p\\
        \|z\|_{2}\leq B}}\left\{ \frac{R^2}{\max\{1,\lambda_{-}\}}\sum_{i=1}^p\frac{1}{\alpha_{i}(z)}\right\}.
    \end{equation*}
\end{proposition}

\begin{proof}
    We begin by proving the first statement of the proposition. This is proved by applying Assouad's Lemma -- a standard technique for proving minimax lower bounds~\citep[Chapter 8]{samworth2024statistics}. We begin by constructing appropriate distributions in $\mathcal{P}_{OSS}^{Gauss}$. We first restrict to the Gaussian sub-class where $\epsilon \sim N(0,\sigma^2), X \sim N(0,\lambda_{-}I_{p})$ and $\epsilon \indep X$. These distributions are parametrised by $\sigma$ and $\beta^*$. Since we are proving a lower bound that holds in the ISS case, we can only vary $\sigma$ and $\beta^*$. For such a  distribution in $\mathcal{P}_{OSS}^{Gauss}$, our data is indexed by $\beta^*$ and $\sigma_{\beta^*}$ and consists of 
\begin{equation*}
    P_{\beta^*} = \prod_{j=1}^{K} N(0,\Omega^{O_{j}}_{\beta^*})^{\otimes n_{j}}.
\end{equation*}
     We now construct the distributions indexed by the hypercube. For $i \in [p]$, set $\delta_{i}=\frac{R}{2\sqrt{\lambda_{-}h_{i}}}$. Set $\Phi = \{0,1\}^p$ and for $\phi \in \Phi$, let $\beta_{\phi}=\sum_{i=1}^p\delta_{i}\phi_{i}e_{i}$ and $\sigma_{\phi}^2=R^2-\lambda_{-}\|\beta_{\phi}\|_{2}^2$. By the conditions of the theorem, we have that $ \sum_{j=1}^p\delta_{j}^2\leq \min\{\frac{R^2}{2\lambda_{-}},B^2\}.$ It follows that $\frac{R^2}{2}\leq \sigma_{\phi}^2\leq R^2$ and $\|\beta_{\phi}\|_{2}^2 \leq B^2$. Note that Gaussian covariates satisfy Assumption \ref{assump:small-ball} with $\chi =1$ and $C=1$; they also are sub-Gaussian with parameter $C_{X}$. By construction, the eigenvalues of $\Sigma$ satisfy $\lambda_{\min}(\Sigma)=\lambda_{\max}(\Sigma)=\lambda_{-}$. These choices do, therefore, define valid distributions in $\mathcal{P}_{OSS}^{Gauss}$. We now control the $\mathrm{KL}$ divergence between neighbouring distributions. Let $\phi$ and $\phi'$ differ in exactly the $i^{th}$ coordinate. It holds that 
     \begin{align}\label{eq:first KL bound first minimax lower bound}
          \KL\left(P_{\phi},P_{\phi'}\right) &= \sum_{j=1}^{K}n_{j}\KL\left(N(0,\Omega^{O_{j}}_{\beta_{\phi}}),N(0,\Omega^{O_{j}}_{\beta_{\phi'}})\right)\nonumber\\
          &\leq\frac{1}{2}\sum_{j=1}^Kn_{j}\left\{\frac{R^2-\lambda_{-}\|(\beta_{\phi'})_{O_{j}}\|_{2}^2}{R^2-\lambda_{-}\|(\beta_{\phi})_{O_{j}}\|_{2}^2}+\frac{R^2-\lambda_{-}\|(\beta_{\phi})_{O_{j}}\|_{2}^2}{R^2-\lambda_{-}\|(\beta_{\phi'})_{O_{j}}\|_{2}^2}-2+\frac{\lambda_{-}\|(\beta_{\phi})_{O_{j}}-(\beta_{\phi'})_{O_{j}}\|_{2}^2}{{R^2-\lambda_{-}\|(\beta_{\phi'})_{O_{j}}\|_{2}^2}}\right\}\nonumber\\
          &\leq \frac{1}{2}\sum_{j=1}^Kn_{j}\left\{\left(\frac{\lambda_{-}\|(\beta_{\phi})_{O_{j}}\|_{2}^2-\lambda_{-}\|(\beta_{\phi'})_{O_{j}}\|_{2}^2}{\frac{R^2}{2}}\right)^2+\frac{\lambda_{-}\|(\beta_{\phi})_{O_{j}}-(\beta_{\phi'})_{O_{j}}\|_{2}^2}{\frac{R^2}{2}}\right\}\nonumber\\
          &= \frac{1}{2}\sum_{j=1}^Kn_{j}\mathbbm{1}_{\{i \in O_{j}\}}\left\{\left(\frac{2\lambda_{-}\delta_{i}^2}{R^2}\right)^2+\left(\frac{2\lambda_{-}\delta_{i}^2}{R^2}\right)\right\}\nonumber\\
          &\leq \frac{2\lambda_{-}\delta_{i}^2}{R^2}\sum_{j=1}^Kn_{j}\mathbbm{1}_{\{i \in O_{j}\}}\nonumber\\
          & = \frac{2\lambda_{-}\delta_{i}^2h_{i}}{R^2}\nonumber\\
          &= \frac{1}{2},
     \end{align}
     where the first inequality follows from Lemma \ref{Lemma:First KL lemma}; the second inequality follows from the fact that $x+\frac{1}{x} \leq (x-1)^2+2$ for $ x \geq 1$ and the fact that for $\phi \in \Phi$, $2\lambda_{-}\|\beta_{\phi}\|_{2}^2 \leq R^2$; the final inequality follows from the fact that $2\lambda_{-}\delta_{i}^2 \leq R^2$. We apply Assouad's lemma~\cite[Lemma 8.7]{samworth2024statistics} and Pinsker's inequality to show 
\begin{align*}
    \inf_{\hat{\beta}\in \hat{\Theta}_{OSS}}\sup_{P \in \mathcal{P}_{OSS}^{Gauss}}\mathbb{E}\left[\|\hat{\beta}-\beta^*(P)\|^2_{2}\right]&\geq \frac{1}{4}\left\{1-\max_{\substack{\phi,\phi'\in\Phi\\\phi \sim \phi'}} \TV(P_{\beta_{\phi}},P_{{\beta}_{\phi'}})\right\}\sum_{i=1}^p\delta_{i}^2\\
    &\geq \frac{1}{4}\left\{1-\max_{\substack{\phi,\phi'\in\Phi\\\phi \sim \phi'}} \sqrt{\frac{\KL(P_{\beta_{\phi}},P_{{\beta}_{\phi'}})}{2}}\right\}\sum_{i=1}^p\frac{R^2}{4\lambda_{-}h_{i}}\\
    &\geq \frac{R^2}{32\lambda_{-}}\sum_{i=1}^p\frac{1}{h_{i}},
\end{align*}
where the final inequality follows from \eqref{eq:first KL bound first minimax lower bound}. 

We now prove the second statement of the proposition. Again this is based on a reduction to a finite number of hypotheses and an application of Assouad's lemma. In order to describe the effects of $B$ and $R$, we instead have to pack the signal parameters away from zero. Define $Q_{1}=0$ and for $r \in [L+1]\setminus\{1\}$ define $Q_{r}=|L_{1}|+\cdots+|L_{r-1}|$ such that $L_{r}=[Q_{r+1}]\setminus [Q_{r}]$. Additionally, define $F_{r}^{(1)}=\{Q_{r}+1,\ldots,Q_{r}+\lceil |L_{r}|/2 \rceil\}$ and $F_{r}^{(2)}=L_{r}\setminus F_{r}^{(1)}$. Since $|L_{r}| \geq 2$ for all $r$, these sets are non-empty. Finally, we define two further sets $F^{(1)}=\cup_{r=1}^LF_{r}^{(1)}$ and $F^{(2)}=\cup_{r=1}^LF_{r}^{(2)}$. Fix $z \in \mathbb{R}^{p}$ with $\|z\|_{2}\leq B$ such that for $i \in [p]$ we have $z_{i}=0$ if $i \in F^{(1)}$. Set $M=\cup_{k=1}^KM_{k}$ and define $y \in \mathbb{R}^p$ by
\begin{equation*}
    y_{i}  =
\begin{cases}
\frac{z_{i}}{4}, & \text{if } |z_{i}| > \frac{B}{4\sqrt{p}} \text{ and } i \in M \\
\frac{B}{4\sqrt{p}}, & \text{if } |z_{i}| \leq \frac{B}{4\sqrt{p}} \text{ and } i \in F^{(2)}\\
0, & \text{otherwise. }
\end{cases}
\end{equation*}
Define $\delta \in \mathbb{R}^p$ such that for $i \in F^{(2)}$ we have $\delta_{i}=0$ and for $i \in F^{(1)}$ we have 
\begin{equation*}
    \delta_{i}=\frac{R}{16\max\{1,\sqrt{\lambda_{-}\}}}\sqrt{\left(\sum_{j=1}^Kn_{j}D_{j}(z)\mathbbm{1}_{\{i \in O_{j}\}}\right)^{-1}}.
\end{equation*} We are now ready to construct our hypotheses. We first restrict to the Gaussian sub-class where $\epsilon \sim N(0,\sigma^2), X \sim N(0,\lambda_{-}I_{p})$ and $\epsilon \indep X$. These distributions are parametrised by $\sigma$ and $\beta^*$. Since we are proving a lower bound that holds in the ISS case, we can only vary $\sigma$ and $\beta^*$. For such a  distribution in $\mathcal{P}_{OSS}^{Gauss}$ our data is indexed by $\beta^*$ and $\sigma_{\beta^*}$ and consists of 
\begin{equation*}
    P_{\beta^*} = \prod_{j=1}^{K} N(0,\Omega^{O_{j}}_{\beta^*})^{\otimes n_{j}}.
\end{equation*} Set $\Phi = \{0,1\}^p$ and for $\phi \in \Phi$ define, 
\begin{align*}
    \beta_{\phi}&=\sum_{r=1}^{L}\left\{\sum_{j\in F_{r}^{(1)}} \delta_{j}\phi_{j}e_{j}+\sum_{j\in F_{r}^{(2)}}y_{j}e_{j}\right\},\\
    \sigma_{\phi}^2&=\frac{R^2}{2}-\lambda_{-}\|\beta_{\phi}\|_{2}^2+\lambda_{-}\|y\|_{2}^2.
\end{align*}
We have by \eqref{eq:min alpha size ISS lower bound} that 
\begin{align*}
    \|\beta_{\phi}\|_{2}^2 \leq \sum_{i=1}^p\delta_{i}^2+\|y\|_{2}^2 &\leq \frac{B^2}{2}+\frac{B^2}{2}=B^2. 
\end{align*}
Additionally, we have that
\begin{align*}
    \left|\sigma_{\phi}^2-\frac{R^2}{2}\right| &=\lambda_{-}|\|\beta_{\phi}\|_{2}^2-\|y\|_{2}^2|\nonumber\\
    &=\lambda_{-}|\|\beta_{\phi}\|_{2}+\|y\|_{2}||\|\beta_{\phi}\|_{2}-\|y\|_{2}|\nonumber\\
    &\leq \frac{3B\lambda_{-}}{2}\|\beta_{\phi}-y\|_{2}\nonumber\\
    &\leq \frac{3B}{2}\lambda_{-}\left(\sum_{i=1}^p\delta_{i}^2\right)^{\frac{1}{2}}\nonumber\\
    &\leq \frac{R^2}{4},
\end{align*}
where the final inequality follows from \eqref{eq:min alpha size ISS lower bound}. It follows that $0\leq \sigma_{\phi}^2 \leq R^2$. Note that Gaussian covariates satisfy Assumption \ref{assump:small-ball} with $\chi =1$ and $C=1$; they are sub-Gaussian with parameter $C_{X}$ and $\mathbb{E}[\epsilon^8]^{\frac{1}{4}}\leq 105^\frac{1}{4}\sigma^2$. By construction, the eigenvalues of $\Sigma$ satisfy $\lambda_{\min}(\Sigma)=\lambda_{\max}(\Sigma)=\lambda_{-}$. We therefore have that these define valid distributions in $\mathcal{P}_{OSS}^{Gauss}$. We next control the KL divergence between neighbouring distributions in $\Phi$. Let $\phi$ and $\phi'$ differ in exactly the $i^{th}$ coordinate. It holds that

\begin{align}\label{eq: ISS lower bound KL control}
    &\KL\left(P_{\phi},P_{\phi'}\right) \nonumber\\
    &\hspace{4mm}= \sum_{j=1}^{K}n_{j}\KL\left(N(0,\Omega^{O_{j}}_{\beta_{\phi}}),N(0,\Omega^{O_{j}}_{\beta_{\phi'}})\right)\nonumber\\
    &\hspace{4mm} \leq \frac{1}{2} \sum_{j=1}^{K} n_{j}\left\{\frac{\sigma_{\phi'}^2+\lambda_{-}\|(\beta_{\phi'})_{M_{j}}\|_{2}^2}{\sigma_{\phi}^2+\lambda_{-}\|(\beta_{\phi})_{M_{j}}\|_{2}^2}+\frac{\sigma_{\phi}^2+\lambda_{-}\|(\beta_{\phi})_{M_{j}}\|_{2}^2}{\sigma_{\phi'}^2+\lambda_{-}\|(\beta_{\phi'})_{M_{j}}\|_{2}^2}-2+\frac{\lambda_{-}\|(\beta_{\phi})_{O_{j}}-(\beta_{\phi'})_{O_{j}}\|_{2}^2}{\sigma_{\phi'}^2+\lambda_{-}\|(\beta_{\phi'})_{M_{j}}\|_{2}^2}\right\}\nonumber\\
    &\hspace{4mm} \leq \frac{1}{2} \sum_{j=1}^{K} n_{j}\left\{\left(\frac{\lambda_{-}\|(\beta_{\phi'})_{O_{j}}\|_{2}^2-\lambda_{-}\|(\beta_{\phi})_{O_{j}}\|_{2}^2}{\frac{R^2}{4}+\lambda_{-}\|y_{M_{j}}\|_{2}^2}\right)^2+\frac{\lambda_{-}\|(\beta_{\phi})_{O_{j}}-(\beta_{\phi'})_{O_{j}}\|_{2}^2}{\frac{R^2}{4}+\lambda_{-}\|y_{M_{j}}\|_{2}^2}\right\} \nonumber\\
    &\hspace{4mm}= \frac{1}{2} \sum_{j=1}^K n_{j}\mathbbm{1}_{\left\{i \in O_{j}\right\}}\left\{\frac{\lambda_{-}\delta_{i}^2}{\frac{R^2}{4}+\lambda_{-}\|y_{M_{j}}\|_{2}^2}+\left(\frac{\lambda_{-}\delta_{i}^2}{\frac{R^2}{4}+\lambda_{-}\|y_{M_{j}}\|_{2}^2}\right)^2\right\}\nonumber\\
    &\hspace{4mm}\leq \sum_{j=1}^K n_{j}\mathbbm{1}_{\left\{i \in O_{j}\right\}}\left\{\frac{\lambda_{-}\delta_{i}^2}{\frac{R^2}{4}+\lambda_{-}\|y_{M_{j}}\|_{2}^2}\right\}\nonumber\\
    &\hspace{4mm}\leq \sum_{j=1}^K n_{j}\mathbbm{1}_{\left\{i \in O_{j}\right\}}\left\{\frac{\lambda_{-}\delta_{i}^2}{\frac{R^2}{4}+\frac{\lambda_{-}\|z_{M_{j}}\|_{2}^2}{16}}\right\}\nonumber\\
    &\hspace{4mm}\leq \frac{16\max\{1,\lambda_{-}\}\delta_{i}^2}{R^2}\sum_{j=1}^K n_{j}D_{j}\mathbbm{1}_{\left\{i \in O_{j}\right\}}\nonumber\\
    &\hspace{4mm} \leq \frac{1}{2}, 
\end{align}
where the first inequality follows from Lemma \ref{Lemma:First KL lemma}; the second inequality follows from the fact that $x+\frac{1}{x} \leq (x-1)^2+2$ for $ x \geq 1$ and the fact that $\sigma_{\phi}^2 \geq \frac{R^2}{4}$; the third inequality follows from the fact that 
\begin{equation*}
    \frac{\lambda_{-}\delta_{i}^2}{\frac{R^2}{4}} \leq \min\{1,\lambda_{-}\}\left(\sum_{j=1}^Kn_{j}D_{j}\mathbbm{1}_{\{i \in O_{j}\}}\right)^{-1}\leq 1,
\end{equation*} by \eqref{eq:min alpha size ISS lower bound}.

For $\phi,\phi' \in \Phi$, we write $\phi \sim \phi'$ if $\phi$ and $\phi'$ differ in precisely one coordinate. We apply Assouad's lemma~\cite[Lemma 8.7]{samworth2024statistics} and Pinsker's inequality to show 
\begin{align*}
    &\inf_{\hat{\beta}\in \hat{\Theta}_{OSS}}\sup_{P \in \mathcal{P}_{OSS}^{Gauss}}\mathbb{E}\left[\|\hat{\beta}-\beta^*(P)\|^2_{2}\right]\\
    &\geq \frac{1}{4}\left\{1-\max_{\substack{\phi,\phi'\in\Phi\\\phi \sim \phi'}} \TV(P_{\beta_{\phi}},P_{{\beta}_{\phi'}})\right\}\sum_{i=1}^p\delta_{i}^2\\
    &\geq \frac{1}{4}\left\{1-\max_{\substack{\phi,\phi'\in\Phi\\\phi \sim \phi'}} \sqrt{\frac{\KL(P_{\beta_{\phi}},P_{{\beta}_{\phi'}})}{2}}\right\}\sum_{i\in F^{(1)}}\frac{R^2}{256\max\{1,\lambda_{-}\}}\left(\sum_{j=1}^Kn_{j}D_{j}\mathbbm{1}_{\left\{i \in O_{j}\right\}}\right)^{-1}\\
    &\gtrsim \frac{R^2}{\max\{1,\lambda_{-}\}}\sum_{i=1}^p\frac{1}{\sum_{j=1}^Kn_{j}D_{j}\mathbbm{1}_{\left\{i \in O_{j}\right\}}},
\end{align*}
where the final inequality follows from \eqref{eq: ISS lower bound KL control}, the fact that for $r \in [L]$, $|F_{r}^{(1)}| \geq \frac{|L_{r}|}{2}$ and the fact that $\sum_{j=1}^Kn_{j}D_{j}\mathbbm{1}_{\left\{i \in O_{j}\right\}}$ depends only on the modality of $i$. This establishes the lower bound for all $z \in \mathbb{R}^{p}$ with $\|z\|_{2}\leq B$ such that for $i \in [p]$ we have $z_{i}=0$ if $i \in F^{(1)}$. The lower bound for all $z \in \mathbb{R}^{p}$ with $\|z\|_{2}\leq B$ then follows by symmetry and adjusting constants. 
\end{proof}

We now prove that the second terms in Theorems \ref{thm:LD structured lower bound} and \ref{thm:LD unstruc lower} lower bound the minimax risk. Again, this follows from constructing a finite packing of distributions and applying Assouad's lemma. We now describe the general form these distributions will take -- these distributions will crop up again in the high-dimensional setting. We then control the KL divergence between these distributions for both the labelled (Lemma \ref{lemma: KL2 miss}) and unlabelled data (Lemma \ref{lemma:KL3}). Finally, we prove the desired lower bound.

For $\lambda_{-}\leq\lambda_{+}$, let $\tau=\tau(\lambda_{-},\lambda_{+})=\frac{\lambda_{-}+\lambda_{+}}{2}$. For $x \in \mathbb{R}^{p-1}$, define $\beta_{x}\in \mathbb{R}^p$ via
\begin{equation*}
    \beta_{x} = \left(- \frac{x}{\tau}, \frac{B}{2} \right)^T.
\end{equation*}
Additionally, we define $\Sigma_{x}\in \mathbb{R}^{p\times p}$ as
\begin{equation*}
    \Sigma_{x} = \begin{pmatrix}
        \tau I_{p-1} & \frac{2x}{B}\\
        \frac{2x^T}{B} & \tau
    \end{pmatrix}.
\end{equation*}
 Note that $\Sigma_{x}$ has $p-2$ eigenvalues equal to $\tau$ with corresponding eigenvectors of the form $(a^T,0)^T$ where $a \in \mathbb{R}^{p-1}$ and $a^Tx =0$. Its two final eigenvalues are $\tau\pm\frac{2\|x\|_{2}}{B}$ with corresponding eigenvectors $(\pm \frac{x^T}{\|x\|_{2}},1)^T$. In particular we see that $\|x\|_{2} \leq \frac{B(\lambda_{+}-\lambda_{-})}{4}$ is a sufficient condition to ensure the eigenvalues of $\Sigma_{x}$ are bounded above by $\lambda_{+}$ and below by $\lambda_{-}$ and that $\|\beta_{x}\|_{2}^2  \leq B^2$. We set $\sigma_{x}=R$ and note that $\Delta_{x} \equiv \beta_{x}^T\Sigma_{x}\beta_{x}+\sigma_{x}^2=\frac{\tau B^2}{4}-\frac{\|x\|_{2}^2}{\tau}+R^2$ only depends on $x$ through $\|x\|_{2}$. Additionally, we define \begin{equation*}
    \Omega_{x} \equiv \begin{pmatrix}
        \Sigma_{x}& \Sigma_{x}\beta_{x} \\
        (\Sigma_{x}\beta_{x})^T & \Delta_{x}
    \end{pmatrix}=  \begin{pmatrix}
        \tau I_{p-1} & \frac{2x}{B} & 0\\
        \frac{2x^T}{B} & \tau & \frac{B\tau}{2}-\frac{2\|x\|_{2}^2}{B\tau}\\
        0 & \frac{B\tau}{2}-\frac{2\|x\|_{2}^2}{B\tau} & \Delta_{x}
    \end{pmatrix},
\end{equation*}
and for $O \subseteq [p]$, let
\begin{equation*}
    \Omega_{x}^O = \begin{pmatrix}
        (\Sigma_{x})_{OO}& (\Sigma_{x})_{O}\beta_{x} \\
        ((\Sigma_{x})_{O}\beta_{x})^T & \Delta_{x}
    \end{pmatrix}.
\end{equation*}
Similar to the proof of Theorem \ref{thm:LD structured lower bound}, we will construct a finite packing of the data-generating mechanisms such that the distribution with parameter $\beta_{x}$ will be $(X,Y) \sim N(0,\Omega_{x})$.
We begin with some lemmas to control the KL divergence between hypotheses of this form. For $O \subseteq [p]$, we let $O' = O\setminus\{p\}$.
\begin{lemma} \label{lemma: KL2 miss} Fix $B,R>0$ and $\lambda_{-}\leq\lambda_{+}$ with $\tau=\tau(\lambda_{-},\lambda_{+})=\frac{\lambda_{-}+\lambda_{+}}{2}$. Fix $x_{1},x_{2} \in \mathbb{R}^{p-1}$ with $\|x_{1}\|_{2}=\|x_{2}\|_{2} \leq \frac{B(\lambda_{+}-\lambda_{-})}{8}$ and $O \subseteq [p]$. Given the above definitions, 
    \begin{align*}
    \KL&(N(0,\Omega_{x_{1}}^O),N(0,\Omega_{x_{2}}^O)) \\&\leq10\max\left\{1,\frac{1}{\tau^2}\right\}\max\left\{\frac{1}{B^2},\frac{1}{R^2}\right\}\left\{\left|\|(x_{1})_{O'}\|_{2}^2-\|(x_{2})_{O'}\|_{2}^2\right|+\sum_{j=1}^{|O'|}|\left\{\left(x_{1}-x_{2}\right)_{O'}\right\}_{j}||\{(x_{2})_{O'}\}_{j}| \right\}.
\end{align*}
Furthermore, if $p \notin O$, then the left-hand side of the above display vanishes. 
\end{lemma}
\begin{proof}
The second statement in the lemma is clear since if $p \notin O$ then $\Omega_{x_{1}}^O=\Omega_{x_{2}}^O$. Therefore, without loss of generality, we set $p \in O$ and $O = \{p_{0}+1,\ldots,p\}$. Applying Exercise 15.13 of~\citet{wainwright2019high} yields that 
    \begin{equation}\label{eq:KL2 miss formula}
    \KL(N(0,\Omega_{x_{1}}^O),N(0,\Omega_{x_{2}}^O))=\frac{1}{2}\left\{\log\left(\frac{\det(\Omega_{x_{2}}^O)}{\det(\Omega_{x_{1}}^O)}\right)-(|O|+1)+\Tr((\Omega_{x_{2}}^O)^{-1}\Omega_{x_{1}}^O) \right\}.
\end{equation}
By~\citet{petersen2008matrix} 9.1.2, we have that
\begin{align}\label{eq:kl2 det blockwise definition}
    \det(\Omega_{x_{2}}^O) &= \det((\Sigma_{x_{2}})_{OO})(\Delta_{x_{2}} - (\Sigma_{x_{2}}\beta_{x_{2}})_{O}^T((\Sigma_{x_{2}})_{OO})^{-1}(\Sigma_{x_{2}}\beta_{x_{2}})_{O}).
\end{align}
Recalling that
\begin{equation*}
    (\Sigma_{x_{2}})_{OO}= \begin{pmatrix}
    \tau I_{|O'|} & (\frac{2x_{2}}{B})_{O'}\\
    (\frac{2x_{2}}{B})_{O'}^T & \tau
\end{pmatrix},
\end{equation*}
it follows, via~\citet{petersen2008matrix} 9.1.2 and 9.1.3, that
\begin{align*}
    \det((\Sigma_{x_{2}})_{OO}) = \tau^{|O'|-1}\left(\tau^2 - \frac{4\|(x_{2})_{O'}\|_{2}^2}{B^2}\right), &\quad \quad \left(\left\{\left(\Sigma_{x_{2}}\right)_{OO}\right\}^{-1}\right)_{p-p_{0},p-p_{0}} = \frac{1}{\tau-\frac{4\|(x_{2})_{O'}\|_{2}^2}{\tau B^2}}.
\end{align*}
From \eqref{eq:kl2 det blockwise definition}, it follows that
\begin{equation*}
    \det(\Omega_{x_{2}}^O) = \tau^{|O'|-1}\left(\tau^2-\frac{4\|(x_{2})_{O'}\|_{2}^2}{B^2}\right)\left(\Delta_{x_{2}} - \frac{(\frac{B\tau}{2}-\frac{2\|x_{2}\|_{2}^2}{B\tau})^2}{\tau-\frac{4\|(x_{2})_{O'}\|_{2}^2}{\tau B^2}}\right).
\end{equation*}
Note that $\Delta_{x_{1}}=\Delta_{x_{2}}$, so for convenience set $\Delta = \Delta_{x_{1}}$. It follows from the above display that
\begin{align}
\label{eq:KL2 miss det bound}
    &\left|\log \left(\frac{\det(\Omega_{x_{2}}^O)}{\det(\Omega_{x_{1}}^O)}\right)\right|\nonumber\\ 
    &\leq \left|\log\left(\frac{\tau^2-\frac{4\|(x_{2})_{O'}\|_{2}^2}{B^2}}{\tau^2-\frac{4\|(x_{1})_{O'}\|_{2}^2}{B^2}}\right)\right|+\left|\log\left(\frac{\Delta - \frac{(\frac{B\tau}{2}-\frac{2\|x_{2}\|_{2}^2}{B\tau})^2}{\tau-\frac{4\|(x_{2})_{O'}\|_{2}^2}{\tau B^2}}}{\Delta - \frac{(\frac{B\tau}{2}-\frac{2\|x_{1}\|_{2}^2}{B\tau})^2}{\tau-\frac{4\|(x_{1})_{O'}\|_{2}^2}{\tau B^2}}}\right)\right|\nonumber\\
    &\leq \left|\frac{\frac{4\|(x_{1})_{O'}\|_{2}^2}{ B^2}-\frac{4\|(x_{2})_{O'}\|_{2}^2}{ B^2}}{\tau^2-\frac{4\|x_{1}\|_{2}^2}{B^2}}\right|+\left|\frac{\frac{(\frac{B\tau}{2}-\frac{2\|x_{1}\|_{2}^2}{B\tau})^2}{\tau-\frac{4\|(x_{1})_{O'}\|_{2}^2}{\tau B^2}}-\frac{(\frac{B\tau}{2}-\frac{2\|x_{2}\|_{2}^2}{B\tau})^2}{\tau-\frac{4\|(x_{2})_{O'}\|_{2}^2}{\tau B^2}}}{\Delta - \frac{(\frac{B\tau}{2}-\frac{2\|x_{1}\|_{2}^2}{B\tau})^2}{\tau-\frac{4\|(x_{1})_{O'}\|_{2}^2}{\tau B^2}}}\right|\nonumber\\
    &\leq \frac{8}{B^2\tau^2}\left|\|(x_{1})_{O'}\|_{2}^2-\|(x_{2})_{O'}\|_{2}^2\right| + \frac{(\frac{B\tau}{2}-\frac{2\|x_{1}\|_{2}^2}{B\tau})^2\left(\frac{\frac{1}{\tau }|\|(x_{1})_{O'}\|_{2}^2-\|(x_{2})_{O'}\|_{2}^2|}{\left(\frac{B\tau}{2}-\frac{2\|(x_{2})_{O'}\|_{2}^2}{\tau B}\right)\left(\frac{B\tau}{2}-\frac{2\|(x_{1})_{O'}\|_{2}^2}{\tau B}\right)}\right)}{R^2}\nonumber\\
    &\leq 10\max\left\{\frac{1}{B^2\tau^2},\frac{1}{R^2\tau}\right\}\left|\|(x_{1})_{O'}\|_{2}^2-\|(x_{2})_{O'}\|_{2}^2\right|,
\end{align}
where the first inequality follows from the triangle inequality; the second inequality follows from the fact that $\log(1+x) \leq x$ for $x > -1$; the third and fourth inequalities follow from the fact $\|x_{1}\|_{2}\leq \frac{B\tau}{4}$ and since $\frac{B\tau}{2}-\frac{2\|x_{1}\|_{2}^2}{B\tau}\leq \frac{B}{2}\left(\tau-\frac{4\|(x_{1})_{O'}\|_{2}^2}{\tau B^2}\right)$. We now bound the trace term in \eqref{eq:KL2 miss formula}. We have that
\begin{equation*}
    \Omega_{x_{1}}^O = \Omega_{x_{2}}^O + \begin{pmatrix}
        0_{|O'|} & \frac{2\left(x_{1}-x_{2}\right)_{O'}}{B}& 0\\
       \frac{2(x_{1}-x_{2})_{O'}^T}{B} & 0 & 0\\
        0 & 0 & 0
    \end{pmatrix}.
\end{equation*}
It follows that 
\begin{align}\label{eq:KL2 trace expansion}
    \Tr((\Omega_{x_{2}}^O)^{-1}\Omega_{x_{1}}^O)& = \Tr((\Omega_{x_{2}}^O)^{-1}(\Omega_{x_{2}}^O+\Omega_{x_{1}}^O-\Omega_{x_{2}}^O))\nonumber\\
    &=|O|+1 + 4\left\{\sum_{j=1}^{|O'|}\left(\frac{\left(x_{1}-x_{2}\right)_{O'}}{B}\right)_{j}((\Omega_{x_{2}}^O)^{-1})_{j|O|} \right\}.
\end{align}
   It remains to calculate for each $j \in O', ((\Omega_{x_{2}}^O)^{-1})_{j|O|}$. From~\citet{petersen2008matrix} 9.1.3 and 3.2.4, it holds that
   \begin{align}\label{eq:KL2 inverse element decomposition}
       ((\Omega_{x_{2}}^O)^{-1})_{j|O|} &= \left(\left\{(\Sigma_{x_{2}})_{OO}-\frac{(\Sigma_{x_{2}}\beta_{x_{2}})_{O}((\Sigma_{x_{2}}\beta_{x_{2}})_{O})^T}{\Delta}\right\}^{-1}\right)_{j|O|}\nonumber\\
       &= \left(((\Sigma_{x_{2}})_{OO})^{-1}+\frac{((\Sigma_{x_{2}})_{OO})^{-1}(\Sigma_{x_{2}}\beta_{x_{2}})_{O}((\Sigma_{x_{2}}\beta_{x_{2}})_{O})^T((\Sigma_{x_{2}})_{OO})^{-1}}{\Delta-((\Sigma_{x_{2}}\beta_{x_{2}})_{O})^T((\Sigma_{x_{2}})_{OO})^{-1}(\Sigma_{x_{2}}\beta_{x_{2}})_{O}}\right)_{j|O|}.
   \end{align}
   Note that by~\citet{petersen2008matrix} 9.1.3, we have that, for some $A_{1} \in \mathbb{R}^{|O'|\times|O'|}$ and $A_{2} \in \mathbb{R}^{1\times |O'|}$,
   \begin{equation}\label{eq:Sigma OO inverse formula}
       ((\Sigma_{x_{2}})_{OO})^{-1} = \begin{pmatrix}
           A_{1} &  - \frac{2(x_{2})_{O'}}{\tau B\left(\tau-\frac{4\|(x_{2})_{O'}\|_2^2}{B^2\tau}\right)}\\
           A_{2} & \frac{1}{\tau-\frac{4\|(x_{2})_{O'}\|_2^2}{B^2\tau}}
       \end{pmatrix}.
   \end{equation}
   Since $\|x_{2}\|_{2}\leq \frac{\tau B}{4}$, it follows that, for $j \in O'$, 
   \begin{equation}\label{eq:KL2 inverse element first bound}
       |\{(\Sigma_{x_{2}})_{OO}\}^{-1}_{j|O|}| \leq \frac{4|\{(x_{2})_{O'}\}_{j}|}{B\tau^2}.
   \end{equation} Note also that 
   \begin{align}\label{eq:KL2 inverse element second bound}
       &\left|\left(\frac{((\Sigma_{x_{2}})_{OO})^{-1}(\Sigma_{x_{2}}\beta_{x_{2}})_{O}((\Sigma_{x_{2}}\beta_{x_{2}})_{O})^T((\Sigma_{x_{2}})_{OO})^{-1}}{\Delta-((\Sigma_{x_{2}}\beta_{x_{2}})_{O})^T((\Sigma_{x_{2}})_{OO})^{-1}(\Sigma_{x_{2}}\beta_{x_{2}})_{O}}\right)_{j|O|}\right|\nonumber\\
       &= \left|\frac{\{((\Sigma_{x_{2}})_{OO})^{-1}(\Sigma_{x_{2}}\beta_{x_{2}})_{O}\}_{j}\{((\Sigma_{x_{2}})_{OO})^{-1}(\Sigma_{x_{2}}\beta_{x_{2}})_{O}\}_{|O|}}{R^2+\frac{\tau B^2}{4}-\frac{\|x_{2}\|_{2}^2}{\tau}-\left(\frac{B\tau}{2}-\frac{2\|x_{2}\|_{2}^2}{B\tau}\right)^2\left(\frac{1}{\tau-\frac{4\|(x_{2})_{O'}\|_2^2}{B^2\tau}}\right)}\right|\nonumber\\
       &\leq \left|\frac{\left(\frac{B\tau}{2}-\frac{2\|x_{2}\|_{2}^2}{B\tau}\right)^2\frac{2((x_{2})_{O'})_{j}}{\tau B\left(\tau-\frac{4\|(x_{2})_{O'}\|_2^2}{B^2\tau}\right)}\frac{1}{\tau-\frac{4\|(x_{2})_{O'}\|_2^2}{B^2\tau}}}{R^2}\right|\nonumber\\
       &\leq \left|\frac{B((x_{2})_{O'})_{j}}{2R^2\tau}\right|,
   \end{align}
   where the first inequality follows from \eqref{eq:Sigma OO inverse formula} and since $\frac{B\tau}{2}-\frac{2\|x_{2}\|_{2}^2}{B\tau}\leq \frac{B}{2}\left(\tau-\frac{4\|(x_{2})_{O'}\|_{2}^2}{\tau B^2}\right)$. Combining \eqref{eq:KL2 inverse element first bound} and \eqref{eq:KL2 inverse element second bound} into \eqref{eq:KL2 inverse element decomposition} yields that
   \begin{align*}
       |(\{\Omega_{x_{2}}^{O}\}^{-1})_{j|O|}| \leq 5\max\left\{1,\frac{1}{\tau^2}\right\}\max\left\{\frac{1}{B},\frac{B}{R^2}\right\}|((x_{2})_{O'})_{j}|. 
   \end{align*}
   Plugging this bound into \eqref{eq:KL2 trace expansion}, it follows that
   \begin{align}\label{eq:KL2 miss trace bound}
       \Tr((\Omega_{x_{2}}^O)^{-1}\Omega_{x_{1}}^O) - |O|-1 \leq  20\max\left\{1,\frac{1}{\tau^2}\right\}\max\left\{\frac{1}{B^2},\frac{1}{R^2}\right\}\left\{\sum_{j=1}^{|O'|}|\left(\left(x_{1}-x_{2}\right)_{O'}\right)_{j}||\{(x_{2})_{O'}\}_{j}| \right\}.
   \end{align}
   Combining \eqref{eq:KL2 miss det bound} and \eqref{eq:KL2 miss trace bound} into \eqref{eq:KL2 miss formula} yields
\begin{align*}
    &\KL(N(0,\Omega_{x_{1}}^O),N(0,\Omega_{x_{2}}^O))\\&\leq 10\max\left\{1,\frac{1}{\tau^2}\right\}\max\left\{\frac{1}{B^2},\frac{1}{R^2}\right\}\left\{\left|\|(x_{1})_{O'}\|_{2}^2-\|(x_{2})_{O'}\|_{2}^2\right|+\sum_{j=1}^{|O'|}|\left(\left(x_{1}-x_{2}\right)_{O'}\right)_{j}||((x_{2})_{O'})_{j}| \right\}.
\end{align*}
\end{proof}

We need the following similar lemma to control the KL divergence for the unlabelled data.

\begin{lemma}\label{lemma:KL3}
Fix $B>0$ and $\lambda_{-}\leq\lambda_{+}$ with $\tau=\tau(\lambda_{-},\lambda_{+})=\frac{\lambda_{-}+\lambda_{+}}{2}$. Fix $x_{1},x_{2} \in \mathbb{R}^{p-1}$ with $\|x_{1}\|_{2}=\|x_{2}\|_{2} \leq \frac{B(\lambda_{+}-\lambda_{-})}{8}$. Consider $\Sigma_{x_{1}}$ and $\Sigma_{x_{2}}$ defined prior to Lemma \ref{lemma: KL2 miss}.
Then we have 
\begin{align*}
    \KL\left(N(0,\Sigma_{x_{1}}),N(0,\Sigma_{x_{2}})\right) &= \frac{4}{B^2\tau}\frac{x_{2}^T(x_{2}-x_{1})}{\tau-\frac{4\|x_{2}\|_{2}^2}{\tau B^2}}\\
    &\leq \frac{8x_{2}^T(x_{2}-x_{1})}{B^2\tau^2}.
\end{align*}

\end{lemma}
\begin{proof}
Using 9.1.3 from~\citet{petersen2008matrix} and the Sherman-Morrison formula, we have
\begin{align}\label{eq:Sigma x2 inverse}
    \Sigma_{x_{2}}^{-1} &= \begin{pmatrix}
        \left(\tau I_{p-1}-\frac{4x_{2}x_{2}^T}{B^2\tau}\right)^{-1} & -\frac{2x_{2}}{B\tau}\frac{1}{\tau-\frac{4\|x_{2}\|_{2}^2}{\tau B^2}} \\ -\frac{2x_{2}^T}{B\tau}\frac{1}{\tau-\frac{4\|x_{2}\|_{2}^2}{\tau B^2}} & \frac{1}{\tau-\frac{4\|x_{2}\|_{2}^2}{\tau B^2}}
    \end{pmatrix}\nonumber\\
     &= \begin{pmatrix}
        \frac{1}{\tau}I_{p-1}+\frac{4x_{2}x_{2}^T}{B^2\tau^2}\frac{1}{\tau-\frac{4\|x_{2}\|_{2}^2}{B^2\tau}} & -\frac{2x_{2}}{B\tau}\frac{1}{\tau-\frac{4\|x_{2}\|_{2}^2}{\tau B^2}} \\ -\frac{2x_{2}^T}{B\tau}\frac{1}{\tau-\frac{4\|x_{2}\|_{2}^2}{\tau B^2}} & \frac{1}{\tau-\frac{4\|x_{2}\|_{2}^2}{\tau B^2}}
    \end{pmatrix}.
\end{align}
Since 
\begin{align*}
    \Sigma_{x_{1}}-\Sigma_{x_{2}} = \begin{pmatrix}
    0_{p-1} & \frac{2(x_{1}-x_{2})}{B}\\
    \frac{2(x_{1}-x_{2})^T}{B} & 0
    \end{pmatrix},
\end{align*} 
it follows from \eqref{eq:Sigma x2 inverse} that
\begin{align*}\Tr\left(\Sigma_{x_{2}}^{-1}\Sigma_{x_{1}}\right) & =\Tr\left(\Sigma_{x_{2}}^{-1}\left(\Sigma_{x_{1}}-\Sigma_{x_{2}}+\Sigma_{x_{2}}\right)\right)\nonumber\\
        & = p+\frac{8}{B^2\tau}\frac{x_{2}^T(x_{2}-x_{1})}{\tau-\frac{4\|x_{2}\|_{2}^2}{\tau B^2}}.
\end{align*}
Using Exercise 15.13~\citep{wainwright2019high}, the previous display and the fact that $\det(\Sigma_{x_{1}})=\det(\Sigma_{x_{2}})$, we see that
\begin{align*}
    \mathrm{KL}\bigl(N(0,\Sigma_{x_{1}}),N(0,\Sigma_{x_{2}}) \bigr) &= \frac{1}{2} \biggl\{ \log \biggl( \frac{\det(\Sigma_{x_{2}})}{\det(\Sigma_{x_{1}})} \biggr)-p+\mathrm{Tr}( \Sigma_{x_{2}}^{-1} \Sigma_{x_{1}}) \biggr\}  \\
    & = \frac{4}{B^2\tau}\frac{x_{2}^T(x_{2}-x_{1})}{\tau-\frac{4\|x_{2}\|_{2}^2}{\tau B^2}}\\
    &\leq \frac{8x_{2}^T(x_{2}-x_{1})}{B^2\tau^2},
\end{align*}
where the final inequality follows since $\|x_{2}\|_{2}\leq \frac{\tau B}{4}$.
\end{proof}

We now are ready to prove the desired minimax lower bound.

\begin{proposition}\label{prop:LD OSS lower bound}
    Suppose we have data drawn from \eqref{eq:datamech1} with distribution lying in $\mathcal{P}_{OSS}^{Gauss}$ satisfying $C_{X}\geq \sqrt{\frac{8\lambda_{+}}{3}}$, $\kappa_{\epsilon}\geq 105^{\frac{1}{4}}$ and $\lambda_{-}<\lambda_{+}$. For $p \geq 2$ and provided that, for all $i \in [p]$,
    \begin{equation*}
        \min\{1,\lambda_{+}^2\}\sum_{j \in [p]\setminus\{i\}} \min\left\{\frac{B^2}{N+n_{\xi(j),\xi(i)}}, \frac{R^2}{n_{\xi(j),\xi(i)}}\right\} \leq 2B^2(\lambda_{+}-\lambda_{-})^2, 
    \end{equation*}
    then, 
    \begin{equation}\label{eq:OSS lower bound equation}
    \inf_{\hat{\beta} \in \hat{\Theta}_{OSS}} 
    \sup_{P \in \mathcal{P}_{OSS}^{Gauss}} 
    \mathbb{E} \left[ \|\hat{\beta} - \beta^*(P)\|_2^2 \right] 
    \geq 
    \frac{\min\{\lambda_{+}^{-2},1\}}{256}\max_{i \in [p]} 
     \sum_{j\in [p]\setminus \{i\}} \min\left\{\frac{B^2}{N+n_{\xi(j),\xi(i)}}, \frac{R^2}{n_{\xi(j),\xi(i)}}\right\}.
\end{equation}
\end{proposition}

\begin{proof}
We prove that the bound in the proposition holds for $i=p$. The $\max$ then holds by symmetry. We begin by constructing hypotheses in $\mathcal{P}_{OSS}^{Gauss}$. Let $\Phi = \left\{0,1\right\}^{p-1}$ and for each $\phi \in \Phi$, we define $x_{\phi}\in \mathbb{R}^{p-1}$ as
\begin{equation*}
    x_{\phi} = \begin{pmatrix}
    \delta_{1}(2\phi_{1}-1)\\
    \delta_{2}(2\phi_{2}-1)\\
    \vdots\\
    \delta_{p-1}(2\phi_{p-1}-1)
\end{pmatrix}
\end{equation*} where $\delta_{j} =\sqrt{\frac{\min\{\tau^2,1\}}{128}\min\left\{\frac{B^2}{N+n_{\xi(j),\xi(p)}}, \frac{R^2}{n_{\xi(j),\xi(p)}}\right\}} \in \mathbb{R}$. We return to the construction outlined in the paragraphs preceding Lemma \ref{lemma: KL2 miss}. For convenience, we re-index everything according to $\phi$ and not $x_{\phi}$. In particular for $\phi \in \Phi$, we let $\beta_{\phi}\equiv\beta_{x_{\phi}}=\left(- \frac{x_{\phi}}{\tau}, \frac{B}{2} \right)^T$, $\sigma_{\phi}^2=R^2, \Delta_{\phi} = \frac{\tau B^2}{4}-\frac{\|x_{\phi}\|_{2}^2}{\tau}+R^2$ and we define $\Sigma_{\phi}\in \mathbb{R}^{p\times p}$ as
\begin{equation*}
    \Sigma_{\phi} \equiv  \begin{pmatrix}
        \tau I_{p-1} & \frac{2x_{\phi}}{B}\\
        \frac{2x_{\phi}^T}{B} & \tau
    \end{pmatrix},
\end{equation*} 
and for $O \subseteq [p]$, set 
\begin{equation*}
    \Omega_{\phi}^O\equiv \begin{pmatrix}
        (\Sigma_{\phi})_{OO}& (\Sigma_{\phi})_{O}\beta_{\phi} \\
        ((\Sigma_{\phi})_{O}\beta_{\phi})^T & \Delta_{\phi}
    \end{pmatrix}.
\end{equation*}
For $\phi \in \Phi$, we let our data be distributed according to $P_{\phi}=\prod_{k=1}^{K} N(0,\Omega^{O_{k}}_{\phi})^{\otimes n_{k}}\otimes N(0,\Sigma_{\phi})^{\otimes N}$. We verify that this distribution lies in $\mathcal{P}_{OSS}^{Gauss}$. This distribution satisfies Assumption \ref{assump:small-ball} with $C=1$ and $\chi =1$ and is sub-Gaussian with parameter $C_{X}$. Additionally, $\mathbb{E}[\epsilon^8]^{\frac{1}{4}}\leq 105^\frac{1}{4}\sigma^2$. Since $\|x_{\phi}\|_{2}^2 \leq \frac{B^2(\lambda_{+}-\lambda_{-})^2}{64}$, this distribution satisfies Assumption \ref{assump:eigen} with parameters $\lambda_{+}$ and $\lambda_{-}$. Additionally, by the discussion preceding Lemma \ref{lemma: KL2 miss}, it holds that $\|\beta_{\phi}\|_{2}^2 \leq B^2$ and $\sigma_{\phi}^2 \leq R^2$. Therefore, these choices define valid hypotheses in $\mathcal{P}_{OSS}^{Gauss}$. Fix $\phi$ and $\phi'$ in $\Phi$ that only differ in the $j^{th}$ coordinate, we have that
    \begin{align*}
        \KL\left(P_{\phi},P_{\phi'}\right)&=\KL\left(\prod_{k=1}^{K} N(0,\Omega^{O_{k}}_{\phi})^{\otimes n_{k}}\otimes N(0,\Sigma_{\phi})^{\otimes N},\prod_{k=1}^{K} N(0,\Omega^{O_{k}}_{\phi'})^{\otimes n_{k}}\otimes N(0,\Sigma_{\phi'})^{\otimes N}\right)\\
        &= \sum_{k=1}^K n_k\KL\left(N(0,\Omega^{O_{k}}_{\phi}),N(0,\Omega^{O_{k}}_{\phi'})\right)\mathbbm{1}_{\left\{\{j,p\} \subseteq O_{k}\right\}}+N\times\KL\left(N(0,\Sigma_{\phi}),N(0,\Sigma_{\phi'})\right)\\
        & \leq n_{\xi(j),\xi(p)}\KL\left(N(0,\Omega_{\phi}),N(0,\Omega_{\phi'})\right) + N\times\KL\left(N(0,\Sigma_{\phi}),N(0,\Sigma_{\phi'})\right) \\
        &\leq 20n_{\xi(j),\xi(p)}\delta_{j}^2\max\left\{1,\frac{1}{\tau^2}\right\}\max\left\{\frac{1}{B^2},\frac{1}{R^2}\right\}+\frac{16N\delta_{j}^2}{B^2\tau^2} \\
        & \leq \frac{1}{2},
    \end{align*}
where the second line follows from the tensorisation of the KL divergence and the fact that $\Omega^{O_{k}}_{\beta_{\phi}}$ and $\Omega^{O_{k}}_{\beta_{\phi'}}$ differ only if $\{j,p\}\subseteq O_{k}$; the third line follows from the data processing inequality; the penultimate line follows from 
applying Lemmas \ref{lemma: KL2 miss} and \ref{lemma:KL3} and the final line follows since $\delta_{j}^2=\frac{\min\{\tau^2,1\}}{128}\min\left\{\frac{B^2}{N+n_{\xi(j),\xi(p)}}, \frac{R^2}{n_{\xi(j),\xi(p)}}\right\}$.

For $\phi,\phi' \in \Phi$, write $\phi \sim \phi'$ if $\phi$ and $\phi'$ differ in precisely one coordinate. By Assouad's lemma~\cite[Lemma 8.7]{samworth2024statistics} and Pinsker's inequality, we have that
\begin{align*}
     \inf_{\hat{\beta}\in \hat{\Theta}_{OSS}}\sup_{P \in \mathcal{P}_{OSS}^{Gauss}}\mathbb{E}\left[\|\hat{\beta}-\beta^*(P)\|^2_{2}\right] 
     &\geq \frac{1}{4}\left(1-\max_{\substack{\phi,\phi' \in \Phi\\\phi \sim \phi'}}\TV(P_{\phi},P_{\phi'})\right)\sum_{j=1}^{p-1}\frac{4\delta_{j}^2}{\tau^2}\\
     &\geq \frac{1}{4}\left(1-\max_{\substack{\phi,\phi' \in \Phi\\\phi \sim \phi'}}\sqrt{\frac{\KL(P_{\phi},P_{\phi'})}{2}}\right)\sum_{j=1}^{p-1}\frac{4\delta_{j}^2}{\tau^2}\\
     \\&\geq \frac{\min\{\tau^{-2},1\}}{256} \sum_{j=1}^{p-1} \min\left\{\frac{B^2}{N+n_{\xi(j),\xi(p)}}, \frac{R^2}{n_{\xi(j),\xi(p)}}\right\}.
\end{align*}
The supremum in \eqref{eq:OSS lower bound equation} for arbitrary $i$ follows by symmetry.
\end{proof}
\begin{proof}[Proof of Theorems~\ref{thm:LD structured lower bound} and~\ref{thm:LD unstruc lower}]
    These follow immediately from Propositions \ref{prop:LD ISS lower bound} and \ref{prop:LD OSS lower bound}.
\end{proof}

\section{High-dimensional proofs}\label{sec:High-Dimensional Proofs}
Appendix \ref{sec:High-Dimensional Proofs} is split into three parts. In Section \ref{B:High-dimensional unstructured lower bound}, we prove the lower bound in the unstructured setting (Theorem \ref{thm:balancing lower bound}). In Section \ref{B:High-Dimensional upper bound proofs}, we prove upper bounds in both the structured (Theorem \ref{thm:high dimensional rate structured}) and the unstructured setting (Theorem \ref{thm:high dimensional upper bound unstructured}). In Section \ref{B:High-dimensional structured lower bound}, we prove the lower bound in the structured setting (Theorem \ref{thm:High-dimensional lower bound}). 
\subsection{High-dimensional unstructured lower bound}\label{B:High-dimensional unstructured lower bound}
The goal of this section is to prove Theorem \ref{thm:balancing lower bound} which gives a lower bound in the unstructured setting. To do this we first derive an extension of the sparse Gilbert-Varshamov theorem (Lemma \ref{lemma:high-dimensional packing}) via the techniques of~\citet{rigollet201518}. We define $\{0,1\}_{s}^p = \{x \in \{0,1\}^p: \|x\|_{0} = s\}$
\begin{lemma}\label{lemma:high-dimensional packing}
    Let $p$ and $s$ be integers with $1 \leq s \leq \frac{p}{32}$ and let $d_{H}$ denote the Hamming distance between two binary vectors. There exists $w_{1},\ldots,w_{M}$ in $\{0,1\}_{s}^p$, where the following hold
    \begin{align*}
        d_{H}(w_{i},w_{j}) &\geq \frac{s}{2}\\
        M &\geq \left\lfloor \left(1+\frac{p}{2s}\right)^{\frac{s}{8}}\right\rfloor.
    \end{align*}
    If in addition, 
    \begin{equation}\label{eq:min dimension sparsity}
        \left(\frac{p}{2s}+1\right)^{\frac{s}{8}} \geq 2 + \frac{p^2}{2s^2}(p+2)\log(2),
    \end{equation}
    then $w_{1},\ldots,w_{M}$ can be chosen such that for all $A \subseteq [p]$,
    \begin{equation}\label{eq:HD packing new conclusion}
        \frac{1}{M}\sum_{j=1}^M \|(w_{j})_{A}\|_{0}\leq \frac{2|A|s}{p}.
    \end{equation}
\end{lemma}
The first two conditions on the packing are standard. The new element is \eqref{eq:HD packing new conclusion}. This is useful because we can choose $A$ to be the missingness patterns so that when Fano's lemma is applied the KL divergence averaged over the packing is well-behaved for each missingness pattern. \eqref{eq:min dimension sparsity} is satisfied for $s$ sufficiently large and $p$ much larger than $s$.

\begin{proof}
    We use the probabilistic method as in~\citet[Lemma 4.14]{rigollet201518} and the first part of the argument follows similarly. Let  $w_{1},\ldots,w_{M}$ be i.i.d.~and uniformly distributed over $\{0,1\}_{s}^p$. Fix an $x_{0} \in \{0,1\}_{s}^p$ and let $w$ be uniformly distributed on $\{0,1\}_{s}^p$. By a union bound and symmetry, we have that 
    \begin{align}\label{eq:sparse packing initial upper bound}
    \mathbb{P}\left(\exists w_{j} \neq w_{k} : d_{H}(w_{j},w_{k}) < \frac{s}{2}\right) &= \mathbb{P}\left(\bigcup_{\substack{\{j,k\}\subseteq [M]^2\\j \neq k}}\left\{d_{H}(w_{j},w_{k}) < \frac{s}{2}\right\}\right)\nonumber\\
    &\leq M^2\mathbb{P}\left(d_{H}(w,x_{0}) < \frac{s}{2}\right).
    \end{align}
    We generate $w$ as follows. Let $U_{1},\ldots,U_{s}$ be $s$ random variables such that $U_{1}$ is drawn uniformly at random from $[p]$ and for $i \in [s]\setminus\{1\}$, conditional on $U_{1},\ldots,U_{i-1}, $\space$ U_{i}$ is drawn randomly from $[p]\setminus\{U_{1},\ldots,U_{i-1}\}$. We then define 
    \begin{align*}
    w_{i}&=
    \begin{cases}
          1 \text{ if } i \in \{U_{1},\ldots,U_{s}\}\\
          0 \text{ otherwise.}
    \end{cases}
    \end{align*}
    For $i \in [s]$, let $Z_{i}=\mathbbm{1}_{\{U_{i} \text{ is in the support of }x_{0}\}}$. We have by construction that
    \begin{align*}
        d_{H}(w,x_{0})= 2s -2\sum_{i=1}^s Z_{i}.
    \end{align*}
     We have that $Z_{1} \sim \Ber(\frac{s}{p})$ and for $i \in [s]\setminus\{1\}$, conditionally on $Z_{1},\ldots,Z_{i-1}$, we have that $Z_{i} \sim \Ber(Q_{i})$, where
    \begin{equation}\label{eq:sparse packing Qi bounds}
        Q_{i} = \frac{s -\sum_{l=1}^{i-1}Z_{l}}{p-(i-1)} \leq \frac{s}{p-s} \leq \frac{32s}{31p},
    \end{equation}
    since $32s \leq p$. We can apply a Chernoff bound to get that for arbitrary $\mu >0$,
    \begin{equation}\label{eq:sparse packing chernoff bound}
        \mathbb{P}\left(d_{H}(w,x_{0}) < \frac{s}{2}\right) \leq \mathbb{P}\left(\sum_{i=1}^s Z_{i} > \frac{3s}{4}\right) \leq \mathbb{E}\left[\exp\left(\mu \sum_{i=1}^s Z_{i}\right)\right]\exp\left(-\frac{3\mu s}{4}\right).
    \end{equation}
    The above MGF can be controlled inductively by the tower law and \eqref{eq:sparse packing Qi bounds}. Setting $\mu = \log\left(1+\frac{31p}{32s}\right)$, we have that
    \begin{align*}
        \mathbb{E}\left[\exp\left(\mu\sum_{i=1}^sZ_{i}\right)\right] & = \mathbb{E}\left[\exp\left(\mu\sum_{i=1}^{s-1}Z_{i}\right)\mathbb{E}\left[\exp(\mu Z_{s}|Z_{1},\ldots Z_{s-1})\right]\right]\\
        &=\mathbb{E}\left[\exp\left(\mu\sum_{i=1}^{s-1}Z_{i}\right)\left(Q_{s}(e^{\mu}-1)+1\right)\right]\\
        &\leq \mathbb{E}\left[\exp\left(\mu\sum_{i=1}^{s-1}Z_{i}\right)\right]\left(\frac{32s}{31p}(e^{\mu}-1)+1\right)\\
        &\quad\vdots\\
        &\leq \left(\frac{32s}{31p}(e^{\mu}-1)+1\right)^s\\
        &\leq 2^{s}.
    \end{align*}
    Combining this and \eqref{eq:sparse packing chernoff bound} into \eqref{eq:sparse packing initial upper bound} yields that
    \begin{align*}
        \mathbb{P}\left(\exists w_{j} \neq w_{k} : d_{H}(w_{j},w_{k}) < \frac{s}{2}\right) &\leq \exp\left(2\log(M)+s\log(2)-\frac{3s}{4}\log\left(1+\frac{31p}{32s}\right)\right)\\
        &< \exp\left(2\log(M)-\frac{s}{2}\log\left(1+\frac{31p}{32s}\right)\right)\\
        & \leq 32^{-\frac{1}{4}}\\
        & < \frac{43}{100},
    \end{align*}
    where the penultimate line follows by setting $M =\lfloor \left(1+\frac{p}{2s}\right)^{\frac{s}{8}}\rfloor$, fact that $p \geq 32s$ and $s\geq 1$. This establishes that with probability at least $\frac{57}{100}$, a random packing satisfies the appropriate separation condition. 
    
    We now turn to prove the second part of the statement, under assumption \eqref{eq:min dimension sparsity}. We again use the probabilistic method and generate $w_{1},\cdots,w_{M}$ as before. 
    Fix an  $A \subseteq [p]$ with $A \neq \emptyset$ and define for each $j \in [M]$
    \begin{equation*}
        X_{j}^A = \|(w_{j})_{A}\|_{0} = \sum_{i \in A}\mathbbm{1}_{\{(w_{j})_{i}\neq 0\}}.
    \end{equation*}
    By the linearity of expectation, we have that
    \begin{equation*}
        \mathbb{E}\left[X_{j}^A\right] = \frac{|A|s}{p} \equiv \mu_{A}.
    \end{equation*}
    Additionally, $X_{j}^A \in [0,|A|]$ for each $j$. Hence, we have that
    \begin{align*}
        \mathbb{P}\left(\frac{1}{M}\sum_{j=1}^{M}X_{j}^A \geq 2 \mu_{A}\right) &\leq \exp \left(\frac{-2M\mu_{A}^2}{|A|^2}\right)\\
        &= \exp \left(\frac{-2Ms^2}{p^2}\right)  \\
        &\leq \frac{1}{2^{p+2}}.
    \end{align*}
    where the first line follows from Hoeffding's inequality and the final line follows from \eqref{eq:min dimension sparsity}.
    Via a union bound, it follows from the above display that 
    \begin{equation*}
            \mathbb{P}\left(\bigcup_{A \subseteq [p]}\left\{\frac{1}{M}\sum_{j=1}^{M}X_{j}^A\geq 2 \mu_{A}\right\}\right) \leq2^{p}\times \frac{1}{2^{p+2}} \leq \frac{1}{4}.
    \end{equation*}
    Therefore, with probability at least $1-\frac{1}{4}-\frac{43}{100}>0$, $w_{1},\ldots, w_{M}$ satisfy the conditions of the lemma. Hence, existence is proven. 
\end{proof}

\begin{proofof}{thm:balancing lower bound}
     Similar to the proof of Proposition \ref{prop:LD OSS lower bound}, this proof involves a joint packing of the space of covariance matrices alongside the space of regression coefficients. We begin by recalling the general form that the hypotheses will take. Without loss of generality, we take $j=p$ in the statement of the Theorem. We define $\lambda_{-}=\Phi_{\Sigma}$, $\lambda_{+}=\frac{3R_{X}^2}{8}$ and set $\tau=\tau(\lambda_{-},\lambda_{+})=\frac{\lambda_{-}+\lambda_{+}}{2}$. For $x \in \mathbb{R}^{p-1}$, define $\beta_{x}\in \mathbb{R}^p$ via
\begin{equation*}
    \beta_{x} = \left(- \frac{x}{\tau}, \frac{B}{2} \right)^T.
\end{equation*}
Additionally, for $x$ such that $\|x\|_2 \leq B \tau/2$, we define $\Sigma_{x}\in \mathbb{R}^{p\times p}$ as
\begin{equation*}
    \Sigma_{x} = \begin{pmatrix}
        \tau I_{p-1} & \frac{2x}{B}\\
        \frac{2x^T}{B} & \tau
    \end{pmatrix}.
\end{equation*}
 Note that $\Sigma_{x}$ has $p-2$ eigenvalues equal to $\tau$ and two final eigenvalues $\tau\pm\frac{2\|x\|_{2}}{B}$. In particular we see that $\|x\|_{2} \leq \frac{B(\lambda_{+}-\lambda_{-})}{4}$ is a sufficient condition  to ensure the eigenvalues of $\Sigma_{x}$ are bounded above by $\lambda_{+}$ and below by $\lambda_{-}$ and that $\|\beta_{x}\|_{2}^2  \leq B^2$. We set $\sigma_{x}=R$ and recall that $\Delta_{x} \equiv \beta_{x}^T\Sigma_{x}\beta_{x}+\sigma_{x}^2$ only depends on $x$ through $\|x\|_{2}$. Additionally, we define \begin{equation*}
    \Omega_{x} \equiv \begin{pmatrix}
        \Sigma_{x}& \Sigma_{x}\beta_{x} \\
        (\Sigma_{x}\beta_{x})^T & \Delta_{x}
    \end{pmatrix}=  \begin{pmatrix}
        \tau I_{p-1} & \frac{2x}{B} & 0\\
        \frac{2x^T}{B} & \tau & \frac{B\tau}{2}-\frac{2\|x\|_{2}^2}{B\tau}\\
        0 & \frac{B\tau}{2}-\frac{2\|x\|_{2}^2}{B\tau} & \Delta_{x}
    \end{pmatrix},
\end{equation*}
and for $O \subseteq [p]$, let
\begin{equation*}
    \Omega_{x}^O = \begin{pmatrix}
        (\Sigma_{x})_{OO}& (\Sigma_{x})_{O}\beta_{x} \\
        ((\Sigma_{x})_{O}\beta_{x})^T & \Delta_{x}
    \end{pmatrix}.
\end{equation*}
    We now split the proof of the theorem into two separate cases. We first assume that $s \geq 33$. Since $p-2 \geq 24(s-1)^2$, we have that
    \begin{align*}
        \left(\frac{p-2}{2(s-1)}+1\right)^{\frac{s-1}{8}} \geq\left(\frac{p-2}{2(s-1)}\right)^{4}=\frac{(p-2)^2}{16(s-1)^2}\frac{p-2}{(s-1)^2}(p-2) \geq 2+\frac{\log(2)}{2}\frac{(p-2)^2}{(s-1)^2}p.
    \end{align*}
    Therefore, for $M= \left\lfloor \left(1+\frac{p-2}{2(s-1)}\right)^{\frac{s-1}{8}}\right\rfloor$, we can construct $w_{1}, \ldots, w_{M}$ from Lemma \ref{lemma:high-dimensional packing} with $p$ and $s$ replaced by $p-2$ and $s-1$ respectively.
    We define for $j \in [M]$, $x_{j} = (\delta w_{j},0)\in \mathbb{R}^{p-1}$ where $\delta=\sqrt{\frac{\min\left\{1,\tau^2\right\}\min\left\{R^2,B^2\right\}\log(M)}{256s\left(C_{\rho}\rho^2n_{\mathcal{L}}+N\right)}}$. For $i \in [M]$, by \eqref{eq: balancing minimum sample size}, we have $\|x_{i}\|_{2}^2\leq s\delta^2 \leq \frac{B^2(\lambda_{-}-\lambda_{+})^2}{64}$, so  $\|\beta_{x_{i}}\|_{2}\leq B $ and $\Phi_{\Sigma}\leq\lambda_{\min}(\Sigma_{x_{i}})\leq\lambda_{\max}(\Sigma_{x_{i}})\leq \frac{3}{8}R_{X}^2$. The lower inequality guarantees that Assumption \ref{assump:restricted eigenvalue} holds. From the upper inequality it follows that if $X \sim N(0,\Sigma_{x_{i}})$ for some $i \in [M]$, then
    \begin{align*}
        \|X\|_{\psi_{2}}  &=\sup_{u \in \mathcal{S}^{p-1}} \|\langle u , X \rangle \|_{\psi_{2}}  = \sup_{u \in \mathcal{S}^{p-1}}\inf_{c>0}\left\{c: \mathbb{E}\left[\exp{\frac{\langle \Sigma_{x_{i}}^{\frac{1}{2}}u ,\Sigma_{x_{i}}^{-\frac{1}{2}} X \rangle^2}{c^2}}\right]\leq 2\right\} \\&=\sup_{u \in \mathcal{S}^{p-1}}\inf_{c>0}\left\{c: \left(1-2\frac{u^T\Sigma_{x_{i}}u}{c^2}\right)^{-\frac{1}{2}}\leq 2\right\}\\
        & = \sqrt{\frac{8\lambda_{\max}(\Sigma_{x_{i}})}{3}} \leq R_{X}.
    \end{align*}
    Therefore, Assumption \ref{assump:sub-Gaussian distribution} is satisfied. Note also that for $i \in [M]$, $\|\beta_{x_{i}}\|_{0}\leq s$, so these hypotheses lie in $\mathcal{P}_{HD}$. Note that our choices guarantee that
    \begin{align*}
        d_{H}(w_{i},w_{j}) &\geq \frac{s-1}{2} \text{ for all } i \neq j,\\
        \|w_{j}\|_{0} &= s-1 \text{ for all } j,\\ 
         \frac{1}{M}\sum_{j=1}^M \|(w_{j})_{A}\|_{0} &\leq \frac{2|A|(s-1)}{p-2} \text{ for all } A \subseteq [p-2].
    \end{align*}
    For $O_{k} \subseteq [p]$, let $O_{k}' = O_{k}\setminus \{p\}$ and $O_{k}''=O_{k}'\setminus \{p-1\}$. Define an additional hypothesis via $x_{0}=(0,\delta\sqrt{s-1})\in \mathbb{R}^{p-1}$ and denote the corresponding distribution by $P_{0}$. We now control the KL divergence between appropriate hypotheses. Denote by $P_{r}$ the distribution $\prod_{k=1}^{K} N(0,\Omega^{O_{k}}_{x_{r}})^{\otimes n_{k}}\otimes N(0,\Sigma_{x_{r}})^{\otimes N}$. We have that for $r \in [M]$
    \begin{align*}
    &\KL\left(P_{r},P_{0}\right)\\
    &=\KL\left(\prod_{k=1}^{K} N(0,\Omega^{O_{k}}_{x_{r}})^{\otimes n_{k}}\otimes N(0,\Sigma_{x_{r}})^{\otimes N},\prod_{k=1}^{K} N(0,\Omega^{O_{k}}_{x_{0}})^{\otimes n_{k}}\otimes N(0,\Sigma_{x_{0}})^{\otimes N}\right)\\
        &= \sum_{k=1}^K n_k\KL\left(N(0,\Omega^{O_{k}}_{x_{r}}),N(0,\Omega^{O_{k}}_{x_{0}})\right)+N\times\KL\left(N(0,\Sigma_{x_{r}}),N(0,\Sigma_{x_{0}})\right)\\
        &\leq 10\max\left\{1,\frac{1}{\tau^2}\right\}\max\left\{\frac{1}{B^2},\frac{1}{R^2}\right\}\sum_{k=1}^K n_k \left|\|(x_{r})_{O'_{k}}\|_{2}^2-\|(x_{0})_{O'_{k}}\|_{2}^2\right|\mathbbm{1}_{\{p \in O_{k}\}}\\
        &\hspace{4mm}+10\max\left\{1,\frac{1}{\tau^2}\right\}\max\left\{\frac{1}{B^2},\frac{1}{R^2}\right\}\sum_{k=1}^Kn_{k}\sum_{j=1}^{|O'_{k}|}\left|\left\{\left(x_{r}-x_{0}\right)_{O'_{k}}\right\}_{j}\right||\{(x_{0})_{O'_{k}}\}_{j}| \mathbbm{1}_{\{p \in O_{k}\}}\\
        &\hspace{4mm}+ \frac{8N}{B^2\tau^2}x_{0}^T(x_{0}-x_{r})\\
        &= 10\max\left\{1,\frac{1}{\tau^2}\right\}\max\left\{\frac{1}{B^2},\frac{1}{R^2}\right\}\sum_{k=1}^K n_k \left|\|(x_{r})_{O'_{k}}\|_{2}^2-\|(x_{0})_{O'_{k}}\|_{2}^2\right|\mathbbm{1}_{\{p \in O_{k}\}}\mathbbm{1}_{\{p-1 \in O_{k}\}}\\
        &\hspace{4mm}+10\max\left\{1,\frac{1}{\tau^2}\right\}\max\left\{\frac{1}{B^2},\frac{1}{R^2}\right\}\sum_{k=1}^Kn_{k}\sum_{j=1}^{|O'_{k}|}\left|\left\{\left(x_{r}-x_{0}\right)_{O'_{k}}\right\}_{j}\right||\{(x_{0})_{O'_{k}}\}_{j}| \mathbbm{1}_{\{p \in O_{k}\}}\mathbbm{1}_{\{p-1 \in O_{k}\}}\\
        &\hspace{4mm}+10\max\left\{1,\frac{1}{\tau^2}\right\}\max\left\{\frac{1}{B^2},\frac{1}{R^2}\right\}\sum_{k=1}^K n_k \left|\|(x_{r})_{O'_{k}}\|_{2}^2-\|(x_{0})_{O'_{k}}\|_{2}^2\right|\mathbbm{1}_{\{p \in O_{k}\}}\mathbbm{1}_{\{p-1 \notin O_{k}\}}\\
        &\hspace{4mm}+10\max\left\{1,\frac{1}{\tau^2}\right\}\max\left\{\frac{1}{B^2},\frac{1}{R^2}\right\}\sum_{k=1}^Kn_{k}\sum_{j=1}^{|O'_{k}|}\left|\left\{\left(x_{r}-x_{0}\right)_{O'_{k}}\right\}_{j}\right||\{(x_{0})_{O'_{k}}\}_{j}| \mathbbm{1}_{\{p \in O_{k}\}}\mathbbm{1}_{\{p-1 \notin O_{k}\}}\\
        &\hspace{4mm}+ \frac{8N}{B^2\tau^2}\|x_{0}\|_{2}^2\\
        &\leq 10\max\left\{1,\frac{1}{\tau^2}\right\}\max\left\{\frac{1}{B^2},\frac{1}{R^2}\right\}\sum_{k=1}^K2n_{k}\delta^2s\mathbbm{1}_{\{p \in O_{k}\}}\mathbbm{1}_{\{p-1 \in O_{k}\}}\\
        &\hspace{4mm}+10\max\left\{1,\frac{1}{\tau^2}\right\}\max\left\{\frac{1}{B^2},\frac{1}{R^2}\right\}\sum_{k=1}^K n_{k} \delta^2s\mathbbm{1}_{\{p \in O_{k}\}}\mathbbm{1}_{\{p-1 \in O_{k}\}} \\
        &\hspace{4mm}+10\max\left\{1,\frac{1}{\tau^2}\right\}\max\left\{\frac{1}{B^2},\frac{1}{R^2}\right\}\sum_{k=1}^K n_k \|(x_{r})_{O'_{k}}\|_{2}^2\mathbbm{1}_{\{p \in O_{k}\}}\\
        &\hspace{4mm}+ \frac{8N}{B^2\tau^2}\|x_{0}\|_{2}^2\\
        &\leq 10\max\left\{1,\frac{1}{\tau^2}\right\}\max\left\{\frac{1}{B^2},\frac{1}{R^2}\right\}\left(3(C_{\rho}\rho^2n_{\mathcal{L}}+N)\delta^2s+\sum_{k=1}^K n_k\delta^2 \|(w_{r})_{O''_{k}}\|_{0}\mathbbm{1}_{\{p \in O_{k}\}}\right)\\
        &\hspace{4mm}+ \frac{8N\delta^2s}{B^2\tau^2},
    \end{align*}
    where the third line follows from the tensorisation of the KL divergence and the fourth line from Lemmas \ref{lemma: KL2 miss} and \ref{lemma:KL3}.
    Applying the previous display, we have that
    \begin{align*}
    \frac{1}{M}\sum_{j=1}^M\KL\left(P_{j},P_{0}\right)
    &\leq 10\delta^2\max\left\{1,\frac{1}{\tau^2}\right\} \max\left\{ \frac{1}{B^2}, \frac{1}{R^2} \right\}
    \cdot \frac{1}{M} \sum_{j=1}^M \sum_{k=1}^K
    n_k \|(w_{j})_{O''_{k}}\|_{0}
    \mathbbm{1}_{\{p \in O_k\}}\\
    &\hspace{4mm}+30\max\left\{1,\frac{1}{\tau^2}\right\}\max\left\{\frac{1}{B^2},\frac{1}{R^2}\right\}(C_{\rho}\rho^2n_{\mathcal{L}}+N)\delta^2s+\frac{8N\delta^2s}{B^2\tau^2} \\
    &\leq 10\delta^2\max\left\{1,\frac{1}{\tau^2}\right\} \max\left\{ \frac{1}{B^2}, \frac{1}{R^2} \right\}
    \cdot\sum_{k=1}^Kn_k\mathbbm{1}_{\{p \in O_k\}} \frac{1}{M} \sum_{j=1}^M 
     \|(w_{j})_{O''_{k}}\|_{0}
    \\
    &\hspace{4mm}+30\max\left\{1,\frac{1}{\tau^2}\right\}\max\left\{\frac{1}{B^2},\frac{1}{R^2}\right\}(C_{\rho}\rho^2n_{\mathcal{L}}+N)\delta^2s+\frac{8N\delta^2s}{B^2\tau^2} \\
    &\leq 10\delta^2\max\left\{1,\frac{1}{\tau^2}\right\} \max\left\{ \frac{1}{B^2}, \frac{1}{R^2} \right\}
    \cdot\sum_{k=1}^Kn_k\mathbbm{1}_{\{p \in O_k\}} \left(\frac{2|O_{k}''|s}{p-2}\right)
    \\
    &\hspace{4mm}+30\max\left\{1,\frac{1}{\tau^2}\right\}\max\left\{\frac{1}{B^2},\frac{1}{R^2}\right\}(C_{\rho}\rho^2n_{\mathcal{L}}+N)\delta^2s+\frac{8N\delta^2s}{B^2\tau^2} \\
    &= \frac{20s}{p-2}\delta^2\max\left\{1,\frac{1}{\tau^2}\right\} \max\left\{ \frac{1}{B^2}, \frac{1}{R^2} \right\}
    \cdot\sum_{k=1}^Kn_k\mathbbm{1}_{\{p \in O_k\}}\sum_{j=1}^{p-2}\mathbbm{1}_{\{j \in O_{k}\}}
    \\
    &\hspace{4mm}+30\max\left\{1,\frac{1}{\tau^2}\right\}\max\left\{\frac{1}{B^2},\frac{1}{R^2}\right\}(C_{\rho}\rho^2n_{\mathcal{L}}+N)\delta^2s+\frac{8N\delta^2s}{B^2\tau^2} \\
    &= \frac{20s}{p-2}\delta^2\max\left\{1,\frac{1}{\tau^2}\right\} \max\left\{ \frac{1}{B^2}, \frac{1}{R^2} \right\}
    \cdot\sum_{j=1}^{p-2}\sum_{k=1}^Kn_k\mathbbm{1}_{\{p \in O_k\}}\mathbbm{1}_{\{j \in O_{k}\}}
    \\
    &\hspace{4mm}+30\max\left\{1,\frac{1}{\tau^2}\right\}\max\left\{\frac{1}{B^2},\frac{1}{R^2}\right\}(C_{\rho}\rho^2n_{\mathcal{L}}+N)\delta^2s+\frac{8N\delta^2s}{B^2\tau^2} \\
    &\leq \frac{20s}{p-2}\delta^2\max\left\{1,\frac{1}{\tau^2}\right\} \max\left\{ \frac{1}{B^2}, \frac{1}{R^2} \right\}
    \cdot\sum_{j=1}^{p-2}C_{\rho}\rho^2n_{\mathcal{L}}\\
    &\hspace{4mm}+30\max\left\{1,\frac{1}{\tau^2}\right\}\max\left\{\frac{1}{B^2},\frac{1}{R^2}\right\}(C_{\rho}\rho^2n_{\mathcal{L}}+N)\delta^2s+\frac{8N\delta^2s}{B^2\tau^2} \\
    &\leq 50\max\left\{1,\frac{1}{\tau^2}\right\}\max\left\{\frac{1}{B^2},\frac{1}{R^2}\right\}(C_{\rho}\rho^2n_{\mathcal{L}}+N)\delta^2s+\frac{8N\delta^2s}{B^2\tau^2} \\
    &\leq 64 s \delta^2\max\left\{1,\frac{1}{\tau^2}\right\} \max\left\{ \frac{1}{B^2}, \frac{1}{R^2} \right\} \left(C_{\rho}\rho^2n_{\mathcal{L}}+N\right)\\
    &= \frac{\log(M)}{4},
\end{align*}
where the third inequality follows from properties of the packing and the last line follows from the fact that $\delta^2 = \frac{\min\left\{1,\tau^2\right\}\min\left\{R^2,B^2\right\}\log(M)}{256s\left(C_{\rho}\rho^2n_{\mathcal{L}}+N\right)}$. Using the last display, we have that
\begin{align}\label{eq: pre-fano control}
    \frac{\frac{1}{M}\sum_{j=1}^M\KL(P_{j},P_{0})+\log(2-\frac{1}{M})}{\log(M)} &\leq \frac{\log(M)}{4}\frac{1}{\log(M)}+\frac{\log(2)}{\log\left(\left\lfloor \left(1+\frac{p-2}{2(s-1)}\right)^{\frac{s-1}{8}}\right\rfloor\right)}\nonumber\\
    &\leq \frac{1}{4}+\frac{\log(2)}{20}\leq \frac{1}{2},
\end{align}
where the final line uses the fact that $p-1\geq 24(s-1)^2$ and $s \geq 33$. It follows that 
\begin{align}\label{eq:unstructured minimax lower bound}
     \inf_{\hat{\beta}\in\hat{\Theta}_{HD}}\sup_{P \in \mathcal{P}_{HD}}\mathbb{E}\left[\|\hat{\beta}-\beta^*\|_{2}^2\right] &\geq \frac{\delta^2}{\tau^2}\frac{s-1}{8}\left(1-\frac{\frac{1}{M}\sum_{j=1}^M\KL(P_{j},P_{0})+\log(2-\frac{1}{M})}{\log(M)}\right)\nonumber\\
     & \geq \frac{\delta^2}{\tau^2}\frac{s-1}{16}\nonumber\\
     &\geq \frac{\min\left\{1,\tau^{-2}\right\}\min\left\{R^2,B^2\right\}s\log(\frac{p-1}{2(s-1)})}{2^{17}\left(C_{\rho}\rho^2n_{\mathcal{L}}+N\right)},
\end{align}
where the first line follows from Fano's method~\cite[e.g.][Lemmas 8.1, 8.11]{samworth2024statistics}, the second line follows from \eqref{eq: pre-fano control} and the final line from the fact that $\lfloor x\rfloor \geq \frac{x}{2}$ for $x \geq 2$. 

It remains to prove the case $2\leq s \leq 32$. We do this for $s=2$ and then argue that the remaining cases follow by adjusting the universal constants. In this case, we apply the same construction as before but with different $\{w_{j}\}_{j=1}^M$. We set for $j \in [p-1]$, $w_{j}=e_{j}, x_{j} = \delta w_{j}$ and $\delta^2 = \frac{\log(p-2)\min\{1,\tau^2\}\min\{R^2,B^2\}}{256(C_{\rho}\rho^2n_{\mathcal{L}}+N)}$. Similar to before, since $\|x_{j}\|_{2}^2=\delta^2\leq \frac{B^2\left(\frac{3R_{X}^2}{8}-\Phi_{\Sigma}\right)^2}{64}$, these choices define distributions in $\mathcal{P}_{HD}$. Via the same argument as before, for $r,l \in [p-1]$, we have that 
    \begin{align*}
        &\KL\left(\prod_{k=1}^{K} N(0,\Omega^{O_{k}}_{x_{r}})^{\otimes n_{k}}\otimes N(0,\Sigma_{x_{r}})^N,\prod_{k=1}^{K} N(0,\Omega^{O_{k}}_{x_{l}})^{\otimes n_{k}}\otimes N(0,\Sigma_{x_{l}})^N\right)\\
        &= \sum_{k=1}^K n_k\KL\left(N(0,\Omega^{O_{k}}_{x_{r}}),N(0,\Omega^{O_{k}}_{x_{l}})\right)+N\cdot\KL\left(N(0,\Sigma_{x_{r}}),N(0,\Sigma_{x_{l}})\right) \\
        &\leq 20\max\left\{1,\frac{1}{\tau^2}\right\}\max\left\{\frac{1}{B^2},\frac{1}{R^2}\right\}\sum_{k=1}^K n_k \delta^2d_{H}((w_{r})_{O_{k}'},(w_{l})_{O_{k}'})\mathbbm{1}_{\{p \in O_{k}\}}+\frac{8N}{B^2\tau^2}\delta^2d_{H}(w_{r},w_{l})\\
        &\leq 20\max\left\{1,\frac{1}{\tau^2}\right\}\max\left\{\frac{1}{B^2},\frac{1}{R^2}\right\}\sum_{k=1}^K n_k \delta^2\left(\mathbbm{1}_{\{r \in O_{k}\}}+\mathbbm{1}_{\{l \in O_{k}\}}\right)\mathbbm{1}_{\{p \in O_{k}\}}+\frac{16N\delta^2}{\tau^2B^2}\\
        &\leq 64\max\left\{1,\frac{1}{\tau^2}\right\}\max\left\{\frac{1}{B^2},\frac{1}{R^2}\right\}\delta^2(C_{\rho}\rho^2n_{\mathcal{L}}+N)\\
        &\leq \frac{\log(p-2)}{4}.
    \end{align*}
    It follows that, for $M=p-1$,
    \begin{align}\label{eq: pre-fano control fixed s}
    \frac{\frac{1}{M}\sum_{j=1}^M\KL(P_{j},P_{1})+\log(2-\frac{1}{M})}{\log(M)} &\leq \frac{\log(p-2)}{4\log(p-1)}+\frac{\log(2)}{\log(p-1)}\nonumber\\
    &\leq \frac{1}{4}+\frac{\log(2)}{\log(24)}\leq \frac{1}{2},
    \end{align}
    where the final line uses the fact that $p-1\geq 24(s-1)^2 = 24$. It follows that 
\begin{align}\label{eq:unstructured minimax lower bound s = 2}
     \inf_{\hat{\beta}\in\hat{\Theta}_{HD}}\sup_{P \in \mathcal{P}_{HD}}\mathbb{E}\left[\|\hat{\beta}-\beta^*\|_{2}^2\right] &\geq \frac{\delta^2}{\tau^2}\frac{1}{2}\left(1-\frac{\frac{1}{M}\sum_{j=1}^M\KL(P_{j},P_{1})+\log(2-\frac{1}{M})}{\log(M)}\right)\nonumber\\
     & \geq \frac{\delta^2}{4\tau^2}\nonumber\\
     &\geq \frac{\min\left\{1,\tau^{-2}\right\}\min\left\{R^2,B^2\right\}\log(p-2)}{2^{10}\left(C_{\rho}\rho^2n_{\mathcal{L}}+N\right)},
\end{align}
where the first line follows from Fano's method~\cite[e.g.][Lemmas 8.1, 8.11]{samworth2024statistics} and the second line follows from \eqref{eq: pre-fano control fixed s}. For $2 \leq s \leq 32$, \eqref{eq:unstructured minimax lower bound s = 2} implies the following
\begin{align}\label{eq:unstructured minimax lower bound fixed s}
     \inf_{\hat{\beta}\in\hat{\Theta}_{HD}}\sup_{P \in \mathcal{P}_{HD}}\mathbb{E}\left[\|\hat{\beta}-\beta^*\|_{2}^2\right] 
     &\geq s\frac{\min\left\{1,\tau^{-2}\right\}\min\left\{R^2,B^2\right\}\log(p-2)}{2^{15}\left(C_{\rho}\rho^2n_{\mathcal{L}}+N\right)}.
\end{align}
Combining this with \eqref{eq:unstructured minimax lower bound} yields the result. 
\end{proofof}
\subsection{High-dimensional upper bound proofs}\label{B:High-Dimensional upper bound proofs}
The goal of this section is to prove Theorems \ref{thm:high dimensional upper bound unstructured} and \ref{thm:high dimensional rate structured}. These give upper bounds for our estimator in both the unstructured setting and the blockwise-missing setting. We begin with three preliminary lemmas that apply to both of these settings (Lemmas \ref{lemma: gamma control}, \ref{lemma: hat sigma control} and \ref{lemma:restricted eigenvalue}). These lemmas may be of independent interest and may be applicable to LASSO based estimator e.g. \cite{loh2011high}. Following the structure of~\citet{bickel2009simultaneous}, we prove a master theorem (Proposition \ref{thm:HD master theorem}) that then specialises to give Theorems \ref{thm:high dimensional upper bound unstructured} and \ref{thm:high dimensional rate structured}. 

 To prove an upper bound for both the structured and unstructured settings we begin by establishing deviation inequalities for $\|\hat{\gamma}-\gamma\|_{\infty}$ and $\|(\hat{\Sigma}-\Sigma)\beta^*\|_{\infty}$, where $\hat{\gamma}$ and $\hat{\Sigma}$ are defined in equations \eqref{eq:gamma estimate definition} and \eqref{eq:covariance estimate definition}. We then show, in either case, the restricted eigenvalue conditions hold with high probability. Throughout this section $F$ and $c$ denote universal constants.
 We begin by controlling $\|\hat{\gamma}-\gamma\|_{\infty}$ for both the structured and unstructured cases.
\begin{lemma}\label{lemma: gamma control}
For both parts of this lemma, we work under Assumption \ref{assump:sub-Gaussian distribution}. 
\begin{enumerate}[label=(\roman*)]
    \item Assume that for all $l \in [L]$,
    \begin{equation*}
        h_{l} \geq \frac{2\log\left( 2L |L_{l}|\right)}{c}.
    \end{equation*}
    Then for all $\lambda$ such that 
    \begin{equation*}
        2R_{X}\geq \frac{\lambda}{\sigma+R_{X}\|\beta^*\|_{2} } \geq \sqrt{\max_{l \in [L]}\left\{\frac{8R_{X}^2\log\left(2L|L_{l}|\right)}{h_{l}c}\right\}},
    \end{equation*} 
    the following holds 
    \begin{equation*}
        \mathbb{P}\left(\|\hat{\gamma}-\gamma\|_{\infty} \geq \frac{\lambda}{2}\right) \leq \frac{1}{2L\min_{l \in [L]}\left\{|L_{l}|\right\}}.
    \end{equation*}

    \item Alternatively, fix a $\rho \in (0,1)$ such that $\min_{g \in [L]}h_{g} \geq \rho n_{\mathcal{L}}$. Then for all $\lambda$ such that 
    \begin{equation*}
        2R_{X} \geq \frac{\lambda}{\sigma+R_{X}\|\beta^*\|_{2}} \geq \sqrt{\frac{8R_{X}^2\log\left(2p\right)}{\rho n_{\mathcal{L}}c}},
    \end{equation*}
    it holds that
    \begin{equation*}
        \mathbb{P}\left(\|\hat{\gamma}-\gamma\|_{\infty} \geq \frac{\lambda}{2}\right) \leq \frac{1}{2p}.
    \end{equation*}
\end{enumerate}
\end{lemma}

\begin{proof}
    Since $X$ is sub-Gaussian with parameter $R_{X}$ and $\epsilon$ is sub-Gaussian with parameter $\sigma$, we have that
    \begin{align*}
        \|Y_{i}(X_{i})_{j}\|_{\psi_{1}} & =\|(X_{i}^T\beta^*+\epsilon_{i})(X_{i})_{j}\|_{\psi_{1}} \\
        & \leq \|X_{i}^T\beta^*\|_{\psi_{2}}\|(X_{i})_{j}\|_{\psi_{2}} + \|\epsilon_{i}\|_{\psi_{2}}\|(X_{i})_{j}\|_{\psi_{2}}  \\
        &\lesssim R_{X}^2\|\beta^*\|_{2}+\sigma R_{X},
    \end{align*}
    where the second line follows from~\citet[][Lemma 2.8.6]{vershynin2025HDP} and the final line follows from~\citet[][Proposition 2.6.6]{vershynin2025HDP}. Hence, by~\citet[][Exercise 2.44]{vershynin2025HDP}, we have that
    \begin{equation*}
        \|Y_{i}(X_{i})_{j}-\mathbb{E}\left[(X_{i})_{j}Y_{i}\right]\|_{\psi_{1}} \lesssim R_{X}^2\|\beta^*\|_{2}+\sigma R_{X}.
    \end{equation*}
    For each $j \in [p]$, 
    \begin{equation*}
        \hat{\gamma}_{j}-\gamma_{j} = \frac{1}{h_{\xi(j)}}\sum_{k\in\mathcal{I}_{\mathcal{L}}}\left((X_{k})_{j}Y_{k}-\mathbb{E}\left[(X_{k})_{j}Y_{k}\right]\right)\mathbbm{1}_{\{j \in O_{\Xi(k)}\}},
    \end{equation*}
    so for some universal $c>0$, by Bernstein's inequality~\citet[][Theorem 2.9.1]{vershynin2025HDP}, 
    \begin{align}\label{eq:hat gamma-gamma control for first prelim lemma}
        \mathbb{P}\left(|\hat{\gamma}_{j}-\gamma_{j}|\geq t\right) &\leq 2 \exp\left(-c\min\left\{\frac{t^2h_{\xi(j)}}{R_{X}^2(\sigma+R_{X}\|\beta^*\|_{2})^2},\frac{th_{\xi(j)}}{R_{X}(\sigma+R_{X}\|\beta^*\|_{2})}\right\}\right) \nonumber\\
        &\leq 2\exp \left(\frac{-ct^2h_{\xi(j)}}{R_{X}^2(\sigma+R_{X}\|\beta^*\|_{2})^2}\right),
    \end{align}
    where the final line holds provided $t\leq R_{X}(\sigma+R_{X}\|\beta^*\|_{2})$. Applying the above for $t = \frac{\lambda}{2}$ and a union bound yields 
    \begin{align}
        \mathbb{P}\left(\|\hat{\gamma}-\gamma\|_{\infty} \geq \frac{\lambda}{2}\right) & \leq 2 \sum_{j=1}^p \exp\left(\frac{-c\lambda^2h_{\xi(j)}}{4R_{X}^2(\sigma+R_{X}\|\beta^*\|_{2})^2}\right)\nonumber\\
        & = 2\sum_{l=1}^L|L_{l}|\exp\left(\frac{-c\lambda^2h_{l}}{4R_{X}^2(\sigma+R_{X}\|\beta^*\|_{2})^2}\right)\nonumber\\
        &\leq \frac{1}{2L\min_{l \in [L]}\left\{|L_{l}|\right\}}\nonumber, 
    \end{align}
    where the final line holds since
    \begin{equation*}
        \lambda \geq \sqrt{\max_{l \in [L]}\left\{\frac{8R_{X}^2(\sigma+R_{X}\|\beta^*\|_{2})^2\log\left(2L|L_{l}|\right)}{h_{l}c}\right\}}.
    \end{equation*}

    The second statement of the lemma also follows from \eqref{eq:hat gamma-gamma control for first prelim lemma}, since 
    \begin{align*}
        \mathbb{P}\left(\|\hat{\gamma}-\gamma\|_{\infty} \geq \frac{\lambda}{2}\right) & \leq 2 \sum_{j=1}^p \exp\left(\frac{-c\lambda^2h_{\xi(j)}}{4R_{X}^2(\sigma+R_{X}\|\beta^*\|_{2})^2}\right)\\
        & \leq \frac{1}{2p},
    \end{align*}
    where the final line follows from the fact that $\lambda \geq \sqrt{\frac{8R_{X}^2(\sigma+R_{X}\|\beta^*\|_{2})^2\log\left(2p\right)}{\rho n_{\mathcal{L}}c}}$ and that $\min_{g \in [L]}h_{g} \geq \rho n_{\mathcal{L}}$.
\end{proof}
\begin{lemma}\label{lemma: hat sigma control}
    For both parts of this lemma, we work under Assumption \ref{assump:sub-Gaussian distribution}. \begin{enumerate}[label=(\roman*)]
        \item For all $\lambda$ such that
        \begin{align}
        \label{eq: hat sigma control lambda size blockwise}
        2LR_{X}^2\|\beta^*\|_{2}\geq \lambda \geq \max_{h}\left\{2L\sqrt{\frac{3R_{X}^4\|\beta^*\|_{2}^2}{c(N+\min_{g}\{n_{g,h}\})}\log\left(2L|L_{h}|\right)}\right\},
        \end{align}
        it holds that
        \begin{align*}
            \mathbb{P}\left(\|(\hat{\Sigma}-\Sigma)\beta^*\|_{\infty}\geq \frac{\lambda}{2}\right) \leq  \frac{1}{4L\min_{g \in [L]}|L_{g}|}.
        \end{align*} 

        \item Alternatively, for any $\rho \in(0,1)$ such that for all  $(g,h) \in [L]^2$, $\rho^2n_{\mathcal{L}}\leq n_{g,h}$ then for all $\lambda$ such that
        \begin{equation*}
            2FR_{X}^2\|\beta^*\|_{2}\geq \lambda \geq \sqrt{\frac{8F^2R_{X}^4\|\beta^*\|_{2}^2\log(p)}{c(\rho^2n_{\mathcal{L}}+N)}},
        \end{equation*}
        we have that
        \begin{equation*}
            \mathbb{P}\left(\|(\hat{\Sigma}-\Sigma)\beta^*\|_{\infty} \geq \frac{\lambda}{2}\right) \leq \frac{2}{p}.
        \end{equation*}
    \end{enumerate}
\end{lemma}

\begin{proof}
    Fix $r \in [p]$ and $v \in \mathbb{R}^p$ and note that $e_{r}^T(\hat{\Sigma}-\Sigma)v$ is centred. For $k \in \mathcal{I}_{\mathcal{U}}$, define $O_{\Xi(k)}=[p]$. Further, we have the following decomposition
    \begin{align}
        e_{r}^T\hat{\Sigma}v & = \sum_{l\in[p]}\frac{v_{l}}{N+n_{\xi(r),\xi(l)}}\left(\sum_{k \in \mathcal{I}_{\mathcal{U}}}(X_{k})_{r}(X_{k})_{l}+\sum_{k \in \mathcal{I}_{\mathcal{L}}}(X_{k})_{r}(X_{k})_{l}\mathbbm{1}_{\left\{\{r,l\}\subseteq O_{\Xi(k)}\right\}}\right)\nonumber\\
        &= \sum_{g \in [L]} \frac{1}{n_{\xi(r),g}+N} 
        \sum_{\substack{k \in \mathcal{I}_{\mathcal{L}} \cup \mathcal{I}_{\mathcal{U}}}}
         \sum_{l \in L_{g}} 
        v_{l} (X_{k})_{r} (X_{k})_{l}\mathbbm{1}_{\{L_g \cup L_{\xi(r)} \subseteq O_{\Xi(k)}\}}\nonumber\\
        &= \sum_{g \in [L]} \frac{1}{n_{\xi(r),g}+N} 
        \sum_{\substack{k \in \mathcal{I}_{\mathcal{L}} \cup \mathcal{I}_{\mathcal{U}}}}
        \langle e_{r},X_{k}\rangle\langle v_{L_{g}},(X_{k})_{L_{g}}\rangle\mathbbm{1}_{\{L_g \cup L_{\xi(r)} \subseteq O_{\Xi(k)}\}}\label{eq:pre bernstein blocks}.
    \end{align}
    Since $X_{k}$ is sub-Gaussian with norm $R_{X}$, it holds by~\citet[][Lemma 2.8.6]{vershynin2025HDP} that
    \begin{align*}
        \|T_{k}^{g,r}\|_{\psi_{1}} &\equiv \|\langle e_{r},X_{k}\rangle\langle v_{L_{g}},(X_{k})_{L_{g}}\rangle\|_{\psi_{1}} \\
        & \leq \|\langle e_{r},X_{k}\rangle\|_{\psi_{2}}\|\langle v_{L_{g}},(X_{k})_{L_{g}}\rangle\|_{\psi_{2}}\\
        &\leq R_{X}^2\|v\|_{2}.
    \end{align*}
    We define for $r \in [p]$, $\delta_{\xi(r)} = \frac{1}{2L^2|L_{\xi(r)}|^3}$. Fix $x>0$ such that $\frac{R_{X}^4\|v\|_{2}^2}{c(N+\min_{g}\{n_{\xi(r),g}\})}\log\left(\frac{4L}{\delta_{\xi(r)}}\right) \leq x^2 \leq R_{X}^4\|v\|_{2}^2$. By Bernstein's inequality~\citet[][Theorem 2.9.1]{vershynin2025HDP}, we have that 
   \begin{align*}
    &\mathbb{P}\left(
        \left|
            \frac{1}{N + n_{\xi(r),g}} 
            \sum_{\substack{k \in \mathcal{I}_{\mathcal{L}} \cup \mathcal{I}_{\mathcal{U}}}}
            \left(T_{k}^{g,r} 
            - \mathbb{E}\left[T_{k}^{g,r}\right]\right)\mathbbm{1}_{\{L_g \cup L_{\xi(r)} \subseteq O_{\Xi(k)}\}}
        \right| 
        \geq x
    \right)\\
    &\leq 
    2 \exp\left(
        -c \cdot \min\left\{
            \frac{(N + n_{\xi(r),g})x^2}{R_{X}^4 \|v\|_2^2}, \,
            \frac{(N + n_{\xi(r),g})x}{R_{X}^2 \|v\|_2}
        \right\}
    \right)\\
    &\leq 2\exp\left(\frac{-c(N + n_{\xi(r),g})x^2}{R_{X}^4  \|v\|_2^2}\right) \\
    &\leq \frac{\delta_{\xi(r)}}{2L},
\end{align*}
where the penultimate line follows since $x \leq R_{X}^2\|v\|_{2}$ and the final line follows since $x^2 \geq \frac{R_{X}^4\|v\|_{2}^2}{c(N+\min_{g}\{n_{\xi(r),g}\})}\log\left(\frac{4L}{\delta_{\xi(r)}}\right)$.
Applying the above to \eqref{eq:pre bernstein blocks} with $x = \frac{\lambda}{2L}, v=\beta^*$ and a union bound over $g$ yields that 
\begin{align}\label{eq:blockwise covariance matrix high dimensional pre union bound}
    \mathbb{P}\left(|e_{r}^T\hat{\Sigma}\beta^*-e_{r}^T\Sigma \beta^*|\geq \frac{\lambda}{2}\right) & \leq \frac{\delta_{\xi(r)}}{2},
\end{align}
by condition \eqref{eq: hat sigma control lambda size blockwise}.
To conclude on bounding $\|\hat{\Sigma}\beta^*-\Sigma\beta^*\|_{\infty}$ we apply \eqref{eq:blockwise covariance matrix high dimensional pre union bound} and a union bound over $r$ yielding
\begin{align*}
    \mathbb{P}\left(\|\hat{\Sigma}\beta^*-\Sigma\beta^*\|_{\infty} \geq \frac{\lambda}{2}\right) &\leq \sum_{r=1}^p \frac{\delta_{\xi(r)}}{2} = \sum_{g=1}^L \frac{|L_{g}|\delta_{g}}{2}\leq \sum_{g=1}^L \frac{1}{4L^2|L_{g}|^2}\leq \frac{1}{4L\min_{g \in [L]}|L_{g}|},
\end{align*}
where the penultimate inequality follows since $\delta_{g} = \frac{1}{2L^2|L_{g}|^3}$.

To prove the second statement of the result, recall that for $r \in [p]$ we have 
    \begin{align*}
        e_{r}^T\hat{\Sigma}v - e_{r}^T\Sigma v& = \sum_{l\in[p]}\frac{v_{l}}{N+n_{\xi(r),\xi(l)}}\left(\sum_{k \in \mathcal{I}_{\mathcal{U}}}(X_{k})_{r}(X_{k})_{l}+\sum_{k \in \mathcal{I}_{\mathcal{L}}}(X_{k})_{r}(X_{k})_{l}\mathbbm{1}_{\left\{\{r,l\}\subseteq O_{\Xi(k)}\right\}}\right)- e_{r}^T\Sigma v.\nonumber\\
    \end{align*}
    Define $u \in \mathbb{R}^p$ via for $l \in [p]$, $u_{l} = \frac{v_{l}}{N+n_{\xi(r),\xi(l)}}$. Then, we have that
    \begin{align*}
        &e_{r}^T\hat{\Sigma}v - e_{r}^T\Sigma v \nonumber\\
        &=\sum_{k \in \mathcal{I}_{\mathcal{U}}} \langle X_{k},e_{r}\rangle\langle X_{k},u\rangle-\mathbb{E}\left[\langle X_{k},e_{r}\rangle\langle X_{k},u\rangle\right] \\
        &\hspace{4mm}+\sum_{\substack{k \in \mathcal{I}_{\mathcal{L}}\\r \in O_{\Xi(k)}}} \langle X_{k},e_{r}\rangle\langle (X_{k})_{O_{\Xi(k)}},u_{O_{\Xi(k)}}\rangle-\mathbb{E}\left[\langle X_{k},e_{r}\rangle\langle (X_{k})_{O_{\Xi(k)}},u_{O_{\Xi(k)}}\rangle\right]\nonumber\\
        & \equiv \sum_{k \in \mathcal{I}_{\mathcal{U}}}Z_{k}+\sum_{\substack{k \in \mathcal{I}_{\mathcal{L}}\\r \in O_{\Xi(k)}}}\tilde{Z}_{k}.
    \end{align*}
    As before, since the $X_{k}$ are sub-Gaussian, for some universal $F$~\citep[][Lemma 2.8.6, Exercise 2.44]{vershynin2025HDP} we have that
    \begin{align*}
        \|Z_{k}\|_{\psi_{1}} &\leq FR_{X}^2\|u\|_{2}\leq \frac{FR_{X}^2\|v\|_{2}}{\rho^2n_{\mathcal{L}}+N},\\
        \|\tilde{Z}_{k}\|_{\psi_{1}} &\leq FR_{X}^2\|u_{O_{\Xi(k)}}\|_{2}\leq \frac{FR_{X}^2\|v\|_{2}}{\rho^2n_{\mathcal{L}}+N}.
    \end{align*}
    Therefore, it follows that
    \begin{align*}
        \sum_{k \in \mathcal{I}_{\mathcal{U}}}\|Z_{k}\|_{\psi_{1}}^2+\sum_{\substack{k \in \mathcal{I}_{\mathcal{L}}\\r \in O_{\Xi(k)}}}\|\tilde{Z}_{k}\|_{\psi_{1}}^2&\leq NF^2R_{X}^4\|u\|_{2}^2+\sum_{\substack{k \in \mathcal{I}_{\mathcal{L}}\\r \in O_{\Xi(k)}}}F^2R_{X}^4\|u_{O_{\Xi(k)}}\|_{2}^2\\
        & = NF^2R_{X}^4\|u\|_{2}^2+F^2R_{X}^4\sum_{\substack{k \in \mathcal{I}_{\mathcal{L}}\\r \in O_{\Xi(k)}}}\sum_{l \in [p]}u_{l}^2\mathbbm{1}_{\{l \in O_{\Xi(k)}\}}\\
        & = NF^2R_{X}^4\|u\|_{2}^2+F^2R_{X}^4\sum_{l \in [p]}u_{l}^2\sum_{k \in \mathcal{I}_{\mathcal{L}}}\mathbbm{1}_{\{r,l \in O_{\Xi(k)}\}}\\
        & = NF^2R_{X}^4\|u\|_{2}^2+F^2R_{X}^4\sum_{l \in [p]}u_{l}^2n_{\xi(r),\xi(l)}\\
        & = NF^2R_{X}^4\|u\|_{2}^2+F^2R_{X}^4\sum_{l \in [p]}\frac{v_{l}^2n_{\xi(r),\xi(l)}}{(N+n_{\xi(r),\xi(l)})^2}\\
        & \leq NF^2R_{X}^4\|u\|_{2}^2+F^2R_{X}^4\sum_{l \in [p]}\frac{v_{l}^2}{(N+n_{\xi(r),\xi(l)})}\\
        &\leq \frac{2F^2R_{X}^4\|v\|_{2}^2}{\rho^2n_{\mathcal{L}}+N}.
    \end{align*}
    Adjusting $F$ and applying Bernstein's inequality~\citep[][Theorem 2.9.1]{vershynin2025HDP} yields for $0\leq x \leq FR_{X}^2\|v\|_{2}$,
    \begin{align*}
       \mathbb{P}\left(|e_{r}^T\hat{\Sigma}v - e_{r}^T\Sigma v|\geq x\right)&\leq 2\exp\left(-c\min\left\{\frac{x^2(\rho^2n_{\mathcal{L}}+N)}{F^2R_{X}^4\|v\|_{2}^2},\frac{x(\rho^2n_{\mathcal{L}}+N)}{FR_{X}^2\|v\|_{2}}\right\}\right)\\
       &\leq 2\exp\left(-\frac{cx^2(\rho^2n_{\mathcal{L}}+N)}{F^2R_{X}^4\|v\|_{2}^2}\right),
    \end{align*}
    where the last line holds since $x\leq FR_{X}^2\|v\|_{2}$. Applying the above concentration inequality with $x = \frac{\lambda}{2}$, $v =\beta^*$ for all $r \in [p]$ and a union bound yields
    \begin{align*}
        \mathbb{P}\left(\|\hat{\Sigma}\beta^*-\Sigma \beta^*\|_{\infty} \geq \frac{\lambda}{2}\right) &\leq 2p \exp\left(-\frac{c\lambda^2(\rho^2n_{\mathcal{L}}+N)}{4F^2R_{X}^4\|\beta^*\|_{2}^2}\right)\\
        & = 2 \exp\left(\log(p)-\frac{c\lambda^2(\rho^2n_{\mathcal{L}}+N)}{4F^2R_{X}^4\|\beta^*\|_{2}^2}\right)\\
        &\leq \frac{2}{p},
    \end{align*}
    where the final line holds provided $\lambda \geq \sqrt{\frac{8F^2R_{X}^4\|\beta^*\|_{2}^2\log(p)}{c(\rho^2n_{\mathcal{L}}+N)}}$.
\end{proof}

Finally, we verify that $\hat{\Sigma}$ satisfies the restricted eigenvalue condition with high probability. 

\begin{lemma}\label{lemma:restricted eigenvalue}
For both parts of this lemma, we work under Assumptions \ref{assump:sub-Gaussian distribution} and \ref{assump:restricted eigenvalue}. 
\begin{enumerate}[label=(\roman*)]
    \item Assume that for all $(g,h)\in [L]^2$ 
    \[
       n_{g,h} + N\geq  \frac{4}{c} \max\left\{ 
    \frac{64 R_{X}^4 s^2}{\Phi_{\Sigma}^2},\,
    \frac{8 R_{X}^2 s}{\Phi_{\Sigma}} 
    \right\}  \max\{\log(|L_{g}|),\log(|L_{h}|)\}.
    \]
     Then with probability at least
    \[
    1-2\sum_{g \in [L],\, h \in [L]}  
    \exp\!\left( 
    -\frac{c}{2} \min\!\left\{ 
    \frac{\Phi_{\Sigma}^2(n_{g,h} + N)}{64 R_{X}^4 s^2},\,
    \frac{\Phi_{\Sigma}( n_{g,h} + N)}{8 R_{X}^2 s} 
    \right\}
    \right),
    \]
    $\hat{\Sigma}$ satisfies $Re(s,\frac{\Phi_{\Sigma}}{2})$.
    
    \item Alternatively, assume there exists $\rho \in (0,1)$  such that 
    for all $(g,h) \in [L]^2$, $ n_{g,h} \geq \rho^2n_{\mathcal{L}}$. Then provided 
    \begin{equation}\label{eq:restricted eigen unstructured min sample size}
       \rho^2 n_{\mathcal{L}}+N \geq  \frac{8}{c} \max\left\{ 
    \frac{64 R_{X}^4 s^2}{\Phi_{\Sigma}^2},\,
    \frac{8 R_{X}^2 s}{\Phi_{\Sigma}} 
    \right\}\log(p),
    \end{equation}
    with probability at least $1-\frac{2}{p}$,
    $\hat{\Sigma}$ satisfies $Re(s,\frac{\Phi_{\Sigma}}{2})$.
\end{enumerate}
\end{lemma}

\begin{proof}
    First, fix a $B \in \mathbb{R}^{p\times p}$ such that $\phi^2(B,s) > 0$, $S \subset [p]$ with $|S|\leq s$ and $\delta \in \mathbb{R}^p$ with $\|\delta_{S^c}\|_{1}\leq\|\delta_{S}\|_{1}$ and $\delta \neq 0$. Recall that $M(\delta) = S \cup \{i \in S^c: |\delta_{i}| \geq |\delta_{S^c}|_{(s)}\}$. For any $A \in \mathbb{R}^{p\times p}$, we have that
    \begin{align*}
        \frac{\delta^T A \delta}{\|\delta_{M(\delta)}\|_2^2}
        &= \frac{\delta^T (B + A - B) \delta}{\|\delta_{M(\delta)}\|_2^2} \\
        &= \frac{\delta^T B \delta}{\|\delta_{M(\delta)}\|_2^2} 
           + \frac{\delta^T (A - B) \delta}{\|\delta_{M(\delta)}\|_2^2} \\
        &\geq \phi^2(B,s) - \left| \frac{\delta^T (A - B) \delta}{\|\delta_{M(\delta)}\|_2^2} \right|  \\
        &\geq \phi^2(B,s) - \frac{\|\delta\|_1^2 \cdot \|A - B\|_\infty}{\|\delta_{M(\delta)}\|_2^2} 
            \\
        &\geq \phi^2(B,s) - \frac{(\|\delta_S\|_1 + \|\delta_{S^c}\|_1)^2 \cdot \|A - B\|_\infty}{\|\delta_{M(\delta)}\|_2^2} \\
        &\geq \phi^2(B,s) - \frac{(2 \|\delta_S\|_1)^2 \cdot \|A - B\|_\infty}{\|\delta_{M(\delta)}\|_2^2} 
            \quad  \\
        &\geq \phi^2(B,s) - \frac{4 s \|\delta_S\|_2^2 \cdot \|A - B\|_\infty}{\|\delta_{M(\delta)}\|_2^2} 
            \quad  \\
        &\geq \phi^2(B,s) - 4 s \|A - B\|_\infty,
    \end{align*}
    where the third line follows from the triangle inequality; the fourth line follows from Hölder's inequality; the sixth line follows since $\|\delta_{S^c}\|_1 \leq \|\delta_S\|_1$; the penultimate line follows since $ \|\delta_S\|_1^2 \leq s \|\delta_S\|_2^2$ for $|S|\leq s$ and last line follows since $\|\delta_{M(\delta)}\|_2^2 \geq \|\delta_S\|_2^2$ because $S \subseteq M(\delta)$.
    
We apply the above result with $B = \Sigma$ and $A = \hat{\Sigma}$. If on some event $ \|\hat{\Sigma} - \Sigma \|_\infty \leq \frac{\Phi_{\Sigma}}{8s}$, then on the same event $\phi^2(A,s) \geq \frac{\Phi_{\Sigma}}{2}$.
Fixing $g,h \in [L]$, $i \in L_{g}$ and $j \in L_{h}$, we recall that 
\begin{equation*}
    \hat{\Sigma}_{ij}-\Sigma_{ij} = \frac{1}{n_{\xi(i),\xi(j)}+N}\left(\sum_{k\in \mathcal{I}_{\mathcal{U}}}(X_{k})_{i}(X_{k})_{j}+\sum_{k \in \mathcal{I}_{\mathcal{L}}}(X_{k})_{i}(X_{k})_{j}\mathbbm{1}_{\{\{i,j\}\subseteq O_{\Xi(k)}\}}\right)-\Sigma_{ij}.
\end{equation*}
We now prove the first statement of the result. Note that by~\citet[][Lemma 2.8.6]{vershynin2025HDP}, we have that
\begin{equation*}
    \|(X_{k})_{i}(X_{k})_{j}\|_{\psi_{1}} \leq \|(X_{k})_{i}\|_{\psi_{2}}\|(X_{k})_{j}\|_{\psi_{2}} \leq R_{X}^2.
\end{equation*}
Via Bernstein's inequality~\citep[][Theorem 2.9.1]{vershynin2025HDP} and the previous display, we have that for some universal constant $c$
\begin{align*}
    \mathbb{P}\left(|\hat{\Sigma}_{ij}-\Sigma_{ij}| \geq \frac{\Phi_{\Sigma}}{8s}\right) &\leq 2\exp \left(-c\min\left\{\frac{\Phi_{\Sigma}^2(n_{\xi(i),\xi(j)}+N)}{64R_{X}^4s^2},\frac{\Phi_{\Sigma}(n_{\xi(i),\xi(j)}+N)}{8R_{X}^2s}\right\}\right).
\end{align*}
By a union bound, it follows that
\begin{align*}
    &\mathbb{P}\left( \|\hat{\Sigma} - \Sigma\|_{\infty} \geq \frac{\Phi_{\Sigma}}{8s} \right) \\
    &\leq 2 \sum_{i \in [p],\, j \in [p]} 
    \exp\left( 
    -c \min\left\{ 
    \frac{\Phi_{\Sigma}^2(n_{\xi(i),\xi(j)} + N)}{64 R_{X}^4 s^2},\,
    \frac{\Phi_{\Sigma}(n_{\xi(i),\xi(j)} + N)}{8 R_{X}^2 s} 
    \right\}
    \right) \\
    &= 2 \sum_{g \in [L],\, h \in [L]} 
    |L_g|\,|L_h| 
    \exp\left( 
    -c \min\left\{ 
    \frac{\Phi_{\Sigma}^2(n_{g,h} + N)}{64 R_{X}^4 s^2},\,
    \frac{\Phi_{\Sigma}(n_{g,h} + N)}{8 R_{X}^2 s} 
    \right\}
    \right) \\
    &\leq 2\sum_{g \in [L],\, h \in [L]}  
    \exp\left( \log(|L_{g}|)+\log(|L_{h}|)
    -c \min\left\{ 
    \frac{\Phi_{\Sigma}^2(n_{g,h} + N)}{64 R_{X}^4 s^2},\,
    \frac{\Phi_{\Sigma}( n_{g,h} + N)}{8 R_{X}^2 s} 
    \right\}
    \right) \\
    &\leq 2\sum_{g \in [L],\, h \in [L]}  
    \exp\left( 
    -\frac{c}{2} \min\left\{ 
    \frac{\Phi_{\Sigma}^2(n_{g,h} + N)}{64 R_{X}^4 s^2},\,
    \frac{\Phi_{\Sigma}( n_{g,h} + N)}{8 R_{X}^2 s} 
    \right\}
    \right), 
\end{align*}
 since, for all  $(g,h) \in [L]$, we have $\quad$$
    c \min\left\{ 
    \frac{\Phi_{\Sigma}^2(n_{g,h} + N)}{64 R_{X}^4 s^2},\,
    \frac{\Phi_{\Sigma}(n_{g,h} + N)}{8 R_{X}^2 s} 
    \right\} \geq 4 \max\{\log(|L_{g}|),\log(|L_{h}|)\}$ by assumption.
The second statement of the lemma now follows immediately from the first, by applying the previous display, the fact that $L\leq p$ and \eqref{eq:restricted eigen unstructured min sample size}. 
\end{proof}

\begin{proposition}\label{thm:HD master theorem}
    Recall the definitions of $\hat{\gamma}$ and $\hat{\Sigma}$ from \eqref{eq:gamma estimate definition} and \eqref{eq:covariance estimate definition}. Let $\lambda>0$ be such that
    \begin{align*}
        \|\hat{\gamma}-\Sigma\beta^*\|_{\infty} \leq \frac{\lambda}{2} & \quad \quad \|\hat{\Sigma}\beta^*-\Sigma\beta^*\|_{\infty} \leq \frac{\lambda}{2}
    \end{align*} and assume that $\hat{\Sigma}$ satisfies $Re(s,\frac{\Phi_{\Sigma}}{2})$ and $\|\beta^*\|_{0} \leq s$. Then for $\hat{\beta}$ a solution of \eqref{eq:ModifiedDantzig}
    \begin{equation*}
        \|\hat{\beta}-\beta^*\|_{2}\leq\frac{16\lambda \sqrt{s}}{\Phi_{\Sigma}}.
    \end{equation*}
\end{proposition}
\begin{proof}
    By the triangle inequality, note that 
    \begin{align*}
        \|\hat{\Sigma}\beta^*-\hat{\gamma}\|_{\infty} \leq \|\hat{\Sigma}\beta^*-\Sigma\beta^*\|_{\infty} + \|\Sigma\beta^*-\hat{\gamma}\|_{\infty} \leq \lambda.
    \end{align*}
    Therefore,  $\beta^*$ is feasible for \eqref{eq:ModifiedDantzig}. Define $\delta = \hat{\beta}-\beta^*\in \mathbb{R}^p$. Recall that $S$ is the active set of $\beta^*$. We have that 
    \begin{align}
        \|\delta_{S^c}\|_{1} &= \|\hat{\beta}_{S^c}\|_{1} = \|\hat{\beta}\|_{1}-\|\hat{\beta}_{S}\|_{1} \leq \|\beta^*\|_{1}-\|\hat{\beta}_{S}\|_{1}  =\|\beta_{S}^*\|_{1}-\|\hat{\beta}_{S}\|_{1}\leq \|\beta_{S}^*-\hat{\beta}_{S}\|_{1}  = \|\delta_{S}\|_{1} \label{eq: cone condition},
    \end{align}
    where the first inequality follows from the optimality of $\hat{\beta}$ and the penultimate line follows from the triangle inequality. This calculation allows us to apply the restricted eigenvalue condition to $\delta$, so we have that 
    \begin{align}\label{eq:mid way though l1 bound}
        \phi^2(\hat{\Sigma},s)&\leq \frac{\delta^T\hat{\Sigma}\delta}{\|\delta_{M(\delta)}\|_{2}^2}\leq \frac{\delta^T\hat{\Sigma}\delta}{\|\delta_{S}\|_{2}^2}\leq \frac{\|\delta\|_{1}\|\hat{\Sigma}(\hat{\beta}-\beta^*)\|_{\infty}}{\|\delta_{S}\|_{2}^2}\nonumber\\
        &\leq \frac{\|\delta\|_{1}\left(\|\hat{\Sigma}\hat{\beta}-\hat{\gamma}\|_{\infty}+\|\hat{\gamma}-\gamma\|_{\infty}+\|\Sigma\beta^*-\hat{\Sigma}\beta^*\|_{\infty}\right)}{\|\delta_{S}\|_{2}^2}\nonumber\\
        &\leq \frac{2\|\delta\|_{1}\lambda}{\|\delta_{S}\|_{2}^2} \\
        &= \frac{2\left(\|\delta_{S}\|_{1}+\|\delta_{S^c}\|_{1}\right)\lambda}{\|\delta_{S}\|_{2}^2}\leq \frac{4\|\delta_{S}\|_{1}\lambda}{\|\delta_{S}\|_{2}^2}\leq \frac{4s\|\delta_{S}\|_{1}\lambda}{\|\delta_{S}\|_{1}^2} = \frac{4s\lambda}{\|\delta_{S}\|_{1}}\nonumber,
    \end{align}
     where the third inequality follows from H\"older's inequality; the fourth inequality follows from the triangle inequality; the fifth inequality follows from the feasibility of $\hat{\beta}$ and by assumption; the sixth inequality follows from \eqref{eq: cone condition} and the final inequality follows since $\|\delta_{S}\|_{1} \leq \sqrt{s}\|\delta_{S}\|_{2}$. Applying this inequality yields, 
    \begin{align}\label{eq: l1 error bound}
        \|\delta\|_{1} = \|\delta_{S}\|_{1}+\|\delta_{S^c}\|_{1}\leq 2 \|\delta_{S}\|_{1}\leq \frac{8s\lambda}{\phi^2(\hat{\Sigma},s)}\leq \frac{16s\lambda}{\Phi_{\Sigma}}.
    \end{align}
    where the final inequality holds by assumption. We can convert this to a bound on $\|\hat{\beta}-\beta^*\|_{2}$ as follows. Note that for $k \geq s$, we have that 
        $k|\delta_{S^c}|_{(k)} \leq \|\delta_{S^c}\|_{1}.$
    It follows that 
    \begin{equation*}
        \|\delta_{M(\delta)^c}\|_{2}^2 \leq \|\delta_{S^c}\|_{1}^2\sum_{k \geq s+1}\frac{1}{k^2} \leq \frac{\|\delta_{S^c}\|_{1}^2}{s}.
    \end{equation*}
    By \eqref{eq: cone condition}, we have that 
    \begin{equation*}
        \|\delta_{M(\delta)^c}\|_{2} \leq \frac{\|\delta_{S^c}\|_{1}}{\sqrt{s}}\leq \frac{\|\delta_{S}\|_{1}}{\sqrt{s}}\leq \|\delta_{S}\|_{2} \leq \|\delta_{M(\delta)}\|_{2}.
    \end{equation*}
    Applying the restricted eigenvalue condition and the previous display yields, 
    \begin{align*}
        \|\delta\|_{2} &\leq \|\delta_{M(\delta)}\|_{2}+\|\delta_{M(\delta)^c}\|_{2}\leq 2\|\delta_{M(\delta)}\|_{2}\leq 2\sqrt{\frac{\delta^T\hat{\Sigma}\delta}{\phi^2(\hat{\Sigma},s)}}\leq 2\sqrt{\frac{2\|\delta\|_{1}\|\hat{\Sigma}\delta\|_{\infty}}{\Phi_{\Sigma}}} \leq 2\sqrt{\frac{4\|\delta\|_{1}\lambda}{\Phi_{\Sigma}}}\leq \frac{16\lambda \sqrt{s}}{\Phi_{\Sigma}},
    \end{align*}
    where the third inequality follows from the restricted eigenvalue condition, the fourth inequality follows from H\"older's inequality, the penultimate inequality follows from the bound implicit in \eqref{eq:mid way though l1 bound} and the final line by \eqref{eq: l1 error bound}.
\end{proof}

\begin{proofof}{thm:high dimensional upper bound unstructured}
    We condition throughout the proof on the high probability events of Lemmas \ref{lemma: gamma control}, \ref{lemma: hat sigma control} and \ref{lemma:restricted eigenvalue}. These events occur with probability at least that stated in the Theorem provided $F$ and $A$ exceed a sufficiently large universal constant. Therefore, $\hat{\Sigma}$ satisfies $Re(s,\frac{\Phi_{\Sigma}}{2})$ and both the following hold
    \begin{align*}
        \|\hat{\gamma}-\Sigma\beta^*\|_{\infty} \leq \frac{\lambda}{2} & \quad \quad \|\hat{\Sigma}\beta^*-\Sigma\beta^*\|_{\infty} \leq \frac{\lambda}{2}.
    \end{align*}
    Applying Proposition \ref{thm:HD master theorem} yields the result. 
\end{proofof}
\begin{proofof}{thm:high dimensional rate structured}
    This follows in exactly the same way as the proof of Theorem \ref{thm:high dimensional upper bound unstructured}
\end{proofof}

\subsection{Proofs for high-dimensional structured lower bounds}\label{B:High-dimensional structured lower bound}
The aim of this section is to prove Theorem \ref{thm:High-dimensional lower bound} which gives a lower bound in the blockwise-missing setting. This will follow from Propositions \ref{prop:lower bound h_l} and \ref{prop:lower bound n_g1g2}.

We begin with a result characterising the dependence of the rate on $\{h_{l}\}_{l=1}^L$. This follows relatively simply, via similar arguments to before. 
\begin{proposition}\label{prop:lower bound h_l}
 Consider the setting \eqref{eq:datamech1} with distribution lying in $\mathcal{P}_{HD}$ and $\Phi_{\Sigma} < \frac{3R_{X}^2}{8}$. Assume that the size of the modalities $\{|L_{i}|\}_{i=1}^L$ and the sparsity $s$ satisfy the following
    \begin{equation*}
        \left\lfloor\left(1+\frac{|L_{i}|}{2s}\right)^{\frac{s}{8}}\right\rfloor \geq 16, \quad\quad 2 \leq s \leq \frac{|L_{i}|}{32}.
    \end{equation*}
    If in addition, for all  $i \in [L]$, 
    \begin{equation}\label{eq: structured modality sample size requirement}
    \frac{R^2s\log\left(1+\frac{|L_{i}|}{2s}\right)}{16B^2\Phi_{\Sigma}} \leq h_{i}, 
    \end{equation}
    then 
    \begin{equation*}
    \inf_{\hat{\beta}\in\hat{\Theta}_{HD}}\sup_{P \in \mathcal{P}_{HD}}\mathbb{E}\left[\|\hat{\beta}-\beta^*\|_{2}^2\right] \geq  \max_{l \in [L]}\left\{\frac{R^2s}{2^{10}h_{l}\Phi_{\Sigma}}\log\left(1+\frac{|L_{l}|}{2s}\right)\right\}.
\end{equation*}
\end{proposition}
\begin{proof}
We fix an $l \in [L]$ and show the inequality holds for this $l$.
Define, for a given observation pattern $O\subseteq[p]$ and parameter $\beta \in \mathbb{R}^{p}$
\begin{align*}
    \Omega_{\beta}^{O} =  \begin{pmatrix}
        \Phi_{\Sigma}I_{O} & \Phi_{\Sigma}\beta_{O}\\
        \Phi_{\Sigma}\beta_{O}^T & \Phi_{\Sigma}\|\beta\|_{2}^2+R^2
    \end{pmatrix}.
\end{align*}
Via Lemma \ref{lemma:high-dimensional packing}, we construct $w_{1}, \ldots, w_{M} \in \left\{0,1\right\}^{|L_{l}|}$ such that
    \begin{align*}
        d_{H}(w_{i},w_{j}) &\geq \frac{s}{2} \text{ for all } i \neq j,\\
        M &= \left\lfloor\left(1+\frac{|L_{l}|}{2s}\right)^{\frac{s}{8}}\right\rfloor,\\
        \|w_{j}\|_{0} &= s \text{ for all } j \in [M]. 
    \end{align*} 
We let $\delta= \sqrt{\frac{R^2\log(M)}{4s h_{l}\Phi_{\Sigma}}}$ and construct hypotheses $\beta_{1},\ldots,\beta_{M} \in \mathbb{R}^{p}$ such that for $k \in [M]$
\begin{equation*}
    (\beta_{k})_{j} = 0 \text{ for } j \notin L_{l}, \quad\quad (\beta_{k})_{L_{l}} = \delta w_{k}.
\end{equation*}
Our hypotheses define the following distribution for $j \in [M]$
\begin{equation*}
    P_{j} = \prod_{k=1}^{K} N(0,\Omega^{O_{k}}_{\beta_{j}})^{\otimes n_{k}}\otimes N(0,\Phi_{\Sigma}I)^{\otimes N}.
\end{equation*}
Identically to the proof of Theorem \ref{thm:balancing lower bound}, our choices satisfy Assumptions \ref{assump:sub-Gaussian distribution} and \ref{assump:restricted eigenvalue}. For $i \in [M]$, we have, by the condition \eqref{eq: structured modality sample size requirement} on $\{h_{i}\}_{i=1}^L$, that $\delta^2s=\|\beta_{i}\|_{2}^2\leq B^2$. Thus, these choices do define distributions in $\mathcal{P}_{HD}$. We can bound the KL divergence between $P_{j}$ and $P_{1}$ as 
\begin{align}\label{eq:blockwise-missing pre fano modality dependence}
    \KL\left(P_{j},P_{1}\right)&=\KL\left(\prod_{k=1}^KN(0,\Omega_{\beta_{j}}^{O_{k}})^{\otimes n_{k}}\otimes N(0,\Phi_{\Sigma}I)^{\otimes N},\prod_{k=1}^KN(0,\Omega_{\beta_{1}}^{O_{k}})^{\otimes n_{k}}\otimes N(0,\Phi_{\Sigma}I)^{\otimes N}\right)\nonumber\\
    &=\sum_{k=1}^Kn_{k}\KL\left(N(0,\Omega_{\beta_{j}}^{O_{k}}),N(0,\Omega_{\beta_{1}}^{O_{k}})\right)+N\KL\left(N(0,\Phi_{\Sigma}I),N(0,\Phi_{\Sigma}I)\right)\nonumber\\
    & \leq \sum_{k=1}^K\frac{\Phi_{\Sigma}n_{k}(\beta_{j}-\beta_{1})^TI_{O_{k}}^TI_{O_{k}O_{k}}^{-1}I_{O_{k}}(\beta_{j}-\beta_{1})}{2R^2}\nonumber\\
    &\leq \sum_{k=1}^Kn_{k}\frac{\Phi_{\Sigma}(\|\beta_{j}\|_{2}^2+\|\beta_{1}\|_{2}^2)\mathbbm{1}_{\{L_{l}\subseteq O_{k}\}}}{2R^2}\nonumber\\
    &\leq \frac{\Phi_{\Sigma} h_{l}\delta^2s}{R^2}\leq \frac{\log(M)}{4},
\end{align}
where the second line follows from the tensorisation of the KL divergence, the third line follows from Lemma \ref{Lemma:First KL lemma}; the penultimate line follows since $\langle \beta_{j},\beta_{1}\rangle \geq 0$ and since if $L_{l}\cap O_{k}=\emptyset$, then $(\beta_{j})_{O_{k}}=(\beta_{1})_{O_{k}}$; the final line follows since $\delta = \sqrt{\frac{R^2\log(M)}{ 4s h_{l}\Phi_{\Sigma}}}$. It follows that 
\begin{align*}
    \inf_{\hat{\beta}\in\hat{\Theta}_{HD}}\sup_{P \in \mathcal{P}_{HD}}\mathbb{E}\left[\|\hat{\beta}-\beta^*\|_{2}^2\right] &\geq \frac{s\delta^2}{8}\left(1-\frac{\frac{1}{M}\sum_{j=1}^M\KL\left(P_{j},P_{1}\right)+\log(2-\frac{1}{M})}{\log(M)}\right)\\
    & \geq \frac{s\delta^2}{16}\\
    & \geq \frac{R^2s\log\left(1+\frac{|L_{l}|}{2s}\right)}{2^{10}h_{l}\Phi_{\Sigma}},
\end{align*}
where the first line follows from Fano's method~\cite[e.g.][Lemmas 8.1, 8.11]{samworth2024statistics}; the second line follows from \eqref{eq:blockwise-missing pre fano modality dependence} and the conditions on the sizes of the modalities and sparsity; and the final line follows from the fact that $\lfloor x\rfloor \geq \frac{x}{2}$ for $x\geq 2$. 
By symmetry, it follows that 
\begin{equation*}
    \inf_{\hat{\beta}\in\hat{\Theta}_{HD}}\sup_{P \in \mathcal{P}_{HD}}\mathbb{E}\left[\|\hat{\beta}-\beta^*\|_{2}^2\right] \geq \max_{l \in [L]}\left\{\frac{R^2s}{2^{10}h_{l}\Phi_{\Sigma}}\log\left(1+\frac{|L_{l}|}{2s}\right)\right\}.
\end{equation*}
\end{proof}
Next, we derive the lower bound dependence on $\left\{n_{g_{1},g_{2}}:g_{1},g_{2} \in [L]\right\}$. 
\begin{proposition}\label{prop:lower bound n_g1g2}
    Consider the setting \eqref{eq:datamech1} with distribution lying in $\mathcal{P}_{HD}$ and $\Phi_{\Sigma} \leq \frac{3R_{X}^2}{16}$. Assume that the size of the modalities $\{|L_{i}|\}_{i=1}^L$ and the sparsity $s$ satisfy the following
    \begin{equation*}
        1 \leq s-1 \leq \min_{l \in [L]}\frac{|L_{l}|}{32} \quad\quad \min_{h \in [L]}\left\lfloor\left(1+\frac{|L_{h}|}{2(s-1)}\right)^{\frac{(s-1)}{8}}\right\rfloor \geq 16.
    \end{equation*}
    If
    \begin{align}\label{eq: blockwise-missing cross modality minimum sample size}
        \max_{g,h}\left\{\frac{16s\log\left(1+\frac{|L_{h}|}{2(s-1)}\right)\min\left\{1,\frac{R^2}{B^2}\right\}\min\left\{1,R_{X}^{-4}\right\}}{9(n_{g,h}+N)}\right\}\leq 1,
    \end{align}
    then 
    \begin{equation*}
        \inf_{\hat{\beta}\in\hat{\Theta}_{HD}}\sup_{P \in \mathcal{P}_{HD}}\mathbb{E}\left[\|\hat{\beta}-\beta^*\|_{2}^2\right] \geq \max_{h \in [L]}\left\{\frac{s\log(1+\frac{|L_{h}|}{2(s-1)})\min\left\{1,\frac{2^6}{3^2R_{X}^4}\right\}\min\{B^2,R^2\}}{2^{17}\left(\min_{g \in [L]}\{n_{g,h}+N\}\right)}\right\}.
    \end{equation*}
\end{proposition}

\begin{proof}
    We begin by fixing $g$ and $h \in [L]$. Without loss of generality, we choose to label our modalities such that $p \in L_g$ and $L_h=\{p-|L_h|,\ldots, p-1\}$. We return to the construction outlined before the proof of Lemma \ref{lemma: KL2 miss}. This involves a joint packing of the space of covariance matrices alongside the space of regression coefficients. We begin by recalling the general form that the hypotheses will take. We define $\lambda_{-}\equiv\Phi_{\Sigma}$,  $\lambda_{+}\equiv\frac{3R_{X}^2}{8}$ and set $\tau=\tau(\lambda_{-},\lambda_{+})=\frac{\lambda_{-}+\lambda_{+}}{2}$. For $x \in \mathbb{R}^{p-1}$, define $\beta_{x}\in \mathbb{R}^p$ via
\begin{equation*}
    \beta_{x} = \left(- \frac{x}{\tau}, \frac{B}{2} \right)^T.
\end{equation*}
Additionally, we define $\Sigma_{x}\in \mathbb{R}^{p\times p}$ as
\begin{equation*}
    \Sigma_{x} = \begin{pmatrix}
        \tau I_{p-1} & \frac{2x}{B}\\
        \frac{2x^T}{B} & \tau
    \end{pmatrix}.
\end{equation*}
 Note that $\Sigma_{x}$ has $p-2$ eigenvalues equal to $\tau$ and its two final eigenvalues are $\tau\pm\frac{2\|x\|_{2}}{B}$. Recall that $\|x\|_{2} \leq \frac{B(\lambda_{+}-\lambda_{-})}{4}$ is a sufficient condition to ensure that the eigenvalues of $\Sigma_{x}$ are bounded above by $\lambda_{+}$ and below by $\lambda_{-}$ and that $\|\beta_{x}\|_{2}^2  \leq B^2$. We set $\sigma_{x}=R$ and recall that $\Delta_{x} \equiv \beta_{x}^T\Sigma_{x}\beta_{x}+\sigma_{x}^2$ only depends on $x$ through $\|x\|_{2}$. Additionally, we define \begin{equation*}
    \Omega_{x} \equiv \begin{pmatrix}
        \Sigma_{x}& \Sigma_{x}\beta_{x} \\
        (\Sigma_{x}\beta_{x})^T & \Delta_{x}
    \end{pmatrix}=  \begin{pmatrix}
        \tau I_{p-1} & \frac{2x}{B} & 0\\
        \frac{2x^T}{B} & \tau & \frac{B\tau}{2}-\frac{2\|x\|_{2}^2}{B\tau}\\
        0 & \frac{B\tau}{2}-\frac{2\|x\|_{2}^2}{B\tau} & \Delta_{x}
    \end{pmatrix},
\end{equation*}
and for $O \subseteq [p]$, let
\begin{equation*}
    \Omega_{x}^O = \begin{pmatrix}
        (\Sigma_{x})_{OO}& (\Sigma_{x})_{O}\beta_{x} \\
        ((\Sigma_{x})_{O}\beta_{x})^T & \Delta_{x}
    \end{pmatrix}.
\end{equation*} Via Lemma \ref{lemma:high-dimensional packing}, we construct $w_{1}, \ldots, w_{M} \in \left\{0,1\right\}^{|L_{h}|}$ such that
    \begin{align*}
        d_{H}(w_{i},w_{j}) &\geq \frac{s-1}{2} \text{ for all } i \neq j,\\
        M &= \left\lfloor\left(1+\frac{|L_{h}|}{2(s-1)}\right)^{\frac{(s-1)}{8}}\right\rfloor,\\
        \|w_{j}\|_{0} &= s-1 \text{ for all } j \in [M]. 
    \end{align*}
    Set $\delta = \sqrt{\frac{\log(M)}{128s(N+n_{g,h})}\min\{1,\tau^2\}\min\{B^2,R^2\}}$. For $i \in [M]$, we define $x_{i}\in \mathbb{R}^{p-1}$ via 
    \begin{equation*}
        (x_{i})_{L_{h}}=\delta w_{i}, \quad (x_{i})_{L_{h}^c}=0.
    \end{equation*}
    Denote by $P_{r}$ the distribution $\prod_{k=1}^{K} N(0,\Omega^{O_{k}}_{x_{r}})^{\otimes n_{k}}\otimes N(0,\Sigma_{x_{r}})^{\otimes N}$. We check that these distributions do indeed lie in $\mathcal{P}_{HD}$. Note that by assumption \eqref{eq: blockwise-missing cross modality minimum sample size} for $i \in [M]$, $\|x_{i}\|_{2}^2=(s-1)\delta^2\leq \frac{B^2(\lambda_{+}-\lambda_{-})^2}{64}$. Therefore, by arguments analogous to those used in the proof of Theorem \ref{thm:balancing lower bound}, these define distributions in $\mathcal{P}_{HD}$. 
    Given this construction, for $r,l \in [M]$
    \begin{align*}
        &\KL\left(P_{r},P_{l}\right)=\KL\left(\prod_{k=1}^{K} N(0,\Omega^{O_{k}}_{x_{r}})^{\otimes n_{k}}\otimes N(0,\Sigma_{x_{r}})^{\otimes N},\prod_{k=1}^{K} N(0,\Omega^{O_{k}}_{x_{l}})^{\otimes n_{k}}\otimes N(0,\Sigma_{x_{l}})^{\otimes N}\right)\\
        &= \sum_{k=1}^K n_k\KL\left(N(0,\Omega^{O_{k}}_{x_{r}}),N(0,\Omega^{O_{k}}_{x_{l}})\right)+N\cdot\KL\left(N(0,\Sigma_{x_{r}}),N(0,\Sigma_{x_{l}})\right)\\
        &\leq 20n_{g,h}\max\left\{1,\frac{1}{\tau^2}\right\}\max\left\{\frac{1}{B^2},\frac{1}{R^2}\right\}\delta^2s+\frac{8Ns\delta^2}{B^2\tau^2}\\
        &\leq \frac{\log(M)}{4},
    \end{align*}
    where the second line follows from the tensorisation of the KL divergence; the third line follows from Lemmas \ref{lemma: KL2 miss} and \ref{lemma:KL3} and the fact that for $k \in [K]$ if either $L_{h} \cap O_{k} =\emptyset$ or if $p \notin O_{k}$, then $\Omega_{\beta_{x_{r}}}^{O_{k}}$ is the same for all $r \in [M]$; the final line follows from the fact that $\delta^2 = \frac{\log(M)}{128s(N+n_{g,h})}\min\{1,\tau^2\}\min\{B^2,R^2\}$.
    
    Using the last display, we have that
\begin{align}\label{eq: unstructured pre-fano control}
    \frac{\frac{1}{M}\sum_{j=1}^M\KL(P_{j},P_{M})+\log(2-\frac{1}{M})}{\log(M)} &\leq \frac{1}{4}+\frac{\log(2)}{\log(M)}\nonumber\\
    &\leq \frac{1}{4}+\frac{\log(2)}{\log(16)}\leq \frac{1}{2},
\end{align}
where the final line uses the fact that $\left\lfloor\left(1+\frac{|L_{h}|}{2(s-1)}\right)^{\frac{(s-1)}{8}}\right\rfloor \geq 16$. It follows that 
\begin{align*}
     \inf_{\hat{\beta}\in\hat{\Theta}_{HD}}\sup_{P \in \mathcal{P}_{HD}}\mathbb{E}\left[\|\hat{\beta}-\beta^*\|_{2}^2\right] &\geq \frac{\delta^2(s-1)}{8\tau^2}\left(1-\frac{\frac{1}{M}\sum_{j=1}^M\KL(P_{j},P_{M})+\log(2-\frac{1}{M})}{\log(M)}\right)\nonumber\\
     & \geq \frac{\delta^2(s-1)}{16\tau^2}\nonumber\\
     &\geq \frac{(s-1)^2\log\left(1+\frac{|L_{h}|}{2(s-1)}\right)}{2^{15}s(N+n_{g,h})}\min\{1,\tau^{-2}\}\min\{B^2,R^2\}\nonumber\\
     &\geq \frac{s\log\left(1+\frac{|L_{h}|}{2(s-1)}\right)}{2^{17}(N+n_{g,h})}\min\{1,\tau^{-2}\}\min\{B^2,R^2\}
\end{align*}
where the first line follows from Fano's method~\cite[e.g.][Lemmas 8.1, 8.11]{samworth2024statistics}; the second line follows from \eqref{eq: unstructured pre-fano control}; the third line follows from the definition of $\delta$ and the fact that $\lfloor x\rfloor\geq \frac{x}{2}$ for $x \geq 2$; the final line follows since $4(s-1)^2 \geq s^2$. Since $g$ and $h$ were arbitrary, the result is proven. 

\end{proof}

\begin{proofof}{thm:High-dimensional lower bound}
   Combining Propositions \ref{prop:lower bound h_l} and \ref{prop:lower bound n_g1g2} yields Theorem \ref{thm:High-dimensional lower bound}. The conditions only on $s$ and $\{L_{l}\}_{l \in [L]}$ in Propositions \ref{prop:lower bound h_l} and \ref{prop:lower bound n_g1g2} are implied by $2 \leq s \leq c_{1}\min_{l \in [L]}|L_{l}|$ for $c_{1}$ a sufficiently small universal constant.
\end{proofof}

\section{Additional details for Section \ref{sec:Simulations}}\label{Appendix:Simulations}

In Section \ref{sec:Comparison to existing methods}, we use an estimator based on the work of \cite{robins1994estimation}. At the population level, define $W=(X_{O},Y)$. The efficient estimator has the following estimating equation
\begin{equation*}
    \frac{n_{1}+n_{2}}{n_{1}}\sum_{i=1}^{n_{1}}X_{i}(Y_{i}-X_{i}^T\hat{\beta})-\frac{n_{2}} {n_{1}}\phi(W_{i},\hat{\beta})+\sum_{i=n_{1}+1}^{n_{1}+n_{2}}\phi(W_{i},\hat{\beta})=0,
\end{equation*}
where $\phi(W_{i},\beta)$ is an estimate of $\mathbb{E}\left[X(Y-X^T\beta)|W\right]\equiv \Phi(W,\beta)$. The Gaussian assumptions allow us to derive an explicit formula for $\Phi(W,\beta)$ in terms of the covariance of $(X,Y)$. We can then estimate the blocks of this covariance from the labelled dataset to yield a version of the estimator in \cite{robins1994estimation}. This relies heavily on the Gaussian assumptions, in contrast to \eqref{eq:OSS estimator definition initial}.

\section{Additional details for Section \ref{sec:Real World Dataset Application}}\label{Appendix:Real World Dataset Application}

We use the Kaggle version of the dataset available from~\citet{Nugent_CaliforniaHousingPrices}. After downloading the data, we created the following derived variables for each region: the average number of rooms, the average number of bedrooms, the average population and the ocean proximity. From an initial $20{,}640$ observations, we removed all 207 incomplete cases\footnote{We separately did the same analysis including the incomplete cases with nearly identical results.}. For the ocean proximity variable, we removed all $5$ samples with `island' recorded as their ocean proximity. This was done because of how infrequently it arose. We then fitted an initial linear regression using the lm function and removed the points whose Cook's distance exceeded $\frac{4}{n-p-1}$ where $n$ is the number of samples and $p$ the dimensionality. This removed $1{,}014$ points, leaving us with a complete dataset of $19{,}414$ samples. We used the lm function again to define the ground truth against which the estimator's output is compared with throughout the simulations. 

\end{document}